\documentclass[11pt, oneside]{amsart}
\usepackage{tikz}
\usepackage{pb-diagram,amssymb,epic,eepic,verbatim,braket, xcolor, mathrsfs}
\usepackage{microtype}
\usepackage{graphics,epsfig,psfrag,braket}
\usepackage{wrapfig} 
\usepackage{hyperref}
\hypersetup{colorlinks}
\usepackage{color}

\usepackage[all,cmtip]{xy}

\usepackage{soul}
\usepackage{bm}
\newcommand{\beqn}{\begin{equation*}}
\newcommand{\eeqn}{\end{equation*}}
\newcommand{\beq}{\begin{equation}}
\newcommand{\eeq}{\end{equation}}
\newcommand{\mc}{\mathcal}
\newcommand{\mb}{\mathbb}
\newcommand{\mf}{\mathfrak}

\newcommand{\us}{{\on{us}}}
\newcommand{\ov}{\overline}
\newcommand{\uds}[1]{\underline{\smash{#1}}}

\numberwithin{equation}{section}

\usepackage{amsfonts}
\DeclareMathAlphabet{\mathpgoth}{OT1}{pgoth}{m}{n}
\DeclareMathAlphabet{\mathesstixfrak}{U}{esstixfrak}{m}{n}
\DeclareMathAlphabet{\mathboondoxfrak}{U}{BOONDOX-frak}{m}{n}
\usepackage{amsmath,amssymb}
\usepackage[bbgreekl]{mathbbol}
\usepackage{yfonts,mathtools}
\definecolor{darkred}{rgb}{0.5,0,0}
\definecolor{darkgreen}{rgb}{0,0.5,0}
\definecolor{darkblue}{rgb}{0,0,0.5}
\hypersetup{ colorlinks,
linkcolor=darkblue,
filecolor=darkgreen,
urlcolor=darkred,
citecolor=darkblue }
\usepackage{tikz}
\usetikzlibrary{arrows}

\makeatletter
\newcommand\@erelb@r[1]{%
  \mathrel{\tikz[baseline=-.5ex]\draw[#1] (0,0)--(0.3,0);}
}
\newcommand{\erelbar}[1]{\@erelbar#1}
\def\@erelbar#1#2{%
  \ifcase\numexpr#1*4+#2\relax
    \@erelb@r{-}\or     
    \@erelb@r{->}\or    
    \@erelb@r{-|}\or    
    \@erelb@r{->|}\or   
    \@erelb@r{<-}\or    
    \@erelb@r{<->}\or   
    \@erelb@r{<-|}\or   
    \@erelb@r{<->}\or   
    \@erelb@r{|-}\or    
    \@erelb@r{|->}\or   
    \@erelb@r{|-|}\or   
    \@erelb@r{|<->|}\or 
    \@erelb@r{|<-}\or   
    \@erelb@r{|<->}\or  
    \@erelb@r{|<-|}\or  
    \@erelb@r{|<->|}    
  \else
    \@wrong
  \fi
}
\makeatother

\newtheorem{theorem}{Theorem}[section]
\newtheorem{hyp}[theorem]{Hypothesis}

\newtheorem{assumption}{Assumption}[section]
\newtheorem{corollary}[theorem]{Corollary}

\newtheorem{notation}[theorem]{Notation}
\newtheorem{proposition}[theorem]{Proposition}
\newtheorem{lemma}[theorem]{Lemma}
\newtheorem{lem}[theorem]{}
\theoremstyle{definition}
\newtheorem{definition}[theorem]{Definition}
\theoremstyle{remark}
\newtheorem{remark}[theorem]{Remark}
\newtheorem{example}[theorem]{Example}
\newcommand{\blem}{\begin{lem} \rm}
\newcommand{\elem}{\end{lem}}

\newcommand\M{\mathcal{M}}
\newcommand\D{\mathcal{D}}

\renewcommand\M{\mathcal{M}}
\renewcommand\D{\mathbb{D}}

\newcommand\wt{\on{wt}}

\newcommand\YY{\mathcal{Y}}

\newcommand\DD{\mathbb{D}}

\newcommand{\LL}{\mathcal{L}} 
\newcommand{\LB}{{\widehat{L}}}  
\newcommand{\WB}{{\bm L}}
\newcommand{\WK}{{\bm K}}

\newcommand{\J}{\mathcal{J}}

\newcommand{\U}{\mathcal{U}}

\newcommand{\F}{\mathcal{F}}
\newcommand{\N}{\mathbb{N}}
\newcommand{\cN}{\mathcal{N}}
\newcommand{\R}{\mathbb{R}}

\newcommand{\C}{\mathbb{C}}

\newcommand{\cS}{\mathcal{S}}
\newcommand{\cM}{\mathcal{M}}
\newcommand{\cU}{\mathcal{U}}

\newcommand{\Z}{\mathbb{Z}}
\newcommand{\Q}{\mathbb{Q}}

\renewcommand{\P}{\mathbb{P}}
\newcommand{\bb}{\mathfrak{b}}
\newcommand{\xx}{\mathpgoth{x}}

\newcommand\lie[1]{\mathfrak{#1}}

\newcommand{\g}{\lie{g}}

\newcommand{\on}{\operatorname}
\newcommand{\white}{\circ} 
\newcommand{\black}{\bullet} 
\newcommand\circt{{\includegraphics[width=.05in]{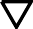}}}
\newcommand\whitet{{\includegraphics[width=.05in]{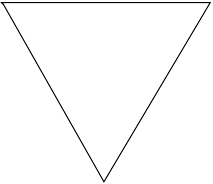}}}
\newcommand\greyt{{\includegraphics[width=.05in]{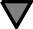}}}
\newcommand\blackt{{\includegraphics[width=.05in]{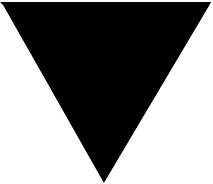}}}

\newcommand{\ainfty}{{$A_\infty$\ }}

\newcommand{\Ob}{\on{Ob}}

\newcommand{\dual}{\vee}

\newcommand{\Obj}{\on{Obj}}

\newcommand{\Edge}{\on{Edge}}
\newcommand{\Leaf}{\on{Leaf}}

\newcommand{\Lag}{\on{Lag}}

\newcommand{\Ver}{\on{Vert}}

\newcommand{\s}{\on{ps}}

\newcommand{\fin}{\on{fin}}

\newcommand{\Aut}{ \on{Aut} }

\newcommand{\Hom}{ \on{Hom}}

\renewcommand{\ker}{ \on{ker}}

\newcommand{\ann}{{\on{ann}}}

\renewcommand{\index}{ \on{index}}
\newcommand{\val}{ \on{val}}
\newcommand{\codim}{\on{codim}}

\newcommand\dirac{/\kern-1.2ex\partial} 
\newcommand\qu{/\kern-.7ex/} 
\newcommand\lqu{\backslash \kern-.7ex \backslash} 

\newcommand\dr{r_+ \kern-.7ex - \kern-.7ex r_-}
 
\newcommand{\labell}\label

\renewcommand{\d}{{\on{d}}}
\newcommand{\ol}{\overline}
\newcommand{\olp}{\ol{\partial}}

\newcommand\Lam{\Lambda}

\newcommand\eps{\epsilon}

\newcommand\om{\omega}

\newcommand{\lan}{\langle}
\newcommand{\ran}{\rangle}

\newcommand{\ti}{\tilde}

\newcommand\pt{\on{pt}}

\newcommand\cL{\mathcal{L}}

\newcommand\cT{\mathcal{T}}
\newcommand\cF{\mathcal{F}}
\newcommand\cI{\mathcal{I}}

\newcommand\Map{\on{Map}}

\newcommand\ev{\on{ev}}

\renewcommand\ul{\uds}
\newcommand\mO{\mathcal{O}}

\newcommand\E{\mathcal{E}}

\newcommand\reg{{\on{reg}}}

\newcommand\bdefn{\begin{definition}}
\newcommand\edefn{\end{definition}}
\newcommand\bea{\begin{eqnarray*}}
\newcommand\eea{\end{eqnarray*}}
\newcommand\bcv{\left[ \begin{array}{r} }
\newcommand\ecv{\end{array} \right] }

\newcommand\bma{\left[ \begin{array}{l} }
\newcommand\ema{\end{array} \right]}
\newcommand\ben{\begin{enumerate}}
\newcommand\een{\end{enumerate}}
\newcommand\bex{\begin{example}}
\newcommand\bsj{\left\{ \begin{array}{rrr} }
\newcommand\esj{\end{array} \right\}}

\newcommand\Id{\on{Id}}

\newcommand\univ{{\on{univ}}}

\newcommand\eex{\end{example}}

\newcommand\sx{*\kern-.5ex_X}
\newcommand{\Ext}{ \on{Ext}}

\newcommand{\Bl}{\on{Bl}}

\newcommand{\Fuk}{\on{Fuk}}

\newcommand{\bGamma}{\mathbb{\Gamma}}
\newcommand{\bPi}{\mathbb{\Pi}}

\newcommand\tensor{\otimes}
\def\mathunderaccent#1{\let\theaccent#1\mathpalette\putaccentunder}
\def\putaccentunder#1#2{\oalign{$#1#2$\crcr\hidewidth \vbox
to.2ex{\hbox{$#1\theaccent{}$}\vss}\hidewidth}}

\newcommand\cwo[1]{ { \color{darkred}  } }

\definecolor{darkred}{rgb}{0.5,0,0}
\definecolor{darkpurple}{rgb}{0.5,0,.5}
\definecolor{darkpink}{rgb}{0,.5,0.5}
\definecolor{cyan}{rgb}{.25,0,0.75}
\definecolor{darkgreen}{rgb}{0,0.5,0}
\definecolor{darkblue}{rgb}{0,0,0.5}

\begin{document}

\title{Fukaya categories of blowups}

\author{Sushmita Venugopalan}
\address{Institute of Mathematical Sciences, CIT Campus, Taramani,
  Chennai 600113, India.} \email{sushmita@imsc.res.in}

\author{Chris Woodward}
\address{Mathematics-Hill Center, Rutgers University, 110
  Frelinghuysen Road, Piscataway, NJ 08854-8019, U.S.A.}
\email{ctw@math.rutgers.edu}

\author{Guangbo Xu}
\address{Department of Mathematics, Texas A\&M University, College
  Station, TX 77843.  } \email{guangboxu@tamu.edu}

\thanks{This work was partially supported by NSF grants DMS 
  2105417 for Woodward and 2204321 for Xu. 
 Any opinions, findings, and conclusions or recommendations 
  expressed in this material are those of the author(s) and do not 
  necessarily reflect the views of the National Science Foundation.  }

\date{\today}

\begin{abstract} 
We compute the Fukaya category of the symplectic blowup of a compact rational symplectic manifold at a point in the   following sense: Suppose a collection of Lagrangian branes satisfy Abouzaid's criterion \cite{Abouzaid_generation} for split-generation of a bulk-deformed Fukaya category of cleanly-intersecting Lagrangian branes.  We show (Theorem \ref{gen}) that for a small blowup parameter, their inverse images in the blowup together with a collection of branes near the exceptional locus split-generate the Fukaya category of the blowup.  This categorifies a result on quantum cohomology by Bayer \cite{Bayer_2004} and is an example of a more general conjectural description of the behavior of the Fukaya category under transitions occurring in the minimal model program, namely that minimal model program transitions generate additional summands.
\end{abstract}

\maketitle

\parskip 0in \tableofcontents \parskip .1in

\section{Introduction}

In this paper we study the Fukaya category of a symplectic manifold obtained by a small symplectic blowup at a point.  In particular, we show that given a collection of branes in a given symplectic manifold satisfying Abouzaid's criterion for split-generation \cite{Abouzaid_generation}, the Fukaya category of the blowup is split-generated by the image of an embedding of the Fukaya category of the original manifold (with bulk deformation) together with a collection of branes near the exceptional locus. This result is a symplectic analog of Orlov's blowup formula \cite{Orlov_1993} that gives a semi-orthogonal decomposition of the derived category of a blowup.  We also show (conditional on a generalization of a result of Ganatra \cite{Ganatra_thesis} to the compact case described in Remark \ref{ganatrarem}) that in this non-degenerate situation the quantum cohomology is isomorphic to the 
Hochschild cohomology of the Fukaya category, c.f. Kontsevich \cite[p.18]{HMS}. 

We first give a non-technical description of our main result. Let $(X, \omega)$ be a compact symplectic manifold and $QH^\bullet (X,\bb)$ its quantum cohomology ring at a bulk deformation $\bb$.  One expects a bulk deformed Fukaya category $\Fuk(X, \bb)$ whose objects are (weakly unobstructed) Lagrangian submanifolds and whose morphisms count pseudoholomorphic disks/polygons. 
We construct a curved \ainfty category $\Fuk_{\cL}^\sim(X,\bb)$ of branes supported
on some cleanly-intersecting collection $\cL$,
and an associated flat \ainfty category
$\Fuk_{\cL}^\flat(X,\bb)$ whose objects
are branes in $\cL$ equipped with weakly bounding 
cochains. There are natural {\it open-closed} and  {\it closed-open} maps
\[
\xymatrix{ HH_\bullet (\Fuk_\cL^\flat ({X},{\bb})) \ar[r]^-{OC}  &  QH^\bullet({X}, {\bb} ) \ar[r]^-{CO} & HH^\bullet  
                           (\Fuk_\cL^\flat ({X},{\bb}))}
                           \]
between the Hochschild (co)homology of the Fukaya category and the bulk deformed quantum cohomology.  Given such, we say that a collection of Lagrangian branes ${\mf G}$ {\it generates} the bulk deformed quantum cohomology if 
\begin{equation}
OC( HH_\bullet( \Fuk^\flat_{\mf G}(X, \bb))) = QH^\bullet(X, \bb),
\end{equation}
where $\Fuk^\flat_{\mf G}(X, \bb)$ is the full sub $A_\infty$ category with objects ${\mf G}$. Recall that by Abouzaid \cite{Abouzaid_generation} and Ganatra \cite{Ganatra_thesis}, in the exact setting with trivial bulk deformation, this generation condition (with quantum cohomology replaced by symplectic cohomology) implies that the collection ${\mf G}$ split-generates the (wrapped) Fukaya category and the open-closed and closed-open maps are isomorphisms.

Our main result regards the change of the Fukaya category under a point blowup in view of the above generation criterion.  The study of the behavior of the Fukaya category under blowups was initiated by Charest-Woodward \cite{flips} for more general Minimal Model Program transitions. The blowup $\pi: \tilde X \to X$ of $X$ at a point $p$ is parametrized by $\epsilon>0$, which is the area of a complex line in the exceptional divisor $\tilde Z \subset \tilde X$. Since the rank of cohomology increases by $n-1$ where $n$ is the complex dimension of $X$, one expects new branes created by the blowup in order to generate the extra cohomology classes under the open-closed map. Indeed, a collection ${\mf E}$ of $n-1$ branes supported near the exceptional divisor, whose Floer cohomology are nontrivial, were identified in \cite{flips}. In this paper, we prove that these new branes are indeed new split-generators (as in Definition \ref{def:split-gen}) of the Fukaya category of the blowup. 

\begin{theorem}\label{gen} (proved in Section \ref{subsec:splitgen})
Let $p\in X$ be a point and $\epsilon>0$ be sufficiently small. Let $\bb$ be a bulk deformation. Suppose ${\mf G}$ is a finite collection of Lagrangian branes in $X$ disjoint from $p$, that generates the bulk deformed quantum cohomology $QH^\bullet( X, \bb + q^{-\epsilon} p)$.   
Then the collection $\pi^{-1}({\mf G}) \cup {\mf E}$ generates $QH^{\bullet}(\tilde{X}, \pi^{-1}(\bb))$ and split-generates 
the Fukaya category $\Fuk_{\ti{\cL}}^\flat (\tilde{X}, \pi^{-1}(\bb))$ of $\tilde{X}$ with bulk deformation $\pi^{-1}(\bb)$ for any cleanly-self-intersecting collection $\ti{\cL}$ containing the split-generators. Moreover, (conditional on the extension of Ganatra \cite{Ganatra_thesis} to the compact case) there are isomorphisms
\begin{equation}\label{ofrings} 
\xymatrix{ HH_\bullet (\Fuk_{\ti{\cL}}^\flat( \tilde{X}, \pi^{-1}(\bb)) ) \ar[r]^-{OC}  & QH^\bullet (\tilde{X}, \pi^{-1}(\bb)) 
\ar[r]^-{CO} & HH^\bullet (\Fuk_{\ti{\cL}}^\flat (\tilde{X},\pi^{-1}(\bb) )) }.
\end{equation}
\end{theorem}

\begin{remark}
The theorem is a special case \label{speccase} of Kontsevich's expectation that Hochschild cohomology of the Fukaya category is isomorphic to the quantum cohomology \cite[p.18]{HMS}.  K. Ono communicated to us that he also proved results in this direction, and some special cases are proved in Sanda \cite{Sanda_2017}.  Pedroza \cite{Pedroza_2019} studied the effect of blowups on the Floer cohomology of Lagrangians disjoint from the blowup point, in the monotone case.  Fukaya categories of certain blowups of toric varieties are studied from the viewpoint of the Strominger-Yau-Zaslow conjecture in Abouzaid-Auroux-Katzarkov \cite{aak}. Our theorem also slightly generalizes a result for small quantum cohomology of Bayer
\cite{Bayer_2004}, who proved that semi-simplicity of quantum cohomology is preserved under point blowups. 
Other works on the Fukaya categories of blowups can be found in P. Seidel \cite{seidel:flux} and I. Smith \cite{smith:quadrics}. \label{flux}
\end{remark}

\begin{remark} \label{ganatrarem} Ganatra has shown in the exact setting \cite{Ganatra_thesis} 
  $\bb$-deformed Hochschild homology and cohomology of
  $\Fuk_\cL(X,\bb)$ are isomorphic as vector spaces (after a degree
  shift):
  \[ HH_\bullet(\Fuk^\flat_\cL(X,\bb)) \cong HH^{\dim(X) -
    \bullet}(\Fuk^\flat_\cL(X,\bb)) \]
  and (in the compact setting here) are both isomorphic to the quantum
  cohomology $QH^{\dim(X)- \bullet}(X,\bb)$.  Ganatra's results
  \cite{Ganatra_thesis} are written for the exact, undeformed, and 
  flat case where
  the main construction is that of a category of
  Lagrangians in $X^- \times X$ including both Lagrangians of split
  form $L \times K$ as well as the diagonal.  Note that if
  $(L,b_L),(K,b_K)$ are Lagrangians equipped with weakly bounding cochains then a result of Amorim \cite{Amorim} implies that
  $L \times K$ may be equipped with a weak Maurer-Cartan solution and
  so defines an object $(L \times K, b_{L \times K})$ of
  $X^- \times X$.  One expects the potential $W(b_{L \times K})$ to
  vanish so that $CF(\Delta, L \times K)$ is a 
  projectively flat \ainfty algebra.
  Furthermore, this construction should interact as expected with the
  open-closed maps in \cite{Ganatra_thesis}.  The results on
  isomorphisms stated in \eqref{ofrings} for the compact case are
  conditional on this extension. 
\end{remark}

\begin{remark}
Theorem \ref{gen} can only be possible under suitable technical assumptions. First, to keep the technicality at a minimum, we assume that the cohomology class of the symplectic form $\omega$ is rational (and $\epsilon$ is rational). We also only consider Lagrangian branes satisfying certain rationality condition (see Definition \ref{srat} and Hypothesis \ref{lagrangian_assumption}). These assumptions allow us to apply the perturbation scheme of Cieliebak-Mohnke \cite{Cieliebak_Mohnke} to define the Fukaya category and the open-closed/closed-open maps. In addition, the Fukaya category is only defined for an arbitrary finite collection of rational Lagrangian branes with clean pairwise intersections, but not all such branes.
\end{remark}

Theorem \ref{gen} is a categorical version of a result of A. Bayer \cite{Bayer_2004}, who proves that blowup creates   algebra summands in the quantum cohomology.  In particular, if $QH^{\bullet}(X,\bb)$ is semisimple for generic $\bb$ (with positive $q$-valuation), then the same holds for slightly negative $q$-valuation (as allowed in Theorem \ref{gen}) and hence so is the quantum cohomology $QH^\bullet(\tilde{X}, \pi^{-1}(\bb))$ of the blowup.

\begin{corollary} (proved in Section \ref{subsec:splitgen}) \label{under}
There is an orthogonal decomposition of the idempotent-completed derived category %
\begin{equation} \label{dpi} 
D^\pi \Fuk_{\ti{\cL}}^\flat (\tilde{X}, \pi^{-1}(\bb) ) \cong D^\pi \Fuk_{\cL}^\flat (X, \bb + q^{-\eps} p) \oplus D^\pi \Fuk^\flat_{\mf E}(\tilde{X}, \pi^{-1}(\bb) ) \end{equation}
into the bulk-deformed Fukaya category $D^\pi \Fuk^\flat (X,\bb + q^{-\eps} p) $ of $X$ and a category of ``exceptional branes'' $D^\pi \Fuk^\flat_{\mf E}(\tilde{X}, \pi^{-1}(\bb))$. The $(n-1)$ objects in ${\mf E}$ have endomorphism algebras 
isomorphic to non-degenerate Clifford algebras. Moreover, the quantum cohomology of $\tilde{X}$ admits a ring isomorphism 
\[
QH^\bullet (\tilde{X}, \pi^{-1}(\bb)) \cong QH^\bullet ( X, \bb + q^{-\eps} p ) \oplus  QH^\bullet (\pt)^{\oplus n-1} .
\]
\end{corollary}

See also Gonz\'alez-Woodward \cite{gw:surject} and Iritani \cite[Theorem 1.3]{iritani:wall}.  

\begin{remark}  \label{clarify1}
\begin{enumerate} 
\item We expect a similar result to hold for flips, and in particular,
  blowups $\tilde{X}$ along non-trivial center $Z \subset X$.  In this
  case, if $L_Z \subset Z$ is a Lagrangian in $Z$ then we expect that
  $L_Z$ admits a ``thickening'' $\tilde{L}_Z \subset \tilde{X}$ that admits
  a collection of $\codim(Z) - 1$ local systems and bounding cochains
  defining objects in the Fukaya category of $\tilde{X}$, so that the
  Fukaya category of $\tilde{X}$ is generated by the proper transforms of
  Lagrangians in $X$ and the thickening of objects in $Z$.  At least
  in the case that the normal bundle of $Z$ admits a reduction in
  structure group to a torus, there is a strategy of proof.
  See for example Schultz \cite{schultz} for results in this
  direction. 
\item Although the decomposition of categories for a fixed bulk
  deformation is orthogonal, the decomposition of quantum cohomologies
  in Corollary \ref{under} is also expected to semi-orthogonal with
  respect to some categorical analog of the quantum connection,
  c.f. Lee-Lin-Wang \cite{lee:flips}.
\end{enumerate}
\end{remark} 

\subsection{Outline of proof}

In this section, we outline the main technical works in this paper and the strategy of proving Theorem \ref{gen} and Corollary \ref{under}. In Section \ref{section2}---Section \ref{section4} we work under the general setting, giving an independent construction of the bulk deformed Fukaya category, the open-closed/closed-open map, and a proof of Abouzaid's generation criterion, as well as  other results. In Section \ref{section5}---Section \ref{section6} we restrict to the case of the blowup; by modifying the previous constructions, we establish the correspondence between Fukaya categories before and after the blowup. 

Before we start the outline, we recall the notion of Novikov field. Let $q$ be a formal variable and let
\[
\Lambda = \left\{ \sum_{i=1}^\infty c_i q^{d_i}, \ c_i \in \C, \ d_i  \in \R, \lim_{i \to \infty} d_i =  + \infty \right\}
\]
be the {\em universal Novikov field}. The valuation by powers of $q$ is denoted
\[
\val_q: \Lambda  - \{ 0 \} \to \R, \quad \sum_{i=1}^\infty c_i q^{d_i} \mapsto \min_{c_i \neq 0 }(d_i).
\]
Denote the subsets with non-negative resp. positive valuation
\begin{equation} \label{novikov} 
\Lambda_{\ge 0} = \{ f \in \Lambda \  | \ \val_q(f) \ge 0 \}, \quad \text{resp.} \quad \Lambda_{> 0} = \{  f \in \Lambda \ | \ \val_q(f) > 0 \}.  
\end{equation}
In the Novikov ring $\Lambda_{\ge 0}$, 
the group of units is the subgroup $\Lambda^\times$ with zero $q$-valuation. \label{units}

\subsubsection{Definition of the Fukaya category}

The first technical construction in this paper is the definition of the Fukaya category using moduli spaces of treed (pearly) disks regularized via the Cieliebak-Mohnke method \cite{Cieliebak_Mohnke}. (A similar construction was also carried out by Perutz-Sheridan \cite{Perutz_Sheridan_2022}.) We allow as objects of the Fukaya category rational compact embedded Lagrangian branes. Let $X$ be a compact symplectic manifold with symplectic form $\omega$ with rational symplectic class $[\omega] \in H^2(X,\Q)$. A {\em Lagrangian brane} is a compact embedded Lagrangian $L \subset X$ equipped with a local system, 
by which we mean a flat $\Lambda^\times$-bundle 
$\LB \to L$, a spin structure, and a grading. Given a finite rational collection (see Definition \ref{srat}) of cleanly-intersecting 
submanifolds $\cL$ and a $\Lambda_{\ge 0}$-valued cycle
$\bb$ \footnote{The same set-up works for pseudocycles, but it notational much more involved. As explained in Zinger \cite{Zinger_2008} any integral homology class may be represented by a pseudocycle.} denote by $\Fuk_\cL^\sim (X,\bb)$
the Fukaya $A_\infty$ category of $X$ supported on $\cL$ with bulk deformation $\bb$. The set of objects is 
\[
\on{Ob}(\Fuk_\cL^\sim (X, \bb)) = \left\{ \LB   |\  L \in {\mc L},\ \LB \to L  \ \text{flat bundle} \right\}
\]
and morphisms are Floer cochains 
\[
\Hom ( \LB, \LB') = CF^\bullet ( \LB, \LB'), \quad \LB, \LB' \in {\rm Ob}(\Fuk_{\mc L}^\sim (X, \bb)).
\]
In the Morse model used here, Floer cochains are formal combinations
of fibers of local systems over critical points of a Morse function 
\[
F_{L, L'}:\ L' \cap L' \to \R .
\]
The composition maps
\[ m_d: \ \Hom( \LB_{d-1}, \LB_d) \otimes \ldots \otimes \Hom(\LB_0, \LB_1) \to
\Hom(L_0,L_d)[2-d], \quad d \ge 0 \]
count treed holomorphic disks $u: C \to X$ with interior markings mapping to the bulk deformation $\bb$. These are maps from combinations $C = S \cup T$ of disks $S_v \subset S$ and segments $T_e \subset T$ that satisfy Gromov's pseudoholomorphicity conditions on the disks $S_v$ and the gradient flow equation on the segments $T_e$ (see Figure \ref{treeddisk0}) for each vertex $v \in \Ver(\Gamma)$ and edge $e \in \Edge(\Gamma)$ of the combinatorial type $\Gamma$ of $C$.
\begin{wrapfigure}{r}{6cm}
\centering
\includegraphics{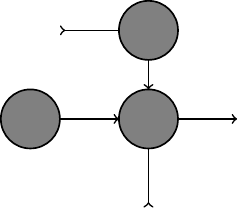}
\caption{A treed disk with two inputs and one output.}\label{treeddisk0}
\end{wrapfigure} 
The Cieliebak-Mohnke perturbation scheme \cite{Cieliebak_Mohnke} depends on choosing a Donaldson hypersurface:   a codimension two submanifold $D \subset X$ whose homology class is Poincar\'e dual to a high multiple of $[\omega]$ such that the union of Lagrangian submanifolds $L \in \LL$ is exact in the complement of $D$. For a suitably chosen almost complex structure, any holomorphic sphere 
in $X$ intersects $D$ at least three times \cite[8.17]{Cieliebak_Mohnke} and any non-constant holomorphic disk intersects $D$ at least once. For any holomorphic treed polygon $u: C \to X$ with boundary pieces labelled by $L_0, \ldots, L_d$, \label{ulphi} the intersections with the divisor $D$ then stabilize the domain $C$. These intersections ``stabilize'' all domains of treed holomorphic maps and allow us to use domain-dependent perturbations of the almost complex structure to overcome the difficulty of regularizing multiply-covered maps. In this paper, we extend the construction of \cite{Charest_Woodward_2017}\cite{flips}\cite{Woodward_Xu_2018} to regularize moduli space of treed holomorphic polygons 
\[
{\mc M}_{d, 1}(L_0, \ldots, L_d)
\]
(see Section \ref{section2}).  Counts of rigid treed holomorphic disks with boundary in the given Lagrangians define the composition maps $m_d$ and the Fukaya category $\Fuk_{\mc L}^\sim(X, \bb)$ as a (curved) strictly unital $A_\infty$ category.

%
%
%

\subsubsection{Spectral decomposition}

Starting from a curved strictly unital \ainfty category, we define flat \ainfty categories
 by restricting to particular values of the curvature. 
For any element $b \in \Hom (\LB, \LB )$ with positive $q$-valuation define 
\[ \mu (b) := \sum_{d \ge 0} m_d(\underbrace{b,\ldots, b}_d) .
\]
Following Fukaya-Oh-Ohta-Ono \cite{fooo} denote by $MC(\LB)$ the space of {\it weakly bounding cochains}, i.e., solutions to the {\em weak Maurer-Cartan equation}
\[
MC(\LB):= \{ b \in \Hom^{\on{odd}} (\LB, \LB) \ | \ \mu (b) \in \Lambda 1_{\LB} \}
\]
and
\[
MC({\mc L}):= \{ \WB = (\LB, b) \ | \ \LB \in \Ob(\Fuk_{\mc L}^\sim (X, \bb)),\ b\in MC(\LB) \}.
\]
For each $w \in \Lambda$, one denotes by $\Fuk_{\cL}(X,\bb)_w$ the flat $A_\infty$ category whose objects are 
\begin{equation}\label{flatdef}
\Ob(\Fuk_{\cL}(X,\bb)_w) = \Big\{ \WB = (\LB, b) \in MC({\mc L}) \ | \ \mu(b) = w 1_{\LB} \Big\}
\end{equation}
and whose sets of morphisms are the Floer cochain groups 
\[
{\rm Hom} (\WB, \WB'):= \Hom ( \LB, \LB') = CF^\bullet( L, L').
\]
Define the composition maps as follows: For $d \ge 1$ define
\begin{multline}\label{mddef} 
{\bm m}_d:  \Hom( \WB_{d-1}, \WB_d) \otimes \ldots \otimes  \Hom (\WB_0, \WB_1) \to \Hom ( \WB_0, \WB_d)[2-d], \\
  x_d \otimes \cdots \otimes x_1  \mapsto \sum_{k_0,\ldots, k_d} m_{d + k_0 +
    \ldots + k_d} (\underbrace{b_d,\ldots, b_d}_{k_d}, x_d, \ldots
     \underbrace{b_1,\ldots, b_1}_{k_1} ,x_1, \underbrace{b_0,\ldots, b_0}_{k_0});
  \end{multline}
when $d = 0$ define for each $\WB \in {\rm Ob} ( \Fuk_\LL^\flat (X, {\mf b})_w)$ that $m_0(1) = 0 \in \Hom( \WB, \WB)$. One checks using $b \in MC(L)$ that the \ainfty axiom holds:
\begin{equation}\label{ainftyassoc} 
0 = \sum_{\substack{i, j \geq 0\\ i + j \leq d}} (-1)^{\maltese_{1}^j} {\bm m}_{d-i+1} ( x_d, \ldots, x_{i+j+1},   {\bm m}_i (x_{j+i},\ldots,x_{j+ 1}), x_j,\ldots,x_{1})
\end{equation}
for all 
homogeneous 
$x_d \in \Hom( \WB_{d-1}, \WB_d), \ldots, x_1 \in \Hom(\WB_0, \WB_1)$ where
\begin{equation}\label{maltese}
\maltese_{l}^{k}:= \sum_{l \leq i \leq k} \Vert x_i \Vert,\ \quad \Vert x_i \Vert:= |x_i| + 1.
\end{equation}
In this way, one obtains a family of {\em flat} $A_\infty$ categories $\Fuk_{\cL}(X,\bb)_w$ indexed by values of the potential $w \in \Lambda$ and bulk deformation $\bb$. More generally, for any subset ${\mf G}$ of weakly unobstructed branes define the flat $A_\infty$ category ${\rm Fuk}_{\mf G}(X, {\mf b})_w$ in a similar way. Denote the flat \ainfty category obtained by disjoint union over all possible
curvatures
\[
\Fuk^\flat_{\mf G}(X, \bb):= \bigsqcup_{w\in \Lambda} \Fuk_{\mf G}(X, \bb)_w.
\]

\subsubsection{The open-closed and closed-open maps}

The open-closed and closed-open maps relate the Hochschild (co)homology with quantum cohomology. In the current framework, we use the Piunikhin--Salamon--Schwarz \cite{PSS} construction and the Cieliebak-Mohnke method to provide an independent construction of the bulk deformed quantum cohomology ring $QH^\bullet(X, \bb)$ (see Subsection \ref{quantumcohomology}). For any collection ${\mf L}$ of branes equipped with weakly bounding cochains, we define the open-closed map (Subsection \ref{section:OC})
\[
OC(\bb): HH_\bullet( \Fuk_{\mf L}^\flat (X, \bb)) \to QH^\bullet(X, \bb)
\]
and the closed-open map (Subsection \ref{section:CO})
\[
CO(\bb): QH^\bullet( X, \bb) \to HH^\bullet( \Fuk_{\mf L}^\flat(X, \bb))
\]
 via counts of treed holomorphic disks with one interior edge.

The specral decomposition of quantum cohomology and the spectral decomposition of the Fukaya category are related by the open-closed map. It has been known that (due to Auroux, Kontsevich, Seidel, see \cite[Section 6]{Auroux_97} and Sheridan \cite[Lemma 2.7]{Sheridan_2016}) in the monotone case, the values of the potential function correspond to eigenvalues of the quantum multiplication by the first Chern class; moreover, the open-closed map shall send the Hochschild homology of the eigen-subcategory to the corresponding generalized eigenspace. In the current situation, we prove a more general statement. Let 
\[ D_q := q \frac{d}{dq} \] 
denote the logarithmic derivative with respect to $q$ and 
define the {\em bulk-deformed
symplectic class} 
\begin{equation} \label{ombb} 
[\omega]^\bb := [\omega]  + D_q \bb  .\end{equation} 
Similarly write 
\[ \bb = \sum_i \bb_i \]
for homogeneous $\bb_i$ of degree $|\bb_i|$ and 
define the {\em bulk-deformed first Chern class} 
\begin{equation} \label{c1bb}
c_1^{\bb}(M) := c_1(M) + \sum_i  \frac{|\bb_i| - 2}{2} \bb_i .\end{equation}
\begin{theorem}\label{thm16}(proved in Section \ref{subsec:ocspect})
For any $w \in \Lambda$, the image 
\[
OC(\bb) (HH_\bullet( \Fuk_{\mf L}^\flat(X, \bb)_w)) \subset QH^\bullet( X, \bb)
\]
lies in the generalized eigenspace of the quantum multiplication by $c_1(M)^\bb$ resp.  the symplectic class $[\omega]^\bb$ corresponding to eigenvalue
$w$ resp. $D_q w$.
\end{theorem}

\subsubsection{The generation criterion}
A criterion for the split generation of the Fukaya category by a subset of branes 
is provided by results of Abouzaid \cite{Abouzaid_generation} and Ganatra
\cite{Ganatra_thesis}.

\begin{definition}
  Given a collection of objects ${\mf G}$
  let $\Fuk_{\mf G}^\flat(X,\bb)$ denote the
  sub Fukaya category with objects ${\mf G}$.
Write
\[ QH_{\mf G}(X,\bb) = (OC(\bb)) ( HH_\bullet(\Fuk_{\mf G}^\flat(X,\bb)) ) \] 
for the image of $HH_\bullet(\Fuk_{\mf G}(X,\bb))$ under the
open-closed map.  Say 
\[ QH^\bullet(X;\bb) \ \text{is {\em generated by ${\mf G} $ }  iff} \ QH_{\mf G}(X,\bb)
= QH^\bullet(X,\bb) .\]
\end{definition}

In our setting where the Fukaya category is the disjoint union of flat \ainfty categories $\Fuk^\flat_{\mf G}(X, \bb)_w$, the quantum cohomology ring $QH^\bullet(X;\bb)$ being generated by $\mf G$ is equivalent to the existence of a generalized eigen-space decomposition
\[QH^\bullet(X;\bb)=\oplus_{w \in w(\mf G)}QH^\bullet(X;\bb)_w\]
 and the condition
\[QH(X,\bb)_w = (OC(\bb)) ( HH_\bullet(\Fuk_{\mf G}^\flat(X,\bb)_w) ) \]
for all curvature values $w \in w({\mf G})$.

\begin{theorem} \label{crit} {\rm (Abouzaid \cite{Abouzaid_generation} in the exact case, extended to the compact case in Section \ref{acrit} below)} 
If $QH^\bullet (X,\bb)$ is generated by ${\mf G} \subset MC({\mc L})$ then for each $w \in \Lambda_{\geq 0}$ there is a subset of ${\mf G}$ that split-generates $\Fuk_\cL (X,\bb)_w$.
\end{theorem} 

The proof is based on an adaption of Abouzaid's original argument \cite{Abouzaid_generation} to the compact setting to incorporate bulk deformations, weakly bounding cochains, and the Cieliebak-Mohnke method. The key is to prove the commutativity of the {\it Cardy diagram} by analyzing two different types of degenerations of treed holomorphic annuli. 

By slightly modifying the moduli spaces of treed holomorphic annuli, one can prove an orthogonality result for images of open-closed maps. The following result will be used in the blowup setting to show that the ``old'' branes and ``new'' branes are orthogonal in the Fukaya category of the blowup.

\begin{theorem}\label{thm110}(proved in Section \ref{subsec:orth}) 
Suppose ${\mf L}_-, {\mf L}_+$ be two disjoint collections of weakly unobstructed branes. Then the images
\begin{align*}
&\ OC(\bb)( HH_\bullet( \Fuk^\flat_{{\mf L}_-}(X, \bb))),\ &\ OC(\bb) (HH_\bullet( \Fuk^\flat_{{\mf L}_+}(X, \bb)))
\end{align*}
are orthogonal with respect to the intersection pairing. 
\end{theorem}

\subsubsection{Old branes in the blowup}

Our main result applies the Abouzaid criterion Theorem \ref{crit} to blowups. Recall that the {\em blowup} of affine space $X = \C^n$ at $p = 0$ is
\begin{equation} \label{blcn}
 \Bl(\C^n,0) = \{ (z,\ell) \in \C^n \times \P^{n-1} | z \in \ell
\} \end{equation} 
and is equipped with a natural holomorphic projection
\[ \pi: \Bl(\C^n,0) \to \C^n, \quad (z,\ell) \mapsto z . \] 
The inverse image of the blowup point
\[ \tilde{Z} = \pi^{-1}(p), \quad \tilde{Z} \cong \P^{n-1} \]
is the {\em exceptional locus} of the blowup.  A symplectic blowup $\pi: \tilde X \to X$ of $X$ at a point $p$ is defined similarly using a Darboux chart $U \ni p$ and gluing in the local
model of the previous paragraph:
\[ \tilde{X} = ((X - \{ p \} ) \cup \pi^{-1}(U) ) / \sim  . \] 
A natural family of symplectic forms $\tilde{\omega}_\eps$ on $\tilde{X}$
arises from the family of symplectic forms on $\Bl(\C^n,0)$ 
considered as a toric variety with moment polytope 
\[ \left\{ (x_1,\ldots, x_n) \in \R_{\ge 0}^n  \ | \  x_1 + \ldots + x_n \ge
\eps \right\}. \]
The resulting symplectic manifold $\tilde{X}$ is the {\em $\eps$-blowup}
of $X$ at $p$, depending on the choice of $\eps$ and Darboux chart
$U$. \label{depending}

An embedding of the original Fukaya category into the Fukaya category
of its blowup will be realized after a shift in bulk deformation given
by homology classes.  
\begin{theorem} \label{oldembed} (proved in Section \ref{embed}) 
Suppose ${\mc L}$ consists of Lagrangian submanifolds that are disjoint from $p$. For $\eps > 0 $ sufficiently small and suitable perturbation data $\ul{P} = (P_\Gamma)$, the structure maps of $\Fuk_{\cL}(X,\bb + q^{-\eps} p )$ are convergent
  and define an \ainfty category with the following property: There exists a homotopy equivalence of curved \ainfty categories
\[
\Fuk_{\cL}^\sim (X,\bb + q^{-\eps} p ) \to \Fuk_{\pi^{-1}(\cL)}^\sim (\tilde{X}, \pi^{-1}(\bb) ).
\]
Moreover, for any collection of weakly unobstructed branes ${\mf L}$, there is a commutative diagram
\[
\xymatrix{   HH_\bullet( \Fuk_{\mf L}^\flat(X, \bb + q^{-\epsilon} p)) \ar[rr]^-{OC(\bb + q^{-\epsilon} p)} \ar[d] & &   QH^\bullet( X ) \ar[d]^{\pi^*} \\
             HH_\bullet( \Fuk_{\pi^{-1}({\mf L})} ^\flat(X, \pi^{-1}(\bb))) \ar[rr]_-{OC(\pi^{-1}(\bb))}   & &  QH^\bullet( \tilde X) }.
\]
\end{theorem}

\begin{remark} 
The bulk deformation $\bb + q^{-\epsilon} p$ has a negative $q$-valuation, so is not of the type usually allowed. The structure maps of the Fukaya category with bulk deformations with negative $q$-valuations may not converge {\it a priori}. 
However, there is a geometric reason for the convergence: holomorphic maps have to ``spend'' a nontrivial amount of energy to pass through a given point $p$ each time. \label{dcirc}
\end{remark}

The proof of Theorem \ref{oldembed} relies on a correspondence between pseudoholomorphic curves induced by the projection. Namely, given any holomorphic curve $\tilde{u}: C \to \tilde{X}$ one obtains a holomorphic curve in the original
manifold by projection $u = \pi \circ \tilde{u}$. This correspondence induces a map between moduli spaces
\begin{equation} \label{mproj} 
{\mc M}_{d, 1}( \pi^{-1}(L_0), \ldots, \pi^{-1}(L_d)) \to {\mc M}_{d, 1}( L_0, \ldots, L_d)
\end{equation}
(and the compactifications) given by composing and collapsing unstable components. 
The projection \eqref{mproj} does not preserve the expected dimension of the moduli spaces 
but {\em does} preserve expected dimension
if the map $u = \pi \circ \tilde u$ is considered as a map with point constraints at $u^{-1}(p)$. Moreover, \eqref{mproj} is a bijection for rigid curves. 
To prove Theorem \ref{oldembed} it therefore suffices to show that perturbation data pulled back under the projection $\pi: \tilde{X} \to X$ make all relevant moduli spaces in $\tilde{X}$ regular; one may then simply compose with the projection to obtain the correspondence.

\subsubsection{Open-closed map from the new branes}

To compare the Fukaya categories, we wish to complete the collection of ``old branes'' before the blowup by adding a collection of new generators after the blowup.   
In a previous paper \cite{flips} Charest and the second author identified a finite collection of
Floer-non-trivial Lagrangian branes near the exceptional locus. Indeed, a neighborhood of the exceptional locus has a toric model ${\mc O}(-1) \to \mb{CP}^{n-1}$ that contains a toric Lagrangian 
$L_{{\bm \epsilon}} \cong (S^1)^n$ that is monotone in  ${\mc O}(-1)$. 
The new branes are given by the Lagrangian $L_{{\bm \epsilon}}$ and $(n-1)$ distinct local systems.  Each local system induces a representation denoted $y: H_1(L_{{\bm \epsilon}}) \to \Lambda^\times$ as $y=(y_1, \ldots, y_n)$.  The calculation in \cite[p145]{flips} shows that the potential function is the Givental potential
\[
 W_0 = q^\epsilon ( y_1 + \cdots + y_n + y_1 \cdots y_n)
\]
plus higher order terms coming from holomorphic curves not contained in the toric region. The local systems of the $(n-1)$ weakly unobstructed branes $\WB_1, \ldots, \WB_{n-1}$ are higher order perturbations of the $(n-1)$ non-degenerate critical points of $W_0$. The toric model also allows us to compute the Floer cohomology rings $HF^\bullet( \WB_i, \WB_i)$, which are isomorphic to Clifford algebras, as well as the leading order terms in the open-closed map on these branes.

\begin{theorem}\label{thm114}(proved in Sections \ref{subsec:floernew}, \ref{subsec:ocnew}) 
Let ${\mf E}= \{ \WB_1, \ldots, \WB_{n-1}\}$ be the collection of exceptional branes described in the the preceeding paragraph. 
\begin{enumerate}
    \item  The potential functions of $\WB_i$ have distinct values,
    
    \item and the composition
\[
\xymatrix{ HH_\bullet( \Fuk^\flat_{\mf E}(\tilde X, \pi^{-1}(\bb)) \ar[rr]^-{OC(\pi^{-1}(\bb))} & & H^\bullet( \tilde X) \ar[r] & \tilde H^\bullet( \tilde Z ) \cong \Lambda^{n-1} }
\]
is surjective. 
\end{enumerate}
\end{theorem}

Theorem \ref{gen} follows from Theorem \ref{thm110}, Theorem \ref{oldembed}, 
and the generation result. Indeed, Theorem \ref{thm110} implies that old and new branes are orthogonal under the open-closed map. As the intersection pairing is non-degenerate on the image of the new branes, 
these two images have trivial intersection. For dimensional reasons, Theorem \ref{oldembed} and Theorem \ref{thm114} imply the surjectivity of the open-closed map. The generation criterion (Theorem \ref{crit}) then applies. One uses the spectral property of the open-closed and closed-open maps (Theorem \ref{spectral} and \ref{CO_spectral}) 
to conclude that the new branes contribute to $(n-1)$ orthogonal one-dimensional pieces of the quantum cohomology, proving Corollary \ref{under}.

\section{Moduli spaces of treed disks}\label{section2}

In this section,  we define the moduli spaces used in the definition of bulk-deformed Fukaya categories and regularize them using Cieliebak-Mohnke perturbations.

\subsection{Trees} \label{subsection21}

First we introduce terminology for trees. Given a tree $\Gamma$, the set of edges $\Edge(\Gamma)$ is equipped with {\em head} and {\em tail} maps
\[
h, t : \Edge({\Gamma}) \to \Ver({\Gamma}) \cup \{ \infty \} .
\]
The {\em valence}  of any vertex $v \in \Ver({\Gamma})$ is the number 
\[
|v | = \# \{ e \in h^{-1}(v) \cup t^{-1}(v) \}
\]
of edges meeting the vertex $v$. An edge $e \in \Edge({\Gamma})$ is
\begin{itemize}
    \item  {\em combinatorially finite} if
$\infty \notin \{ h^{-1}(e), t^{-1}(e) \}$,
\item {\em semi-infinite} or
{\em a leaf} if $\{ h^{-1}(e), t^{-1}(e) \} = \{v , \infty \}$ for
some $v \in \Ver({\Gamma})$, and 
\item {\em infinite} if
$h(e) = t(e) = \infty$.
\end{itemize}
Denote
\[
{\Edge}_{\rm \fin}(\Gamma)\ {\rm resp.}\ {\Edge}_\rightarrow(\Gamma) \subset {\Edge}(\Gamma)
\]
the set of finite resp. semi-infinite edges, that is, leaves. 

For now,  we assume that trees are {\it rooted}, which means that when $\Ver(\Gamma) \neq \emptyset$ there is a distinguished vertex $v_{\rm root} \in \Ver(\Gamma)$ called the {\em root} and a distinguished semi-infinite edge $e_{\rm out} \in {\Edge}_\rightarrow(\Gamma)$ with $t(e_{\rm out}) = v_{\rm root}$ called the {\em output}. All edges are then oriented towards the output. This will suffice for defining the Fukaya category. Later on, we will consider not-necessarily-rooted trees. 

There is a special tree which does not have vertices: a {\em vertex-free tree} is a tree $\Gamma$ with ${\rm Vert}(\Gamma) = \emptyset$ with one infinite edge. However we set ${\Edge}_{\rightarrow} = \{ e_{\rm in}, e_{\rm out}\}$, the incoming
and the outgoing ends of the vertex-free tree. In any case, for a tree $\Gamma$, denote by $ {\Edge}_{\rm in}(\Gamma)$ resp.  $ {\Edge}_{\rm out}(\Gamma)$ the incoming and outgoing leaves. \label{leaves}

Our trees will be composed of two parts corresponding to the sphere and disk vertices.  We color these vertices black and white respectively, and call the resulting structure a {\it two-colored tree.}

\begin{definition}\label{defn21}(Two-colored trees)
\begin{enumerate}
\item A {\em ribbon structure} on a tree $\Gamma$ consists of a cyclic ordering
$o_v: \{ e \in \Edge({\Gamma}), e \ni v \} \to \{ 1, \ldots, |v| \}$ of the edges \label{ribbon} incident to each vertex
$v \in \Ver({\Gamma})$; a cyclic ordering is an equivalence class $[o_v]$ of orderings where two orderings $o_v, o_v'$ are equivalent if
they are related by a cyclic permutation.
\item A {\em rooted subtree} of a tree $\Gamma$ is a connected
subgraph
  ${\Gamma_\circ}$ whose vertices
  $\Ver({\Gamma_\circ})$ contain the root $v_{\rm root}$ of
  $\Gamma$,\footnote{When defining the open-closed map we will
    consider {two-colored tree}s whose root vertex is not in the {disk part $\Gamma_\circ$}.}
  and whose edges contain all finite edges $e \in \Edge(\Gamma)$ connecting vertices in
  $\Ver({\Gamma_\circ})$ and a subset of semi-infinite edges 
  $e \in \Edge(\Gamma)$ connected
  to vertices in $\Ver({\Gamma_\circ})$.
\item A {\em {two-colored tree}} is a tree $\Gamma$ together with a rooted subtree ${\Gamma_\circ}$ with a ribbon structure on ${\Gamma_\circ}$.

\item A {two-colored tree} $\Gamma$ is {\em stable} if each {\em sphere
    vertex} 
    \[ v \in \Ver_\black(\Gamma) := \Ver(\Gamma) \setminus \Ver(\Gamma_\circ) \]  
    has
  valence at least three and for each {\em disk vertex}
  $v \in \Ver({\Gamma_\circ})$ the number of edges $e \in \Edge({\Gamma_\circ})$
  connected to $v$ plus twice of the number of {\it interior edges}
  \[ e \in 
  \Edge_\black(\Gamma) := \Edge(\Gamma) \setminus \Edge(\Gamma_\circ)\]
  connected to $v$ is at least three.
\end{enumerate}
\end{definition}
We distinguish between boundary and interior leaves and disk and sphere components. Objects related to the rooted subtree (which are usually related to disks and boundary insertions) are labelled with $\circ$ while the corresponding notions related to the complement of the rooted subtree (which are related to spheres and interior insertions) are labelled with $\bullet$. For example, we denote by 
\[
{\rm Edge}_\circ(\Gamma):= {\rm Edge}(\Gamma_\circ) \subset {\rm Edge}(\Gamma)
\]
the set of {\it boundary edges} and we used above
\[
{\rm Edge}_\bullet(\Gamma):= {\rm Edge}(\Gamma) \setminus {\rm Edge}_\circ(\Gamma)
\]
the set of {\it interior edges}. Semi-infinite edges are also called {\it leaves} and we denote 
\[ {\rm Leaf}_\circ(\Gamma):={\Edge}_\rightarrow(\Gamma) \cap
\Edge_\circ(\Gamma),\  {\rm
  Leaf}_\black(\Gamma):={\Edge}_\rightarrow(\Gamma) \cap
\Edge_\black(\Gamma). \]

A moduli space of metric trees is obtained by allowing the finite edges on the {disk part} to acquire lengths. 

\begin{definition}\label{defn22}
Let $\Gamma$ be a {two-colored tree}. A {\em metric} on $\Gamma$ is a non-negative function on the space of finite boundary edges
\[
\ell: {\Edge}_{\rm \fin} (\Gamma_\circ) \to [0, +\infty).
\]
A {\em metric type} on $\Gamma$, denoted by $\uds{\ell}$, is the associated decomposition 
\[
{\Edge}_{\rm \fin}(\Gamma_\circ) = {\Edge}_0( \Gamma_\circ) \sqcup {\Edge}_+( \Gamma_\circ)
\]
corresponding to edges with zero or positive lengths. 
\label{changephrasing} \footnote{To define
  the Fukaya category we only need to consider metric on boundary
  edges. When we define the open-closed and closed-open maps we need
  more general metric types.}
\end{definition}

To compactify the set of gradient segments we allow the lengths of the
edges to go to infinity and break.  A {\em broken metric tree} is
obtained from a finite collection of metric trees by gluing outputs
with inputs as follows: Given two metric trees $(\Gamma_1, \ell_1)$
and $(\Gamma_2, \ell_2)$ with specified leaves
$e_1 \in \Leaf(\Gamma_1)$ and $e_2 \in \Leaf(\Gamma_2)$, let
$\ol{\Gamma}_1$ resp. $\ol{\Gamma}_2$ denote the space obtained by
adding a point $\infty_1$ resp. $\infty_2$ at the open end of $e_1$
resp. $e_2$. The space
\begin{equation}\label{glueT}
\Gamma := \ol{\Gamma}_1 \cup_{\infty_1 \sim \infty_2} \ol{\Gamma}_2
\end{equation}
is a broken metric tree, the point $\infty_1 \sim \infty_2$ being called a {\em breaking}. To obtain a well-defined root for the glued tree we require that exactly one of $e_1$ and $e_2$ is the output. See Figure \ref{cuttree}.
\begin{figure}[ht]
\includegraphics[height=.85in]{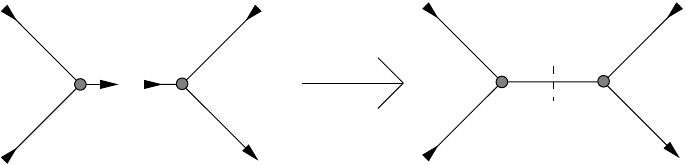} 
\caption{Creating a broken tree}
\label{cuttree}
\end{figure}
\noindent In general, a broken metric tree $\Gamma$ is obtained \label{obtainedgram} from
broken metric trees $\Gamma_1, \Gamma_2$ as in \eqref{glueT} in such a
way that the resulting space $\Gamma$ is connected and has no
non-contractible loops, that is, $\pi_0(\Gamma)$ is a point and
$\pi_1(\Gamma)$ is the trivial group\footnote{Later when we consider
  treed annuli we will allow loops.}. We think of the gluing points as
breakings rather than vertices, so that there are no new vertices in
the glued tree $\Gamma$.

In order to obtain Fukaya algebras with strict units, we wish for our moduli spaces to admit forgetful maps. For this we introduce weightings on certain edges, as in for example Ganatra \cite[Section 10.5]{Ganatra_thesis}.

\begin{definition}\label{defn23}
Consider an unbroken tree $\Gamma$. 
\begin{enumerate}
\item A weighting on $\Gamma$ is a map 
\[ \on{wt}: {\Edge}_{\rightarrow} (\Gamma) \to [0,1] \]
satisfying 
\[
\on{wt}|_{{\rm Leaf}(\Gamma)} \equiv 0, 
\]
and
\begin{equation}\label{weight_relation} 
\prod_{e\in {\Edge}_{\rm in}(\Gamma)} \on{wt}(e) = \on{wt}(e_{\rm out})
\end{equation}
The underlying decomposition 
\begin{eqnarray*} {\Edge}_\rightarrow(\Gamma) &=& {\Edge}^\blackt
  (\Gamma) \sqcup {\Edge}^\greyt (\Gamma) \sqcup {\Edge}^\circt
  (\Gamma) \\ &:=& \on{wt}^{-1}(0) \sqcup \on{wt}^{-1}((0,1)) \sqcup \on{wt}^{-1}(1).
\end{eqnarray*} 
is called a {\em weighting type}, denoted by $\uds{\on{wt}}$; elements of
$ {\Edge}^\blackt (\Gamma)$ resp. ${\Edge}^\greyt(\Gamma) $
resp. ${\Edge}^\circt (\Gamma)$ are called {\em unforgettable}
resp. {\em weighted} resp. {\em forgettable. }  A tree $\Gamma$ with a
weighting is called a {\em weighted tree}. \label{fpage}

\item If the output $e_{\rm out}$ of $\Gamma$ is unweighted then an
{\em  isomorphism 
of weighted trees} is an isomorphism of trees %
$ \psi: (\Gamma, \on{wt} ) \to (\Gamma', \on{wt}') $ that preserves the
weightings. If the output $e_{\rm out}$ of $\Gamma$ is weighted (which
implies $\Gamma$ has no interior incoming edge and all boundary
incoming edges are weighted or forgettable), then an isomorphism
$\psi: (\Gamma, \on{wt}) \to (\Gamma', \on{wt}')$ is an isomorphism of trees such
that there is a positive number $\alpha$ such that
\[ \on{wt} (e) = \on{wt}'(\psi(e))^\alpha, \quad \forall e \in
{\Edge}_{\rightarrow}(\Gamma). \]

\item If $\Gamma$ is broken, then a weighting on $\Gamma$ consists of weightings on all unbroken components that agree over breakings.
\end{enumerate} 
\end{definition}

\subsection{Treed disks}\label{subsection22}

The domains of treed holomorphic disks are unions of disks, spheres, and lines, rays, and line segments. A {\em disk} is a bordered Riemann surface biholomorphic to the complex unit disk $\DD = \Set{ z \in \C | \Vert z \Vert \leq 1 }$. The automorphism group of $\DD$ is $\Aut(\D) \cong PSL(2,\R)$. A {\em nodal disk} with a single boundary node is a topological space $S$ obtained from a disjoint union of disks $S_1,S_2$ by identifying pairs of boundary points $w_{12} \in S_1, w_{21} \in S_2$ on the boundary of each component so that
\begin{equation} \label{glueS} 
S = S_1 \cup_{w_{12} \sim w_{21}} S_2 .\end{equation} 
See Figure \ref{cutfig2}. The image of $w_{12}, w_{21}$ in the space $S$ is the {\em nodal point}. 
\begin{wrapfigure}{r}{6cm}
\centering
\includegraphics[scale=0.65]{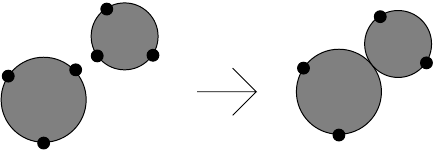} 
\caption{Creating a nodal disk}
\label{cutfig2}
\end{wrapfigure}
A nodal disk $S$ with multiple nodes $w_{ij}, i,j \in \{ 1,\ldots, k\}, i \neq j$ is obtained by repeating this construction \eqref{glueS} with $S_1,S_2$ nodal disks with fewer nodes, and $w_{12}, w_{21}$ distinct from the other nodes. More generally we allow boundary and interior markings. For an integer $d \ge 0$ a {\em nodal disk with $d+1$ boundary markings} is a nodal disk $S$ equipped with a finite ordered collection of points $\ul{x} = (x_0,\ldots,x_d)$ on the boundary $\partial S$,  disjoint from the nodes, in counterclockwise cyclic order around the boundary $\partial S$. A $(d+1)$-marked nodal disk $(S,\ul{x})$ is {\em stable} if each component $S_v$ has at least three special (nodal or marked) points, or equivalently the group $\Aut(S,\ul{x})$ of automorphisms of $S$ leaving $\ul{x}$ pointwise fixed is trivial. The moduli space of $(d+1)$-marked stable disks $ [(S,\ul{x})]$ forms a compact cell complex, isomorphic as a cell complex to the associahedron from Stasheff \cite{Stasheff_63, Stasheff_1970}.

More complicated configurations involve spherical components. A {\em marked sphere} is a complex surface biholomorphic to the projective line $S^2 \cong \P^1$ together with a distinct ordered list of markings $z_1, \ldots, z_k\in S^2$. A nodal disk $S$ with a single interior node $w \in S$ is defined similarly to that of a boundary node by using the construction \eqref{glueS}, except in this case $S$ is obtained by gluing together a nodal disk $S_1$ with a  marked sphere $S_2$ with $w_{12}, w_{21}$ points in the interior $\on{int}(S)$.
 \label{glueI}   \label{glueIp} 
 
General treed disks are defined as in Oh \cite{Oh_1993}, Cornea-Lalonde \cite{cl:clusters}, Biran-Cornea \cite{Biran_Cornea, Biran_Cornea_09}, and Seidel \cite{Seidel_2011}.

\begin{definition}[Treed disks, domain types]\label{defn24}
\begin{enumerate}
    
\item A {\em combinatorial type for treed disks} (or a {\it domain type}) is a two-colored tree $\Gamma$ together with a metric type $\ul\ell$ (see Definition \ref{defn22}) and a weighting type $\ul{\on{wt}}$ (see Definition \ref{defn23}). To save notations, we often abbreviate a domain type $(\Gamma, \ul\ell, \ul{\on{wt}})$ by $\Gamma$.

\item A {\em treed disk} $C$ of domain type $(\Gamma, \ul\ell, \ul{w})$ consists of the {\em surface part}
\[
S = (S_v, \ul{x}_v, \ul{z}_v)_{v\in \Ver(\Gamma)}
\]
(where $\ul{x}_v$ resp. $\ul{z}_v$ denotes the ordered set of boundary resp. interior markings), a {\em tree part}
\[
T = (T_e)_{e \in {\Edge}(\Gamma)},
\]
(where $T_e$ is a finite interval of a certain length $\ell(e)$ if $e$
is combinatorially finite, a semi-infinite interval $[0, +\infty)$ or
$(-\infty, 0]$ if $e$ is semi-infinite, so that $(\Gamma, \ell)$
becomes a metric tree whose metric type agrees with $\ul\ell$\footnote{If $e$ is an infinite edge, then regard $T_e$ as
  the real line $(-\infty, +\infty)$ which is also the union of two
  rays labelled by the input and the output.}), a weighting
\[ \on{wt}: {\Edge}_{\rightarrow}(\Gamma) \to [0,1]  \]  
whose underlying weighting type agrees with $\ul{\on{wt}}$, and {\em nodal
  points}
\[
z_{e, +} \in S_{h(e)},\ z_{e, -} \in S_{t(e)},\ \forall e \in {\Edge}(\Gamma).
\]
These data must satisfy the following conditions: for each vertex $v\in {\rm Vert}(\Gamma)$, the set of {\em special points}, i.e., the collection of boundary and interior markings and nodal points are distinct. See Figure \ref{treeddisk} for a typical configuration of a treed disk. To a treed disk $C$ we associate a compact topological space $S \cup T$ obtained by gluing different components in the obvious way.  Each such space $S \cup T$
includes a finite subset of points corresponding to the breakings and infinities of semi-infinite edges.
    
\item An {\em isomorphism} of treed disks $\phi$ from $C = S \cup T$
  to $C' = S' \cup T'$ consists of an isomorphism
  $\psi: (\Gamma, \ell, \on{wt}) \to (\Gamma', \ell', \on{wt}')$ of underlying
  weighted metric trees, a collection of conformal isomorphisms
\[ \phi_v: S_v \to S_{\psi(v)}', \quad v \in \Ver(\Gamma) \]  
of disks or spheres preserving the
  markings and special points, and a collection of length-preserving
  isomorphisms 
\[ \phi_e: T_e \to T_{\psi(e)}', \quad e \in \Edge(\Gamma) \] 
of intervals. \label{addstr}
    
\item A treed disk is stable if its underlying combinatorial type is
  stable (see Definition \ref{defn21}).\footnote{As in the case of
    spheres, treed disk is stable if and only if its automorphism
    group is trivial.}
\end{enumerate}
\end{definition}
\begin{figure}[ht]
\includegraphics[height=1.5in]{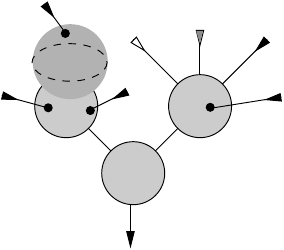}
\caption{A treed disk with three disk components and one sphere
  component, and its combinatorial type}
\label{treeddisk}
\end{figure} 

\begin{remark}\label{rem:partialorder}
There is a natural partial order among all stable domain types, denoted by $\Gamma' \preceq \Gamma$. Instead of giving the full definition, we only recall the typical situations. These typical situations include the case of bubbling off holomorphic spheres, bubbling off holomorphic disks, and breaking of gradient lines, in which $\Gamma'$ is obtained from $\Gamma$ by a change of the underlying tree. Moreover, when the length of an edge of $\Gamma$ changes from positive to zero, one obtains a different type $\Gamma' \prec \Gamma$ by changing the metric type; when the weighting of one or more semi-infinite edges of $\Gamma$ changes to zero or one, one also obtains a different type $\Gamma' \prec \Gamma$ by changing the weighting type accordingly. In general, $\Gamma' \preceq \Gamma$ if $\Gamma'$ can be obtained from $\Gamma$ by finitely many such changes.  We emphasize that each partial order relation $\Gamma' \preceq \Gamma$ induces a unique tree map from $\Gamma'$ to $\Gamma$, that is, a surjective map $\rho_V: {\rm Vert}(\Gamma') \to {\rm Vert}(\Gamma)$ that preserves the partial order among vertices and sends ${\rm Vert}(\Gamma_\circ')$ onto ${\rm Vert}(\Gamma_\circ)$, as well as a bijection
  ${\rm Leaf}_\circ(\Gamma') \cong {\rm Leaf}_\circ(\Gamma)$ and a
  bijection
  ${\rm Leaf}_\bullet(\Gamma') \cong {\rm Leaf}_\bullet(\Gamma)$.
\end{remark}

The moduli spaces of stable weighted treed disks are naturally cell
complexes. Suppose $\Gamma$ is a stable domain type with $d(\circ)$
boundary inputs and $d(\black)$ interior leaves. Let ${\cM}_\Gamma$ denote
the set of all isomorphism classes of treed disks of type $\Gamma$,
with its natural topology induced by embedding in the product of the moduli space of stable trees and stable disks.  The space ${\cM}_\Gamma$ is a
manifold of dimension
\begin{multline*}
\dim (\M_\Gamma) =
d(\circ)+ 2 d(\black)  + \# \Edge^{\greyt}(\Gamma) - \# {\Edge}_0(\Gamma) - 2 \# {\Edge}_{\rm interior}(\Gamma)\\
+ \begin{cases} -2 &\ {\rm if\ } e_{\rm out} \notin {\Edge}^{\greyt} (\Gamma), \\
-4 &  {\rm if\ } e_{\rm out} \in \Edge^{\greyt}(\Gamma).
                                    \end{cases}
                                    \end{multline*}
Denote 
\[
\ov{\cM}_\Gamma = \bigsqcup_{\substack{\Gamma' \preceq \Gamma\\ \Gamma'\ {\rm stable}} {\cM}_{\Gamma'}}.
\]
As in the definition of Gromov convergence of pseudoholomorphic curves, there is a natural way to endow $\ov{\cM}_\Gamma$ a compact Hausdorff topology that agrees on the manifold topology on each stratum ${\cM}_{\Gamma'}$, so that $\ov{\cM}_\Gamma$ is a cell complex with ${\cM}_\Gamma$ equal to the top cell.

\begin{remark} \label{examples} The moduli spaces of weighted treed disks are related to unweighted moduli spaces by taking products with intervals: If $\Gamma$ has at least one vertex and $\Gamma'$ denotes the domain type obtained from $\Gamma$ by setting the weights $w (e)$ to zero and the output $e_{\rm out}$ of $\Gamma$ is unweighted then 
\[
\M_\Gamma \cong \M_{\Gamma'} \times (0,1)^{|\Edge^{\greyt} (\Gamma)|} .
\]
If the outgoing edge $e_{\rm out}$ is weighted then
\[
\M_\Gamma \cong \M_{\Gamma'} \times (0,1)^{|\Edge^{\greyt} (\Gamma)| - 2}
\]
because of the way we define isomorphism of weighted types (see Definition \ref{defn23}). Figure \ref{weighted} illustrates a one-dimensional moduli space with weighted output and its boundary strata.
\end{remark}

\begin{figure}[ht]
    \centering
    \includegraphics[width=5in]{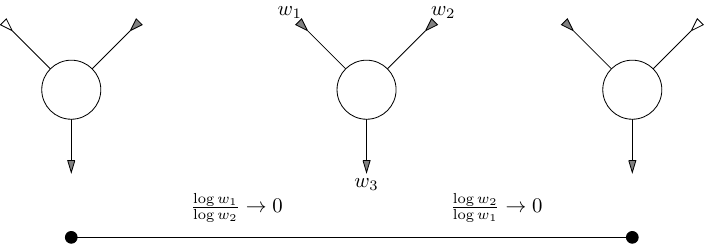}
    \caption{A one-dimensional moduli space of weighted treed disks with all three semi-infinite edges weighted.}
    \label{weighted}
\end{figure}

In general, moduli spaces of stable curves only admit universal curves in an orbifold sense. In the setting here, orbifold singularities are absent and the moduli spaces of stable treed disks admit honest universal curves.  For any stable domain type $\Gamma$ let $\ol{\U}_\Gamma$ denote the {\em universal treed disk} (or called the {\em universal curve}) consisting of isomorphism classes of pairs $(C,z)$ where $C$ is a treed disk of type $\Gamma$ and $z$ is a point in $C$, possibly on a disk component $S_v \cong \{ |z| \leq 1 \}$, a sphere component $S_v \cong \P^1$, or one of the edges $e$ of the tree part $T \subset C$ (the infinities of semi-infinite edges are allowed). The map
\[
\pi_\Gamma: \ol{{\U}}_\Gamma \to \ol{{\M}}_\Gamma, \quad [C,z] \to [C]
\]
is the universal projection.  Moreover, for each $[C]\in \ov{\mc M}_\Gamma$ represented by $C$, the fibre $\pi_\Gamma^{-1}([C])$ is homeomorphic to $C$.   In case $\Gamma$ has no vertices we define $\ol{\U}_\Gamma$ to be the real line, considered as a fiber bundle over the point $\ol{\M}_\Gamma$.

\begin{figure}[ht]
\includegraphics[width=4in]{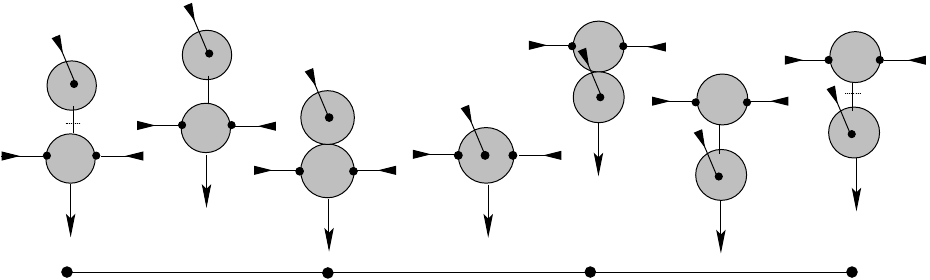}
\caption{Treed disks with interior leaves}
\label{interior}
\end{figure} 

We introduce notation for particular subsets of the universal curves. First, for each vertex $v \in {\rm Vert}(\Gamma)$, let 
\[
\ov{\cU}_{\Gamma, v} \subset \ov{\cU}_\Gamma
\]
denote the closed subset corresponding to points on the surface component $S_v$ and on the semi-infinite edges attached to $v$. 
For each boundary edge $e \in {\Edge}(\Gamma_\circ)$, let 
\[ \ov{\cU}_{\Gamma, e} \subset \ov{\cU}_{\Gamma} \]
the closed subset corresponding to points on the tree component $T_e$. Denote
\begin{equation} \label{US}
\ov{\cS}_\Gamma:= \bigcup_{v \in {\rm Vert}(\Gamma)} \ov{\cU}_{\Gamma, v}
\end{equation}
and 
\begin{equation} \label{UT}
\ov{\cT}_\Gamma:= \bigcup_{e \in {\Edge}(\Gamma_\circ)} \ov{\cU}_{\Gamma, e}.
\end{equation} 
Moreover, for each subtree $\Pi \subset \Gamma$ (not necessarily containing the root), denote by 
\[
\ov{\cU}_{\Gamma, \Pi}  \subset \ov{\cU}_{\Gamma}
\]
the set of points $z$ on components $S_v,T_e$ corresponding to
vertices $v$ and edges $e$ of $\Pi$. There is a contraction map
$\ov{\cU}_{\Gamma, \Pi} \to \ov{\cU}_{\Pi}$ contracting edges not in
$\Gamma$. In particular, for the {disk part} $\Gamma_\circ$, one has
$ \ov{\cU}_{\Gamma, \Gamma_\circ} \subset \ov{\cU}_{\Gamma}.$ Lastly,
for $\Pi \preceq \Gamma$, one has a boundary stratum
$ \ov{\cU}_{\Pi} \subset \ov{\cU}_\Gamma.$

The boundary of each treed disk is divided up into parts between the boundary inputs.
Denote by $(\partial C)_i$ the component of $\partial C$ between the
$i$-th and $i+1$-st leaves, in cyclic order. Similarly, denote $i$-th
boundary part of the universal curve
$ \partial_i \ol{\U}_\Gamma \subset \ov{\cU}_\Gamma$.

\subsection{Rational branes}\label{subsection23}

In this subsection,  we specify assumptions and additional data on the Lagrangian submanifolds in our construction.  Let $(X, \omega)$ be a compact symplectic manifold.  Let ${\mc L} = \{ L_1, \ldots, L_m\}$ be a collection of embedded Lagrangian submanifolds in $(X, \omega)$. Denote the {\it support} of ${\mc L}$ to be 
\[
|{\mc L}|:= \bigcup_{L \in {\mc L}} L \subset X.
\]
By a {\it brane} we will mean a flat $\Lambda^\times$ bundle $\LB$ over $L$ where $L$ is an (oriented, spin, embedded) Lagrangian submanifold.  By a {\it weakly unobstructed brane} we will mean a pair ${\bm L} = (\LB, b)$ where $\LB$ is a brane and $b$ is a Maurer-Cartan element of the curved $A_\infty$ algebra of $\LB$.

\begin{definition}[Rational Lagrangian] \label{srat}  
The collection of Lagrangians $\LL$ is called {\em rational} if there exists a line bundle with connection $\hat X \to X$ with curvature $(2\pi k /i ) \omega$ for some $k \in \Z_+$ and there exists a smooth section $s \in \Gamma(\hat X)$ that is nowhere vanishing along $|\LL|$, whose restriction to each $L \in \LL$ is flat with respect to the
connection on $\hat X$. The collection $\LL$ is called {\it exact} in
an open subset $U\subset X$ if
\begin{enumerate}

\item $|\LL| \subset U$;
    
\item there is a 1-form $\theta \in \Omega^1(U)$ such that $\omega|_U = d\theta$;
    
\item there exists a continuous function $f: |\LL| \to {\mb R}$ whose restriction to each $L \in \LL$ is smooth and 
\[ 
\theta|_L = d(f|_L). 
\]
\end{enumerate}
\end{definition}

From now on, we assume that $[\omega]$ is integral, for simplicity, and
$\LL$ is rational:

\begin{hyp}\label{lagrangian_assumption}
The collection ${\mc L}$ satisfies the following conditions.
\begin{enumerate}
    \item Each $L \in {\mc L}$ is connected, oriented, and equipped with a spin structure.  

    \item For each pair $L, L'$ in ${\mc L}$, their intersection is clean, oriented, and equipped with a spin structure.\footnote{In general, the clean intersection of two orientable submanifolds may not be orientable.} 
    
    \item ${\mc L}$ is rational as in Definition \ref{srat}. 
\end{enumerate}
\end{hyp}

\subsubsection{Bulk deformation}

Bulk deformations used in this paper are linear combinations of disjoint embedded closed submanifolds $\bb_1,\ldots, \bb_N \subset X$ denoted by
\[
\bb = \sum_{i=1}^N c_i \bb_i,\ c_i \in \Lambda\setminus \{0\}.
\]
We assume that all $\bb_i$ are oriented and have even and positive codimensions. The support of $\bb$ is the union 
\[
|\bb|:= \bigcup_{i = 1}^N \bb_i .
\]

\begin{remark}
There will be no essential difference but only notational complexities if we allow bulk deformation to be pseudocycles rather than closed submanifolds. 
\end{remark}

\subsubsection{Donaldson hypersurface}

The Cieliebak-Mohnke scheme relies on the existence of Donaldson hypersurfaces, defined as follows. \label{dsurf} Given a rational symplectic manifold $(X, \omega)$ a {\it Donaldson hypersurface} is a compact codimension two symplectic submanifold $D \subset X$ whose Poincar\'e dual is a multiple $k[\omega]$ of $[\omega]$. The positive integer $k$ is called the {\it degree} of
  the Donaldson hypersurface. 

\begin{lemma}\label{countable} {\rm (c.f. Charest-Woodward \cite[Section 3.1]{Charest_Woodward_2017}, \cite[Lemma 8.7]{Cieliebak_Mohnke})} 
Let $J$ be an $\omega$-compatible almost complex structure on $X$ such that all Lagrangians in $\LL$ are totally real. For $l \in {\mb N}$ sufficiently large there exist a sequence of degree $l$ Donaldson hypersurfaces $D = D_l \subset X$ disjoint from $|\LL|$ with the following properties.
\begin{enumerate}  
\item $\LL$ is exact in the complement $X - D_l$.

\item For each $l$, there is a tamed almost complex structure $J_0 \in {\mc J}_{\rm tame}(X, \omega)$ making $D_l$ almost   complex such that all nonconstant $J_0$-holomorphic spheres in $X$ intersect $D_l$ at finite but at least three points and all nonconstant $J_0$-holomorphic disks with boundary in $|\LL|$ intersect $D_l$ in the interior.

\item $D_l$ is transverse to each component of $\bb$.
\end{enumerate}
\end{lemma} 

\begin{proof} 
The construction is an extension of the original constructions \cite{Donaldson_96}  \cite{Auroux_97} \cite{Auroux_Gayet_Mohsen} \cite[Theorem 3.3]{pt:wall}. Let $\hat X \to X$ be a line bundle whose curvature is the symplectic form (up to a factor of $2\pi /i $).  The argument of \cite[Theorem 3.3]{pt:wall} (which is purely local and so applies to the cleanly-intersecting Lagrangian considered here) gives the existence of an approximately $J$-holomorphic section
\[
s: X \to \hat{X}^l
\]
of some tensor power $\hat{X}^l$ and so that the restriction of $s$ to each $L \in \LL$ is close to the given flat section on $L$.  One obtains a symplectic hypersurface as the zero-set:
\[
D = s^{-1}(0) .
\]
The connection one-form $\alpha$ in the trivialization provided by $s$ provides a primitive for the symplectic form $\omega$, and the fact that $s$ is approximately flat on $\LL$ implies that the integral of $\alpha$ over any loop in $|\LL|$ vanishes, so that $\LL$ is exact; see \cite[Theorem 8.1]{Cieliebak_Mohnke} and the modification in \cite[Theorem 3.6]{Charest_Woodward_2017}. By Cieliebak-Mohnke \cite[Corollary 8.16]{Cieliebak_Mohnke}, for sufficiently generic tamed almost complex structures $J$, each
non-constant $J$-holomorphic sphere $u: \P^1 \to D $ is not contained in $D $ and intersects $D$ in at least three points:
\[
\# u^{-1}(D) \ge 3.
\]
On the other hand, since $\LL$ is exact in the complement of $D$, each nonconstant pseudoholomorphic disk $u: \D \to X$ with boundary in $\LL$ intersects $D$ in at least one interior point. Transversality of $D$ to the bulk deformation $\bb$ \label{ofu} 
follows as in \cite[Corollary 5.8]{Cieliebak_Mohnke}. 
\end{proof}

\begin{remark}\label{rem_asymptotically_holomorphic}
The notion of approximately holomorphic can be made more precise as follows: A sequence of sections $s_l$ of $\hat{X}^l$ (for large $l$) is said to be {\it asymptotically holomorphic} with respect to the given connections and almost-complex structure if the following bounds hold: There exists a constant $C>0$ such that, for all $l$ and at every point of $X$, 
\[
|s_l| + | \nabla s_l | +| \nabla^2 s_l | \leq C, \quad  | \overline{\partial} s_l| + | \nabla ( \overline{\partial} s_l) |  \leq C l^{-\frac{1}{2}}
\]
where the norms of the derivatives are evaluated with respect to the metrics 
defined by the rescaled two-form $l \omega$.
Such a sequence is said to be {\it uniformly transverse to $0$} with constant $\eta$ if the derivative of $s_l$ is non-zero whenever $|s_l(x)| < \eta $ 
and has a right inverse bounded by $\eta^{-1}$, as in \cite[Definition 1]{Auroux_Gayet_Mohsen}. Donaldson's construction  shows the existence of asymptotically holomorphic sections
uniformly transverse to the zero section, using sequences of asymptotically holomorphic sections concentrated near a point.  \end{remark}

\subsection{Perturbations}\label{subsection24}

We consider domain-dependent perturbation data defined on the
universal curves. We first define a condition called 
{\em locality}, which our perturbation data will be required to satisfy. A similar condition plays an important role in Cieliebak--Mohnke's
approach \cite{Cieliebak_Mohnke}.

\begin{notation}
Let $\Gamma = (\Gamma, \ul\ell, \ul{\on{wt}})$ be a stable domain type. Recall that $\Gamma_\circ$ is the subtree corresponding to disk components and boundary edges. For each spherical vertex $v \in \on{Vert}(\Gamma)\setminus {\rm Vert}(\Gamma_\circ)=: {\rm Vert}_\bullet(\Gamma)$, let $\Gamma(v)$ denote the subtree of $\Gamma$ consisting of the vertex $v$ and all
edges $e$ of $\Gamma$ meeting $v$. 
Let
\begin{equation}
  \pi: \pi_1 \times \pi_2: {\cU}_\Gamma \to {\cM}_{\Gamma_\circ} \times {\cU}_{\Gamma(v)}
\end{equation}
be the product of maps where $\pi_1$ is given by projection followed
by the forgetful morphism and $\pi_2$ is the contraction $C \to S_v$.
\end{notation} 

\begin{definition}[Locality]\label{locality}
Let $Z$ be a set. A map $f: {\cU}_\Gamma \to Z$ is called {\it local} if the following two conditions are satisfied.
\begin{enumerate}
    \item For each spherical vertex $v\in {\rm Vert}_\bullet (\Gamma)$, the restriction of $f$ to ${\mc U}_{\Gamma, v}$ 
    factors through a map $f_v$ as in the commutative diagram
  \[
  \xymatrix{ {\mc U}_{\Gamma, v} \ar[d]_{\pi_1\times \pi_2} \ar[r]^f & Z \\
             {\mc M}_{\Gamma_\circ}\times {\mc U}_{\Gamma(v)} \ar[ru]_{f_v} &  }
  \]
  
\item Let ${\mc U}_{\Gamma, \Gamma_\circ}$ be the union of the tree part ${\mc T}_\Gamma$ and ${\mc U}_{\Gamma, v}$ for all disk vertices $v \in {\rm Vert}_\circ(\Gamma)$. Then there is a contraction map 
\[
{\mc U}_{\Gamma, \Gamma_\circ} \to {\mc U}_{\Gamma_\circ}.
\]
We require that the restriction of $f$ to ${\mc U}_{\Gamma, \Gamma_\circ} \subset {\mc U}_\Gamma$ is equal to the pullback of a map $f_\circ: {\mc U}_{\Gamma_\circ} \to Z$.
\end{enumerate}
    A map $f: \ol{\cU}_\Gamma \to Z$ is local if the restriction of
    $f$ to any stratum ${\cU}_\Pi \subset \ov{\mc U}_\Gamma$ for $\Pi \prec \Gamma$
    is a local map.
\end{definition}

\begin{remark} 
Locality implies the following gluing construction: for any sphere  vertex $v \in \on{Vert}_{\black} (\Gamma)$ let 
  $\Gamma'$ denote the  type of graph obtained
  by removing all but one interior leaf $e \in \Leaf (\Gamma)$ meeting $v$ and collapsing any unstable component. Then on the complement of $S_v$ and possibly other collapsed components, $f$ is equal to the pull-back of a map from 
 $\cU_{\Gamma'}$ to $Z$.
\end{remark}

\subsubsection{Supports of perturbations}

In this section, we construct open sets where the perturbations 
are required to vanish.  Let $\ol{\cS}_\Gamma$ and $\ol{\cT}_\Gamma$ be the universal surface and tree from \eqref{US} and \eqref{UT}.

\begin{lemma}\label{lemma216}
For all stable combinatorial types $\Gamma$, there exist collections of open subsets (where the complex structure $J$, the Hamiltonian perturbations $H$, or the Morse functions $F$ will be fixed)
\[
\ov{\cS}_{\Gamma, J} \subset \ov{\cS}_\Gamma,\ \ \ov{\mc S}_{\Gamma, H}\subset \ov{\mc S}_\Gamma, \quad  \ov{\cT}_{\Gamma, F} \subset \ov{\cT}_{\Gamma}
\]
satisfying the following properties.
\begin{enumerate}

\item The open set $\ov{\cS}_{\Gamma, J}$ intersects with any fiber
  $C = S \cup T \subset \ov{\cU}_\Gamma$ at a neighborhood of all
  special points on the surface part so that for all
  $v \in \on{Vert}(\Gamma)$, the complement of $\ov{\cS}_{\Gamma, J}$ has
  non-empty intersection with $S_v$;
  
\item The intersection of $\ov{\mc S}_{\Gamma, H}$ with each fibre $C = S \cup T \subset \ov{\mc U}_\Gamma$ contains all spherical components and a neighborhood of all nodal points. Moreover, the complement of $\ov{\mc S}_{\Gamma, H} \cap C$ has a nonempty intersection with each disk component $S_v \subset S$. 
\item The open set $\ov{\cT}_{\Gamma, F}$ is a neighborhood of the
  locus corresponding to infinities of semi-infinite edges in all
  degenerations $\Pi \prec \Gamma$.

\item If $\Gamma$ is separated by a breaking into two subtrees $\Gamma_1$ and $\Gamma_2$, then $\ov{\cS}_{\Gamma, *}$ resp. $\ov{\cT}_{\Gamma, *}$ is the product 
\[
\ov{\cS}_{\Gamma_1, *}\boxtimes \ov{\cS}_{\Gamma_2, *}\ {\rm resp.}\ \ov{\cT}_{\Gamma_1, *}\boxtimes \ov{\cT}_{\Gamma_2, *}
\]
{ where $$\ov{\cS}_{\Gamma_1, *}\boxtimes \ov{\cS}_{\Gamma_2, *} =
\pi_1^{-1} (\ov{\cS}_{\Gamma_1, *}) \times \pi_2^{-1}(  \ov{\cS}_{\Gamma_2, *}) $$} etc. 

\item The characteristic functions of $\ov{\cS}_{\Gamma, *}$ and $\ov{\cT}_{\Gamma, *}$, viewed as maps
  from $\ov{\cU}_\Gamma$ to $\{0, 1\}$, are local maps. 
\end{enumerate}
\end{lemma}

The proof is left to the reader.  We need to specify certain Banach
space norms on perturbation data. After taking away the
open sets $\ov{\cS}_{\Gamma, J}$ and
$\ov{\cT}_{\Gamma, F}$, the surface part and the tree part of the
universal curve $\ov{\cU}_\Gamma$, the complements
\[
\ov{\cS}_{\Gamma} \setminus \ov{\cS}_{\Gamma, J}\ {\rm resp.}\ \ov{\mc S}_\Gamma \setminus \ov{\mc S}_{\Gamma, H}\ {\rm resp.}\ \ov{\cT}_{\Gamma} \setminus \ov{\cT}_{\Gamma, F}
\]
are smooth manifolds. To measure the norms of smooth functions we
choose Riemannian metrics on these complements in a way that the
metrics are local functions on the universal curve and respect
degeneration of curves. We omit the details.

\subsubsection{Spaces of almost complex structures}

In this section, we introduce domain-dependent perturbations and their spaces. The Cieliebak-Mohnke method in \cite{Cieliebak_Mohnke} provides for each energy bound $E>0$ an open neighborhood ${\mc J}_{\rm tame}^E(X, \omega)$ of the base almost complex structure $J_0$ consisting of almost complex structures $J$ for which all nonconstant $J$-holomorphic spheres $u: S^2 \to X$ of energy at most $E$ intersect the Donaldson hypersurface $D$ at finite but at least three points, that is, $\# u^{-1}(D) \ge 3$. On the other hand, the domain types, especially the numbers of interior markings 
$\ul{z}$
provide {\it a priori} bounds for energy, which allow us to define suitable spaces of perturbations.

\begin{notation}
Let $\Gamma$ be a stable domain type.  Let
\begin{equation}\label{energybound}
E(\Gamma):= \frac{\# {\rm Leaf}_\bullet(\Gamma) + 1}{k} \in \Q
\end{equation}
where $k$ is the degree of the Donaldson hypersurface $D$. 
\end{notation}

We define suitable spaces of almost complex structures that do not
allow holomorphic spheres in the Donaldson hypersurface.  Assume that
$ J_0 \in {\mc J}_{\rm tame}(X, \omega) $ is a {\em base almost
  complex structure} satisfying the conditions in Proposition
\ref{countable}.

\begin{lemma} \label{lemma218} \cite[Corollary 8.16]{Cieliebak_Mohnke}
  For any $E>0$, there exists an open neighborhood
  ${\mc J}_{\rm tame}^E(X, \omega) \subset {\mc J}_{\rm tame}(X,
  \omega)$
  of $J_0$ in the $C^\infty$-topology satisfying the following
  property: For every $J \in {\mc J}_{\rm tame}^E(X, \omega)$, all
  nonconstant $J$-holomorphic spheres with energy at most $E$
  intersect $D$ in finitely many but at least three points.
\end{lemma}

\subsubsection{The space of domain-dependent perturbations}

\begin{definition}\label{defn219}
A {\em perturbation datum} for a stable domain type $\Gamma$ is a collection $P_\Gamma = (J_\Gamma, H_\Gamma, F_\Gamma, M_\Gamma)$ consisting of
\begin{enumerate}

\item A {\em domain-dependent almost complex structure}
\[
J_\Gamma: \ov{\cS}{}_\Gamma \to \J_{\rm tame} (X, \omega)
\]
satisfying the following conditions.
\begin{enumerate}
    \item For any vertex $v \in {\rm Vert} (\Gamma)$, let $\Gamma_{(v)}$ be the maximal subtree containing $v$ which has no boundary edges with positive length. Then 
    \[
    J_\Gamma( \ov{\mc U}_{\Gamma, v}) \subset \J_{\rm tame}^{E(\Gamma_{(v)})}(X, \omega).
    \]
    Here $E(\Gamma)$ is the energy bound defined by \eqref{energybound}.
    
    \item $J_\Gamma$ is equal to the base almost complex structure $J_0$ over the open subset $\ov{\cS}_{\Gamma, J}$ and in a fixed neighborhood of $D \cup \{p\}$. 
\end{enumerate}
The last condition implies that $J_\Gamma$ extends canonically to a map on $\ov{\mc U}_\Gamma$.

\item A {\it domain-dependent Hamiltonian perturbation}
\[
H_\Gamma: \ov{\mc S}_\Gamma \to \Gamma( T(\ov{\mc S}_{\Gamma})^* / T \ov{\mc M}_\Gamma)^*)\otimes C^\infty(X).
\]
that is zero over the open set $\ov{\mc S}_{\Gamma, H}$. Here $T(\ov{\mc S}_{\Gamma} / \ov{\mc M}_\Gamma)$ is the vertical tangent bundle, which is a smooth vector bundle away from nodal points. The last condition implies that $H_\Gamma$ extends canonically to a map defined over $\ov{\mc U}_\Gamma$.

\item A {\em domain-dependent smooth function}
\[
F_\Gamma: \ov{\cT}_\Gamma \times \left(\bigsqcup_{(L, L')
    \in {\LL}^2} (L \cap L' ) \right) \to {\mb R}
\]
that is zero over the open set $\ov{\mc T}_{\Gamma, F}$. 

\item A {\em domain-dependent perturbation of the evaluation map} given by a collection of continuous maps for the interior inputs
\[
M_{\Gamma, e}: \ov{\cM}_\Gamma \to {\rm Diff}(X)\ \forall e \in {\rm Leaf}(\Gamma)
\]
that are smooth in the interior ${\cM}_\Gamma$ (with respect to the manifold structure of ${\cM}_\Gamma$). Each $M_{\Gamma, e}$ can be viewed as a map from the universal curve by pullback via $\ov{\cU}_\Gamma \to \ov{\cM}_\Gamma$.
\end{enumerate}
Moreover, the tuple $P_\Gamma = (J_\Gamma, H_\Gamma, F_\Gamma, M_\Gamma)$ can be viewed as a map from $\ov{\cU}_{\Gamma}$ to a
certain set. We require that this map is a local map (see Definition
\ref{locality}).
\end{definition} 

We take perturbations in a small-in-the-Floer-norm neighborhood of the base perturbation. Given a sequence of positive numbers $\eps = (\eps_i)_{i=1}^\infty$ Floer's (complete) $C^\eps$-norm on functions on a Riemannian manifold is defined by
\[
\| f \|_{{C^\eps}}:= \sum_{i=0}^\infty \eps_i \| \nabla^i f \|_{C^0}.
\]
For a suitably chosen sequence $\eps$, in all dimensions the space of smooth functions with finite $C^\eps$-norms contains bump
functions of arbitrary small supports (see \cite{Floer_unregularized}).  For each stable $\Gamma$, there is a base perturbation datum in which $J_\Gamma$ is the base almost complex structure $J_0$ specified by Lemma \ref{countable}, $H_\Gamma = 0$, $F_\Gamma$ is the given Morse function, and $M_\Gamma = {\rm Id}_X$. The tangent space of ${\mc J}_{\rm tame}(X, \omega)$ at
$J_0$ is
\[
T_{J_0} {\mc J}_{\rm tame}(X, \omega) = \{ \zeta \in {\rm End}(TX) \ |\ J_0 \xi + \xi J_0 = 0 \}.
\]
For $\delta>0$ sufficiently small we identify the $\delta$-neighborhood of $J_0$ in ${\mc J}_{\rm tame}(X, \omega)$ with
respect to the $C^0$-norm with the $\delta$-ball of the tangent space $T_{J_0} {\mc J}_{\rm tame}(X, \omega)$. Then a domain-dependent almost complex structure $J_\Gamma: \ov{\cS}_{\Gamma} \to {\mc J}_{\rm tame}(X, \omega)$ that is $C^0$-close to the base $J_0$ can be viewed as a vector in the linear space $C^\infty( \ov{\cS}_\Gamma \setminus \ov{\cS}_{\Gamma, J}, T_{J_0}
{\mc J}_{\rm tame}(X, \omega))$ so one can measure its norms. Similarly, a domain-dependent
diffeomorphism $M_{\Gamma, e}: X \to X$ that is $C^0$-close to the
identity can be identified with a $C^0$-small vector field on $X$,
denoted by $M_{\Gamma, e} - {\rm Id}_X$. On the other hand, $H_\Gamma$ and $F_\Gamma$
are elements of certain vector spaces. For each stable $\Gamma$, define
\begin{equation} {\mc P}_\Gamma:= \Big\{ P_\Gamma = (J_\Gamma, H_\Gamma,
  F_\Gamma, M_\Gamma)\ |\ \|J_\Gamma - J_0\|_{C^\eps} + \| H_\Gamma \|_{C^\epsilon} +  \| F_\Gamma
  \|_{C^\eps} + \| M_\Gamma - {\rm Id}_X \|_{C^\eps} < \infty \Big\}.
\end{equation}
This set with the $C^\eps$-norm is a separable Banach manifold (in
fact an open set of a separable Banach space).
 
 Once a perturbation datum for a stable domain type is
 fixed we obtain perturbations for not-necessarily-stable types as
 follows. Let $C$ be a treed disk of domain type $\Gamma$ not
 necessarily stable.  Let \label{Cst} $C^{st}$ denote its
 stabilization, obtained by collapsing unstable components of $S$ and
 gluing the associated edges.  The stabilization $C^{st}$ is naturally
 identified with a fiber of the universal curve $\U_{\Gamma^{st}}$ for
 the type $\Gamma^{st}$.  Via the stabilization map $C \to C^{st}$ the
 perturbation data $P_{\Gamma^{st}}$ pulls back to perturbation data
 $P_{\Gamma}$ for $\Gamma$.

\subsection{Holomorphic treed disks}\label{subsection25}

Holomorphic treed disks are combinations of holomorphic disks and gradient flow segments. We first state the assumptions on the boundary conditions. Let $(X, \omega)$ be a compact symplectic manifold. For each pair of Lagrangians $(L, L') \in \cL^2$ (including the case $L = L'$) let
\[ F_{L, L'}: L \cap L' \to \R. \]
be a Morse function on the clean intersection. Its critical points will be asymptotic constraints for gradient rays. In order to obtain strict units, we expand the set of critical points as follows. Define
\begin{equation}\label{expand_critical_point_1}
{\mc I}(L, L') =  \left\{ \begin{array}{ll} {\rm crit}(F_{L, L'}),\ &\ L \neq L',\\
                                                      {\rm crit}(F_{L, L'}) \cup {\mc I}_L^{hu},\ &\ L = L'. \end{array}\right. 
\end{equation}
where 
\begin{equation}\label{expand_critical_point_2}
{\mc I}_L^{hu} = \{ 1_{L, c}^\greyt, 1_{L, c}^{\circt} \}.
\end{equation}

Interior labelling data provide constraints of maps at interior markings. 
The stabilizing divisor $D \subset X$ which intersects each $\iota_i$ transversely. Denote
\begin{equation} \label{IX} 
{\mc I}_{X}:= \cI_{X,\on{Stab}} \sqcup \cI_{X,\on{Bulk}} \sqcup {\mc I}_{X,\on{Mix}}
\end{equation} 
where 
\begin{eqnarray*} 
\cI_{X,\on{Stab}} &:=& \{ (D, 1), (D, 2) \}, \\
\cI_{X,\on{Bulk}} &:=& \{  {\mf b}_i\ |\ i = 1, \ldots, N \}, \\
{\mc I}_{X,\on{Mix}} &:= & \{ D \cap {\mf b}_i\ |\ i = 1, \ldots, N\} \cup \{ {\mf b}_i \cap {\mf b}_j\ |\ i \neq j \}.
\end{eqnarray*}
which will be used to label all possibly interior constraints (eventually corresponding to whether each interior leaf 
$T_e$ corresponds to a Morse trajectory, intersection with stabilizing divisor $D$, or intersection with the bulk deformation
$\bb$.) The elements $(D, 1)$, $(D, 2)$ will indicate tangency order to the stabilizing divisor, so a map $u: C \to X$ with constraint $z_e \in C$ of type $(D,1)$ has $u(z_e) \in D$ resp. of type $(D, 2)$ has $u(z_e)$ and the normal derivatives of $u$ at $z_e$ vanish.

\begin{definition}[Map types]\label{maptype}
Given a domain type $\Gamma$ of treed disks with $d$ inputs, a {\em map type} consists of 
\begin{enumerate}
\item A {\em boundary constraint datum} given by a sequence of Lagrangian
  branes
\[
\ul{L}:= (L_0, L_1, \ldots, L_d)
\]
labelling the boundary components of treed disks. For each boundary edge $e \in {\Edge}(\Gamma_\circ)$ there is then an ordered pair $(L_{e, -}, L_{e, +})$ of branes induced from the datum $\uds L$. Abbreviate
\[ L_e:= (L_{e, -} \cap L_{e, +}) \]
\item A {\em corner constraint datum} given by a sequence of elements
\[
\uds x:= (x_0, x_1, \ldots, x_d) \in {\mc I}(L_0, L_d) \times {\mc I}(L_0, L_1) \times \cdots \times {\mc I}(L_{d-1}, L_d)
\]
satisfying the following requirement regarding the weighting types. The $i$-th leaf $e_i$ is forgettable resp. weighted if
and only if for some $L \in \LL$
\[ x_i = 1_L^\circt \quad \text{resp.} \quad  x_i = 1_L^\greyt.\]

\item A {\em homology datum} which is a map 
\[ \uds{\beta}: {\rm Vert}(\Gamma) \to H_2(X, |{\mc L}|). \]

\item An {\em interior constraint datum} which is a collection of labels
\begin{equation} \label{ullam}
 \ul{\lambda}: {\rm Leaf}_\bullet(\Gamma) \to {\mc I}_X\end{equation} 
such that on each maximal subtree of $\Gamma$ which has no boundary edges with positive length, there is at most one interior marking $z_e$ labelled by $(D ,2)$.

\end{enumerate}

A map type is denoted by ${\bGamma} = (\Gamma, \uds{x}, \uds{\beta}, \uds{\lambda})$ (notice that $\uds x$ determines $\uds L$). We write $\bGamma \mapsto \Gamma$ if the underlying domain type of $\bGamma$ is $\Gamma$. 
\end{definition}
%

Perturbed treed holomorphic disks are defined by allowing the almost
complex structure, and Morse function to vary in the domain. Let
$\Gamma$ be a domain type (not necessarily stable). Let $\Gamma^{\rm st}$ be the stabilization of $\Gamma$ (which is not empty). Let $C$ be a treed disk of type $\Gamma$ and $C^{\rm   st}$ its stabilization which is of type $\Gamma^{\rm st}$. Suppose we are given a perturbation datum $P_{\Gamma^{\rm st}}$ for type $\Gamma^{\rm st}$.  On each surface part $S_v$ of $C$,
$P_{\Gamma^{\rm st}}$ induces a domain-dependent almost complex
structure $J_v$, and a domain-dependent Hamiltonian function $H_v$. On each tree part $T_e$ of $C$, $P_{\Gamma^{\rm st}}$ induces a domain-dependent function
\[
F_{e}: T_e \times \bigsqcup_{(L_-, L_+)\in {\mc L}^2} (L_- \cap L_+) \to {\mb R}.
\]
These data allow one to define the equations on each component. For each
surface component $S_v$ and a smooth map $u_v: S_v \to X$, define
\begin{multline*}
\ov\partial_{J_v, H_v} u_v:= (\d_{H_v} u_v)^{0,1} :=(\d u_v + X_{H_v}(u_v) )^{0,1}\\
:= \frac{1}{2} ( \d u_v + J_v \circ
\d u_v \circ j_v)  +  (X_{H_v}(u_v))^{0,1}\in \Omega^{0,1}(S_v, u_v^* TX). 
\end{multline*}
We say that $u_v$ is $(J_v, H_v)$-holomorphic if $\ov\partial_{J_v, H_v} u_v = 0$.
For each tree component $T_e$ and a smooth map
\[
u_e: T_e \to \bigsqcup_{(L_-,L_+)} L_- \cap L_+
\]
we say that $u_e$ is a perturbed negative gradient segment if
\[
u_e'(s) + \nabla F_{e}(s, (u_e(s))) = 0.
\]

\begin{definition} \label{tdisk} Let
  ${\bGamma} = (\Gamma, \ul{x}, \ul{\beta}, \ul{\lambda})$ be a map
  type with underlying combinatorial type $\Gamma$ of treed disks. Let
  $C = S \cup T$ be a treed disk of type $\Gamma$. Let
  $\Gamma^{\rm st}$ be the stabilization of $\Gamma$ and
  $P_{\Gamma^{\rm st}}$ be a perturbation datum on
  $\ov{\cU}_{\Gamma^{\rm st}}$.  A {\em
    $P_{\Gamma^{\rm st}}$-perturbed adapted treed holomorphic map}
  from $C$ to $X$ of map type ${\bGamma}$ is a 
  continuous map $u : C \to X$ satisfying the following conditions
  (using notations specified before this definition).

\begin{enumerate}
    
\item The restriction of $u$ to the surface component $S_v$, denoted by $u_v: S_v \to X$, is $(J_v, H_v)$-holomorphic; moreover, if $v\in {\rm Vert}_\circ(\Gamma)$, then $u_v$ maps each component of the boundary of $S_v$ to the Lagrangian in $\LL$ labelled by $\bGamma$.

\item The restriction of $u$ to the tree component $T_e$ is contained
  in $L_e$, denoted by $u_e: T_e \to L_e$, is a perturbed negative
  gradient segment, namely
\[ u_e'(s) + \nabla F_{e}(s, (u_e(s))) = 0. \]

    
\item For each semi-infinite edge $e$, the map $u_e$ converges to the
  limit specified by the datum $\ul{x}$ in the sense that
\[ \ev_e(u) := \lim_{s \to \infty} u(s) = x_e , \quad \forall e \in
\Edge_{\rightarrow}(\Gamma) \cap \Edge_\circ(\Gamma) \] 
where $s \in \pm [0,\infty)$ is a coordinate on the (incoming or
outgoing) edge $T_e$.

\item For each interior leaf $e$ attached to a vertex $v_e$, we require (see notations in \eqref{IX})
\[
u_{v_e}(z_e) \in \left\{ \begin{array}{ll}  D,\ &\ \uds\lambda(e) = (D, 1)\ {\rm or}\ (D, 2),\\
                                           M_{\Gamma, e}^{-1}( \bb_i),\ &\ \uds \lambda(e) = \bb_i,\\
                                           M_{\Gamma, e}^{-1}( D \cap \bb_i),\ &\ \uds \lambda(e) = D \cap \bb_i \end{array} \right.
                                           \]
Here $M_{\Gamma, e}: X \to X$ is the diffeomorphism contained in the perturbation datum. Moreover, if $\uds\lambda(e) = (D, 2)$ and if $u_{v_e}$ is not a constant map, then the tangency order of $u_{v_e}$ with $D$ is $2$.
\end{enumerate} 
The triple $(C, u)$ is called an (adapted) {\em treed holomorphic disk} of map type ${\bGamma}$. \end{definition}

Isomorphisms of perturbed treed holomorphic disks are defined in a way similar to that for stable pseudoholomorphic maps. A perturbed treed holomorphic disk is called {\em stable} if its automorphism group is finite, or equivalently
\begin{enumerate}
\item every sphere component $u_v: S_v \to X$ with $\d u_v = \d_{H_v} u_v \equiv 0$ has at least three special points, and 

\item every disk component $u_v: S_v \to X$ with $\d_{H_v} u_v \equiv 0$ either has at least three boundary special points, or one boundary special point and one interior special point, or at least two interior special points. 

\item over each infinite edge $T_e \subset C$ the map $u_e: T_e \to L_e$ is nonconstant.
\end{enumerate}
Given a map type ${\bGamma} = (\Gamma, \uds x, \uds \beta, \uds \lambda)$, denote by 
$
{\cM}_{\bGamma}(P_{\Gamma^{\rm st}})
$
the set of isomorphism classes of stable $P_{\Gamma^{\rm st}}$-perturbed adapted treed holomorphic disks of
map type $\bGamma$. One can also define a Gromov topology and compactify the moduli spaces (we omit the details). We only consider the compactification for the case $\Gamma$ is stable. In this case, the Gromov compactification is
\begin{equation}
\ov{\cM}_{\bGamma}(P_\Gamma):= \bigsqcup_{\bPi \preceq \bGamma} {\cM}_{\bPi} \left(P_\Gamma|_{\ov{\cU}_{\Pi^{\rm st}}} \right).
\end{equation}
Here the partial order $\bPi \preceq {\bGamma}$ naturally extends the partial order $\Pi \preceq \Gamma$ among domain types (see Remark \ref{rem:partialorder} and below).

\begin{definition}[Partial order among map types]
Let ${\bGamma} = (\Gamma, \ul{x}, \ul{\beta}, \ul{\lambda})$ and ${\bGamma}' = (\Gamma', \ul{x}', \ul{\beta}', \ul{\lambda}')$ be two map types. We write ${\bGamma}' \preceq {\bGamma}$ if $\Gamma' \preceq \Gamma$ (which induces a morphism $\psi: \Gamma' \to \Gamma$ and a natural identification ${\rm Leaf}_\bullet(\Gamma) \cong {\rm Leaf}_\bullet(\Gamma')$ with respect to which $\uds x = \uds x'$) and 
\[
\beta(v) = \sum_{v' \in \psi^{-1}(v)} \beta'(v');
\]
moreover, for each interior leaf $e\in {\rm Leaf}_\bullet (\Gamma)$ with corresponding $e'\in {\rm Leaf}_\bullet (\Gamma')$, one has $X_{e'} \subset X_e$.
%
\end{definition}

The composition laws of Fukaya algebras rely on the following relation
among perturbation data.

\begin{definition}[Coherent perturbations]\label{defncoherent}
A collection of perturbation data
\[ \ul{P}:= ( P_\Gamma)_{\Gamma} \]
for all stable domain types $\Gamma$ are called a {\em coherent system of perturbation data} if the following conditions are satisfied.

\begin{enumerate}

\item {\rm (Cutting-edges axiom)} If a (boundary) breaking separates $\Gamma$
  into $\Gamma_1$ and $\Gamma_2$ then $P_\Gamma$ is the product of the
  perturbations $P_{\Gamma_1},P_{\Gamma_2}$ under the isomorphism
  \[
    \ov{\cU}_\Gamma \simeq \pi_1^* \ov{\cU}_{\Gamma_1} \cup \pi_2^*
    \ov{\cU}_{\Gamma_2}.
  \]
  
\item {\rm (Degeneration axiom)} If $\Gamma'\prec \Gamma$, then the
  restriction of $P_\Gamma$ to $\ov{\cU}_{\Gamma'}$ is equal to
  $P_{\Gamma'}$. Notice that these degenerations include the case that the weight 
  $\on{wt}(e) $ on one weighted edge $e$ limits to $0$ or $1$ (see Remark \ref{rem:partialorder}).

\item {\rm (Forgetful axiom)} For a forgettable boundary input $e \in \Edge_\circ(\Gamma)$, let $\Gamma_e$ be the domain type obtained from $\Gamma$
  by forgetting $e$ and stabilizing. Then $P_{\Gamma}$ is equal to the
  pullback of $P_{\Gamma_e}$ via the contraction
  $\ov{\cU}_{\Gamma} \to \ov{\cU}_{\Gamma_e}$.
  
\end{enumerate}
\end{definition}


\subsection{Transversality}\label{regularize}\label{subsection26}

In this subsection we regularize the moduli spaces used in our construction. We first review very briefly the Fredholm theory
associated to treed holomorphic maps.  Let ${\bGamma}$ be a map type. We specify Sobolev constants $k \in \N$ and $p>0$ and a {\em decay constant} $\delta>0$ with 
$kp > 2$ \label{kp} and $\delta$ sufficiently small. The set  $
{\mc B}^{k,p, \delta}(C, {\bGamma}) $
of maps of type ${\bGamma}$ has the structure of a Banach manifold.  In the case without branch changes in the boundary condition, an element $u \in {\mc B}^{k,p,{\delta}}(C, {\bGamma})$ is defined as in
Definition \ref{tdisk} without requiring the holomorphic curve and
gradient flow equations and {instead requiring} $u$ to be of class
$W^{k,p}_{\rm loc}$ over each surface or tree component.  In the case
with branch changes, that is, for each disk component $S_v \cong {\mb D} \subset C$ with a boundary node or marking $z \in \partial S_v$ with two sides of $z$ are labelled by two different Lagrangian submanifolds, then we require the map $u$ is of class $W^{k, p, \delta}$ with respect to a cylindrical type metric, {which means
that it differs from a map constant near infinity by exponentiation of a section of class $W^{k,p,\delta}$}.  If the two sides of $z$ are labelled by the same Lagrangian submanifold (with possibly different branes structures), then we alternatively require that the map is of class $W^{k, p}$ with respect to the smooth metric. Tangency conditions for
a maximal order $m$ {as in
  \eqref{IX} are defined} for $k$ sufficiently large. Choose a
perturbation datum $P_\Gamma = (J_\Gamma, H_\Gamma, F_\Gamma, M_\Gamma)$. Over
the Banach manifold 
${\mc B}^{k,p, \delta}(C, {\bGamma}) $ there is a Banach vector bundle
${\mc E}^{k-1,p,{\delta}}(C, {\bGamma})$ of $0,1$-forms of class $k-1,p$ so
that the defining equations of Definition \ref{tdisk} provides a
section
\[
{\mc F}: {\mc B}^{k,p,{\delta}}(C, {\bGamma} ) \to {\mc E}^{k-1,p,{\delta}}(C, {\bGamma})
\]
{combining the perturbed Cauchy-Riemann operators on the surface parts
and gradient flow operators on the edges}.    To include the variations of the domains, one takes an open neighborhood $\M^i_\Gamma\subset {\cM}_\Gamma$ of $[C]$ over which the universal curve $\U_\Gamma$ has a trivialization
\[ \U_\Gamma|_{{\cM}_\Gamma^i} \cong {\cM}_\Gamma^i \times C. \]
The linearization of the map ${\mc F}$ at a perturbed treed holomorphic disk $(C, u)$ is a Fredholm operator
\[ \ti{D}_u : T_{(u, \partial u)} {\mc B}^{k,p,{\delta}}(C, {\bGamma}) \times
T_{[C]} {\cM}_\Gamma \to {\mc E}^{k-1,p,{\delta}}(C, {\bGamma})|_{u}. \]
Since the Lagrangians are always totally real with respect to the domain-dependent almost complex structures, the linearized operator
is Fredholm.  Its index can be calculated using Riemann-Roch for surfaces with boundary and gives the expected dimension of the moduli space 
\[
{\rm ind}(\Gamma) := {\rm dim} {\cM}_{\bGamma}(P_\Gamma) = \on{Ind}(\ti{D}_u) = {\rm dim} {\cM}_\Gamma + \mu(\ul{\beta}) + i(\ul{x}) - i(\ul{\lambda})
\]
where $\mu(\ul{\beta})$ is the total Maslov index of the disk class, $i(\ul{x})$ is the sum of Morse indices of asymptotic constraints, and $i (\ul{\lambda})$ is the effect of interior constraints. For example, if $\Gamma$ has $k$ interior leaves, all of which are labelled by $(D, 1)$, then $i(\ul{\lambda}) = 2k$. 

Following Cieliebak-Mohnke \cite{Cieliebak_Mohnke}, we introduce a collection of map types for which transversality can be achieved
by domain-dependent perturbations.

\begin{definition}\label{uncrowded}
A map type ${\bGamma} = (\Gamma, \ul{x}, \ul{L}, \ul{\lambda})$ is called {\em uncrowded} if each ghost sphere bubble tree contains at most one interior leaf $e$ whose interior constraint is $(D, 1)$ or $(D, 2)$. Otherwise ${\bGamma}$ is called {\em crowded}.
\end{definition} 

\begin{remark} 
Cieliebak-Mohnke perturbations can never make crowded configurations $u: C \to X$ transversely cut out, since one can replace an interior leaf $T_e$ with a given label $D$ and replace it with a sphere bubble $S_v$ with two interior leaves $T_{e_1}, T_{e_2}$ attached with the same label $D$, which reduces the expected dimension of a stratum by two.  Repeating this process eventually produces a non-empty moduli space of negative expected dimension, which is a contradiction if the perturbations are regular.
\end{remark}

We will need certain forgetful maps to treat crowded configurations. Let $\Gamma$ be a stable domain type.  Choose a subset 
\[
W \subset {\rm Vert}_\bullet (\Gamma) = {\rm Vert}(\Gamma) \setminus {\rm Vert}(\Gamma_\circ)
\]
of spherical vertices. Define $\Gamma_W$ to be the domain type obtained by the following operation:   For each connected component $W_i \subset W$, remove all interior leaves except the one with the largest labelling on $W_i$, and stabilize the remaining configuration.  The set $W$ descends to a (possibly empty) subset $W' \in {\rm Vert}_{\circ}(\Gamma_W)$. A  consequence of the locality condition on the perturbation data is that each $P_\Gamma \in {\mc P}_\Gamma$ descends to a perturbation datum $P_{\Gamma_W} \in {\mc P}_{\Gamma_W}$ whose restriction to surface components $S_v$ for $v \in W'$ equals to the base almost complex structure $J_0$ and the zero Hamiltonian perturbation.\footnote{The descent $P_{\Gamma_W}$ may not agree with a member of any prechosen coherent collection of perturbation data.} Let 
$
{\mc P}_{\Gamma_W, W'} \subset {\mc P}_{\Gamma_W}
$
be the subset of perturbations that agree with the base almost complex structure $J_0$ and the zero Hamiltonian over surface components corresponding to vertices in $W'$. This forgetful construction gives a smooth map of Banach manifolds
\begin{equation}\label{forgetw}
{\mc P}_{\Gamma} \to {\mc P}_{\Gamma_W, W'}.
\end{equation}
Indeed, this is 
a surjective map and essentially a projection, hence admits a smooth right inverse. 

\begin{definition}\label{stronglyregular}
  Let $\Gamma$ be a stable domain type. A perturbation
  $P_\Gamma \in {\mc P}_\Gamma$ is called {\em regular} if all
  uncrowded maps of type ${\bGamma}$ with underlying domain type
  $\Gamma$ are regular.  The perturbation $P_\Gamma$ is called
  {\em strongly regular} if for any subset
  $W \subset {\rm Vert}_\bullet(\Gamma)$ and for any uncrowded map
  type ${\bGamma}_W$ whose underlying domain type is $\Gamma_W$ and whose homology classes on surface components
  corresponding to vertices in $W'$ are zero, every map of type
  $\bGamma_W$ is regular.
\end{definition}

The main result of this section is the regularity of moduli spaces for uncrowded map types and the selection of a coherent collection of perturbation data. 

\begin{theorem}\label{regular} There exist a
  coherent collection of perturbation data $\ul{P} = (P_\Gamma)$ whose
  elements $P_\Gamma$ are all strongly regular. \end{theorem}

\begin{proof} 
The proof is an induction on the possible domain types according to the
  partial order (see Remark \ref{rem:partialorder}). First we introduce an equivalence relation among stable domain types.  We write $\Gamma \sim \Pi$ for the equivalence relation generated by \label{genby}
\[
\Pi \preceq \Gamma\ {\rm and}\ \rho|_{\Pi_\circ}: \Pi_\circ \to
\Gamma_\circ\ {\rm is\ an\ isomorphism};
\]
that is, if roughly they have isomorphic disk part. Here $\rho$ is the tree map induced from the partial order relation $\Pi \preceq \Gamma$ (see Remark \ref{rem:partialorder}). Let $[\Gamma]$
denote the equivalence class of $\Gamma$. The partial order relation
among domain types descends to an equivalence relation among their
equivalence classes.

The inductive step is the following.  Fix an equivalence class $[\Gamma]$. Suppose we have chosen strongly regular perturbation data $P_\Pi$ for all stable domain types $\Pi$ with $[\Pi] \prec [\Gamma]$ as well as domain types with strictly fewer boundary inputs or the same number of boundary inputs but strictly fewer interior leaves, such that the chosen collection is coherent in the sense of Definition \ref{defncoherent}. 

\begin{definition}  For each $\Gamma$ in this class $[\Gamma]$, 
denote by 
$
{\mc P}_\Gamma^* \subset {\mc P}_\Gamma
$
the closed Banach submanifold
consisting of perturbation data whose values over all lower strata $\ov{\cU}_\Pi$ with $\Pi \prec \Gamma$ and $[\Pi] \prec [\Gamma]$ agree with the prechosen one $P_\Pi$. 
\end{definition} 

We prove the following sublemma.

\vspace{0.1cm}

\noindent {\bf Sublemma.} There is a comeager subset ${\mc P}_\Gamma^{*, {\rm reg}} \subset {\mc P}_\Gamma^*$ whose elements are regular. 

\vspace{0.1cm}

\noindent {\it Proof of the sublemma.}
Let $\M^i_\Gamma$ be a subset of $\M_\Gamma$ over which the universal curve $\U_\Gamma$ is trivial, and $\U_\Gamma^i$ the restriction of $\U_\Gamma$ to $\M^i_\Gamma$. For each uncrowded map type ${\bGamma}$ with underlying domain type $\Gamma$, consider the universal moduli space
\[
\M^{i,\univ}_{\bGamma}({\mc P}_\Gamma^*) = \{ ([u: C \to X], P_\Gamma) | P_\Gamma \in {\mc P}_\Gamma^*,\ C \subset \U_\Gamma^i, \ [u] \in {\cM}_{\bGamma}(P_\Gamma) \}.
\]
of maps with domain in $\U_\Gamma^i$ together with a perturbation
datum $P_\Gamma$. By the Sard-Smale theorem, this sublemma can be proved
once we show the regularity of the local universal moduli
space. Suppose this is not the case, so that for some $(u, P_\Gamma)$
the linearization of the defining equation of the universal moduli is
not surjective. 
Then there exists a nonzero section $\eta$ in the $L^2$-orthogonal complement of the image of the linearized operator, or equivalently, in the kernel of the formal adjoint of the linearized operator. By elliptic regularity, $\eta$ is actually smooth. 
%
We will derive a contradiction by showing that each component of $\eta$ vanishes identically on that component.

\noindent {\it Step One. The form $\eta$ vanishes on any nonconstant sphere component  $u_v: S_v \to X$.}
\label{izero} By our assumption on the domain-dependent Hamiltonian perturbation $H_\Gamma$ (see Lemma \ref{lemma216} and Definition \ref{defn219}), $H_\Gamma$ vanishes on spherical components. Since the support of the perturbation $J_{\Gamma, v}$ has nonzero intersection with $S_v$, the restriction $\eta_v$ of $\eta$ to $S_v$ must vanish over a nonempty open set of $S_v$. The unique continuation principle for first order elliptic equation implies that $\eta_v$ vanishes identically. 

\noindent {\it Step Two. The form $\eta$ vanishes on any disk component $u_v: S_v \cong {\mb D} \to X$ corresponding to the vertex $v \in {\rm Vert}(\Gamma_\circ)$. }  Suppose that $\d_{H_v} u_v$ is not identically zero.
Then $\d_{H_v} u_v$ is nonzero over an nonempty open subset $U \subset S_v$ with $u_v(U)$ is disjoint from the neighborhood of $D$ where $J_v \equiv J_0$.  By orthogonality 
to images of deformation of $J_v$ over $U$,  $\eta_v$ is zero over $U$. Unique continuation principle shows that $\eta_v$ vanishes identically. Suppose $u_v$ is covariantly constant, i.e., $d_{H_v} u_v \equiv 0$. Then $u_v(S_v)$ has a nonempty intersection with the neighborhood of $|\LL|$ where one can perturb the Hamiltonian $H_v$. This again shows that $\eta_v$ vanishes on a nonempty open subset of $S_v$, 
and so vanishes identically on $S_v$ by the unique continuation principle. 

\noindent {\it Step Three. The form $\eta$ vanishes on each edge $T_e$ with positive length $\ell(e) > 0$.}  First, for an edge $T_e$ with positive or infinite length $\ell(e) \in \{ 0, \infty \}$, if the gradient segment
$u_e: T_e \to L_e$ is mapped into a positive dimensional target $L_e$,
then since the support of the perturbation $F_{e}$ is nonempty, it
also follows that the restriction $\eta_e$ to $T_e$ vanishes
identically. If $L_e$ is zero-dimensional, then by definition $\eta_e \equiv 0$.

\noindent {\it Step Four. The form $\eta$ vanishes on each ghost sphere component $S_v \cong \P^1, \d u_v = 0$. }  Let $u_v: S_v \to X$ be a constant map with value $x_v \in X$. For any domain-dependent almost complex structure, the linear map 
\[
\ov\partial_{J_v}: \Omega^0(S_v, T_{x_v} X) \to \Omega^{0,1}( {\mb P}^1, T_{x_v} X)
\]
is surjective with kernel equal to the finite dimensional subspace of
constant vector fields $\xi$ on $S_v$.  However, there might be constraints coming
from special points on this component.  For this we use the
uncrowdedness condition. Consider a maximal ghost sphere tree
$W \subset {\rm Vert}(\Gamma)$. By the above argument for ghost disk components, we may assume that $W$ contains only spherical vertices $v \in \Ver_{\black}(\Gamma)$. There is at most one special point $z_e$ on  the corresponding curve $C_W \subset C$ which is constrained by $(D, m)$; this condition puts a two-dimensional constraint on a constant vector field $\xi$ restricted to $C_W$. For any other interior marking $z_e$, one can use the deformation of the diffeomorphism $M_{\Gamma, e}$ to allow variations in $\xi$ while preserving the constraints.    For any node connecting $C_W$ to a nonconstant component, the
constraints are transversely cut out using the fact that the linearized operator is surjective on deformations on the adjacent nonconstant components even vanishing at the node.   Thus the form $\eta$ vanishes on components in
$W$. \hfill {\it End of the proof of the sublemma.}

Next we construct a comeager subset 
of strongly regular perturbations. For any subset $W \subset {\rm Vert}_\bullet (\Gamma)$, consider the domain type $\Gamma_W$ with a descent subset
$W' \subset {\rm Vert}_\bullet (\Gamma_W)$.  The trees $\Gamma_W$ and $\Gamma$ have isomorphic disk parts. Hence the prechosen perturbations provides a subset ${\mc P}_{\Gamma_W, W'}^* \subset {\mc P}_{\Gamma_W, W'}$ consisting
perturbations whose values are fixed precisely over strata $\Pi'$ with
$[\Pi'] \prec [\Gamma_W]$. Moreover, the forgetful map \eqref{forgetw}
restricts to a forgetful map
\[
\pi_W: {\mc P}_\Gamma^* \to {\mc P}_{\Gamma_W, W'}^*
\]
which has a right inverse given by pullback. By the same argument as
the proof of the above sublemma, there is a comeager subset
${\mc P}_{\Gamma_W, W'}^{*, {\rm reg}}$ consisting of perturbations
$P_{\Gamma_W}$ that regularize moduli spaces
${\cM}_{{\bGamma}_W}(P_{\Gamma_W})$ for map types
${\bGamma}_W$ that are ghost on surface components corresponding to
vertices in $W'$. Define
\[
{\mc P}_{\Gamma}^{*, {\rm s.reg}}:= \bigcap_{W \subset {\rm Vert}_\bullet(\Gamma)} \pi_W^{-1} ( {\mc P}_{\Gamma_W, W'}^{*, {\rm reg}}).
\]
This subset is still comeager  and all its elements are strongly regular. 

Lastly, we choose perturbations for each strata extending the prechosen perturbations on lower-dimensional strata. We define smaller comeager subsets ${\mc P}_\Gamma^{**}$ inductively as follows. If $\Gamma$ is a smallest element of the equivalence class $[\Gamma]$, then define ${\mc P}_{\Gamma}^{**}:= {\mc P}_{\Gamma}^{*, {\rm s.reg}}$. Suppose for a general $\Gamma$ in $[\Gamma]$ one has defined ${\mc P}_{\Gamma'}^{**}$ for all $\Gamma' \prec \Gamma$ with $[\Gamma'] = [\Gamma]$.  Define 
\[
{\mc P}_{\Gamma}^{**}:= {\mc P}_{\Gamma}^{*, {\rm s.reg}} \cap \bigcap_{\substack{\Gamma' \prec \Gamma \\ [\Gamma'] = [\Gamma]}} \pi_{\Gamma, \Gamma'}^{-1}( {\mc P}_{\Gamma'}^{**}).
\]
Here $\pi_{\Gamma, \Gamma'}: {\mc P}_\Gamma \to {\mc P}_{\Gamma'}$ is the map defined by restricting to boundary strata. Then we have defined ${\mc P}_{\Gamma}^{**}$ for all $\Gamma$ in this equivalence class. 
The equivalence class $[\Gamma]$ has a unique maximal element $\Gamma_{\rm max}$. Choose an arbitrary perturbation $P_{\Gamma_{\rm max}}\in {\mc P}_{\Gamma_{\rm max}}^{**}$. By boundary restriction this choice induces $P_\Gamma$ for all $\Gamma$ in this equivalence class. By construction, all these $P_\Gamma$ extend the existing perturbations on lower-dimensional strata.  By  induction, one obtains the claimed collection $\ul{P}$. 
\end{proof}

\begin{remark}\label{rem228}
We could use more restricted types of Hamiltonian perturbations. Indeed, we only need to regularize the potential constant disks at intersections of two or more different Lagrangian submanifolds. Therefore, we only need to turn on the Hamiltonian perturbation in a small open neighborhood of such intersections. In particular, if there is an embedded Lagrangian submanifold $L$ supporting different branes in $\LL$ and $L$ does not intersect other branes in $\LL$, then without using Hamiltonian perturbations the constant disks at points in $L$ are already regular for any almost complex structure. This condition
holds in particular in the blowup case where the exceptional collection of Lagrangian branes are supported on the same Lagrangian torus which does not intersect the ``old'' branes. 
\end{remark}

\subsection{Boundary strata}\label{subsection27}

In this section,  we describe the boundary of the moduli spaces we use, which are those of expected dimension at most one.      Fix a coherent collection of strongly regular perturbation data $\ul{P} = (P_\Gamma)$ and abbreviate all moduli spaces ${\cM}_{\bGamma}(P_\Gamma)$ by ${\cM}_{\bGamma}$.

\begin{definition}\label{essentialtype}
\begin{enumerate}

\item A map type ${\bGamma}$ is called {\it essential} if it has no broken edges
$T_e = T_{e_1} \cup T_{e_2}$, no edges $T_e$ of length zero or infinity except for 
the leaves and root, no spherical components $S_v, v \in \Ver_\black(\Gamma)$, if all interior constraints are either $(D, 1)$ or $\bb$ and for each disk vertex $v \in {\rm Vert}(\Gamma_\circ)$, the number of interior leaves labelled by $(D, 1)$ is equal to $k \omega ( \beta_v)$ where $k$ is the degree of the Donaldson hypersurface.

\item Given asymptotic data $\ul{x} = (x_0, x_1, \ldots, x_d)$ (see Definition \ref{maptype}) and $i = 0, 1$, let 
\[
{\mc M}_{d, 1}(\uds x)_i = {\mc M}_{d, 1}(x_0, x_1, \ldots, x_d)_i
\]
be the union of moduli spaces ${\cM}_{\bGamma}$ for essential map types of expected dimension $i$ whose asymptotic data is $\ul{x}$. 
\end{enumerate}
\end{definition}

\begin{remark}
As in \cite{Wehrheim_Woodward_orientation} the determinant lines of the linearized operators become equipped with orientations induced by relative spin structures. In particular, if all strata of ${\cM}(\ul{x})_0$ are regular then there is a map
\[
\eps: {\mc M}_{d, 1}(\uds x)_0 \to \{\pm 1\}.
\]
\end{remark} 

The following lemma classifies types of topological boundaries of one-dimensional moduli spaces.

\begin{lemma} \label{refinedcompactness} Suppose  $\ul P = (P_\Gamma)$ is a 
coherent and regular collection of perturbations. For an essential map type $\bGamma$ of expected dimension zero, the moduli space ${\cM}_{\bGamma}(P_\Gamma)$ is compact. For a map type ${\bGamma}$ of expected dimension one, the boundary of the compactified one-dimensional moduli $\ov{\cM}_\bGamma(P_\Gamma)$ is the disjoint union of moduli spaces ${\cM}_\bPi(P_\Pi)$ where $\bPi$ is a map type related to $\Gamma$ by exactly one of the following operations:
\begin{enumerate}
    \item collapsing an edge $e \in \Edge(\Pi)$ of length zero;
    
    \item shrinking a finite edge 
    $e \in \Edge(\Pi)$ to length zero or breaking into two semi-infinite edges;
    
    \item in the case when the output edge $e_0 \in \Edge(\Gamma)$ is not weighted, setting  the weight $\rho(e)$ of exactly one weighted input $e \in \Edge(\Gamma)$ to be zero or one; or
    
    \item in the case when the output $e_0 \in \Edge(\Gamma)$ is weighted, changing the weight $\rho(e)$ of exactly one weighted input  $e \in \Edge(\Gamma)$ so that it becomes one. 
\end{enumerate}
\end{lemma}

\begin{proof}[Sketch of proof] It suffices to check sequential compactness.
Let $(C_\nu, u_\nu)$ be a sequence of treed holomorphic disks representing a sequence of points in ${\cM}_{\bGamma}(P_\Gamma)$. By a combination of Gromov compactness for  (pseudo)holomorphic disks and compactness of the moduli space of gradient trajectories, there is a subsequence (still indexed by $\nu$) that converges to a limiting treed holomorphic disk $(C_\infty, u_\infty)$ of certain map type $\bPi$.
We first claim that the domain type $\Pi$ is stable. Suppose on the contrary it is not the case. Then, there is either a domain-unstable disk or a domain-unstable sphere component 
$u_{\infty,v}: 
S_v \to X, v \in \Ver(\bPi)$.
By the stability condition, $u_v$ must be a nonconstant map. Moreover, $u_v$ is pseudoholomorphic with respect to a constant tamed almost complex structure $J  = J_{\Pi^{\rm st}} (z)$ on $X$, where $z \in C_\infty^{\on{st}}$ is the image of $S_v$ in the stabilization. On the other hand, there exists a unique maximal subtree $\Gamma_{(v)} \subset \Gamma$ which contains no boundary edge with positive edge, such that the bubble $u_v$ comes from energy blow up on components belonging to $\Gamma_{(v)}$. Therefore, by the conditions of the perturbation data (see (a) of Definition \ref{defn219}), one has
\[
J \in {\mc J}_{\rm tame}^{E(\Gamma_{(v)})} (X, \omega).
\]
Since the convergence preserves the total energy, the disk or the
sphere has energy at most $E(\Gamma_{(v)})$.  By Lemma \ref{lemma218}, the $J$-holomorphic
map $u_v: S_v \to X$ is not contained in $D$ and
must intersect $D$ in at \label{atat} least three points resp. one point in the sphere resp. disk case.  Topological invariance of intersection numbers, as in Cieliebak-Mohnke \cite{Cieliebak_Mohnke}, implies that for $\nu$ sufficiently large, $u_v$
intersects with $D$ in at least three resp. one nearby point. Since the
type ${\bGamma}$ is essential, all these intersection points are
marked points labelled by $(D, 1)$. The convergence implies
that the intersection points of $u_\infty$ with $D$ must all be
marked points, contradicting the assumption that the domain of the
domain $S_v$ is unstable. Therefore, the domain type $\Pi$ is stable.  Since $\Pi \preceq \Gamma$, the perturbation
datum $P_\Gamma$ induces by restriction a perturbation datum
$P_\Pi$ so that $[(C_\infty, u_\infty)] \in {\cM}_{{\bPi}}(P_\Pi).$

Next we show that type of the limiting map constructed in the previous
paragraph is uncrowded. Suppose this is not the case, then let
$W \subset {\rm Vert}_{\circ}(\Pi)$ be the (nonempty) set of ghost
sphere components. By the locality property, the perturbation data
$P_\Pi$ descends to a perturbation $P_{\Pi_W}$ which is equal to $J_0$
over $W'$. The limiting configuration $[(C_\infty, u_\infty)]$ then
descends to an element
\[ [(C', u')] \in {\cM}_{{\bPi}_W}(P_{\Gamma_W}). \]
Since $\bPi_W$ is uncrowded, the above moduli space is regular and nonempty. However, similar to the argument of \cite{Cieliebak_Mohnke}, the reduction drops the expected dimension by at least two. This contradiction shows that $\bPi$ must be uncrowded.

Finally, we claim $\bPi$ has no sphere components.  Indeed, each sphere component will drop the dimension of the domain moduli space by two and $\Gamma$ has no sphere components. It follows from the dimension formula for ${\cM}_\Gamma$ that when ${\rm dim} {\cM}_{\bGamma} = 0$, $\Pi$ must be identical to $\Gamma$ and hence ${\cM}_{\bGamma}$ is compact. When ${\rm dim} \ {\cM}_{\bGamma} = 1$, the only possibly types of $\Pi$ are described in the above list. 
\end{proof}

Moreover, we distinguish the boundary strata as either {\em true} or {\em fake} boundary components. The true boundaries are those corresponding to edge breaking and weight changing to zero or one while the fake boundaries are those corresponding to disk bubbling or edges shrinking to zero. For  one-dimensional moduli strata ${\cM}_{d, 1} (\ul{x})_1$, define  
\[
\ov{\cM}_{d, 1} (\ul{x})_1 = \bigcup_{{\bGamma}}\ov{\cM}_{\bGamma}
\]
to be the union of all compactified moduli space of expected dimension one while identifying fake boundaries.  Standard gluing constructions (gluing disks or gradient lines) show that $\ov{\cM}_{d, 1} (\ul{x})_1$ a topological 1-manifold with boundary and its cutoff at any 
energy 
level (indexed by the number of interior markings) is compact. The boundaries are strata corresponding to edge breaking and weight changing to zero or one.

\section{Fukaya categories and quantum cohomology}\label{section3}

In this section,  we introduce bulk-deformed Fukaya categories
associated to a given 
rational
finite collection of
Lagrangian immersions. Given the regularization of the moduli spaces
in the Section \ref{regularize} above, the material in this section is
fairly straightforward adaptation of that in Fukaya-Oh-Ohta-Ono
\cite{fooo}.

\subsection{Composition maps} 

In this section,  we apply the transversality results above to construct Lagrangian Floer theory. In the Morse model the generators of the Lagrangian Floer cochains are critical points of a Morse function on the Lagrangian intersection, assumed clean. Recall for each pair $(L, L') \in \LL^2$, the intersection $L\cap L'$ is a smooth manifold. We have chosen a Morse function
\[
F_{L, L'}: L \cap L' \to {\mb R}.
\]
Consider the set of all {\it Lagrangian branes} supported on $\LL$, i.e., 
\[
\widehat {\mc L}:= \big\{ \LB \to L \ |\  L \in {\mc L} \big\}
\]
For each pair of branes $(\LB, \LB')$, denote the set of critical points 
on the intersection
\[
{\mc I}( \LB, \LB' )=  \left\{ \begin{array}{cc} {\rm crit} ( F_{L, L'}),\ &\ \LB \neq \LB',\\
{\rm crit} ( F_{L, L'} ) \cup {\mc I}_{\LB}^{hu},\ &\ \LB = \LB', \end{array} \right.                                   
\]
where ${\mc I}_{\LB}^{hu}$ contains $1_{\LB}^\greyt$ and $1_{\LB}^\circt$ for all connected component $c$ of $L$ (see \eqref{expand_critical_point_1} and \eqref{expand_critical_point_2}). The ``Morse indices'' of the extra generators are defined by
\begin{align*}
&\ {\rm index}_{\rm Morse}(1_{\LB}^\greyt) = {\rm dim}(L \cap L') + 1,\ &\ {\rm index}_{\rm Morse}(1_{\LB}^\circt) = {\rm dim}(L \cap L').
\end{align*}

\subsubsection{Gradings}

In order to obtain graded Floer cohomology groups a grading on the set of generators is defined as follows. Let $N \in \Z$ be an even integer and let 
\begin{equation} \label{mcover}
\pi^N: \Lag^N(X) \to \Lag(X)
\end{equation}
be an $N$-fold Maslov cover of the bundle of Lagrangian subspaces as in Seidel \cite{Seidel_2000}; we always assume that the induced $2$-fold cover $\Lag^2(X) \to \Lag(X)$ is the bundle
of oriented Lagrangian subspaces. A {\em $\Z_N$-grading} of $L \in \LL$ is a lift \label{zngrad} 
\begin{equation} \label{zngradeq}
\xymatrix{ & \Lag^N(X) \ar[d] \\
           L \ar[r] \ar[ru]^{\phi_L^N} & \Lag(X) }\end{equation}
where the horizontal arrow is the map assigning to each $x \in L$ its tangent space. Given such a grading, there is a natural $\Z_N$-valued map
\[ \cI( \LB_0, \LB_1) \to \Z_N, \quad x \mapsto |x| \]
{defined as follows. Recall that for each pair of paths $\lambda_0, \lambda_1: [a, b] \to \Lag( T_x X)$, there is a Maslov index 
\begin{equation} \label{mindex}
\mu(\lambda_0, \lambda_1) \in \frac{1}{2} {\mb Z}
\end{equation}
which is an integer if and only if ${\rm dim} (\lambda_0(a) \cap \lambda_1(a)) \equiv {\rm dim} (\lambda_0(b) \cap \lambda_1(b))\ {\rm mod}\ 2$.  For any $x \in {\mc I}(\LB_0, \LB_1)$, choose two paths 
\[
\tilde \lambda_0, \tilde \lambda_1: [a, b] \to \Lag^N(T_x X)
\]
such that $\tilde \lambda_0(a) = \tilde \lambda_1(a)$ and 
\begin{align*}
&\ \tilde \lambda_0(b) = \phi_{L_0}^N( T_x L_0),\ &\  \tilde \lambda_1(b) = \phi_{L_1}^N( T_x L_1)
\end{align*}
with notation from \eqref{zngradeq}. 
Define 
\[
|x|:= \frac{n}{2} - \mu( \pi^N (\tilde\lambda_0), \pi^N(\lambda_1)) + \frac{1}{2} {\rm dim} (L_0 \cap L_1) - {\rm index}_{\rm Morse}(x) \in {\mb Z}/N{\mb Z}.
\]
with notation from \eqref{mindex} and \eqref{mcover}. 
For example, when $L_0 = L_1$ and $x$ is an ordinary critical point, then $|x|$ is the dimension of the stable manifold of $x$ under the negative gradient flow.

\subsubsection{Weighted counting}

The moduli space of holomorphic disks is non-compact, and to remedy
this the structure maps of the Fukaya algebra are defined over Novikov
rings in a formal variable. The Floer cochain space is a free module
over generators corresponding to Morse critical points 
and the two additional generators from \eqref{expand_critical_point_2} necessary
to achieve strict units. Given two branes $\LB, \LB'$ let
\[CF^\bullet (\LB, \LB' ) = \bigoplus_{ x \in \cI (\LB, \LB' ) } 
\Hom(\LB_x \otimes_{\Lambda^\times} \Lambda,\LB_x' \otimes_{\Lambda^\times} \Lambda) \]
the sum of space of linear maps between the fibers
of the local systems.  The space of Floer cochains is naturally $\Z_N$-graded by 
\[ CF^\bullet (\LB, \LB') = \bigoplus_{k \in \Z_N} CF^k (\LB, \LB'), \quad CF^k (\LB, \LB') = \bigoplus_{x \in \cI^k(\LB, \LB')} \Lambda x.\]
The $q$-valuation on $\Lambda$ extends naturally to $CF^\bullet ( \LB, \LB')$:
\[ \val_q: CF^\bullet ( \LB, \LB' ) - \{ 0 \} \to \R, \quad
\sum_{x \in
  \cI( \LB, \LB' )} c(x) x \mapsto \min(\val_q(c(x))) .\]

The local systems contribute to the coefficients of the composition maps in the expected way. For any holomorphic treed disk $u: C \to X$ with boundary in some collection $\ul{\LB} = (\LB_0, \ldots, \LB_d)$ and mapping the corners to $x_0,\ldots, x_d$, generators
\[ a_i \in 
\Hom(\LB_{i-1,x_i} \otimes_{\Lambda^\times} \Lambda,\LB_{i,x_i} \otimes_{\Lambda^\times} \Lambda) \]
denote by 
\begin{equation} \label{yu}
y(u)  =   T_{d-1} a_{d-1} \ldots a_1 T_0 a_0 T_0 \in 
\Hom(\LB_{0,x_0} \otimes_{\Lambda^\times} \Lambda,\LB_{d,x_0} \otimes_{\Lambda^\times} \Lambda)  
\end{equation} 
the product of parallel transports $T_i$ along the restrictions $u_v | (\partial C)_i$.   For more complicated treed disks the holonomy 
is defined recursively starting with the components furthest away from the root.

\begin{definition}\label{defn_composition} 
Fix a coherent collection of strongly regular perturbation $\ul{P} = (P_\Gamma)_\Gamma$ (whose existence is provided by Theorem \ref{regular}). For each $d \geq 0$ define {\em higher composition maps}
\[
m_d: CF^\bullet ( \LB_{d-1}, \LB_d) \otimes \ldots \otimes CF^\bullet  ( \LB_0, \LB_1) \to CF^{\bullet} ( \LB_0, \LB_d )[2-d]
\]
on generators by the weighted count of treed disks by
\begin{equation}\label{composition}
m_d(a_d, \ldots, a_1)= \sum_{x_0 \in {\mc I}(\LB_0, \LB_d)} \sum_{u \in {\cM}_{d, 1} ( x_0, \ldots, x_d)_0}  (-1)^{\heartsuit} \wt(u)  
\end{equation}
where the weightings
\begin{equation} \label{weightings} 
\wt(u) := c(u,\bb)p(u) y(u) q^{A(u)} o(u) d(u)^{-1}\end{equation} 
are defined as follows:
\begin{itemize}
\item if the domain type of $u$ is $\Gamma$, then 
\begin{equation} \label{du}
 d(u):= ( k A(u) )! = (\# {\rm Leaf}_\bullet(\Gamma))!
\end{equation} 
which is the number of permutations of interior markings $z_e$ mapped into $D$;
\item the coefficient $c(u,\bb)$ is a product of coefficients $c_i$ of the
  bulk deformation, with product taken over interior leaves mapping to $\bb$,
\item the coefficient $p(u)$ is the coefficient $p_i$ of the multivalued perturbation $P_\Gamma$ of \eqref{mv} evaluated at the branch containing $u$,
   and
\item $y(u)$ is the holonomy of the local system as defined in 
\eqref{yu}; 
  \item the exponent $A(u)$ is the symplectic area of the map $u$.
\item the sign $o(u)$ arises from the choice of coherent
  orientations and the overall sign $\heartsuit$ is given by
\begin{equation} \label{heartsuit}  \heartsuit = {\sum_{i=1}^d i|x_i|} .\end{equation}
\end{itemize} 
\end{definition} 

We first define a curved \ainfty category with infinitely many objects 
supported on the given Lagrangian submanifolds.  Later, we will consider a modified definition of the objects so that the \ainfty category is flat.

\begin{theorem} 
For any strongly regular coherent perturbation system $\ul{P} = (P_\Gamma)$ the maps $(m_d)_{d \ge 0}$  define a (possibly curved) \ainfty category $\Fuk_{\LL}^\sim (X, \bb)$ with 
\begin{enumerate}

\item the set of objects given by ${\rm Ob} \left( \Fuk_{\LL}^\sim (X, \bb) \right):= \widehat{\mc L}$,

\item the set of morphisms from $\LB$ to $\LB'$ given by ${\rm Hom}(\LB, \LB'):= CF^\bullet( \LB, \LB')$,

\item the composition maps $(m_d)_{d\geq 0}$  defined by Definition \ref{defn_composition}, and

\item for each object $\LB$ the strict unit equal to $1_{\LB}^\circt\in CF^0 (\LB, \LB)$.
\end{enumerate}
\end{theorem}

\begin{proof} 
We must show that the composition maps $m_d$ satisfy the {\it $A_\infty$-associativity equations} \eqref{ainftyassoc}. Up to
  sign the relation \eqref{ainftyassoc} follows from the description
  of the boundary in Lemma \ref{refinedcompactness} of the
  one-dimensional components. The strict unit axiom follows in the
  same way as in
  \cite{flips}, by noting that by definition for any edge $e \in \Edge^{\circt}(\Gamma)$ the perturbation data is pulled back under the morphism of universal moduli spaces forgetting $e$ and stabilizing (whenever such a map exists).
  \end{proof}

\begin{remark} \label{indep} 
The \ainfty homotopy type of $\Fuk_{\LL}^\sim (X,\bb)$ (as a curved \ainfty algebra with curvature with positive $q$-valuation over the Novikov ring $\Lambda_{\ge 0}$) is  independent of the choice of almost complex structures, perturbations,\footnote{We do not use Hamiltonian perturbations of Lagrangians in this paper, so the Fukaya category we use here is defined over the Novikov ring $\Lambda_{ \geq 0}$ as opposed to the Novikov field $\Lambda$.}
stabilizing divisors, and depend only on the isotopy class of bulk deformation.  The argument uses a moduli space of {\em quilted disks} with seams labelled by the diagonal, as in \cite[Section 5.5]{flips}. Suppose first that the Donaldson hypersurface is fixed. Let $J_{\Gamma,t}$ be an isotopy of almost complex structures,  and $\bb_t $ be an isotopy of 
cycles $\bb_0$ to $\bb_1$. Requiring that the markings map to $\bb_t$, the treed disks $u| S_v$ are $J_{\Gamma,t}$-holomorphic
on components at distance $d(v) =  1/(1-t) - 1/t$ as in Equation \eqref{dv}
produces a moduli space 
fibered over the multiplihedron as in \cite[Section 5.5]{flips} producing a homotopy equivalence given by a collection of maps 
\[ \phi_d: CF^\bullet ( \LB_{d-1}, \LB_d)_0 \otimes \ldots  \otimes CF^\bullet (\LB_0, \LB_1)_0   \to   CF^\bullet ( \LB_0, \LB_d )_1[1-d] \] 
as in Seidel \cite[Section 1d]{Seidelbook} where the groups $CF^\bullet (\LB_k,\LB_{k+1})_t$ are the morphism spaces for the categories defined using the data for the given value $t$ in the family. The independence from the choice of Donaldson hypersurface is shown in the appendix. We do not address the question of invariance under Hamiltonian isotopy of Lagrangians and the relation to the Fukaya categories defined by other methods, and dependence only on the cobordism class of the cycle $\bb$, and so the homology class $[\bb]$.
\end{remark}

\subsubsection{Maurer-Cartan equation, potential function, and Floer cohomology}

Floer cohomology is defined for projective solutions to the {\it Maurer-Cartan equation}.  For each $\LB \in {\rm Ob}( \Fuk_{\LL}^\sim(X, {\mf b}))$, the element
\[
m_0 (1) \in CF^\bullet ( \LB, \LB )
\]
 is the {\em curvature} of the Fukaya algebra $CF^\bullet ( \LB, \LB )$. Its $q$-valuation $\val_q(m_0(1))$ is positive because the bulk deformation $\bb$ and the Lagrangian submanifolds do not intersect. The Fukaya algebra $CF^\bullet ( \LB, \LB)$ is called {\em flat} if $m_0(1)$ vanishes and {\em projectively flat} if $m_0(1)$ is a multiple of the identity $1_\LB^\circt$. \label{flat}
\footnote{With these definitions, the Fukaya algebra is rarely projectively flat, because any disk with positive energy has an unforgettable output, while the strict unit $1_L^\circt$ only labels a forgettable semi-infinite edge.  In such cases one needs a  nontrivial weakly bounding cochain.} Consider the sub-space of $CF^\bullet ( \LB, \LB)$ consisting of elements with positive $q$-valuation with notation from \eqref{novikov}:
\[
CF^\bullet ( \LB, \LB)_+ = \bigoplus_{x \in \cI (\LB, \LB) } \Lambda_{>0} x .
\]
Define the {\em Maurer-Cartan map}
\[
\mu : CF^{\rm odd} ( \LB, \LB)_+ \to CF^\bullet ( \LB, \LB ), \quad b \mapsto m_0(1) + m_1(b) + m_2(b,b) + \ldots .
\]
Let $MC (\LB )$ denote the space of {\em weak solutions to the Maurer-Cartan space}
\[
MC( \LB ):= \{ b \in CF^{\on{odd}} ( \LB, \LB ) \ | \ \mu (b) = W(b) 1_\LB^\circt, \quad W(b) \in \Lambda \}.
\]
The value $W(b)$ for $b \in MC (\LB )$ defines the {\em disk potential}
\[
W: MC (\LB) \to \Lambda .
\]

\begin{definition}[Weakly unobstructed branes and Floer cohomology]\hfill
\begin{enumerate}
\item A {\it weakly unobstructed brane} is a triple $\WB = (\LB, b)$ where $\LB \in \widehat{\mc L}$ and $b \in MC(\LB)$. 

\item The set of all weakly unobstructed branes supported on $\LL$ is denoted by 
\[
MC( \LL ):= \Big\{ \WB = (\LB, b)\ |\ \LB \in \widehat{\mc L},\ b \in MC(\LB) \Big\}.
\]

\item The {\it Floer cohomology} of a weakly unobstructed brane $\WB$ is
\[
HF^\bullet ( \WB, \WB):= {\rm ker} (m_1^b)/ {\rm im}(m_1^b)
\]
where $m_1^b$ is defined by 
\[
m_1^b: CF^\bullet( \LB, \LB ) \to CF^\bullet ( \LB, \LB ),\ m_1^b(a) = \sum_{k_1, k_2\geq 0} m_1( \underbrace{b, \ldots, b}_{k_1}, a, \underbrace{b, \ldots, b}_{k_2}).
\]
\end{enumerate}
\end{definition}

\subsubsection{Flat $A_\infty$ category and spectral decomposition}

Given a curved $A_\infty$ category, flat $A_\infty$ categories are obtained by restricting to particular values of the curvature.

\begin{definition}{\rm(Flat $A_\infty$ category)}
Let ${\mc F}^\sim$ be a strictly unital curved $A_\infty$ category over $\Lambda$. The  flat $A_\infty$ category ${\mc F}$ associated to ${\mc F}^\sim$ is defined as the disjoint union
\[
{\mc F}^\flat:=  \bigsqcup_{w\in \Lambda} {\mc F}_w
\]
where each component ${\mc F}_w$ has the set of objects
\[
{\rm Ob} ({\mc F}_w):= \left\{ \WB = (\LB, b)\ |\ \LB \in {\rm Ob}({\mc F}^\sim),\ b \in MC(\LB),\ W_\LB (b) = w \right\},
\]
the set of morphisms
\[
{\rm Hom} ( \WB, \WB' ):= {\rm Hom} (\LB, \LB'),
\]
and the higher compositions for $d \geq 1$
\[
m_d(a_d,\ldots,a_1) = \sum_{k_0,\ldots, k_d \ge 0} m_{d + k_0 + \ldots + k_d}(\underbrace{b_d,\ldots, b_d}_{k_d} , a_d,
  \ldots,  \underbrace{b_1,\ldots, b_1}_{k_1}, a_1,\underbrace{b_0,\ldots, b_0}_{k_0}).
\]
Define $m_0 = 0$. If ${\rm Ob}({\mc F}_w) \neq \emptyset$, then we say ${\mc F}_w$ is an {\it eigen-subcategory} of ${\mc F}$.
\end{definition}

\begin{proposition} 
For any $w \in \Lambda$, the category ${\mc F}_w$ is a flat strictly unital \ainfty category as defined in \eqref{flatdef}.
\end{proposition} 

\begin{proof} 
The flatness condition $m_0(1) = 0$ holds by definition. The $A_\infty$ relation
  \begin{multline} 0 = \sum m_{d -i + 1} (\underbrace{b_d,\ldots,
      b_d }_{i_d}  ,    {a}_{d} ,
\ldots, , b_{j+i} \\
    m_i(  
\underbrace{b_{j+i}, \ldots,
    b_{j+i}}_{k_{j+i}},
    {a}_{j+i},\ldots,
    {a}_{j+1},
    \underbrace{b_{j+1},\ldots, b_{j+1}}_{k_j}) , \\ \ldots, b_{j+1}, {a}_{j+1},
    \ldots,         \underbrace{b_1,\ldots, b_1}_{i_1}, 
 {a}_1, \underbrace{b_0,\ldots, b_0}_{i_0}
 ) \end{multline}
  follows from the $A_\infty$ relation for ${\mc F}^\sim$, the strict identity relation, and the inclusion
\[ \sum_{i \ge 0} m_i(\underbrace{b_j,\ldots, b_j}_i) \in
\on{span}(1_\LB^\circt ) \quad \forall j = 0,\ldots, d. \qedhere \]
\end{proof} 

In the case of the bulk deformed curved Fukaya category $\Fuk_{\mc L}^\sim (X, {\mf b})$, we have the associated flat category.

\begin{definition} 
Define a flat $A_\infty$ category
\[
\Fuk_\LL^\flat(X,\bb):= \bigsqcup_{w \in \Lambda}   \Fuk_\LL(X,\bb)_w 
\]
whose set of objects is the disjoint union of all objects in the eigen-subcategories, and the space of morphisms between objects in different eigen-subcategories is the zero vector space. More generally, given a subset ${\mf L} \subset MC(\LL)$, denote by 
\[ \Fuk_{\mf L}^\flat(X, \bb) \]
be the full $A_\infty$-subcategory with the set of objects equal to ${\mf L}$.
\end{definition}

\subsection{Hochschild (co)homology}

Hochschild homology of a category is the homology of a contraction operator on the space of all composable sequences of morphisms. In the case of curved \ainfty categories, there seems to be no good definition at the moment, although we understand from Abouzaid that he and Varolgunes and Groman are developing such a theory. For our purposes it suffices to use the Hochschild theory for flat categories in combination with a spectral decomposition. We first recall the definition from, for example, \cite[Section 2]{Seidel_2008}.

\begin{definition}\label{bimodule}
Let $\F$ be a flat ${\mb Z}_N$-graded $A_\infty$-category.
\begin{enumerate} 

\item As in Seidel \cite[Section 2]{Seidel_2008} an {\em \ainfty bimodule} $\M$ over $(\F,\F)$ consists of
\begin{enumerate} 
\item a map assigning to any pair of objects $\WB, \WB'$ a graded vector space $ \M (\WB, \WB')$ and
\item multiplication maps for integers $d, d' \geq 0$ and objects  $\WB_0, \ldots, \WB_d$, $\WB_0', \ldots, \WB_{d'}'$ of $\cF$
\begin{multline} 
m_{d, d'} : 
 \Hom ( \WB_d, \WB_{d-1} ) \otimes \cdots \otimes \Hom ( \WB_1, \WB_0) 
 \otimes
 \M (\WB_0, \WB_0') \otimes  \\
 \Hom (\WB_0', \WB_1') \otimes \cdots \otimes \Hom ( \WB_{d'-1}', \WB_{d'}' ) \to \M (\WB_d, \WB_{d'}')
\end{multline}  
\end{enumerate}
satisfying the \ainfty bimodule axiom, see \cite[Section 2]{Seidel_2008}.
\item Given an \ainfty bimodule $\M$ over $(\F,\F)$, the space of {\em
    Hochschild chains} with values in $\M$ is the direct sum
\begin{multline} 
CC_\bullet(\F,\M) = \bigoplus_{\WB_0,\ldots, \WB_d \in  \on{Ob}(\cF)}     \Hom (\WB_{d-1}, \WB_d)\otimes 
 \\
  \ldots  \otimes \Hom (\WB_{i+1}, \WB_{i+2}) \otimes   \M ( \WB_i, \WB_{i+1}) \otimes \Hom( \WB_{i-1}, \WB_i) 
 \otimes \ldots \\
 \otimes \Hom( \WB_1, \WB_2)  \otimes \Hom ( \WB_0, \WB_1)  \otimes \Hom ( \WB_d, \WB_0).
\end{multline}
In particular $\F$ is itself a bimodule over $(\F,\F)$, called the
{\em diagonal bimodule}. The Hochschild chain group in this case is denoted by $CC_\bullet(\cF)$. A generator of the summand
\[
\Hom (\WB_{d-1}, \WB_d)\otimes \cdots  \otimes \Hom ( \WB_0, \WB_1)  \otimes \Hom ( \WB_d, \WB_0)
\]
is typically denoted by ${a}_d \otimes \cdots \otimes {a}_0$.

\item The {\it Hochschild differential} on $CC_\bullet( \cF)$ is defined by summing over all possible contractions:
\begin{multline} \label{hh} 
\delta_{CC}: {a}_d \otimes \ldots \otimes {a}_0 \mapsto \sum_{i+j \leq d} (-1)^{\mathsection} 
  m_{d- j-1}({a}_{i-1} \otimes \ldots \otimes {a}_{i+j+1}) \otimes {a}_{i+j} \otimes \ldots \otimes {a}_i \\ 
+ \sum_{i+j \leq d} (-1)^{{\maltese_0^{i-1}}} {a}_d \otimes \ldots \otimes {a}_{i+j+1} \otimes 
  m_{j+1}( a_{i+j} \otimes \ldots \otimes {a}_{i}) \otimes {a}_{i-1}
  \otimes \ldots \otimes {a}_0 \end{multline}
where $\maltese_k^l$ is given by \eqref{maltese} and 
\[
{\mathsection := \maltese_0^{i-1} ( 1 + \maltese_i^d) + \maltese_i^{i+j}.} 
\]
For $\F$ flat as above denote by
\[ HH_\bullet(\cF) : = \frac{\on{ker}(\delta_{CC})}{\on{im}(\delta_{CC} )} \]
the homology of $\delta_{CC}$.  

\item For a curved \ainfty category $\F^\sim$ let 
\[
{\mc F}^\flat:= \bigsqcup_{w \in \Lambda} {\mc F}_w
\]
the flat $A_\infty$ category obtained via the spectral decomposition.  Denote by %
\begin{equation}\label{sdecomp}
CC_\bullet (\F^\sim): = CC_\bullet( {\mc F}^\flat ) \cong \bigoplus_{w\in \Lambda} CC_\bullet ({\mc F}_w ) \end{equation}
the direct sum over possible values $w$ of the potential of the
Hochschild homologies of the flat categories obtained by fixing the
value of the curvature.
\end{enumerate}
\end{definition}

The Hochschild cohomology is defined for a flat \ainfty category as
follows.  A {\em Hochschild cochain} $\tau$ on a flat \ainfty category $\F$
valued in $\F$ is a collection
\begin{equation}
  \label{eq:hoc-cochain}
  \tau:= (\tau_{\WB, d} )_{\WB \in {\rm Ob}({\mc F}^\flat),d \geq 0}  
\end{equation}
where $\tau_{\WB, d}$ is a linear map
\[
\tau_{\WB, d}: \bigoplus_{\WB_1, \ldots, \WB_d} \Hom ( \WB_{d-1}, \WB_d )  \otimes \ldots \otimes
  \Hom ( \WB, \WB_1) \to  \Hom ( \WB, \WB_d ). 
\]
The space $CC^*(\F,\F)$ is an \ainfty algebra whose composition maps
are
%
\begin{multline}\label{eq:ccm1} 
  (m^1_{CC^*} \tau)^d (a_d,\ldots,a_1) = \sum_{i,j} (-1)^\dagger
  m^{d-j+1}_{\F}( a_{d},\ldots, a_{i+j+1},\tau^{j}(a_{i+j},\\ \ldots,
  a_{i+1}), a_i,\ldots,a_{1}, ) - \sum_{i,j} (-1)^{\clubsuit}
  \tau^{d - j + 1}(a_d,\ldots,a_{i+j+1}, \\ m^j_{\F}(a_{i+j},\ldots,
  a_{i+1}),a_{i},\ldots,a_1),
\end{multline}
where
\[\dagger = (|\tau|-1)( |a_1| + \ldots + |a_{i_1 + \ldots +
  i_{k-1}}| - i_1 - \ldots - i_{k-1}) , \quad \clubsuit:=i + \sum_{j=1}^i |a_j| + |\tau| - 1,\]
and for $e \geq 2$
\begin{multline}\label{eq:ccmd}
  (m^e_{CC^*}(\tau_e,\dots,\tau_1))^d(a_d,\dots,a_1):=\sum_{\substack{i_1,\dots,i_e
      \\ j_1,\dots, j_e}}(-1)^\circ 
m_\F^{d - \sum j_k}(a_d,\ldots,
a_{i_e+j_e+1},
\tau_e^{j_e}(\dots),
 a_{i_e}, 
\\ \ldots,
a_{i_1+j_1+1}, 
\tau_1^{j_1}(a_{i_1+j_1}, \dots,a_{i_1+1}), 
a_{i_1},\dots,a_{1}, 
)
\end{multline}
where
\[\circ := \sum_{j=1}^e \sum_{k=1}^{i_j}(|\tau_j|-1)(|a_k|-1).\]
The boundary operator $m^1_{CC^*} $ squares to zero, and we denote by
$HH^\bullet(\F) := HH^\bullet(\F,\F)$ the Hochschild cohomology of $\F$ valued in $\F$.

\begin{remark} \label{hhid}
The Hochschild cohomology is equipped with a natural identity. 
Suppose
the $A_\infty$ category is strictly unital. Consider the cochain $1_{\mc F} \in CC^0({\mc F}, {\mc F})$ defined by
\[
1_{{\mc F},d} ( {a}_d \otimes \cdots \otimes {a}_1) = \left\{ \begin{array}{cc} 0, &\ d > 0,\\
1_{\WB}, &\  d= 0.\end{array} \right.
\]
Then $m^1_{CC}(1_{\mc F}) = 0$ \label{inconsistent} and the cohomology
class of $1_{\mc F}$ is the unit of the Hochschild cohomology
ring. 
\end{remark} 

\begin{remark} \label{indep4} 
Any \ainfty functor $\Phi:  {\mc F} \to {\mc F}'$ (between flat $A_\infty$ categories) induces a map of Hochschild homologies  $HH_\bullet(\Phi): 
 HH_\bullet({\mc F}, {\mc F}) \to HH_\bullet({\mc F}', {\mc F}')$ as in \cite[Section 2.9]{Ganatra_thesis}, depending only on the homotopy type of the functor,
and the resulting maps are functorial for \ainfty functors with respect to composition.   In particular, the isomorphism class of  the Hochschild homology of the Fukaya category $HH_\bullet({\mc F}, {\mc F})$ is independent of the 
is  independent of the choice of almost
  complex structures, perturbations,
stabilizing divisors, and depend only on the
  isotopy class of bulk deformation. 
  Similarly the Hochschild cohomology is the cohomology of the space of endomorphisms
  of the identity functor, and so independent up to isomorphism of all such choices.  
\end{remark} 

\begin{remark} 
Any $A_\infty$ homotopy equivalence $\Phi$ between {\it curved} $A_\infty$ categories $\cF^\sim, {\mc G}^\sim$ induces homotopy equivalences between eigencategories $\Phi_w: \cF_w \to {\mc G}_w$ for each $w$. The \ainfty morphism axiom
implies that for any object $\LB\in {\rm Ob}(\cF^\sim)$ the induced map on Maurer-Cartan spaces $\Phi: MC (\LB ) \to MC(\Phi( \LB ))$ preserves the potential function. For any object $\WB = (\LB, b)\in {\rm Ob}(\cF_w)$ we obtain an object
\[
\Phi_w (\WB) := (\Phi(L), \Phi(b) )\in {\rm Ob}({\mc G}_w).
\]
The induced map of morphisms spaces $\Hom ( \WB, \WB' ) \to \Hom ( \Phi_w(\WB), \Phi_w(\WB'))$ then satisfy the \ainfty homotopy axiom with vanishing curvature terms.
\end{remark} 

\begin{remark}  The length filtration induces a spectral sequence computing 
the Hochschild (co)homology of any flat \ainfty category $\cF$  with first page the Hochschild (co)homology of the  (co)homology category $H(\cF)$ whose morphism groups are $H ( \Hom (\WB_1, \WB_2))$, see \cite[Lemma 5.3]{Getzler_Jones_1990}.
\end{remark}

\subsection{Quantum cohomology} \label{diag} \label{quantumcohomology}

Before discussing the open-closed and the closed-open maps, we first give a construction of quantum cohomology using the Morse model and Cieliebak-Mohnke's method which can be incorporated into the constructions of Fukaya category and the open-closed/close-open maps. We first extend the terminology of trees and treed disks to include spheres.
\begin{definition}
\begin{enumerate}

\item A {\em domain type of treed spheres} consists of a rooted tree
  $\Gamma$ with empty disk part $\Gamma_\circ$ and a decomposition 
\[
{\rm Leaf}(\Gamma) = {\rm Leaf}_{\rm grad}(\Gamma) \sqcup {\rm Leaf}_{\rm const}(\Gamma)
\]
of the set of leaves into subsets of {\em {gradient}} and {\em constrained} leaves
(which eventually will correspond to leaves that map to gradient
trajectories in $X$, or leaves that map to the stabilizing divisor or
bulk deformation.) A domain type of treed spheres $\Gamma$ is {\em stable} if
the valence of each vertex $v \in \Ver(\Gamma)$ is at least three.\footnote{To define the
  quantum multiplication we do not need to allow finite edges to
  acquire length. However it is necessary for the proof of the associativity.}

\item A treed sphere $C = S \cup T$ of $\Gamma$ is obtained from a nodal sphere $S'$ whose combinatorial type is described by $\Gamma$ by attaching 
a copy of $(-\infty, 0]$ for each {gradient} leaf
and an interval $[0, +\infty)$ for the output. The surface part $S$ is the union of spherical components $S_v$ labelled by vertices $v \in {\rm Vert}(\Gamma)$ while the tree part $T$ is the union of these semi-infinite intervals $T_e, e \in \Edge(\Gamma)$.

\item Given a stable domain type of treed spheres $\Gamma$, the universal curve ${\cU}_\Gamma$ is formally the disjoint union
\[
{\cU}_\Gamma = \bigsqcup_{[C] \in {\cM}_\Gamma} C.
\]

\end{enumerate}
\end{definition}
A natural partial order among domain types can be defined in a similar way as Section \ref{section2}. 

Quantum cohomology will be defined via the choice of a Morse-Smale pair on the manifold.
Each Morse-Smale pair $(f_X, h_X)$ on $X$ induces a Morse-Smale-Witten complex 
\[ (CM^\bullet(f_X, h_X), \delta_{\rm Morse}) \] 
generated by critical points over $\Lambda$ and graded by $2n$ minus the Morse index; the Morse differential $\delta_{\rm Morse}$ counts trajectories of the negative gradient flow of $f_X$ and hence increases the grading.   The cohomology $HM^\bullet( f_X, h_X)$ of $\delta_{\rm Morse}$ is isomorphic to the (co)homology of $X$ over ${\mb Z}$.  More precisely, if $\sum c_i {\xx}_i$, where $c_i \in {\mb Z}$ and ${\xx}_i \in {\rm crit}(f_X)$ is a Morse cocycle of degree $k$, then the linear combination
\[ \label{Wsum}
\sum_i c_i W^s({\xx}_i)
\]
of stable manifolds is a $2n-k$-dimensional 
pseudocycle (see Schwarz \cite{Schwarz_1999}), 
hence defines a $k$-dimensional cohomology class. 

\subsubsection{Perturbations and transversality}

We introduce domain-dependent perturbations on the universal curves
similar to those before. 
Let $D \subset X$ be a Donaldson hypersurface and let $J_0\in {\mc J}_{\rm tame}(X, \omega)$ be a tamed almost complex structure satisfying (b) of Lemma \ref{countable}. In the case of quantum cohomology, for treed spheres of domain type $\Gamma$, the perturbation data $P_\Gamma$ include a domain-dependent almost complex structure $J_\Gamma$ which are sufficiently close to $J_0$, a domain-dependent perturbation of the Morse function $f_X$, and diffeomorphisms $M_\Gamma$ of $X$.

Quantum multiplication is defined by counting treed spheres with two {gradient} leaves $T_{e_1}, T_{e_2}$ and one output $T_{e_{\rm out}}$. Let $\Gamma_n$ bea domain type with only one vertex $v \in \Ver(\Gamma_n)$ and $n+3$ leaves in total. When $n =  0$, $\Gamma_0$ is a trivalent graph.  In this case we set $J_{\Gamma_0} \equiv J_0$ and perturb $f_X$ away from infinities of the semi-infinite edges. Such a perturbation induces two perturbations $W_{e_1}^u(x)$, $W_{e_2}^u(x)$ of each unstable manifold $W^u(x)$ of $f_X$ and one perturbation $W_{e_{\rm out}}^s(x)$ of each stable manifold $W^s(x)$. This perturbation of $f_X$ on the trivalent graph induces perturbations of $f_X$ on all $\ov{\mc U}_{\Gamma_n}$. Furthermore, when $n \geq 2$, we require that $J_{\Gamma_n}$ and $M_{\Gamma_n}$ do not depend on the positions of the {gradient} leaves. More precisely, let $\Gamma_n'$ be the domain type obtained from $\Gamma_n$ by forgetting the two gradient, leaves (which is still stable).  Let $\ov{\mc U}_{\Gamma_n} \to \ov{\mc U}_{\Gamma_n'}$ denote the naturally induced contraction.  We require that $J_{\Gamma_n}$ and $M_{\Gamma_n}$ are equal to pullbacks of perturbations on $\ov{\mc U}_{\Gamma_n'}$. We also require the locality property: for each $\Pi \prec \Gamma_n$, let $P_\Pi$ be the restriction of $P_{\Gamma_n}$ to $\ov{\mc U}_\Pi \subset \partial \ov{\mc U}_{\Gamma_n}$. For each $v \in {\rm Vert}_{\Pi}$, the restriction of $P_\Pi$ to ${\mc U}_{\Pi, v}$ is equal to the pullback from a function defined on ${\mc U}_{\Pi(v)}$ (see the relevant notations in Definition \ref{locality}).

One can achieve transversality in the same way as the case for treed disks (see Theorem \ref{regular}). An essential map type $\bGamma$ with underlying domain type $\Gamma_n$ contains a labelling of critical points $x_1, x_2, x_{\rm out}$ at the {gradient} leaves of $\Gamma_n$, $n_1$ constrained leaves labelled by $D$ and $n_2 = n- n_1$ constrained leaves labelled by components of the bulk deformation $\bb$, and a homology class $\beta\in H_2(X; {\mb Z})$ satisfying $n_1 = k\omega(\beta)$ (where $k$ is the degree of the Donaldson hypersurface).  A generic perturbation of $f_X$ on the trivalent graph $\Gamma_0$ and perturbing $J_{\Gamma_n}, M_{\Gamma_n}$ for each $n\geq 1$,  makes each moduli space ${\mc M}_{\bGamma}^{\rm QH} (P_{\Gamma_n})$ transverse; and, in addition, if ${\rm index}\ \bGamma = 0$ resp. ${\rm index}\ \bGamma = 1$, then ${\mc M}_{\bGamma}^{\rm QH} (P_{\Gamma_n})$ is compact resp. compact up to at most one breaking of gradient trajectories. 

\subsubsection{Bulk deformed quantum cohomology ring}

Now we repeat the Piunikhin-Salamon-Schwarz construction \cite{PSS} under the current setting. Fix critical points ${\xx}_1, {\xx}_2, {\xx}_\infty\in {\rm crit}(f_X)$. Let $\bGamma$ be an essential map type with incoming {gradient} leaves labelled by ${\xx}_1, {\xx}_2$ and outgoing leaf labelled by ${\xx}_\infty$. If the (only) vertex of $\bGamma$ is labelled by $\beta\in H_2(X; {\mb Z})$, then the expected dimension of the moduli space ${\mc M}_{\bGamma}(P_\Gamma)$ is 
\[
{\rm \index} \ \bGamma = 2c_1(\beta) - {\rm deg} {\xx}_1 - {\rm deg} {\xx}_2 + {\rm deg} {\xx}_\infty.
\]
Let 
\[
{\mc M}^{\rm QH}_{2, 1}({\xx}_1, {\xx}_2, {\xx}_\infty)_0
\]
be the union of all moduli spaces ${\mc M}_{\bGamma}^{\rm QH} (P_\Gamma)$
with labelled map types $\bGamma$  having index zero.  Define a bilinear map 
\[
\star_\bb: CM^\bullet( f_X, h_X) \otimes CM^\bullet( f_X, h_X) \to CM^\bullet( f_X, h_X)
\]
whose values on the generators ${\xx}_1 \otimes {\xx}_2$ are
\[
\star_\bb({\xx}_1, {\xx}_2) = \sum_{{\xx}_\infty\in {\rm crit}(f_X)} \left( \sum_{[C, u]\in {\mc M}_{2, 1}^{\rm QH} ({\xx}_1, {\xx}_2, {\xx}_\infty)_0 } c(u, \bb)  q^{A(u)} o(u) d(u)^{-1} \right) {\xx}_\infty
\]
where $A(u)$ is the symplectic area of $u$, $o(u) \in \{1, -1\}$ is the sign determined by the orientation, $d(u) = (k A(u))!$, and $c(u, \bb)$ is the coefficient determined by the interior leaves mapped to components of the bulk deformation. Strong transversality implies that the above is a finite sum 
for each area bound.  Arguments involving the boundary of the one-dimensional moduli spaces show that $\star_\bb$ is a chain map and hence induces a bilinear map 
\[
\star_\bb: HM^\bullet( f_X, h_X) \otimes HM^\bullet( f_X, h_X) \to HM^\bullet( f_X, h_X).
\]
A cobordism argument shows that  the operation $\star_\bb$ on cohomology is independent of the choice of perturbation data and the choice of the Morse-Smale pair. Lastly, by allowing treed sphere with three incoming {gradient} leaves and interior edges with positive lengths, one can prove that the quantum multiplication $\star_\bb$ is associative.    We denote this graded unital ring 
\[
QH^\bullet (X, \bb) = ( HM^\bullet( f_X, h_X), \star_b) 
\]
and call it the {\it $\bb$-deformed quantum cohomology ring} of $X$.

The quantum cohomology has a natural identity element defined as follows. Since $X$ is connected, one can choose $f_X$ such that it has a unique critical point $x_{\max}$ of maximal Morse index. It is clearly a cochain and the fact that its cohomology class is the identity follows from the fact that the Morse--Smale pairs on {gradient} leaves are fixed and the perturbation respects forgetting {gradient} leaves. 

\begin{remark} \label{indep3} 
The quantum cohomology $QH^\bullet (X, \bb)$ is also independent of the choice the stabilizing divisor. The relevant constructions are carried out in the Appendix.
  \end{remark} 

\subsubsection{Quantum multiplication by submanifolds}

We will use a particular chain-level definition of the quantum multiplication by the class of a submanifold. Recall that the class of a submanifold may be expressed in 
terms of the stable manifolds as follows.  If $ Y \subset X$ be an oriented submanifold
then 
\[ [Y] = \left[\sum_i c_i {\xx}_i \right] \in H(X) \] 
as in \eqref{Wsum}.  Each coefficient $c_i$ may be written as a signed count of intersections
\[ c_i = \# (Y \cap W^u({\xx}_i) ) \]  
may be taken to be the number of intersection points
of $Y$ with $W^u({\xx}_i)$.  Equivalently, $c_i$ is the number of rigid gradient trajectories connecting $Y$ with ${\xx}_i$, counted up to sign.

We may express  quantum multiplication by the class of a submanifold in terms
of  treed holomorphic spheres with a constraint in the submanifold, as follows. 
Above quantum multiplication is defined by counting treed holomorphic spheres with two inputs and one outputs all of which are labelled by Morse critical points of $f_X: X \to {\mb R}$.   Consider the space of configurations of treed spheres $u: C \to X$ where the first incoming interior leaf has been replaced  a marking $z_\bullet$ 
that maps to $u(z_\bullet) \in Y$.   Let $\Gamma$ be a combinatorial type of treed disks of this type, i.e., one interior incoming leaf $T_{e_{\bullet,1}}$, one interior outgoing leaf $T_{e_{\bullet,0}}$, a number of normal interior markings
$z_e$ and one auxiliary marking $z_\bullet$. We do not allow interior edges to acquire length. A map type $\bGamma$ refining $\Gamma$ consists of homology classes $\beta(v) \in H_2(X) \cup H_2(X,L)$ labelling vertices $v \in \Ver(\Gamma)$ and labelling leaves 
$e \in \Edge(\Gamma)$ indicating the limit of the corresponding Morse trajectory
and the type of weighting, and for the auxiliary marking
%
$z_\bullet$
a label of either $Y$, $Y \cap D$, or $Y \cap \bb$
indicating the constraint at that marking.  A map type $\bGamma$ is {\em essential} if there are no broken edges, no sphere components, all normal interior markings 
$z_e$  
are labelled by either $(D,1)$ or $\bb$, and the auxiliary marking 
%
$z_\bullet$ 
labelled by $Y$, and the number of interior markings labelled by $(D,1)$ is the expected number $k \sum \lan \beta(v), [\omega] \ran$. Let $P_\Gamma$ be a perturbation datum defined on the universal moduli space $\ov{\mc U}_\Gamma$ which does not depend on the position of the auxiliary marking. Let ${\mc M}_\bGamma(P_\Gamma)$ be the moduli space of treed spheres of map type $\bGamma$. 
For generic $P_\Gamma$ each zero-dimensional moduli space ${\mc M}_\bGamma(P_\Gamma)$ with essential map type $\bGamma$ is compact and regular, and each one-dimensional component of the moduli space with essential map types is compact up to one breaking
at either the incoming interior edge or the outgoing interior edge.  Given two critical points $\xx, \xx'$, let 
\[
{\mc M}^{\rm QH}(Y, \xx, \xx')_i,\ i = 0, 1
\]
be the union of $i$-dimensional moduli spaces with essential map types with the input and output labelled by $\xx$ and $\xx'$ respectively. Define the chain map 
\begin{equation}\label{qm_chain}
Y \star_\bb: CM^\bullet(f_X, h_X) \to CM^\bullet(f_X, h_X)
\end{equation}
by 
\[
Y \star_\bb (\xx) = \sum_{\xx'} \sum_{[u] \in {\mc M}^{\rm QH}(Y, \xx, \xx')} \frac{1}{(k E(u))!} {\rm Sign}(u) \xx'.
\]

\begin{proposition} $Y \star_\bb$ is a chain map and induces the map on cohomology
\[
[Y] \star_\bb: QH^\bullet( X; \bb) \to QH^\bullet( X; \bb).
\]
\end{proposition}

\begin{proof} The fact that $Y \star_\bb$ is a chain map follows by considering boundaries of 1-dimensional moduli spaces with $z_\bullet$ 
constrained to map to $Y$:   Boundary configurations occur when a Morse trajectory bubbles off  on the incoming edge or the outgoing edge.  
The equality with $[Y] \star_\bb$ in cohomology is proved by 
considering the moduli space of configurations where the marking
$z_\bullet$ is replaced by  leaf $T_\bullet$
of some length $\ell(T_\bullet) \in [0,\infty]$.  In the case
$\ell(T_\bullet) = \infty$ one obtains $ (\sum c_i W_i^s({\xx}_i)) \star_\bb$,
while the case $\ell(T_\bullet) = 0$ gives $Y \star_\bb$.  
The other boundary configurations involve breaking off a Morse trajectory
at one of the other two edges, and we obtain 
\[ [Y \star_\bb c ] = [Y] \star_\bb [c ]  \] 
for any cocycle $c \in CM^\bullet( f_X, h_X)$ as desired.
\end{proof}

\subsection{Open-closed maps}\label{section:OC}

The open-closed maps, roughly speaking, are defined via counts of treed holomorphic disks where the inputs are on the boundary
(generators of morphisms spaces of the Fukaya category) and the outputs are critical points in the ambient symplectic manifolds.  The combinatorial structure underlying  the maps that will be used to define the open-closed map  combine the features of  treed disks and spheres in the construction of Floer and quantum cohomology.

\subsubsection{Open-closed domains and perturbations}

\begin{definition} \label{octype}
  {\rm(Open-closed domain type)} 
The {\em open-closed domain type} consists of a variation of the rooted two-colored tree where all the semi-infinite edges on the disk part 
$e \in \Edge_\white(\Gamma)$ are inputs, and the output $e_{\rm out} \in \Edge_\black (\Gamma)$ is an interior semi-infinite edge and the only {gradient} leaf, along with the metric type $\ul \ell$ and a weighting type
  $\ul{\on{wt}}$ of the boundary edges defined as follows:
  Similar to rooted two-colored trees, a metric on an open-closed domain is a map
\[
\ell: {\Edge}_{\rm \fin} ( \Gamma_\circ) \to [0, +\infty),
\]
and the weighting is a map 
\[
\on{wt}: {\Edge}_\rightarrow(\Gamma) \to [0,1]
\]
that is zero on all interior semi-infinite edges. We do not require
here the relation \eqref{weight_relation} on weightings. The discrete datum
underlying $\ell$ resp. $w$ is called a {\em metric type} resp. {\em
  weighting type} and denoted by $\ul{\ell}$ resp. $\ul{\on{wt}}$.
\end{definition}

Open-closed domain types describe treed disks with an interior output. The stability condition is defined in the usual way, as the absence of non-trivial infinitesimal automorphisms.  A broken open-closed domain type may have unbroken components which are domain types of treed disks or infinite edges supporting flow lines in $X$; however, we remark that there is no unbroken component that is the domain type for treed spheres. Perturbation data for open-closed domain types  extend the existing perturbation data chosen for defining the Fukaya category and a Morse-Smale pair $(f_X, h_X)$ on $X$. We also require the perturbation $P_\Gamma$ does not depend on the position of the {gradient} leaf. More precisely, we require the following: if $\Gamma'$ is the domain type obtained by forgetting the (only) {gradient} leaf $e$ on $\Gamma$ and stabilization, then with respect to the contraction map ${\mc U}_\Gamma \to {\mc U}_{\Gamma'}$, the perturbation $P_\Gamma$ is naturally induced from a perturbation defined on ${\mc U}_{\Gamma'}$. In particular, if $\Gamma'$ becomes empty, then in $P_\Gamma = (J_\Gamma, F_\Gamma, H_\Gamma, M_\Gamma)$, $J_\Gamma \equiv J_0$, $F_\Gamma = 0$, and $H_\Gamma = 0$.  Transversality in this case 
requires that the stable manifolds of the Morse function $f_X$ be transverse to the unstable manifolds on the Lagrangian, which  can be achieved by generic choice of $f_X$.

\subsubsection{Open-closed moduli spaces}

The moduli space of open-closed maps admits a stratification by type.
An {\em open-closed map type} $\bGamma$ includes an extra labelling on the interior {gradient} leaf $e \in \Leaf_{\rm grad} (\Gamma)$ by a critical point of a chosen Morse function $f_X$. For any perturbation datum $P_\Gamma$, a treed holomorphic disk of map type $\bGamma$ consists of a treed disk $C$ of type $\Gamma$ and a continuous map $u: C \to X$ that is a perturbed holomorphic map on each surface component $S_v \subset C$, a perturbed negative gradient line/ray/segment on each boundary edge, and a negative gradient ray of $f_X$ on the out-going {gradient} leaf $T_e \cong [0, +\infty)$. Let ${\mc M}_{\bGamma}^{\rm OC} (P_\Gamma)$ denote the moduli space of treed holomorphic disks of map type $\bGamma$.  Transversality for uncrowded strata is proved in the same way as Lemma \ref{regular}. As in the case of treed disks with gradient trajectories in the Lagrangians, a version of Gromov compactness implies that the union
\[
\ov{\cM}{}_{\bGamma}^{\rm OC} (P_\Gamma) = \bigsqcup_{\bPi \preceq {\bm
    \Gamma}} {\cM}_{\bPi}^{\rm OC} \left(P_\Gamma|_{\ov{\cU}_{\Pi^{\rm st}}}
  \right)
\]
over all open-closed types is a compact Hausdorff space, {and only finitely many types appear for any given energy bound.} Lower strata also include (arbitrarily many) breakings in the distinguished interior semi-infinite edge.

\begin{definition}
An open-closed map type ${\bGamma}$ is called {\em essential} if it has no spherical components $S_v, v \in \Ver_\black(\Gamma)$
nor edges $T_e$ of length $\ell(e)$ zero, all interior markings are either $(D, 1)$ or $\bb$ and for each disk component $v$, the number of interior markings labelled by $(D, 1)$ is equal to $k \omega(\beta_v)$ where $k$ is the degree of the Donaldson hypersurface.
\end{definition}

The following lemma can be proved in the same way as Lemma \ref{refinedcompactness}.

\begin{lemma} \label{tboundplus}
Let ${\bGamma}$ be an essential open-closed map type. If the expected dimension of ${\bGamma}$ is zero, then ${\cM}_{\bGamma}^{\rm OC}$ is compact. If the expected dimension is one, then $\ov{\cM}_{\bGamma}^{\rm OC}$ is a compact topological 1-manifold with boundary where the boundary strata consist of moduli spaces ${\cM}_{\bPi}^{\rm OC}$ where $\Pi$ is either obtained from $\Gamma$ by one of the operations listed in Lemma \ref{refinedcompactness}, or obtained from $\Gamma$ by breaking the interior semi-infinite edge once.
\end{lemma}

The open-closed map is defined by counting treed holomorphic disks whose output edge is an interior edge.
See Figure \ref{co} for an example.

\begin{figure}[ht]
\includegraphics[height=2in]{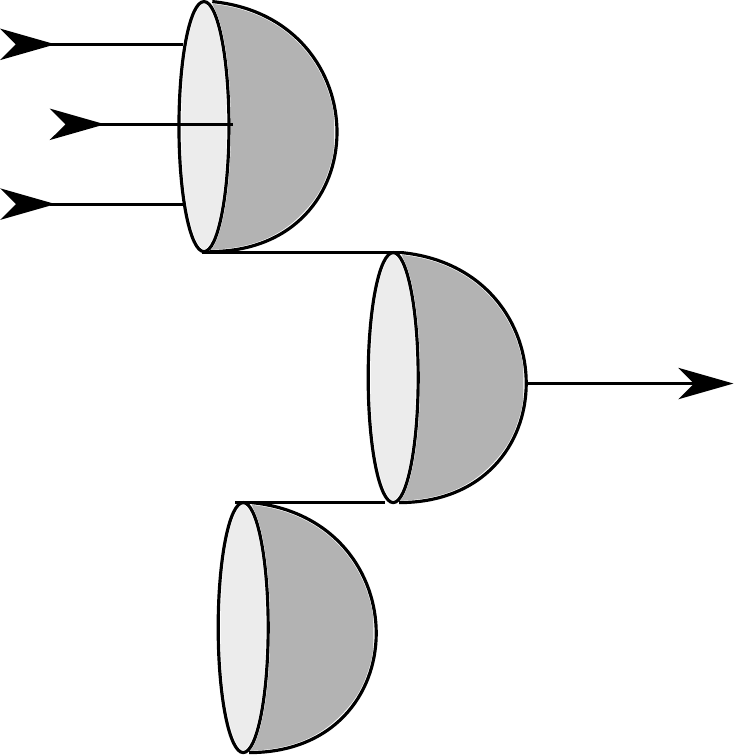} 
\caption{A typical configuration that possibly contributes to the definition of the open-closed map. Interior markings to be mapped to the Donaldson hypersurfaces and the bulk deformation and boundary edges with Maurer-Cartan insertions are omitted.} 
\label{co}
\end{figure} 

\begin{lemma}\label{oc_unit}
Let $\bGamma$ be an essential open-closed map type with index $0$. Suppose the outgoing {gradient} leaf is labelled by $x_{\min}$, the only critical point with minimal Morse index. Then either ${\mc M}_{\bGamma}^{\rm OC} (P_\Gamma) = \emptyset$, or $\bGamma$ contains exactly one incoming semi-infinite edge labelled by a critical point of $f_L: L \to {\mb R}$ for some $L \in {\mc L}$ with minimal Morse index.
\end{lemma}

\begin{proof}
Let $\Gamma'$ be the domain type obtained from $\Gamma$ by forgetting the {gradient} leaf and stabilization. If $\Gamma' \neq \emptyset$, then $\bGamma$ induces a map type $\bGamma'$ with negative index. Since the perturbation $P_\Gamma$ does not 
depend on the position of the {gradient} leaf, it induces a perturbation $P_{\Gamma'}$ on $\ov{\mc U}_{\Gamma'}$. As ${\mc M}_{\bGamma'}^{\rm OC} (P_{\Gamma'}) = \emptyset$ by transversality, ${\mc M}_{\bGamma}^{\rm OC} (P_\Gamma) = \emptyset$. Therefore, $\Gamma' = \emptyset$. As a consequence, $\Gamma$ has one or two incoming semi-infinite edges and no interior constrained leaves.  This implies that $\bGamma$ has zero energy. Hence any configuration in ${\mc M}_{\bGamma}^{\rm OC} (P_\Gamma)$ is a constant map (since no Hamiltonian perturbation in this case) on the surface part. If $\Gamma$ has two boundary inputs, then ${\mc M}_{\bGamma}^{\rm OC} (P_\Gamma)$ cannot be zero-dimensional. Hence $\Gamma$ has only one input.  By the zero index condition, the input 
must be labelled by a critical point of $f_L: L \to {\mb R}$ for some $L \in {\mc L}$ with minimal Morse index.
\end{proof}

We introduce the following notation for the moduli spaces with fixed limits along the semi-infinite edges. Suppose 
\[
x_d \in {\mc I}(\LB_{d-1}, \LB_d), \ldots, x_1 \in {\mc I}(\LB_d, \LB_1),\ \xx \in {\rm crit}(f_X).
\]
For $i = 0, 1$, denote  
\[
{\mc M}_{d, 1}^{\rm OC}(x_1, \ldots, x_d, \xx)_i:= \bigcup_{\bGamma} {\mc M}_\bGamma^{\rm OC}(P_\Gamma)
\]
where the union is taken over all open-closed map types $\bGamma$ of expected dimension $i$ whose boundary inputs are labelled by $x_1, \ldots, x_d$ (in counterclockwise order) and the outgoing interior leaf is labelled by $\xx$.

\subsubsection{The open-closed map}

Recall that 
we have fixed a collection of Lagrangian submanifolds ${\mc L}$ from which we have 
constructed
a curved $A_\infty$ category $\Fuk_{\mc L}^\sim(X, \bb)$ and an associated flat $A_\infty$ category $\Fuk_{\mc L}^\flat(X, \bb)$. Given a subset of weakly unobstructed branes ${\mf L}\subset MC({\mc L})$ we have a full subcategory $\Fuk^\flat_{\mf L}(X, \bb)$ 
whose objects are the branes ${\mf L}$.

\begin{definition}[Open-closed map] 
Write for simplicity
\[
CC_\bullet (\Fuk^\flat_{\mf L}(X, \bb)):=CC_\bullet ( \Fuk^\flat_{\mf L}(X, \bb), \Fuk^\flat_{\mf L}(X, \bb) ).
\]
Define the {\em bulk-deformed open-closed map}
\begin{multline}
  OC_d^\sim (\bb): CC_\bullet  ( \Fuk^\flat_{\mf L}(X, \bb))  \to CM^\bullet (f_X, h_X) \\
  a_d\otimes \cdots \otimes a_1 \mapsto \sum_{\xx\in {\rm crit}(f_X)} \sum_{ [u] \in {\mc
      M}^{\rm OC}_{d, 1}(x_1, \ldots, x_d, \xx)_0} (-1)^{\heartsuit + |x_{d(\circ)}|} \wt(u) \xx
\end{multline}
with weightings as in \eqref{weightings}, but with the product of parallel transports and generators $a_1,\ldots, a_d$ now an element of $\Lambda$.  The chain-level open-closed map $OC(\bb)$ is the direct sum $OC_d^\sim$ deformed by the Maurer-Cartan data on each Lagrangian brane: 
\begin{multline*}
OC(\bb):CC_\bullet ( \Fuk^\flat_{\mf L}(X, \bb)) \to CM^\bullet( f_X, h_X), \\
a_d \tensor \dots \tensor a_1 \mapsto \\ \sum_{j_1, \dots, j_d \geq 0} OC_{d+ j_1 + \cdots + j_d}^\sim (\underbrace{b_d, \dots, b_d}_{j_d},a_d,,\dots,, a_3,\underbrace{b_2,\dots, b_2}_{j_2},a_2, \underbrace{b_1,\dots, b_1}_{j_1},a_1)
\end{multline*}
where $a_i \in {\rm Hom}( \WB, \WB_i )$ and $\WB_i = (\LB_i, b_i)$.
\end{definition} 

\begin{proposition} \label{occhain} 
The open-closed map $OC(\bb): CC_\bullet( \Fuk^\flat_{\mf L}(X, \bb)) \to CM^\bullet(f_X, h_X)$ is a chain map, that is,
\[
OC(\bb) \circ \delta_{CC_\bullet}(\bb) = \delta_{\rm Morse} \circ OC(\bb)
\]
where $\delta_{CC_\bullet} (\bb)$ is the Hochschild differential on $CC_\bullet ( \Fuk^\flat_{\mf L}(X, \bb))$. Therefore $OC(\bb)$ induces a map between (co)homology
\[
[OC(\bb)]: HH_\bullet(  \Fuk^\flat_{\mf L} (X, \bb)) \to HM^\bullet(f_X, h_X)\cong QH^\bullet(X, \bb).
\]
\end{proposition} 

\begin{proof}[Sketch of proof]
The identity follows from the description of the boundary strata of open-closed moduli spaces ${\mc M}_{d, 1}^{\rm OC}(a_1, \ldots, a_d, \xx)_1$ in Lemma \ref{tboundplus} with verification of signs. We remark that the terms involving the curvature $m_0(1)$ vanish,
since by assumption the output of $m_0(1)$ is a multiple $w 1_{\LB_i}$ is the identity of each brane $\LB_i$.  The strict
identity axiom implies that $OC$ vanishes except in the case of two inputs, in which case the two terms
involving $m_0(1)$ cancel.
\end{proof}

\begin{remark} \label{indep2} 
Continuing Remark \ref{indep}, the 
open-closed map is independent of the choice of the stabilizing divisor, the 
perturbation, and only depends on the isotopy class of the bulk deformation. The proof of independence uses a moduli space of quilted holomorphic disks shown in Figure \ref{concentric}. 
\begin{figure}[h]
  \centering \scalebox{.8}{
\begingroup%
  \makeatletter%
  \providecommand\color[2][]{%
    \errmessage{(Inkscape) Color is used for the text in Inkscape, but the package 'color.sty' is not loaded}%
    \renewcommand\color[2][]{}%
  }%
  \providecommand\transparent[1]{%
    \errmessage{(Inkscape) Transparency is used (non-zero) for the text in Inkscape, but the package 'transparent.sty' is not loaded}%
    \renewcommand\transparent[1]{}%
  }%
  \providecommand\rotatebox[2]{#2}%
  \newcommand*\fsize{\dimexpr\f@size pt\relax}%
  \newcommand*\lineheight[1]{\fontsize{\fsize}{#1\fsize}\selectfont}%
  \ifx\svgwidth\undefined%
    \setlength{\unitlength}{448.99819147bp}%
    \ifx\svgscale\undefined%
      \relax%
    \else%
      \setlength{\unitlength}{\unitlength * \real{\svgscale}}%
    \fi%
  \else%
    \setlength{\unitlength}{\svgwidth}%
  \fi%
  \global\let\svgwidth\undefined%
  \global\let\svgscale\undefined%
  \makeatother%
  \begin{picture}(1,0.40377869)%
    \lineheight{1}%
    \setlength\tabcolsep{0pt}%
    \put(0,0){\includegraphics[width=\unitlength,page=1]{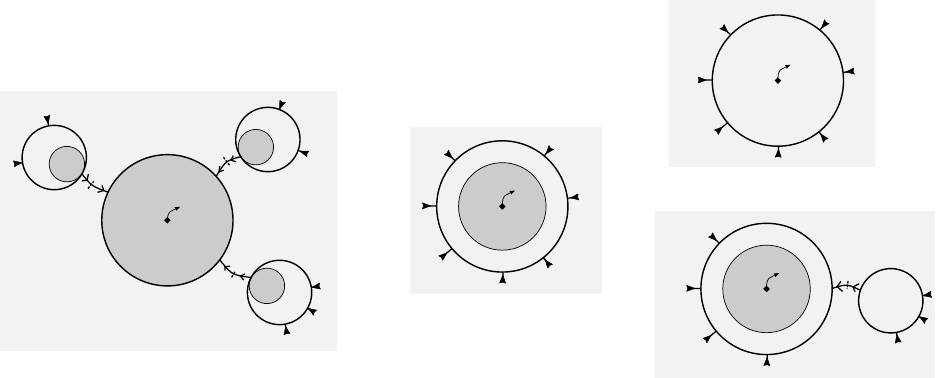}}%
    \put(0.01425502,0.04749745){\color[rgb]{0,0,0}\makebox(0,0)[lt]{\lineheight{1.25}\smash{\begin{tabular}[t]{l}(a)\end{tabular}}}}%
    \put(0.72125218,0.24829136){\color[rgb]{0,0,0}\makebox(0,0)[lt]{\lineheight{1.25}\smash{\begin{tabular}[t]{l}(b)\end{tabular}}}}%
    \put(0.71956482,0.01712525){\color[rgb]{0,0,0}\makebox(0,0)[lt]{\lineheight{1.25}\smash{\begin{tabular}[t]{l}(c)\end{tabular}}}}%
  \end{picture}%
\endgroup%
}
  \caption{Curve types (a), (b), (c) that can occur on the boundary of a one-dimensional moduli space of quilted disks with concentric seam (center).} 
\label{concentric}
\end{figure}
Each domain $C$ is a collection of disks $S_v, v \in \Ver(\Gamma)$
possibly with an additional {\em seam} which is an embedded circle
$Q_v \subset S_v$ either tangent to a single point on the boundary $\partial S_v$, or a concentric dilation of the boundary circle $Q_v$ towards the outgoing marking, as in
Figure \ref{concentric}. Given an isotopy of Donaldson hypersurfaces $D_t \subset X, t \in [0,1]$ the resulting moduli spaces with $d$ boundary inputs and one interior output are denoted $\M_{d,1}(L,D_t)$.  
On the components without seams "before" the seam, with respect to the ordering of components starting with the incoming edges, 
   the complex structure, divisor, and bulk deformation used are 
  $J_0,D_0, \bb_0$, while on components "after" the seam those used
  are $J_1,D_1,\bb_1$.  
  
  The one-dimensional components of the moduli spaces so defined are compact one-manifolds with boundary corresponding to three types of configurations:  (a) Configurations $u: C \to X$ where the inner seam $Q$ has "bubbled off" onto the boundary $\partial S$ creating a number of quilted disks with seams tangent to the boundary (b) configurations $u: C \to X$ where the inner seam $Q$
  has collapsed onto the output inner marking and (c) configurations
  $u: C \to X$ where an unquilted disk $S_v$ has broken off.  
  
  The description of the boundary configurations gives a chain homotopy as follows.
  Configurations of the first type (a) contribute to the map 
  \[  OC(\bb)_1 \circ CC_\bullet(\phi):  
   CC_\bullet(\Fuk^\flat_{\cL}(X,\bb)_0) \to  CM(f_X,h_X)  \] 
 while configurations of the second type are exactly those of $OC_0$.
 Configurations of the third type are of the form  $ OC_\bullet \circ \delta $
 where $OC_\bullet$ is a variant of the open closed map (shifted by degree) 
 that counts rigid treed quilted disks where the radius of the seam is allowed
 to vary between $0$ and $1$.   Restricted to Hochschild cycles, contributions of this type vanish giving an equality between $OC_0$ and $OC_1 \circ HH_\bullet (\phi)$.   Since $HH_\bullet(\phi)$ induces an isomorphism of Hochschild homologies 
 by Remark \ref{indep4}, this gives an identification of the images. 
 \end{remark}

\subsection{Spectral decomposition under open-closed map}\label{subsec:ocspect}

In this subsection,  we prove Theorem \ref{thm16} in the introduction which says that the open-closed map respects 
the spectral decomposition of the Fukaya category and quantum cohomology.
 The components in the spectral decomposition of the quantum cohomology may be viewed  as generalized eigen-spaces of quantum multiplication by either the symplectic class $[\om]$ or the first Chern class $c_1(X)$. To work with the the latter viewpoint, we need the additional assumption that $c_1(X)$ is representable in the following sense: 
  We say  that $c_1(X)$ is {\em representable with respect to $\LL$} if some
 multiple of the Poincar\'e dual of  $c_1(X)$ can be represented by a smooth submanifold $Y$ in $X$ disjoint from $|{\mc L}|$, so that 
for each component $L \in \LL$ the submanifold $Y$ represents the Maslov 
class in $H^2(X,L)$.  This condition is automatic if $\LL$ consists of a single 
brane $L$,  since we may take $Y$ to be the zero locus of a generic section of the anticanonical bundle, trivialized over $L$ using the orientation.  The following result on the open-closed map subsumes Theorem \ref{thm16}.  Recall the definitions of $[\omega]^\bb$
and $c_1(M)^{\bb}$ from \eqref{ombb}, \eqref{c1bb}.

\begin{theorem}\label{spectral} 
The image of $HH_\bullet( \Fuk_{\mf L} (X,\bb)_w )$ in $QH^\bullet (X,\bb)$ under the open-closed map $OC(\bb)$ lies in the generalized eigenspace for quantum multiplication $[\omega]^\bb \star_\bb$ by the symplectic class $[\omega]^\bb$ with eigenvalue $D_q w$. 

Furthermore, suppose that $c_1(X)$ is representable with respect to $\LL$.  The image of $HH_\bullet( \Fuk_{\mf L} (X,\bb)_w )$ in $QH^\bullet (X,\bb)$ under the open-closed map $OC(\bb)$ lies in the generalized eigenspace for quantum multiplication $c_1(X)^\bb \star_\bb$ with eigenvalue $w$.
\end{theorem} 

\begin{remark} \label{eigenval} 
In the monotone, non-bulk-deformed case the image of each summand 
\[
OC(\bb) ( HH_\bullet(\Fuk_{\mf L} (X,\bb)_w) ) \subset  QH^\bullet (X, \bb)
\]
is contained in the $w$-eigenspace of the operator given by quantum multiplication by the first Chern class, as in Sheridan  \cite{Sheridan_2016} (see also Yuan \cite{Yuan_2021} in a more general setting). The above statement is compatible with this fact since in the monotone case $w$ has only terms of power $q^{1/\lambda}$ where $\lambda [\omega] = c_1(X)$. Hence $D_q w= w/ \lambda$ in this case. To see why $w$ is an eigenvalue of 
$c_1(X)^\bb \star_\bb$, one can also replace the Donaldson hypersurface  by an anticanonical divisor and make use of the fact that only Maslov 2 disks contribute to the potential function.
\end{remark}

\subsubsection{The case of a length-one Hochschild chain}
\label{samplecase}

We first give a simplified argument for Theorem 
\ref{spectral} assuming that the Hochschild chain has length one, all weakly bounding cochains and bulk deformations are zero and the disks are transversely cut out without using a domain-dependent perturbation.   Thus, by assumption, the curvature $m_0(1)$ is a multiple of the strict unit for all involved branes.   The proof in this simplified case is based on the study of moduli spaces of open-closed domains with an auxiliary interior marking with a specified offset angle in comparison with the first boundary marking.  Given a treed disk $C$ of an open-closed domain type (we allow interior edges to acquire length), there is a unique disk component $S_0 \subset C$ that is closest to the unique outgoing interior semi-infinite edge $T_0$. We call $S_0$ the {\it central disk} and let $z_0 \in S_0$ be the interior special point that is connected to $T_0$. There is also a boundary special point $w_0\in \partial S_0$ that is closest to the $0$-th boundary incoming semi-infinite edge.  Identify $S_0$ with $\mb{D}$ biholomorphically so that $z_0$ resp. $w_0$ is identified with $0\in {\rm Int} {\mb D}$ resp. $1 \in \partial {\mb D}$ and such an identification $S_0 \cong {\mb D}$ is unique. There is a contraction map 
\begin{equation}
  \label{contractionmap} 
  \sigma_C: C \to S_0
\end{equation}
which is the identity on the central disk $S_0 \subset C$ and which contracts points on other surface components $S_v$ or edges $T_e$ to the corresponding attaching points on the central disk $S_0$. A point $z \in C$ is said to have {\it offset angle $\theta\in S^1$} if 
\[
\sigma_C(z) \in ( e^{i \theta} {\mb R}_+ \cap S_0 ) \cup \{z_0\}.
\]
Let $\Gamma$ be a stable open-closed type consisting of disks with a single boundary leaf $T_1$, the interior leaf $T_0$, and an interior marking $z_\bullet \in S_0$.  
Fix an angle $\theta \in S^1$. Define the subspace
\[
{\mc M}_\Gamma^\theta \subset {\mc M}_\Gamma
\]
consisting of isomorphism classes of open-closed domains $C$ of type $\Gamma$ such that the auxiliary marking $\mathring{z}$ has offset angle $\theta$, as in Figure \ref{fig:spectral0}.  
\begin{figure}[h]
    \centering
    \includegraphics[scale=0.9]{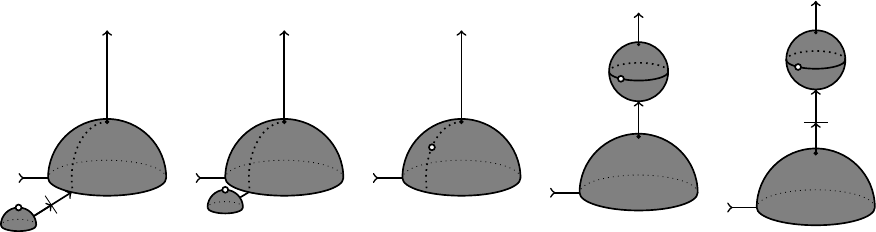}
    \caption{A one-dimensional moduli space considered to show the spectral property of the open-closed map. The auxiliary marking, which is hollow in the picture, must have a fixed angle 
    shown as 
    the dashed curve in the first three configurations.}
    \label{fig:spectral0}
\end{figure}
Suppose $[\omega]$ is integral and let 
\[ Y \subset X \] 
be a representative of $[\omega]$ transverse to $D$.  
We construct a moduli space $\M_{1,1}^\theta(L,Y) $
of open-closed maps bounding $L$ equipped with a map 
\[ \M_{1,1}^\theta(L,Y) \to \M_{1,1}^\theta \]
as follows.  Configurations in $\M_{1,1}^\theta(L,Y)$
consist of holomorphic treed disks $u:C \to X$ with an open-closed domain type $\Gamma$ with the boundary edge $T_1$ labelled by components of $\alpha$, the interior {gradient} leaf  $T_0$ labelled by a critical point $\xx\in {\rm crit}(f_X)$, and for the interior auxiliary marking $\mathring{z}$ we require that 
\[ u(\mathring{z}) \in Y .\] 
%
We may assuming that the perturbations are independent of the position of the
point $\mathring{z}$, since the intersections with $D$ stabilize the domain.

Using the moduli spaces above, we define a modified version of the open-closed map.
Given a Floer cochain 
\[ \alpha \in CF^\bullet(\WB, \WB) \] 
for some weakly unobstructed brane $\WB$ with potential $w$, define a map 
\[ OC^Y(\alpha) \in CM(f_X,h_X) \]  
by weighting the contributions of the moduli space  $ \M_{1,1}^\theta(L,Y) $ by the 
coefficients of $\alpha.$  Each (true) boundary stratum of  ${\mc M}^\theta_{1,1}(L,Y)_1$ consists of configurations $(C,u: C \to X)$ with exactly one broken edge $T_e \subset C$ 
and belongs to the following types, as in  Figure \ref{fig:spectral0}:
\begin{enumerate}
    \item configurations $u: C \to X$ with a broken incoming edge $T_1$ contributing to 
    $OC^Y(m_1(\alpha))$;
    \item configurations $u: C \to X$ with a broken interior edge $T_1$, contributing to $Y \star OC(\alpha)$ (such as the right-most configuration in Figure \ref{fig:spectral0});
    \item configurations $u: C \to X$ with one disk component $S_v$ containing $z_\bullet$ and no boundary labels, connected to other components by a broken boundary edge $T_e$ (for example, the left-most configuration in Figure \ref{fig:spectral0}), to be explained below; and

    \item configurations $u: C \to X$ with a broken interior leaf $T_1$, 
    contributing to
    $ \delta_{\rm Morse} OC^Y(\alpha)$.
\end{enumerate} 
To understand the contributions arising from the third type of boundary configuration,
write the potential function 
\[ w = \sum_i c_i q^{A_i} \]  
as a sum over contributing holomorphic disks $u_i: C \to X$ of area
$A_i$ with coefficients $c_i \in \Q$. Holomorphic disks with energy $A_i$ intersect $Y$ at $A_i$ points counted with sign, since by construction the intersection is transversal and the perturbations are independent of the choice of auxiliary marking. 
Each configuration contributes 
$ A_i OC(\alpha)$, as we have $ A_i$ choices of the auxiliary marking $z_\bullet$.  Since the signed count of boundary points of the one-dimensional moduli space vanishes, we obtain the relation
\[
 Y \star OC(\alpha) = \sum_i c_i A_i q^{A_i} OC(\alpha) =  (D_q w) OC(\alpha) \quad
 \text{mod} \ \on{Im} \delta_{\rm Morse}.
\]
The argument for quantum multiplication by the first Chern class $c_1(X)^\bb$ is similar. Suppose $Y \subset X$ is a smooth submanifold representing half the Maslov class in $H^2(X,L)$. Then
\[
 Y \star OC(\alpha) = \sum_i c_i \frac{1}{2} I_i q^{A_i} OC(\alpha) =  w OC(\alpha)
  \ \text{mod} \ \on{Im} \delta_{\rm Morse}
\]
where $I_i$ is the Maslov index of the $i$-th disk contributing to $w$, necessary equal to $2$ since the bulk and boundary deformations vanish.

\subsubsection{Treed disks with auxiliary markings and specified offsets}

The proof of the Spectral Theorem \ref{spectral} in the general case involves moduli spaces with a collection of interior markings at fixed offset angles.  These moduli spaces define generalized open closed maps which satisfy a recursive relation, equivalent to the 
image of the open-closed map lying in a generalized eigenspace for quantum multiplication.

\begin{definition}
\label{angleseq} 
\begin{enumerate}
\item An {\em angle sequence} is a collection 
$\theta_1, \ \theta_2, \  \ldots \in {\mb R}/ 2\pi {\mb Z} \cong S^1$
of distinct, non-zero angles. 
    \item An {\em open-closed domain type with auxiliary markings} is an open-closed domain type $\Gamma$ together with decompositions
    %
    $ {\rm Leaf}_{\bullet, {\rm const}}(\Gamma) = {\rm Leaf}_{\bullet, {\rm normal}}(\Gamma) \sqcup {\rm Leaf}_{\bullet, {\rm auxiliary}}(\Gamma),
    $
    %
    and 
    %
    $ {\rm Leaf}_{\circ} = {\rm Leaf}_{\circ, {\rm normal}}(\Gamma) \sqcup {\rm Leaf}_{\circ, {\rm auxiliary}}(\Gamma)$ 
    %
    such that the $0$-th boundary leaf is normal and such that the path connecting each auxiliary leaf to the central disk $S_0$ contains at least one broken boundary edge $T_e$. Such a domain type $\Gamma$ is said to be {\em of type $(l_{\black},l_{\white})$} if there are $l_{\black}$ interior auxiliary leaves and $l_{\white}$ boundary auxiliary leaves. 

    \item An {\em open-closed treed disk with $m$ auxiliary markings} of type $\Gamma$ and angle sequence $\ul{\theta}$ is a treed disk $C = S \cup T$ of type $\Gamma$ such that the offsets of the auxiliary markings $\mathring{z}_{\on{aux},1}, \ldots, \mathring{z}_{\on{aux},m}$ are $\theta_1, \ldots, \theta_m$.
    
    \item By cutting along breakings, a treed disk $C$ is separated to tree disks $C_1,\ldots, C_k$  with no breakings.  The components $C_1,\ldots, C_k$ will be called the {\it unbroken components}. The {\it branch} of $C$ at offset angle $\theta$ of a treed disk $C$ with one auxiliary marking $\mathring{z}_i$ is the union of unbroken components $C_j$ that are connected to the central treed disk $C_0$ via the boundary special point on $C_0$ with offset angle $\theta$.
    \end{enumerate}
\end{definition}

\begin{center}
\begin{figure}

\includegraphics{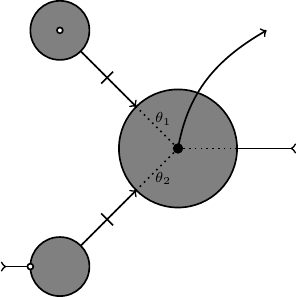}
\caption{A treed disk with one interior auxiliary marking and one boundary auxiliary marking (the hollow markings). The semi-infinite on the right is at the $0$-th boundary marking.}
\end{figure}
\end{center}

\subsubsection{Generalized open-closed maps}

Define a generalized open-closed map using the above kind of domains
as follows.   Open-closed map types are map types with the following 
labelling of edges and boundary arcs: Each interior auxiliary leaf $T_{\black,i}$ is labelled either by $Y$ or by the bulk deformation $D_q \bb$, and the two short arcs on both sides of a boundary auxiliary leaf $T_{\white,i}$ are labelled by the same Lagrangian brane.  Consider a sequence of weakly unobstructed branes $(\WB_0, \ldots, \WB_d)$.  
Define a (not necessarily chain) {\em fixed-angles map} 
\[
OC_m(\bb): CC_\bullet( \Fuk_{\mf L}(X,\bb)_w) \to CM^\bullet( f_X, h_X)
\]
by counting map types $\bGamma$ with $m$ auxiliary markings with the following conditions as in Figure \ref{fig:OCmm}:
\begin{enumerate}
\item \label{aa} each interior auxiliary marking $\mathring{z}_i$ maps either to $Y$ or to the bulk deformation $D_q \bb$;
\item \label{bb} each auxiliary boundary marking $\mathring{z}_i$, if it is on a boundary arc labelled by $\WB_i$, then the contributions to 
$OC_m(\bb)$ are weighted by the coefficients of $D_q b_{\WB_i}$; and
\item \label{cc}  each branch $C_j$ contains a normal (that is, non-auxiliary) boundary marking 
weighted by the coefficients of the Hochschild chain $\alpha$.
\end{enumerate}
We remark that the bulk deformation at normal (non-auxiliary) interior markings is $\bb$ and the bulk deformation at auxiliary markings is $D_q \bb$.  Analogously the boundary insertion at an auxiliary marking is $D_qb$ and a normal boundary marking is either weighted by the coefficients of the Hochschild chain $\alpha$ or has an insertion $b$. 
\begin{figure}[h]
    \centering
    \includegraphics[scale = 0.7]{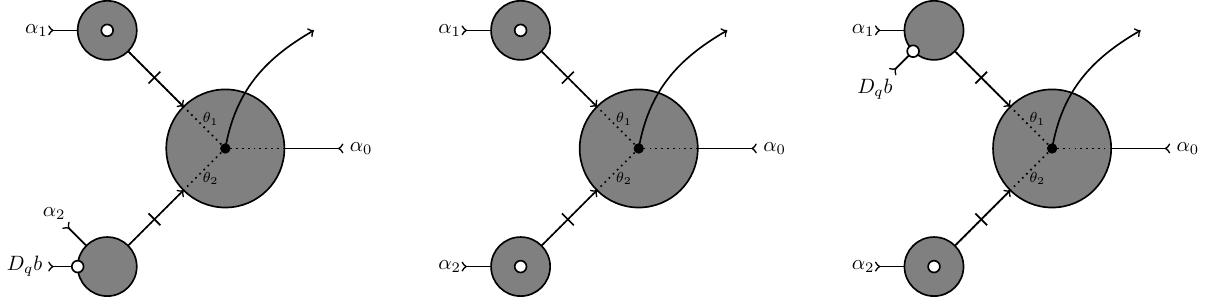}
    \caption{Configurations possibly contributing to the map $OC_2$ with three boundary insertions $\alpha_0$, $\alpha_1$ and $\alpha_2$. The insertion by weakly bounding cochains are omitted.  Each branch contains an auxiliary marking (either interior or boundary) and at least one boundary insertion labelled $\alpha$.}
    \label{fig:OCmm}
\end{figure}

We will prove in Lemma  \ref{lemma_spectral_2} that the fixed-angles maps $OC_m$ have image lying in generalized eigenspaces of quantum multiplication by $[\omega]^\bb$. To prove the Lemma, we introduce some variants of the fixed-angle open-closed map. Define the maps
\[
OC_{m,\white}(\bb) \quad \text{resp.} \quad OC_{m,\black}(\bb): CC_\bullet( \Fuk_{\mf L}(X,\bb)_w) \to CM^\bullet( f_X, h_X)
\]
as counts of the same treed holomorphic disks  as those counted by $OC_{m+1}(\bb)$ with the additional condition that the last auxiliary marking $\mathring{z}_{m}$ is a boundary resp. interior marking. Thus,
\[OC_{m}(\bb)=OC_{m,\white}(\bb)+ OC_{m,\black}(\bb).\]
We need another variation of the map $OC_{m,\white}$. 
Consider treed disks with $m+1$ auxiliary markings $\mathring{z}_i$ (either interior or boundary) with the last one $\mathring{z}_{m+1}$ being a boundary auxiliary marking; the first $m$ branches $C_1,\dots, C_m$ satisfying conditions \eqref{aa},  \eqref{bb} \eqref{cc} listed above; and the last branch $C_{m+1}$ satisfying \eqref{bb} but not \eqref{cc}, that is,
 none of the normal boundary markings in the branch $C_{m+1}$ are labelled by the Hochschild chain $\alpha$. 
The count of such configurations defines a (not necessarily chain) map
\[
OC_{m, +}(\bb): CC_\bullet( \Fuk_{\mf L}(X,\bb)_w) \to CM^\bullet( f_X, h_X)
\]
as in Figure \ref{fig:OCmp}.
\begin{figure}[h]
    \centering
    \includegraphics[scale=.7]{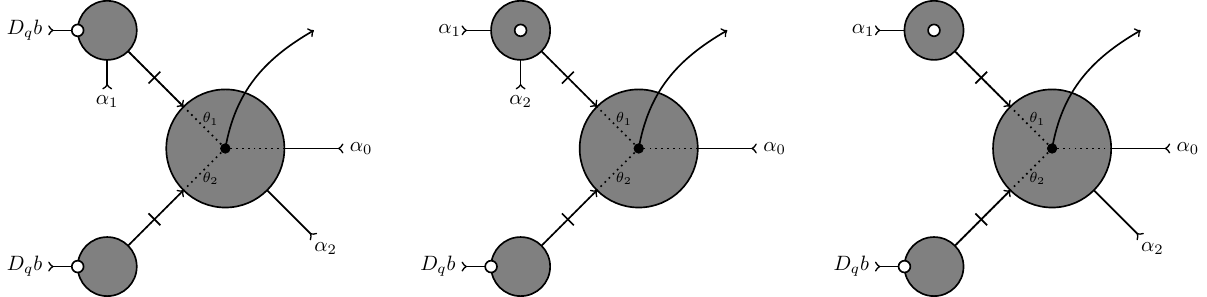}
    \caption{Configurations possibly contributing to the map $OC_{1, +}$ with three boundary insertions $\alpha_0$, $\alpha_1$, and $\alpha_2$. The insertions by weakly bounding cochains are omitted.  The second (last) branch can only contain a boundary auxiliary marking (hollow) labelled by $D_q b$ and does not contain insertions by $\alpha_i$.}
    \label{fig:OCmp}
\end{figure}

\begin{lemma}\label{lemma_spectral_1}
Let $\alpha \in CC_\bullet( \Fuk_{\mf L}(X,\bb)_w)$ be a Hochschild cycle. For any $m \ge 1$,
\[ 
OC_{m+1,\white}(\bb)(\alpha) + OC_{m, +}(\bb)(\alpha) \in {\rm Im} (\delta_{\rm Morse} ).
\]
\end{lemma}

\begin{proof}   The relation above follows from studying a moduli space where
the number of enforced breakings is one less than the number of auxiliary markings.
Consider the moduli spaces ${\mc M}^\theta_{\bGamma}(L,D)_1$ for types 
$\bGamma$ with $m+1$ auxiliary markings and exactly $m$ breakings, with one breaking 
between each of the first $m$ branches $C_1,\ldots, C_m$ and the treed central disk $C_0$; and whose last auxiliary marking $\mathring{z}_{m+1}$ is a boundary marking not separated from $C_0$ by any broken edges.  Counts of rigid such configurations define yet another
map 
\[ OC_{m, ++}(\bb): CC_\bullet( \Fuk_{\mf L}(X,\bb)_w) \to CM^\bullet( f_X, h_X) .\]
True boundary strata of the one-dimensional space of such configurations consist of configurations $u: C \to X$ with an additional breaking at some edge $T_e$.  If the additional breaking is at an interior edge $T_e$, then the contribution of such configurations gives an element in ${\rm Im}(\delta_{\rm Morse})$. If the additional breaking is at a boundary edge $T_e$ creating an unbroken component $C' \subset C$ containing no auxiliary markings $z_{{\diamond},i}$ or normal boundary edges labelled by $\alpha$, then the contribution of $u: C \to X$ is zero by forgetful property of the perturbation data and the definition of weakly bounding cochain. 
The other possibilities where the additional breaking is at a boundary edge
are as follows; in each case we identify  the  contributions of 
the corresponding strata.
\begin{enumerate}
\item The additional breaking is at the offset angle $\theta_{m+1}$
  such that the $(m+1)$-st auxiliary marking $\mathring{z}_{m+1}$ is
  separated from the central treed disk $C_0$ by one breaking. The
  contribution of these configurations is the term
  $OC_{m+1,\white}(\bb)(\alpha) + OC_{m, +}(\bb)(\alpha)$.
\item The additional breaking is at a generic offset angle $\theta$
  different from the fixed ones $\theta_1, \ldots, \theta_{m+1}$.  The
  breaking separates the central treed disk $C_0$ with a treed disk
  $C'$ labelled by some subset of the $\alpha_i$ and weakly bounding
  cochains $b_{\WB_i}$.  These configurations contribute to
  $(OC_{m,++})(\bb)_k(\delta_{l}(\alpha))$ where $(OC_{m,++})(\bb)_k$ is the
  open-closed map acting on $k$ Hochschild inputs and $\delta_{l}$ is
  the component of the Hochschild differential $\delta$ of $\alpha$
  involving contractions of $l$ elements.
\item The additional breaking is at one of the fixed offset angles
  $\theta_i$, $i = 1, \ldots, m$, such that the auxiliary marking
  $\mathring{z}_i$ in the branch $C_i$ is still separated from the
  central disk $S_0$ by the breaking of a single edge $T_e$.  These
  configurations also contribute to $(OC_{m,++})(\bb)_k(\delta_{l}(\alpha))$
  where $(OC_{m,++})(\bb)_k$ is the open-closed map acting on $k$
  Hochschild inputs and $\delta_{l}$ is the component of the
  Hochschild differential $\delta$ of $\alpha$ involving contractions
  of $l$ elements. 
\item The additional breaking is at one of the fixed offset angle
  $\theta_i$ such that the auxiliary marking $\mathring{z}_i$ in the
  branch $C_i$ is separated from the central disk by two broken edges,
  say $T_{e'}$ and $T_{e''}$.  Moreover, the unbroken component $C'$
  containing $\mathring{z}_i$ contains a normal boundary marking
  labelled by some $\alpha$. Such configurations can be viewed as the
  boundary of two strata of one dimension higher.  Indeed, either of
  the broken edges $T_{e'}$ and $T_{e''}$ may be glued, as Figure
  \ref{fig_spectral_4}. Therefore, we can regard this type of boundary
  strata as a fake boundary component of the one-dimensional moduli
  space.
    
\begin{figure}[h]
  \includegraphics[trim={0 2.5cm 0 0}, clip]{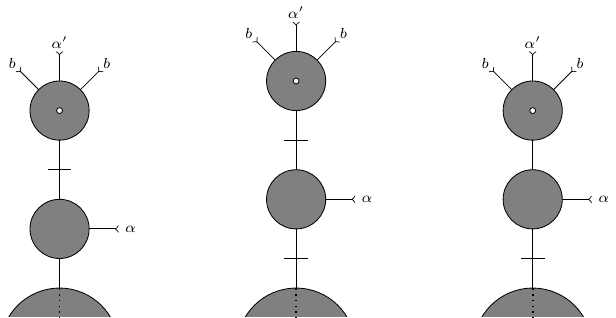}
  \caption{A fake boundary stratum.}\label{fig_spectral_4}
\end{figure}

\item The additional breaking is at one of the fixed offset angle
  $\theta_i$ such that the auxiliary marking $z_{{\diamond},i}$ in
  this branch $C_i$ is separated from the central disk $C_0$ by
  breakings of two edges, say $T_{e'}$ and $T_{e''}$.  Moreover, the
  unbroken component $C'$ containing this auxiliary marking contains
  no normal boundary markings labelled by $\alpha$.
\end{enumerate}

\begin{figure}[h]
  \includegraphics[trim ={0 2.5cm 0 0}, clip]{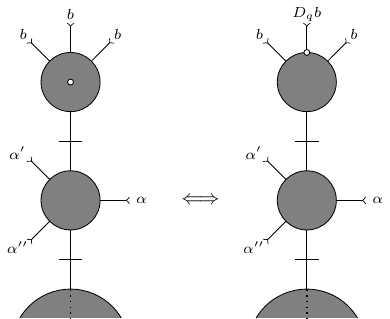} 
\caption{The cancellation of two boundary contributions.}
   \label{fig:spectral_3}
\end{figure}

In the rest of the proof we study the contribution from configurations in the last item.  More precisely, those with an interior auxiliary marking at angle $\theta_i$, together with the corrections arising from the bulk deformation $D_q \bb$, cancel with those with a boundary auxiliary marking, as in Figure \ref{fig:spectral_3}.  Abbreviate $b_\WB$ by $b$ and write 
\[ m_k (a_k,\ldots,a_1) = \sum_{l,d \geq 0} q^l m_{k,l,d}(a_k,\ldots,a_1;\bb,\ldots,\bb) \] 
where $m_{k,l,d}$ is the contribution from treed disks of area $l$ with 
$d$ bulk insertions.
Suppose $Y$ represents $[\omega]$.  We have 
\begin{multline}\label{eqn_spectral}
(  D_q w) 1_\WB =  D_q \sum_{k,d \geq 0} m_k(b, \ldots, b;\bb,\ldots,\bb) =  D_q \sum_{k ,d\geq 0} \sum_{l \geq 0} q^l m_{k,l}(b, \ldots, b;\bb,\ldots,\bb)\\
= \sum_{k\geq 0} \sum_{l \geq 0}  lq^l m_{k,l}(b, \ldots, b;\bb,\ldots,\bb) + \sum_{k ,d\geq 0} \sum_{1 \leq s \leq k} m_k( \underbrace{b, \ldots, b}_{s-1},  D_q b, b, \ldots, b;\bb,\ldots,\bb) \\  + \sum_{k ,d \geq 0} \sum_{1 \leq s \leq d} m_k( b,\ldots, b, \underbrace{\bb, \ldots, \bb}_{s-1},  D_q \bb, \bb,\ldots, \bb) 
\\
= \sum_{k,d \geq 0} m_k^+(b, \ldots, b;\bb,\ldots,\bb) + \sum_{k,d \geq 0} \sum_{1 \leq s \leq k}  m_k( \underbrace{b, \ldots, b}_{s-1}, D_q b, b, \ldots, b;\bb,\ldots,\bb)
\\ 
 + \sum_{k ,d \geq 0} \sum_{1 \leq s \leq d} m_k( b,\ldots, b, \underbrace{\bb, \ldots, \bb}_{s-1},  D_q \bb, \bb,\ldots, \bb) 
\end{multline}
where $m_k^+$ is the count of treed disks with an interior auxiliary marking constrained to lie in $Y$ with $k$ inputs; recall that the perturbations are independent of the auxiliary markings.   The terms involving $D_q \bb$ are produced by the 
quantum corrections to the symplectic form in \eqref{ombb}.  Since the count of configurations with one boundary edge labelled by the identity vanishes automatically by the forgetful axiom, 
the two types of contributions in  Figure \ref{fig:spectral_3} combine with the corrections $D_q \bb$ to give zero.   A similar identity holds
in the case that $Y$ represents $c_1(X)$, with the modification that 
$D_q b$ for a weakly bounding cochain $b$ is replaced by 
\[ \left( \frac{1-E}{2} \right) b = \sum_{k=0}^\infty \left( \frac{1-k}{2} \right) b_k \]  
where $b_k$ is the component of $b$ of degree $k$ and $E$ is the grading (Euler) operator.
\end{proof}

\begin{lemma}\label{lemma_spectral_2}   For each Hochschild cycle $\alpha \in CC_\bullet( \Fuk_{\mf L}(X,\bb)_w)$, one has 
\[
 (Y + D_q \bb) \star_\bb OC_m (\bb) (\alpha)
-  (D_q w ) \ OC_m (\bb)(\alpha)
- OC_{m+1}(\bb) (\alpha) \in {\rm Im}(\delta_{\rm Morse}).
\]
Similarly if $c_1(X)$ is representable then
\[
 \left( Y +  \sum_i  \frac{|\bb_i| - 2}{2} \bb_i \right) \star_\bb OC_m (\bb) (\alpha)
-  w \ OC_m (\bb)(\alpha)
- OC_{m+1}(\bb) (\alpha) \in {\rm Im}(\delta_{\rm Morse}).
\]
\end{lemma}

\begin{proof} The proof of the statement of the Lemma follows by
  considering families of fixed-angle configurations with an
  additional interior marking at some additional fixed angle that is
  not separated from the central disk by a broken edge.  Let
  $\theta_1,\ldots, \theta_{m+1}$ be an angle sequence as in
  Definition \ref{angleseq}. Consider open-closed treed disks
  $u: C \to X$ with $m+1$ auxiliary markings
  $\mathring{z}_1, \dots, \mathring{z}_{m+1}$ constrained to lie at
  angles $\theta_1,\ldots, \theta_{m+1}$ (as defined by
  \eqref{contractionmap}); the domain $C$ has exactly $m$ broken edges
  that split $C$ into a central disk $C_0$ and branches
  $C_1,\dots, C_m$ such that for $i=1,\dots, m$, $\mathring{z}_i$ lies
  in the branch $C_i$; and the last auxiliary marking
  $\mathring z_{m+1}$ is an interior marking and lies in the central
  disk $C_0$. Consider the moduli space
 $\M(\LL,Y)$ of maps $u: C \to X$ for which the images $u( \mathring{z}_i)$ lies either on $Y$ or on the bulk deformation $D_q \bb$.
 The rigid count of such maps defines a (not necessarily chain) map 
 \[ OC_{m,1}(\bb) : CC_\bullet( \Fuk^\flat_{\mf L}(X, \bb)) \to CM^\bullet(f_X, h_X) . \]
The true boundary points of the one-dimensional component of $\M(\LL,Y)$ are configurations $u: C \to X$ with an additional breaking, say at an edge $T_e$. If the new breaking is along an interior edge $T_e$, then such configurations contribute to $Y \star_\bb (OC_m(\bb)(\alpha))$, $D_q\bb \star_\bb (OC_m(\bb)(\alpha))$, or to the Morse coboundary $\delta_{\rm Morse}(OC_m(\bb)(\alpha))$.  On the other hand, if the broken edge $T_e$ is a boundary
 edge and $C'$ is the treed disk separated from $C_0$ one of the following possibilities occur:
 \begin{enumerate}
 \item If the last auxiliary marking $z_{{\diamond},m+1}$ 
   is not separated from $C_0$ by broken edges then $u: C \to X$ contributes
    to the expression $OC_{m,1}(\bb)(\delta(\alpha))$.

  \item If the new breaking $T_{e'}$ is at the $(m+1)$-st fixed offset angle $\theta$ that separates $\mathring{z}_{m+1}$ from the central disk $C_0$, then consider the $(m+1)$-st branch $C_{m+1}$.
    Two cases arise.
    \begin{enumerate}
    \item If the branch $C_{m+1}$ has at least one  normal boundary marking labelled by $\alpha$, then the configuration contributes to
      \[OC_{m+1,\black}(\bb)(\alpha). \]
    \item If the branch $C_{m+1}$ does not have any normal boundary marking labelled by $\alpha$, then as in our sample case in Section \ref{samplecase}, the configuration contributes to the difference 
    \[ (D_q w) OC_m(\bb)(\alpha) - OC_{m,+} (\bb) (\alpha) \]
    as in Figure \ref{fig_spectral_5}, where $OC_{m,+}(\bb)$ counts configurations with a branch $C_{m+1}$ containing a boundary edge labelled $D_q b$ but with no boundary edge labelled $\alpha$.
    \end{enumerate}

   \end{enumerate}

\begin{figure}[h]
\includegraphics[trim ={0 2.5cm 0 0}, clip]{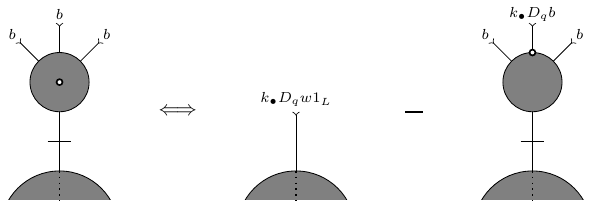} 
\caption{The appearance of the eigenvalue. In the leftmost figure, the auxiliary marking lies on $Y$ or on the bulk deformation $D_q \bb$.}
\label{fig_spectral_5}
\end{figure}
The signed count of the true boundary points of the one-dimensional components of the moduli space 
$\M(\LL,Y)$ is zero, and thus we obtain
\begin{multline*}
0 = ( Y + D_q \bb) \star_\bb OC_m (\bb)(\alpha) \\
 -  (D_q w ) \ OC_m (\bb)(\alpha)
+  OC_{m,+}(\bb)(\alpha) - OC_{m+1,\black}(\bb)(\alpha)  \ \text{mod}   \ {\rm Im}(\delta_{\rm Morse}).   
\end{multline*}
  By Lemma \ref{lemma_spectral_1} 
  $OC_{m,+}(\bb)(\alpha) + OC_{m+1,\white}(\bb)(\alpha)$ vanishes modulo boundary terms.
  Together with the fact that $OC_{m+1,\black}(\bb) + OC_{m+1,\white}(\bb)=OC_{m+1}(\bb)$, we obtain the
claimed identity 
\[
 (Y + D_q \bb) \star_\bb OC_m (\bb) (\alpha)
-  (D_q w ) \ OC_m (\bb)(\alpha)
- OC_{m+1}(\bb) (\alpha) \in {\rm Im}(\delta_{\rm Morse}).
\]
  The proof in the case that $Y$ represents $c_1(X)$
is similar.
\end{proof}

\begin{proof}[Proof of Theorem \ref{spectral}]
Suppose $Y$ represents the Poincar\'e dual of $[\omega]$ and is disjoint from $|\LL|$.
By definition, any Hochschild chain $\alpha$ has bounded length.
Hence for $m$ sufficiently large, $OC_m(\bb)(\alpha) = 0$.  
By Lemma \ref{lemma_spectral_2},
\[
( ((Y + D_q \bb) \star ) -  (D_q w ) \on{Id} )^m (OC(\bb)(\alpha)) \in {\rm Im}(\delta_{\rm Morse}).
\]
Therefore, passing to cohomology we have
\[
( ( [\omega]^\bb \star ) - (D_q w) \on{Id} )^m ([OC(\bb)]([\alpha])) = 0.
\]
Hence $[OC(\bb)]([\alpha])$ is in the generalized eigenspace of $[\omega]^\bb \star$ with eigenvalue $D_q w$.
Similarly, if $c_1(X)$ is representable with respect to $\LL$ by some $Y$ then 
\[
( ( c_1(X)^\bb \star ) - w \on{Id} )^m ([OC(\bb)]([\alpha])) = 0 \]
and $[OC(\bb)]([\alpha])$, if non-zero, has eigenvalue $w$.
\end{proof}

\begin{corollary} \label{orthogonal1}  For any $w \neq w'$,
the images of the Hochschild homology groups 
$ HH_\bullet(\Fuk_{\mf L} (X,\bb)_w) $ 
and 
$HH_\bullet( \Fuk_{\mf L} (X,\bb)_{w'}) $ 
in $QH(X,\bb)$ under $[OC(\bb)]$ are orthogonal with respect to the Poincar\'e pairing.  
\end{corollary}

\begin{proof} Since quantum multiplication by even classes is self-adjoint with respect
to the Poincar\'e pairing, the claim follows from Theorem \ref{spectral}. 
\end{proof} 

\subsection{Closed-open maps}\label{section:CO}

The closed-open map takes as input a quantum cohomology class and its output is an element of Hochschild cohomology: 
\[ 
[CO(\bb)]: QH^\bullet (X,\bb) \to HH^\bullet (\Fuk^\flat_{\cL}(X,\bb)).
\]
In the monotone situation, the construction of this map is a special case of the construction of the functor described in the work of Ma'u, Wehrheim, and the second author \cite{Mau_Wehrheim_Woodward}. For monotone symplectic manifolds $X_0$, $X_1$, \cite{Mau_Wehrheim_Woodward} defines an \ainfty functor
\[
\Phi: \Fuk(X_0^- \times X_1) \to \on{Func}(\Fuk(X_0), \Fuk(X_1)).
\]
If $X_0:=X_1:=X$, $\Phi$ maps the diagonal to the identity functor:
\[ \Phi( \Delta \subset X^-\times X) = \Id_{\Fuk(X)} \] 
and $\Phi$ restricts to an \ainfty map from the Fukaya algebra of the diagonal to the space of natural transformations on the identity functor, i.e., the space of Hochschild cochains. 

\subsubsection{Definition of the chain-level closed-open maps}

We first describe the combinatorics of the domains responsible for the chain-level maps.
  
\begin{definition} 
A {\em closed-open domain type} consists of a two-colored tree $\Gamma$ with the output $e_\infty \in \Edge_\rightarrow(\Gamma_\circ)$ in the disk part $\Gamma_\circ$ and with exactly one {gradient} leaf $e_\black \in {\rm Leaf}_{\rm grad}(\Gamma)$, a metric type $\ul\ell$ and a weighting type $\ul{\on{wt}}$ for boundary semi-infinite edges and interior constrained leaves,
as in Definition \ref{octype}. Moreover, we require that the weighting type comes from a weighting on semi-infinite edges that satisfies \eqref{weight_relation}. A {\em closed-open map type} $\bGamma$ consists of a closed-open domain type $\Gamma$ (which has $d$ boundary inputs), a collection
\[
\ul{x} = (x_0, x_1, \ldots, x_d) \in \cI( \LB_d, \LB_0)  \times \cI( \LB_0, \LB_1) \times \dots \times \cI( \LB_{d-1}, \LB_d) 
\]
of critical points corresponding to a sequence of Lagrangians
\[
\ul{L} = (L_0, \ldots, L_d ),
\]
a collection 
\[
\ul{\beta} = (\beta_v)_{v\in {\rm Vert}(\Gamma)}
\]
of homology classes, a critical point $\xx$ of the Morse function $f_X : X \to \R$, and additional interior labelling data $\ul{\lambda}$ as in \eqref{ullam} indicating whether
the interior leaf represents a Morse
trajectory in $X$ or an intersection with the Donaldson hypersurface $D$.
\end{definition}

Moduli spaces for closed-open maps with Lagrangian boundary conditions are defined similarly as the case of open-closed maps, but now the Morse trajectory on the {gradient} leaf goes in the opposite direction. Given a closed-open map type $\bGamma$ and a perturbation $P_\Gamma$, let
${\mc M}_{\bGamma}(P_\Gamma)$ denote the moduli space of stable holomorphic treed disks of map type ${\bGamma}$. Regularization of these moduli spaces can be achieved using Donaldson hypersurfaces constructed in the same way as in Theorem \ref{regular}.  Here, we also require that the system of coherent perturbations extends the existing system for defining the Fukaya category. 
In the case of no incoming boundary markings, we fix perturbations that are independent of the position of the incoming interior leaf; 
such perturbations may be chosen since the interior markings constrained to 
map to the Donaldson hypersurface already stabilize the components on which the map is non-constant.  On the other hand, transversality 
on the constant components (including transversality of the matching conditions)
may be achieved by a generic perturbation of the function on the incoming edge, independent of the domain.\footnote{This is no longer true in the case of several incoming edges, if one wants to obtain a strictly unital \ainfty morphism from $CF(\Delta,\Delta)$.  However we will not need that such a morphism is strictly unital.}   A closed-open map type is {\em essential} if it has no spherical components, all boundary edges $e \in \Edge_\circ(\Gamma)$ have positive lengths $\ell(e) > 0$, and the number of edges labelled by the Donaldson hypersurface $D$ on each surface component $S_v$ is equal to the expected number $k \omega(\beta_v)$, where $k$ is the degree of the Donaldson hypersurface. For a collection of boundary inputs $\ul{x} = (x_0, x_1, \ldots, x_d)$ and a critical point $\xx$ of $f_X$, let 
\[
{\mc M} (\ul{x}, \xx)_0:= \bigsqcup_{{\bGamma}} {\cM}_{\bGamma}(P_\Gamma)
\]
denote the union of the moduli spaces of closed-open treed disks of essential map types whose expected dimensions are zero, whose {gradient} leaf $T_{e_\black}$ is labelled $\xx$, and whose boundary insertions are $x_0, x_1, \ldots, x_d$.    

\begin{definition}{\rm(Closed-open map, without bounding cochains)}
For an integer $d \geq 0$ and a sequence of branes $\uds{ \LB}:=(\LB_0, \ldots, \LB_d)$, define a map
\begin{multline}
CO_{d, \ul{\LB }}^\sim (\bb): CM^\bullet(f_X, h_X) \to \\ {\rm Hom} \Big( CF^\bullet ( \LB_{d-1}, \LB_d) \otimes \cdots \otimes CF^\bullet  (\LB_0, \LB_d)), CF^\bullet (\LB_0, \LB_d) \Big)
\end{multline}
between the Morse complex of $(f_X, h_X)$ and Hochschild cochain complex (see \eqref{eq:hoc-cochain}) as follows. For any generator $\xx$ of $CM^\bullet(f_X, h_X)$ and generators $a_1 \in CF^\bullet (\LB_0, \LB_1)$, $\cdots$, $a_d \in CF^\bullet (\LB_{d-1}, \LB_d)$, define
\[
CO_{d, \ul{\LB}}^\sim (\bb)(\xx) (a_1\otimes \cdots \otimes a_d) := \sum_{x_0 \in \cI (\LB_0, \LB_d)} \left( \sum_{ [u] \in {\mc M}(\ul{x}, \xx)_0}  (-1)^{\heartsuit} \wt(u) \right) 
\]
with weightings $\wt(u)$ from \eqref{weightings}, extended linearly over $\Lambda$.
\end{definition}

\begin{definition}{\rm(Closed open map, with bounding cochains)} \label{defn412}
Given a subset of weakly unobstructed branes ${\mf L}$, the chain level closed-open map from the $\bb$-deformed quantum cohomology of $X$ to the flat $A_\infty$ category $\Fuk^\flat_{\mf L} (X, \bb)$ is a map
\[
CO(\bb): CM^\bullet( f_X, h_X ) \to CC^\bullet (\Fuk^\flat_{\mf L}(X, \bb))
\]
defined as follows. Suppose 
\[ \WB_i = (\LB_i, b_i), \ i = 0, \ldots, d \] 
are weakly unobstructed branes in ${\mf L}$. For each $\xx$, $CO(\bb)(\xx)$ is the cochain that maps
\[
a_d \otimes \cdots \otimes a_1 \in  {\rm Hom} ( \WB_{d-1}, \WB_d ) \otimes \cdots \otimes {\rm Hom} ( \WB_0, \WB_1)
\]
to the following element in ${\rm Hom}(\WB_0, \WB_d)$:
\[
\sum_{j_1,\ldots,j_d \geq 0} CO_{d + j_1 + \cdots + j_d, \ul{\LB}}^\sim (\bb) (\xx) (\underbrace{b_d,\ldots, b_d}_{j_d},a_d,
\ldots, a_2,
\underbrace{b_1,
  \ldots,b_1}_{j_1},
a_1,  \underbrace{b_0,\ldots, b_0}_{j_0} ).
\]
\end{definition}

See Figure \ref{closedopen} for an illustration of a typical configuration possibly contributing to the closed-open map

\begin{figure}[ht]
    \centering
   \includegraphics[height=2in]{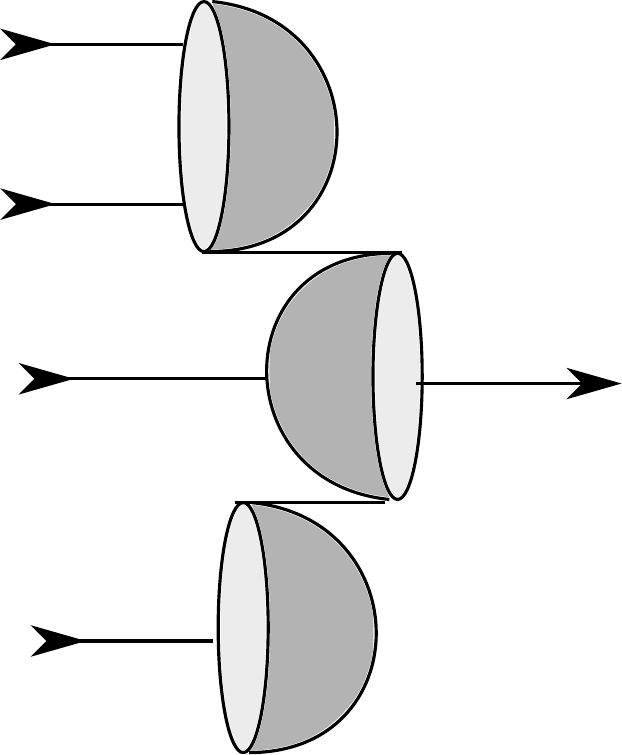}
    \caption{A configuration that possibly contributes to the closed-open map (Maurer--Cartan insertions are omitted).}
    \label{closedopen}
\end{figure}

First we show the closed-open map induces a map from the quantum cohomology to the Hochschild cohomology. 
\begin{theorem} 
The map $CO(\bb): CM^\bullet (f_X, h_X ) \to CC^\bullet (\Fuk^\flat_{\mf L}(X, \bb))$ defined by Definition \ref{defn412} has the following properties. 
\begin{enumerate}
    \item \label{cochainy} $CO(\bb)$ is a cochain map.
    
    \item \label{unity} If $d \geq 1$ and $a_i = 1_{L_i}^\circt$ for some $i = 1, \ldots, d$, then 
    \[
    CO (\bb)(\xx) (a_d \otimes \cdots \otimes a_1) = 0.
    \]
\end{enumerate}
\end{theorem}

\begin{proof} 
To prove the claim \eqref{cochainy} that $CO(\bb)$ is a chain map, consider the one-dimensional moduli space ${\mc M}(\ul{x}, \xx)_1$ for fixed labelling data $\ul{x}$ and $\xx$. A compactness theorem similar to Lemma \ref{refinedcompactness} shows that the boundary of such a moduli space consists of once-broken configurations, that is, strata ${\mc M}_{\bGamma}(P_\Gamma)$ of expected dimension zero with one infinite-length edge $e \in \Edge(\bGamma)$. The broken edge $e$ could be on the boundary, so that part of the configuration contributes to the differential $\delta$ of the Hochschild cochain complex, or on the interior {gradient} leaf which corresponds to the Morse differential on $X$.
To prove property \eqref{unity}, notice that the perturbation data $P_\Gamma$ for types involving forgettable boundary inputs 
$e \in \Edge_{\rightarrow,\white}(\Gamma)$ is pulled back under the forgetful map that removes such inputs and collapses unstable components.  Property \eqref{unity} holds in the same way as the unitality of the Fukaya category.
\end{proof}

\subsubsection{A spectral property of the closed-open map}

We prove a spectral property of the closed-open map similar to the spectral property of the open-closed map. This result allows us to obtain a refined generation result for the Fukaya category (see the statement of Theorem \ref{crit}).

\begin{theorem}\label{CO_spectral}
Suppose $\lambda, w \in \Lambda$ with $\lambda \neq D_q w$. Suppose ${\bm K} \in MC({\mc L})$ is a weakly unobstructed brane with potential function $w$ and $\gamma\in QH^\bullet(X; \bb)$ is a generalized eigenvector of the quantum multiplication by $[\omega]^\bb$ corresponding to eigenvalue $\lambda$.
Then
\[
[CO(\bb)_{0, {\bm K}}](\gamma) = 0 \in HF^\bullet( {\bm K}, {\bm K}).
\]
The same identity holds if $\gamma$ is a generalized eigenvector of 
 the quantum multiplication by $c_1(X)^\bb$ corresponding to eigenvalue $\lambda \neq w$.
\end{theorem}

\begin{proof}[Sketch of Proof] Similar to the case of the open-closed map, we introduce a new kind of closed-open domain with an interior marking at a a distinct angle $\theta \in (0, 2\pi)$ and a chain-level map
\[
CO_+: CM^\bullet( f_X, h_X) \to CF^\bullet( {\bm K}, {\bm K})
\]
that counts configurations described by Figure \ref{fig:CO_spectral}. 
\begin{figure}[h]
    \centering
    \includegraphics{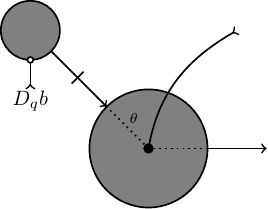}
    \caption{A configuration contributing to $CO_+$. There is one auxiliary boundary marking labelled by $D_q b_{\bm K}$. Boundary insertions by weakly bounding cochains are omitted.}
    \label{fig:CO_spectral}
\end{figure}
It follows in a way similar to Lemma \ref{lemma_spectral_1} that if the input $\xx$ is a Morse cocycle, then  $ 
CO_+(\xx) \in {\rm Im} (m_{1, {\bm K}})$ is a Floer coboundary. 
Choose a submanifold $Y\subset X$ that represents the Poincar\'e dual of $[\omega]$ and that intersects transversely with the Donaldson hypersurface $D$.  Via the moduli space similar to that described by Figure \ref{fig:spectral0} with the directions on all semi-infinite edges reversed, one obtains that on the chain level
\[
CO(\bb)_{0, {\bm K}} ( (Y + D_q \bb) \star_\bb(\xx)) = (D_q w) CO(\bb)_{0, {\bm K}}(\xx) + CO_+(\xx).
\]
Hence, for $\gamma$ a generalized eigenvector corresponding to $\lambda$, one has
for some positive integer $m$
\[
(D_q w - \lambda)^m [CO(\bb)_{0, {\bm K}}] (\gamma) = 0.
\]
As $\lambda \neq D_q w$, the first part of the theorem follows.
The second part is similar, using that since ${\bm K}$ is orientable, 
the first Chern class $c_1(X)$ is representable by a generic section of the 
anticanonical bundle that is non-vanishing on ${\bm K}$.   
\end{proof}

\subsubsection{The homomorphism property}

Lastly, we show that the map on the cohomology level intertwines with the ring structures. 

\begin{theorem}\label{homo}
The cohomology-level closed-open map
\[
[CO(\bb)]: QH^\bullet(X ,\bb) \to HH^\bullet \left( \Fuk^\flat_{\mc L} (X,\bb) \right)
\]
is a unital ring homomorphism.
\end{theorem}

Before giving the proof, we note the following consequence for automorphism algebras of objects in the Fukaya category. Given a weakly unobstructed brane $\WB = (\LB, b) \in MC({\mc L})$, we may consider the component 
\[
CO_{0,\WB}(\bb):  CM^\bullet (f_X, h_X) \to {\rm Hom}( \WB, \WB) \cong CF^\bullet( \LB, \LB)
\]
of the closed-open map that outputs Hochschild cochains of length zero lying in $CF^\bullet (\LB, \LB)$. Theorem \ref{homo} implies that the resulting map is a ring homomorphism from the quantum cohomology of $X$ to the Floer cohomology of $\WB$, and in particular non-vanishing of Floer cohomology 
gives eigenvalues for quantum multiplication; we thank Marco Castronovo 
for discussions on this point.
\begin{corollary} 
For any weakly unobstructed brane $\WB\in MC({\mc L})$, the closed-open map
\[
[CO_{0,\WB}(\bb)] : QH^\bullet (X, \bb) \to HF^\bullet (\WB, \WB)
\]
is a unital ring homomorphism.  In particular, if $HF^\bullet (\WB, \WB)$
is non-zero and $\WB$ lies in 
$ \Fuk_{\mf L}(X,\bb)_w$ then 
$w$ is an eigenvalue for quantum multiplication by $c_1(X)^\bb$.
\end{corollary}

\begin{proof}  The first part of the statement of the Corollary 
is immediate from Theorem \ref{homo} and Definition \ref{hhid}. 
The second part follows from the unitality property and the fact
that if $HF(\WB,\WB)$ is non-zero, then the unit must be non-vanishing.
By Theorem \ref{CO_spectral}, the unit in $QH(X,\bb)$ must have a non-vanishing eigen-component for quantum multiplication by $c_1(X)^\bb$.
\end{proof} 

The central ingredient of the proof of Theorem \ref{homo} is the notion of
a balancing condition on some interior markings on a disk which is similar
to the notion of {\em quilted disks} in \cite{Mau_Wehrheim_Woodward}. 

\begin{definition}{\rm(Balanced marked disks)}\label{defn_balanced_disk}
Consider a marked disk $S \simeq \DD$ with two interior markings $z', z''$ and boundary markings
 $\ul z = (z_0,\dots,z_d)$. The marked disk $(S, z', z'', \ul z)$  is {\em balanced} if the interior markings $z', z''$ and the  boundary output marking $z_0$ lie on a circle $S \subset \DD$ tangent to $\partial \DD$ at $z_0$.  In the combinatorial type of a  balanced treed disk, the interior markings $z', z''$ correspond to {gradient} leaves.  This ends the Definition.
\end{definition}

The moduli space of balanced disks is defined as follows. \label{tangentcircle} The balanced condition is
invariant under the action of $PSL(2; {\mb R}) \cong {\rm Aut}(\DD)$, 
and we denote by $\M^b$ the set of isomorphism classes of balanced disks.
This moduli space can be equipped with a Hausdorff topology in the same way as for marked disks.

\begin{figure}[ht]
 \centering \scalebox{.8}{
\begingroup%
  \makeatletter%
  \providecommand\color[2][]{%
    \errmessage{(Inkscape) Color is used for the text in Inkscape, but the package 'color.sty' is not loaded}%
    \renewcommand\color[2][]{}%
  }%
  \providecommand\transparent[1]{%
    \errmessage{(Inkscape) Transparency is used (non-zero) for the text in Inkscape, but the package 'transparent.sty' is not loaded}%
    \renewcommand\transparent[1]{}%
  }%
  \providecommand\rotatebox[2]{#2}%
  \newcommand*\fsize{\dimexpr\f@size pt\relax}%
  \newcommand*\lineheight[1]{\fontsize{\fsize}{#1\fsize}\selectfont}%
  \ifx\svgwidth\undefined%
    \setlength{\unitlength}{542.54996712bp}%
    \ifx\svgscale\undefined%
      \relax%
    \else%
      \setlength{\unitlength}{\unitlength * \real{\svgscale}}%
    \fi%
  \else%
    \setlength{\unitlength}{\svgwidth}%
  \fi%
  \global\let\svgwidth\undefined%
  \global\let\svgscale\undefined%
  \makeatother%
  \begin{picture}(1,0.91235834)%
    \lineheight{1}%
    \setlength\tabcolsep{0pt}%
    \put(0,0){\includegraphics[width=\unitlength,page=1]{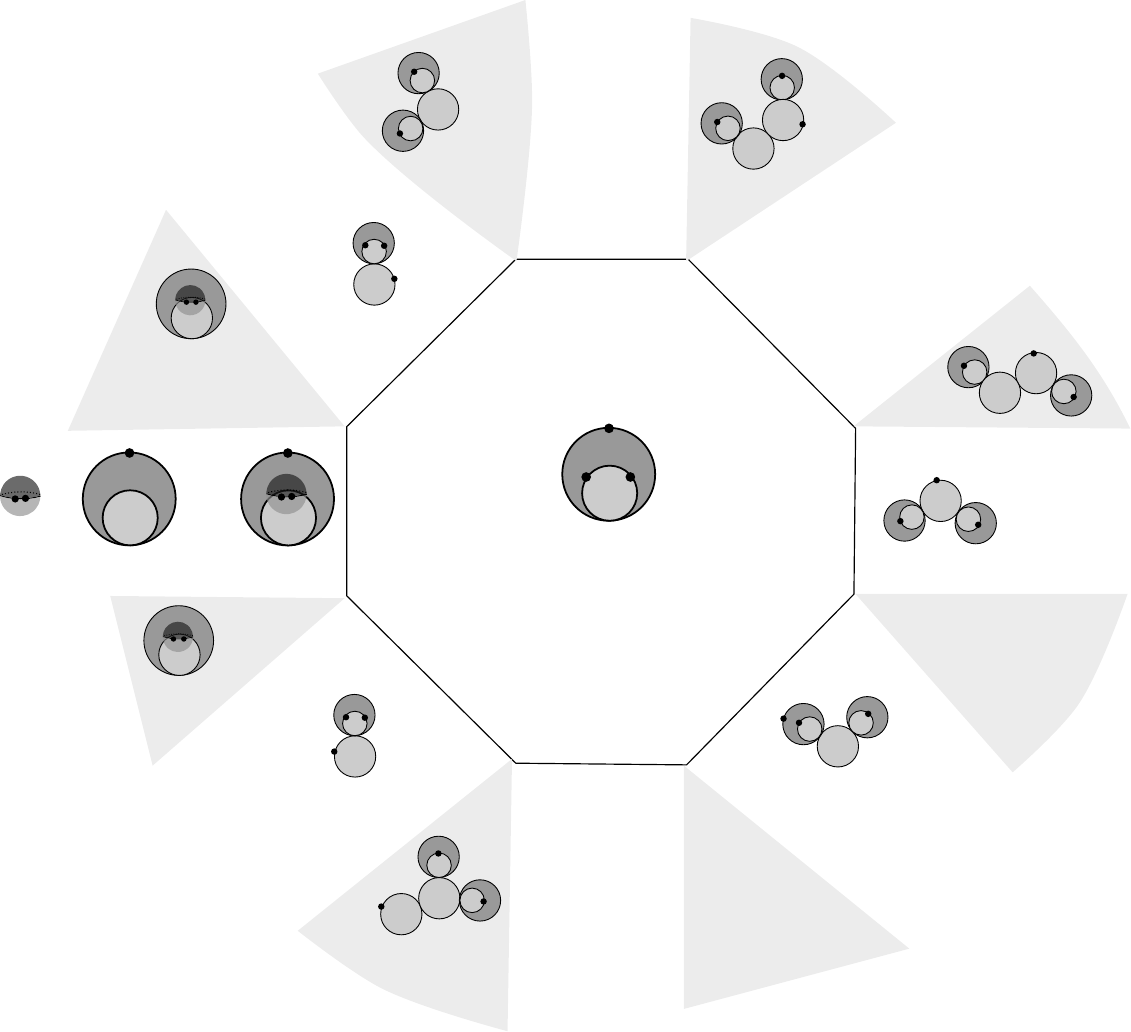}}%
    \put(0.17228804,0.46521938){\color[rgb]{0,0,0}\makebox(0,0)[lt]{\lineheight{1.25}\smash{\begin{tabular}[t]{l}=\end{tabular}}}}%
    \put(0.04591423,0.46591759){\color[rgb]{0,0,0}\makebox(0,0)[lt]{\lineheight{1.25}\smash{\begin{tabular}[t]{l}+\end{tabular}}}}%
    \put(0,0){\includegraphics[width=\unitlength,page=2]{mwfig.pdf}}%
  \end{picture}%
\endgroup%
}
 \caption{The compactified moduli space of balanced disks with two interior markings (long leaves), one boundary input marking, and one boundary output marking.}
  \label{fig:mwfig}
\end{figure}

\begin{remark} {\rm(Balanced versus quilted disks)} 
The balancing condition in the above definition is similar to the condition in Ma'u-Wehrheim-Woodward \cite{Mau_Wehrheim_Woodward} that the interior markings lie on the same seam. Therefore in Figure \ref{fig:mwfig}, balanced disks are depicted in a similar manner to quilted disks; the circle containing  $z', z'', z_0$ resembles a quilting circle and the part of the disk above resp. below the circle is colored dark  resp. light. However the balancing condition differs from the  quilting condition in that in the compactification of the moduli  space, disk/sphere components that do not contain either of the markings  $z'$, $z''$ are unquilted, and boundary inputs are allowed to  be incident on light components as well. Figure \ref{fig:mwfig}  shows the compactified moduli space of balanced disks with one  boundary input, which may be contrasted with the quilted version in  \cite[Figure 11]{Mau_Wehrheim_Woodward}.
\end{remark}

We generalize the balanced condition to treed disks.
\begin{definition}\label{balanced}{\rm(Balanced treed disk)}
Consider a treed disk $C = S \cup T$ of domain type $\Gamma$ that has two {gradient} leaves $e', e''$ and one boundary output. We say that $C$ is {\em balanced} if the following conditions are satisfied. Let $v', v'' \in {\rm Vert}(\Gamma_\circ)$ be the two vertices in the disk part that are closest to the two {gradient} leaves $e'$ and $e''$ respectively.

\begin{enumerate}
\item If $v' \neq v''$, then for the (unique) path $e_1, e_2, \ldots, e_k$ in $\Gamma_\circ$ connecting $v'$ and $v''$, require 
\[
   \sum_{i=1}^k \pm \ell(e_i) = 0
\]
 where the signs depend on whether the direction of the path is towards the root or away from the root.
    
 \item If $v' = v'' = v$, then let $z', z'' \in S_v \simeq \DD$ be the node corresponding to them and let $z_0\in \partial S_v$ be the node towards the output. Then we require that the marked disk $(S_v, z', z'', z_0)$ is balanced (see Definition \ref{defn_balanced_disk}).
\end{enumerate}
\end{definition}

We make a few remarks on the differences between the moduli space of balanced treed disks and the moduli spaces of treed disks used before.  For any stable domain type $\Gamma$ of treed disks with two {gradient} leaves, inside the moduli space ${\cM}_\Gamma$ of stable treed disks the locus of balanced treed disks, denoted by $ {\mc   M}_\Gamma^b \subset {\mc   M}_\Gamma$ is a real codimension one submanifold. See Figure \ref{ringmap} for an illustration of a compactified moduli space of balanced treed disks with two {gradient} leaves.  As opposed to stable maps, the number of nodes is not equal to the codimension of the stratum; instead, there are relations on the gluing parameters arising from the fact that the markings must lie on the same interior circle. See \cite{Mau_Wehrheim_Woodward} for more details (for disks rather than treed disks).

We introduce moduli spaces of balanced disks with Lagrangian boundary conditions as follows. For any map type ${\bGamma}$ let ${\cM}_{\bGamma}^b(P_\Gamma)$ denote the moduli space of maps from balanced disks with perturbation data $P_\Gamma$. The transversality argument of Section \ref{regularize} can be extended to guarantee that ${\cM}_{\bGamma}^b (P_\Gamma)$ is cut out transversely as {gradient} as ${\bGamma}$ is uncrowded. We also require that the coherent system of perturbations extends the existing one used for defining the Fukaya category and the quantum multiplication.

\begin{figure}[ht]
\centering \centering \scalebox{.8}{
\begingroup%
  \makeatletter%
  \providecommand\color[2][]{%
    \errmessage{(Inkscape) Color is used for the text in Inkscape, but the package 'color.sty' is not loaded}%
    \renewcommand\color[2][]{}%
  }%
  \providecommand\transparent[1]{%
    \errmessage{(Inkscape) Transparency is used (non-zero) for the text in Inkscape, but the package 'transparent.sty' is not loaded}%
    \renewcommand\transparent[1]{}%
  }%
  \providecommand\rotatebox[2]{#2}%
  \newcommand*\fsize{\dimexpr\f@size pt\relax}%
  \newcommand*\lineheight[1]{\fontsize{\fsize}{#1\fsize}\selectfont}%
  \ifx\svgwidth\undefined%
    \setlength{\unitlength}{497.22014099bp}%
    \ifx\svgscale\undefined%
      \relax%
    \else%
      \setlength{\unitlength}{\unitlength * \real{\svgscale}}%
    \fi%
  \else%
    \setlength{\unitlength}{\svgwidth}%
  \fi%
  \global\let\svgwidth\undefined%
  \global\let\svgscale\undefined%
  \makeatother%
  \begin{picture}(1,0.16243626)%
    \lineheight{1}%
    \setlength\tabcolsep{0pt}%
    \put(0,0){\includegraphics[width=\unitlength,page=1]{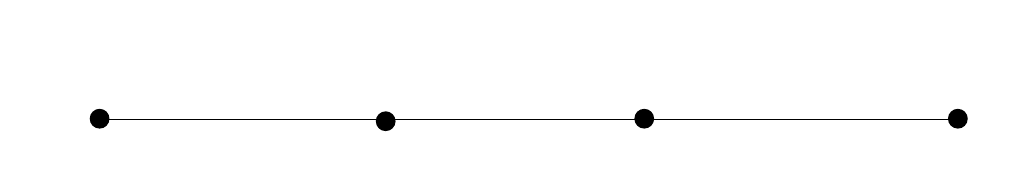}}%
    \put(0.08523909,0.00720978){\color[rgb]{0,0,0}\makebox(0,0)[lt]{\lineheight{1.25}\smash{\begin{tabular}[t]{l}$\rho=-\infty$\end{tabular}}}}%
    \put(0.35202463,0.00479285){\color[rgb]{0,0,0}\makebox(0,0)[lt]{\lineheight{1.25}\smash{\begin{tabular}[t]{l}$\rho=0$\end{tabular}}}}%
    \put(0.59085508,0.00720912){\color[rgb]{0,0,0}\makebox(0,0)[lt]{\lineheight{1.25}\smash{\begin{tabular}[t]{l}$\rho=1$\end{tabular}}}}%
    \put(0.88929997,0.00720912){\color[rgb]{0,0,0}\makebox(0,0)[lt]{\lineheight{1.25}\smash{\begin{tabular}[t]{l}$\rho=\infty$\end{tabular}}}}%
    \put(0,0){\includegraphics[width=\unitlength,page=2]{ring2.pdf}}%
    \put(0.71529543,0.10485339){\color[rgb]{0,0,0}\makebox(0,0)[lt]{\lineheight{1.25}\smash{\begin{tabular}[t]{l}{\tiny $\tau$}\end{tabular}}}}%
    \put(0.78059501,0.10394225){\color[rgb]{0,0,0}\makebox(0,0)[lt]{\lineheight{1.25}\smash{\begin{tabular}[t]{l}{\tiny $\tau$}\end{tabular}}}}%
    \put(0,0){\includegraphics[width=\unitlength,page=3]{ring2.pdf}}%
  \end{picture}%
\endgroup%
}
    \caption{The compactified moduli space of balanced treed disks with two interior markings (long leaves) and one output. This moduli space is one-dimensional and the points $\rho=0, 1$ are fake boundary strata.}
   \label{ringmap}
\end{figure} 

\begin{proof}[Proof of Theorem \ref{homo}] \label{homoproof} Denote by $\ast$ the Yoneda product $\mu^2_{CC^\bullet}$ of \eqref{eq:ccmd}. Fix Morse cocycles ${\xx}_1, {\xx}_2 \in CM^\bullet(f_X, h_X)$. We will show that the difference
\[
CO(\bb)( {\xx}_1 \star_\bb {\xx}_2) - CO(\bb)({\xx}_1) \ast CO(\bb)({\xx}_2)
\]
is a coboundary in the Hochschild cochain complex. To reduce notational complexities, we assume that both ${\xx}_1$ and ${\xx}_2$ are single critical points. We first construct the coboundary. For any $w \in \Lambda$, consider weakly unobstructed branes $\WB_0, \cdots, \WB_d$ with potential function having value $w$; consider generators $\ul{a} = (a_d, \ldots, a_0)$ where
\[
a_1 \in CF^\bullet ( \LB_0,  \LB_1), \ldots, a_d \in CF^\bullet ( \LB_{d-1}, \LB_d), a_0 \in CF^\bullet( \LB_0, \LB_d).
\]
Fix $j_0, j_1, \ldots, j_d \geq 0$ and consider balanced domain types $\Gamma$ with two {gradient} leaves, $d + j_0 + \cdots + j_d$ boundary inputs and essential map types ${\bGamma}$ of expected dimension zero whose {gradient} leaves are labelled by ${\xx}_1, {\xx}_2$ and whose boundary inputs are labelled by 
\[
\underbrace{b_0, \ldots, b_0}_{j_0}, a_1, \underbrace{b_1, \ldots, b_1}_{j_1}, a_2, \cdots, a_d, \underbrace{b_d, \ldots, b_d}_{j_d}, a_0
\]
(in counterclockwise orientation, the last one is the output). For
each such moduli space ${\cM}_{\bGamma}^b$ the count of rigid
elements defines an element
\[
\tau({\xx}_1, {\xx}_2)(\ul{a}) \in \Lambda. 
\]
and so one obtains a cochain
\[
\tau({\xx}_1, {\xx}_2) \in CC^\bullet (\Fuk_{\mc L} (X, \bb)_w ).
\]
We claim that 
\begin{equation}\label{eqn37}
CO(\bb)( {\xx}_1 \star_\bb {\xx}_2) - CO(\bb)({\xx}_1) \ast CO(\bb)({\xx}_2) = m_1^{CC} (\tau({\xx}_1, {\xx}_2)).
\end{equation}
To show this relation, consider a one-dimensional balanced moduli space with {gradient} leaves labelled by ${\xx}_1, {\xx}_2$ and any number of boundary inputs. There are three types of true boundary strata, see types (a), (b), (c) in Figure \ref{fig:cobdry}. In the first type (a) of boundary strata, there is a broken treed segment at an interior node. In the second type (b), there are two boundary breakings on a path connecting the two disk components having the two {gradient} leaves. In the third type, there is one boundary breaking that is not in the path connecting the disk components having the two {gradient} leaves. These types correspond to the three terms in \eqref{eqn37}.  

It remains to show that the closed-open map is unital.  
 One possible argument would be to construct $CO$ on the chain level so that it maps strict units
 to strict units. Rather than take this route, we  note that because the perturbations 
were chosen to be independent of the position of the leaf labelled by $\xx$, configurations
with input $\xx$ equal to the geometric unit $x_{\max}$ and no other boundary inputs can be rigid only if the underlying configuration is unstable, which  means that the map is constant and has a single output, necessarily the geometric unit in $CF(\WB,\WB).$   Since the geometric unit minus the strict unit 
$1_\WB^\whitet$ is the boundary of 
$1_\WB^\greyt$ up to terms with higher $q$-valuation, this implies that
with notation from Remark \ref{hhid} 
the difference $[CO(\bb)] (1_{QH(X)}) - 1_{HH(\cF^\flat)} $ has positive $q$-valuation, or vanishes.  Suppose that the difference is non-vanishing. Then for some $\Upsilon_0,\Upsilon_1 \in 
HH^\bullet \left( \Fuk^\flat_{\mc L} (X,\bb) \right)$ we have 
\[  [CO(\bb)] (1_{QH(X)}) = 1_{HH(\cF^\flat)} + \Upsilon_0  + \Upsilon_1\]
where $\Upsilon_1$ 
has length at least one in the length filtration on $HH(\Fuk^\flat_{\mc L} (X,\bb))$.
Write  $\Upsilon_0  = \Upsilon_0' + \Upsilon_0''$ where
$\Upsilon_0'$ is homogeneous in $q$ and $\val_q(\Upsilon_0'') > \val_q(\Upsilon_0')$, if non-vanishing.   We view $\Upsilon_0'$
as the leading order term in $\Upsilon_0 + \Upsilon_1$.
The homomorphism property and preservation of the length filtration implies that 
\begin{eqnarray*}
[CO(\bb)] (1_{QH(X)})^2 &=&  (1_{HH(\cF^\flat)} + \Upsilon_0' + \ldots)^2 \\
&=& (1_{HH(\cF^\flat)} + 2\Upsilon_0'  + \ldots ) \\
&=& [CO(\bb)] (1_{QH(X)}) = 1_{HH(\cF^\flat)} + \Upsilon_0'  + \ldots   .\end{eqnarray*}
Hence $2\Upsilon_0' = \Upsilon_0'$ which forces $\Upsilon_0$ to vanish.
If $\Upsilon_1$ has length $\ell \ge 1$ (that is, can be represented by a cochain
which vanishes unless the number of inputs is at least $\ell$)
then $\Upsilon_1^2$ has  length at least $2\ell$, which is a contradcition unless $\Upsilon_1$ vanishes.
\end{proof}

\begin{figure}[ht]
\centering \scalebox{.8}{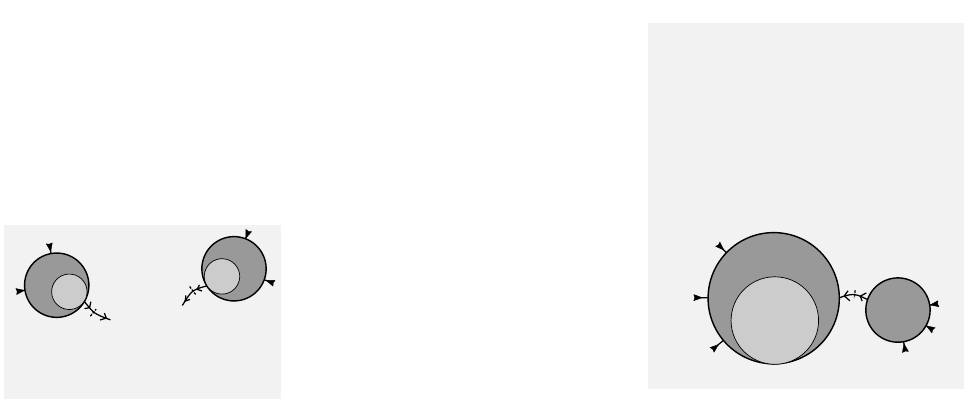}
\caption{Curve types (a), (b), (c) that can occur on the boundary of a one-dimensional moduli space of balanced treed disks with two {gradient} leaves. These three types contribute to the relation \eqref{eqn37}.}
\label{fig:cobdry}
\end{figure}

\section{Abouzaid's split-generation criterion}\label{acrit}\label{section4}

In this section,  we adapt Abouzaid's criterion \cite{Abouzaid_generation} for the split-generation of the Fukaya category to the non-exact case in which the \ainfty composition maps are defined by counts of treed
disks. We follow the argument of \cite{Abouzaid_generation} to prove Theorem \ref{crit}. The main technical input is the use of moduli spaces of treed annuli and a particular way of degenerating treed annuli. Using a  different degeneration we also prove that disjoint branes have orthogonal images under the open-closed maps, i.e., Theorem \ref{thm110}.

\subsection{The Cardy diagram} 

The idea of Abouzaid's construction is to produce the maps necessary for writing a Lagrangian as a mapping cone by degenerating holomorphic annuli to pairs of disks.  Given a collection ${\mf G}$ of objects of $\Fuk^\flat_{\mc L} (X,\bb)$, we wish to show that any object ${\bm K}$ of $\Fuk^\flat_{\mc L} (X,\bb)$ is split-generated (see Definition \ref{def:split-gen}) by the objects  ${\mf G}$.  For example, we might hope to show that ${\bm K}$ is a sub-object of some object $\WB$ of ${\mf G}$; to show this we want morphisms
\begin{align*}
&\ \alpha \in \Hom ( {\bm K}, \WB ),  &\ \beta \in \Hom ( \WB, {\bm K} ) 
\end{align*}
such that 
\[
m_2 (\alpha,\beta) = 1_{\bm K} \in \Hom ({\bm K}, {\bm K} ) .
\]
Naturally one hopes that the chains $\alpha,\beta$ can be produced geometrically as a count of holomorphic disks with two outputs.  If this is the case, one can glue to obtain holomorphic annuli with an output labelled by the identity $1_{\bm K}$.  A degeneration of the annulus to ``infinite length'' shows that a count of holomorphic disks with a single output must be non-trivial, see Figure \ref{mannuli}.

The result, Abouzaid's criterion Theorem \ref{crit}, gives a factorization of the open-closed and closed-open maps through the tensor product of Yoneda modules. 

\begin{definition}{\rm(Yoneda modules, collapsing map)}
Let $\WK$ be an object of the flat Fukaya category $\Fuk^\flat_{\mc L}(X, \bb)$. 
\begin{enumerate}
\item For any $w \in \Lam$ and $\WK \in \Obj(\Fuk_{\mc L}(X,\bb)_w)$, 
denote by $ \YY_\WK^{\rm L}$ resp. $\YY_\WK^{\rm R}$ the left resp. right {\em Yoneda module} over $\Fuk^\flat_{\mc L} (X,\bb)$ defined on objects by
\begin{align*}
&\ \YY_\WK^{\rm R}( \WB ) = \Hom( \WB, \WK ), &\  \YY_\WK^{\rm L} ( \WB ) = \Hom ( \WK, \WB ) 
\end{align*}
 for $\WB \in \Obj(\Fuk_{\mc L}(X,\bb)_w)$.

\item The tensor product of Yoneda modules is an \ainfty bimodule over $\Fuk^\flat_{\mc L} (X,\bb)$ denoted $\YY_\WK^{\rm L} \otimes \YY_\WK^{\rm R}$. It is hence an $A_\infty$ bimodule over any full subcategory $\Fuk^\flat_{\mf G}(X, \bb)$ by restricting to a subset of weakly bounding cochains ${\mf G} \subset MC({\mc L})$. The Hochschild homology
\[
HH_\bullet (\Fuk^\flat_{\mf G} (X,\bb), \YY_\WK^{\rm L} \otimes \YY_\WK^{\rm R} ) = H_\bullet ( \YY_\WK^{\rm R} \otimes_{\Fuk^\flat_{\mf G} (X,\bb)} \YY_\WK^{\rm L} )
\]
is computed by the {\em bar complex} 
\begin{multline} \label{barcomplex} 
B (\YY_\WK^{r} \otimes_{\Fuk^\flat_{\mf G} (X,\bb)} \YY_\WK^{l})\\  = \bigoplus_{k=0}^\infty
  \bigoplus_{\substack{\WB_1, \ldots, \WB_k \in {\mf G}\\ w(\WB_i)=w(\WK) }} {\rm Hom} (\WB_{k}, \WK)  \otimes \cdots \otimes {\rm Hom} (\WB_{1}, \WB_{2}) \otimes  {\rm Hom} (\WK, \WB_{1})  
\end{multline}
with differential given by the possible ways of collapsing. Here the $k=0$ summand is $\Hom(\WK, \WK)$. 

\item The {\it collapsing map}
\[
\mu_\WK: B(\YY_\WK^{\rm R} \otimes_{\Fuk^\flat_{\mf G} (X,\bb)} \YY_\WK^{\rm L}) \to \Hom (\WK,\WK )
\]
is defined by composing all factors in \eqref{barcomplex}:
\[
\mu_\WK: a_+ \otimes a_k \otimes \ldots \otimes a_1 \otimes a_- \mapsto (-1)^{\diamondsuit} m_{k+2}(a_+ ,a_k,\ldots, a_1, a_-)
\]
where $\diamondsuit$ is the Koszul sign 
\[
| a_-| + \sum_{j=1}^k \Vert a_j \Vert .
\]
The $A_\infty$ relation implies that $\mu_\WK$ is a chain map, hence induces a map
\begin{equation}
\mu_\WK: H_\bullet ( {\mc Y}_\WK^{\rm R} \otimes_{\Fuk^\flat_{\mf G}(X, \bb)} {\mc Y}_\WK^{\rm L} ) \to HF^\bullet( \WK, \WK).
\end{equation}

\end{enumerate}
\end{definition}

The following characterization of split-generation (see \cite[Lemma 1.4]{Abouzaid_generation}) will serve as the definition for our purposes.
\begin{definition}{\rm(Split-generation)} \label{def:split-gen}
  A flat $A_\infty$-category ${\mf L}$ is split-generated by a set of objects ${\mf G}$ if for any object $\WK$ in ${\mf L}$, the image $\mu_\WK(H_\bullet ( {\mc Y}_\WK^{\rm R} \otimes_{\Fuk^\flat_{\mf G}(X, \bb)} {\mc Y}_\WK^{\rm L} ))$ contains the identity element $1_\WK^\whitet \in HF(\WK, \WK)$.
\end{definition}

\subsubsection{The coproduct}

We define an \ainfty coproduct functor 
\[
\delta_\WK: \Fuk^\flat_{\mf G}(X, \bb) \to {\mc Y}_\WK^{\rm L} \otimes {\mc Y}_\WK^{\rm R}
\]
by counting treed disks with two outputs:  Such a functor consists of a collection of maps $\{\delta_{r|1|s}\}_{r,s \geq 0}$ where 
\begin{multline}
  \delta_{r|1|s} : \ \Hom (\WB_{r-1}, \WB_r) \otimes \cdots \otimes \Hom (\WB_0, \WB_1)
  \otimes \Hom (\WB_0, \WB_0') \otimes \WB(\WB_0', \WB_1') \otimes \\ \cdots \otimes
  \Hom (\WB_{s-1}', \WB_s') \to \Hom (\WK, \WB_r) \otimes \Hom (\WB_s', \WK), 
  \end{multline}
satisfying an \ainfty axiom (see \cite[(4.13)]{Abouzaid_generation}).  We will define these maps $\delta_{r|1|s}$ by counting holomorphic disks with two outputs. As the moduli spaces are different from what we have been using, the construction deserves a separate discussion.

We briefly discuss the moduli spaces of treed disks with two outputs. The domain types are two-colored trees with weighting types on semi-infinite edges and metric types on finite edges. However, in comparison with the types used for the construction of the Fukaya algebras, the trees are no longer rooted. The restriction on the weighting types is different from Definition \ref{defn23}. We require that both outputs are unforgetful (labelled by $\blackt$) while inputs can still be forgetful ($\circt$), unforgetful ($\blackt$), or weighted ($\greyt$). The stability condition remains the same. For each stable domain type $\Gamma$, there is a moduli space ${\mc M}_\Gamma$ and its compactification $\ov{\mc M}_\Gamma$ as well as the universal curve $\ov{\mc U}_\Gamma \to \ov{\mc M}_\Gamma$. Notice that treed disks with two outputs can degenerate to broken treed disks whose unbroken components can have only one output. In order to define maps respecting the existing $A_\infty$ structure on the Fukaya category, we require the perturbations used here to extend the existing ones used for defining the Fukaya category. The notion of map types is similar to previous cases.  For any stable domain type $\Gamma$, a perturbation datum $P_\Gamma$, and a map type $\bGamma$, one has a moduli space ${\mc M}_\bGamma(P_\Gamma)$ of treed disks with two outputs of type 
$\bGamma$.  There exists a coherent system of strongly regular perturbations $P_\Gamma$, so that for all uncrowded map type $\bGamma$, ${\mc M}_\bGamma(P_\Gamma)$ is a smooth manifold of the expected dimension. Moreover, for essential map types (see Definition \ref{essentialtype}) of expected dimension zero or one, a refined compactness result similar to Theorem \ref{refinedcompactness} holds.

The structure maps for the coproduct functor are defined as follows. For $r, s \geq 0$, Lagrangian branes $\LB_r, \ldots, \LB_0, \LB_0', \ldots, \LB_s'$, and generators $x_i \in {\mc I}( \LB_{i-1}, \LB_i)$ for $i = r, \ldots, 1$, $x_0 \in {\mc I}(\LB_0, \LB_0')$, $x_j' \in {\mc I}(\LB_{j-1}', \LB_j' )$ for $j = 1, \ldots, s$, and $y^{\rm L} \in {\mc I}(\widehat K, \LB_r)$, 
$y^{\rm R} \in {\mc I}( \LB_s', \widehat K)$, $y_1, \ldots, y_t \in {\mc I}^{\rm odd} (\widehat K, \widehat K)$, one considers the moduli space 
\[
{\mc M}_{r|1|s}( \uds x; y^{\rm R}, \uds y, y^{\rm L} )_0:= {\mc M}_{r|1|s} ( x_r, \ldots, x_1, x_0, x_1', \ldots, x_s'; y^{\rm R}, y_1, \ldots, y_t, y^{\rm L} )_0
\]
given by the union of moduli spaces ${\mc M}_{\bGamma}(P_\Gamma)$ of essential map types $\bGamma$ whose boundary edges are labelled by these generators (see Figure \ref{fig:coproduct}).    Define
\begin{multline}\label{eq:deltaexp}
  (a_{r},\ldots, a_1, a_0, a_1',\ldots, a_{s}'; a''_1, \ldots, a''_t) \\
  \mapsto \sum_{y^{\rm R} \in {\mc I}( \widehat L_s', \widehat K)} \sum_{y^{\rm L} \in {\mc I}(\widehat K, \widehat L_r)} \sum_{u  \in {\mc M}_{r|1|s}( \uds x; y^{\rm R}, \uds y, y^{\rm L})_0} (-1)^{\ddagger} \wt(u) 
  \end{multline}
where the sum is over rigid maps $u$ with two output leaves and one distinguished input (in this case $x_0$) among a list of input leaves, the $a'_i$ resp. $a''_i$ are the generators corresponding 
to $x_i$ and $y_i$ respectively, and the product of holonomies over $u$
is interpreted as an element in the tensor product of identification of 
local systems in  $y^{\rm L}$ and $y^{\rm R}$.

\begin{figure}[h]
    \centering
    \includegraphics{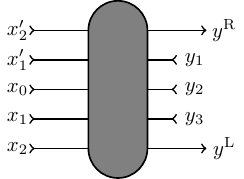}
    \caption{Moduli spaces defining the $A_\infty$ coproduct.}
    \label{fig:coproduct}
\end{figure}
The sign $\ddagger$ is given as in Abouzaid
\cite[4.17]{Abouzaid_generation} by
\[ \sum_{j=1}^s (s - j + 1) |x'_j| 
+ s | x_0 |   + \sum_{j=1}^r (j + s) | x_j |   . \] 
The coproduct map is defined by summing over all possible ways of inserting weakly bounding cochains; especially, we insert $a''_1 = a''_2 = \cdots = a''_t = b_{\bm K}$ in \eqref{eq:deltaexp}). For any $r, s \geq 0$, we can define 
\begin{multline*}
\delta_{r|1|s}: \bigotimes_{i=1}^r {\rm Hom}( \WB_{i-1}, \WB_i) \otimes {\rm Hom}( \WB_0, \WB_0') \otimes \bigotimes_{j=1}^s \Hom (\WB_{j-1}', \WB_j') \\
\to  \Hom ( \WK, \WB_r) \otimes \Hom( \WB_s' , \WK),
\end{multline*}
where $\WB_i, \WB_i' \in \Obj(\Fuk_{\mc L}(X,\bb)_{w(\WK)})$. 
We obtain a coproduct map on Hochschild chains
\begin{multline}\label{eq:ccdelta}
\delta_\WK : CC_d (\Fuk^\flat_{\mc L} (X, \bb)) \to \YY^{\rm R}_\WK
\otimes_{\Fuk^\flat_{\mc L} (X,\bb)}  \YY^{\rm L}_\WK\\
a_d \otimes \ldots \otimes a_0 
\mapsto \sum_{r,s} (-1)^{\diamond} {\mc T} \Big( a_{r+1} \otimes \ldots a_{d -s}
\otimes \delta_{r|1|s} ( a_r \otimes \cdots \otimes a_1 \otimes a_0 \\ \otimes a_d \otimes \cdots \otimes a_{d-s+1} ) \Big)   \\
\end{multline}
where the map $\cT$ reorders the factors
\[
\cT(a_{r+1} \otimes \ldots a_{d -s} \otimes  y^{\rm L} \otimes y^{\rm R} )  = (-1)^{\circ} y^{\rm R} \otimes a_{r+1} \otimes \ldots \otimes a_{d - s - 1} \otimes y^{\rm L}
\]
and the signs are given by the formulas 
\begin{equation} \label{diamondsign} \diamond = \maltese_1^r (1 +
  \maltese_{r+1}^d) + \dim(X) \maltese_{r+1}^{d - s- 1} \end{equation}
and 
\begin{equation} \label{circsign} 
\circ  = \deg(y^{\rm R})( \deg(y^{\rm L}) + \maltese_{r+1}^{d - s- 1}) .
\end{equation} 

\begin{proposition}
For any subset ${\mf G}\subset MC({\mc L})$ of weakly unobstructed branes and $\WK \in MC({\mc L})$, the coproduct map $\delta_{\bm K}: CC_\bullet( \Fuk^\flat_{\mf G}(X, \bb)) \to B( {\mc Y}_{\WK}^{\rm R} \otimes_{\Fuk^\flat_{\mf G}(X, \bb)} {\mc Y}_\WK^{\rm L})$ is a chain map.
\end{proposition}

\begin{proof}
The statement of the Proposition is a consequence of the classification of boundary strata of moduli spaces of treed disks with two outputs and the verification of signs. Indeed, for any one-dimensional moduli space of treed disks with two outputs, there are two types of boundary strata: either the two outputs are in the same unbroken components, or they are in different unbroken components (see Figure \ref{fig:coproduct2}). 
\begin{figure}[h]
    \centering
    \includegraphics{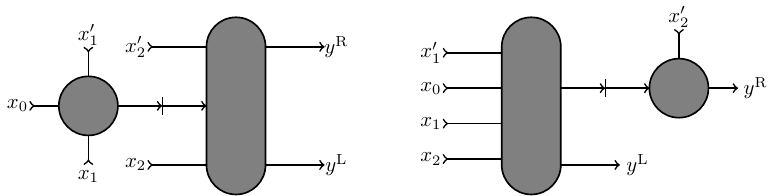}
    \caption{Two types of boundary strata of moduli spaces of treed disks with two boundary outputs.}
    \label{fig:coproduct2}
\end{figure}
These two boundary types correspond to (part of) the differentials on the Hochschild complex and the bar complex. The other parts of these differentials are carried over 
in the terms in the definition of $\delta_{\bm K}$ which do not count holomorphic disks. 
\end{proof}

We denote the induced map on homology still by the same notation:
\[
\delta_\WK: HH_\bullet( \Fuk^\flat_{\mf G}(X, \bb)) \to H_\bullet( {\mc Y}_\WK^{\rm R} \otimes_{\Fuk^\flat_{\mf G}(X, \bb)} {\mc Y}_\WK^{\rm L}).
\]
Note that $\delta_\WK=0$ on $HH_\bullet(\Fuk^\flat_{\mf G}(X,\bb)_{w'})$ for $w' \neq w(\WK).$

\subsubsection{The Cardy diagram}

The coproduct map, collapse map $\mu$, and open-closed and closed-open
maps $CO(\bb),OC(\bb)$ fit into a commutative-up-to-sign {\em Cardy
  diagram}:
  
\begin{theorem} \label{cardy} (Abouzaid \cite{Abouzaid_generation} in the exact, embedded case; see also Ganatra \cite{Ganatra_thesis}) For any collection ${\mf G} \subset MC({\mc L})$ and any object ${\bm K} \in MC({\mc L})$, there is a {\em Cardy diagram}
\begin{equation}\label{cardy_diagram}
\vcenter{ \xymatrix{  HH_\bullet ( \Fuk^\flat_{\mf G} (X,\bb)) \ar[rr]^{\delta_\WK} \ar[d]_{[OC(\bb)]} & & H_\bullet ( \YY_{\bm K}^{\rm R} \otimes_{\Fuk^\flat_{\mf G} (X,\bb)} 
    \YY^{\rm L}_{\bm K} ) \ar[d]^{\mu_\WK} \\ 
            QH^\bullet_{\mf G} (X,\bb) \ar[rr]_{[CO_\WK(\bb)]} &  & HF^\bullet ( {\bm K}, {\bm K} ) } }
            \end{equation}
that commutes up to an overall sign of $(-1)^{\dim(X)(\dim(X) + 1)/2}$.
\end{theorem} 

Abouzaid's generation criterion (Theorem \ref{crit}) follows as a consequence of the commutativity of the Cardy diagram.
\begin{proof}[Proof of Theorem \ref{crit}]
  The flat $A_\infty$ category $\Fuk^\flat_{\mc L}(X,\bb)$ is split
  generated by ${\mf G}$ if for any weakly unobstructed brane
  ${\bm K} \in MC(\mc L)$, the image of $\mu_{\bm K}$ contains the
  identity $1_{\bm K}^\whitet$. By the commutativity of the Cardy
  diagram (Theorem \ref{cardy}), this containment holds if the identity $1_{QH^\bullet(X)}$ lies in
  the image of $OC(HH_\bullet(\Fuk^\flat_{\mf G}(X,\bb)))$, or in other
  words, $QH^\bullet(X,\bb)$ is generated by
  $\mf G \subset MC(\mc L)$.
\end{proof}

\subsection{Holomorphic treed annuli}\label{hannuli}

To prove the commutativity of the Cardy diagram Theorem \ref{cardy} in the version of the Fukaya category considered here, we begin with some preliminaries. Notice that the composition of maps on either direction of the diagram \eqref{cardy_diagram} consist of maps which count certain degenerate treed holomorphic annuli.
\begin{definition}  Given $0<\rho_1<\rho_2$, an {\em annulus} is a complex curve with boundary 
of the form
\[ A_{\rho_1,\rho_2} = \Set{ z \in \C | \ \rho_1 \leq |z| \leq \rho_2
  }. \]
The boundary components are denoted by
\[\partial_- A_{\rho_1,\rho_2}:=\{z:|z|=\rho_1\}, \quad \partial_+ A_{\rho_1,\rho_2}:=\{z:|z|=\rho_2\}. \]
\end{definition}

\begin{definition} {\rm(Stable treed annuli)}
  \label{def:treedann}
  \begin{enumerate}
  \item {\rm(Marked annulus)} For $d = (d_-, d_+)$ a pair of positive integers and $d_\bullet\geq 0$ 
  a {\em $(d_-, d_+, d_\bullet)$-marked annulus} consists of the following data: an inner and outer radii $\rho_1 < \rho_2$, a collection of interior markings
\[
z_{\black,i} \in \on{int}( A_{\rho_1,\rho_2} ), \quad 1 \leq i    \leq d_\bullet
\]
and a collection of boundary marked points
\[
z^\pm_{\white,j} \in \partial_\pm A_{\rho_1,\rho_2}, \quad 1     \leq j \leq d_\pm. 
\]
We always require that the boundary markings on the outer circle are counterclockwise ordered, while the boundary markings on the inner circle are clockwise ordered.
  \item {\rm(Treed annulus)} 
  There is a compactification of the moduli space of marked annuli by allowing {\em stable} nodal annuli: nodal annuli $S$ with no non-trivial infinitesimal automorphisms.   
  As in the case of stable marked disks, a combinatorial type underlying a stable annulus is a graph $\Gamma$.  A {\em treed annulus} $C$
    is obtained from a nodal annuli by replacing each boundary node $w_e$, $e \in \Edge_{\white,-}(\Gamma)$ by a segment $T_e$ equipped with a length $\ell(e) \in [0,\infty)$, and attaching a semi-infinite treed segment $T_e$ at each boundary
    marking $z_e$, $e \in \Edge_{\white,\to}(\Gamma)$. We then allow the finite lengths to increase to infinity and the finite edges to break.

  \item {\rm (Additional features)} We consider treed annuli with some
    additional features to prove the Cardy relation and an orthogonality relation in Section \ref{subsec:orth}. 
    \begin{enumerate}
    \item {\rm (Distinguished leaves and Balanced lengths)} The leaves/markings $z_{0}^+$ and $z_{0}^-$ are {\em distinguished leaves} and are constrained to have an angle offset of $\pi$:
      \begin{equation}\label{eq:fixp}
       {\rm\text{(Angle offset)}} \quad \exists \theta : \enspace z_{0}^+ = \rho_2e^{\iota \theta}, \quad z_{0}^-=\rho_1 e^{\iota(\theta+\pi)}. 
     \end{equation}
     The lengths of treed segments are subject to a balancing condition: let $S_{v_\pm} \subset C$ be the surface component containing $z_0^\pm$. When $S_{v_-} \neq S_{v_+}$, there are exactly two paths $\gamma_1, \gamma_2$ connecting them. We require that the two paths have the same total length:
    \begin{equation}\label{eq:balann}
      {\rm(Balanced)} \quad  \sum_{e \in \gamma_1} \ell(e) = \sum_{e \in \gamma_2 } \ell(e) ;  
    \end{equation}    
    See Figure \ref{mannuli} where the paths $\gamma_\pm$ are the vertical paths in the two left diagrams.
      
    \item {\rm(Treed segment at an interior node)} An interior node
      that disconnects $z_{0}^+$ from $z_{0}^-$ is called a {\em path node}. We allow path nodes to be replaced by treed
      segments which can have a positive length.
    \end{enumerate}
  \end{enumerate}
\end{definition}

We introduce a moduli space of stable treed annuli with fixed angle offset as follows.  Denote by 
$ \ov{\mc M}{}_{d_-, d_+, d_\bullet}^{\rm ann}$
the moduli space of stable treed annuli for which the $0$-th boundary markings on the inner circle and outer circle have an angle offset of $\pi$ (as in \eqref{eq:fixp}) and that satisfy the balancing condition \eqref{eq:balann} for treed segments at path edges. Standard arguments show that the moduli space $\ov{\mc M}{}_{d_-, d_+, d_\bullet}^{\rm ann}$ is compact and Hausdorff.  The subspace of $\ov{\mc M}{}^{\rm ann}_{d_-,d_+, d_\bullet}$ that parametrizes curves with at most one path node is a topological manifold of dimension
\[
{\rm dim} \ov{\mc M}{}^{\rm ann}_{d_-,d_+, d_\bullet} = d_- + d_+ + 2d_\bullet - 1.
\]
The moduli space is equipped with a universal curve $\ov{\mc U}{}^{\rm ann}_{d_-,d_+, d_\bullet}$ which decomposes into a surface part $\ov{\mc S}{}^{\rm ann}_{d_-,d_+, d_\bullet}$ and tree part $\ov{\mc T}{}^{\rm ann}_{d_-,d_+, d_\bullet}$.  There is a forgetful map %

\begin{equation}  \label{eq:forgetann}
\ov{\mc M}{}_{d_-, d_+, d_\bullet}^{\rm ann} \to \ov{\mc M}{}_{1, 1, 0}^{\rm ann} 
\end{equation}
that forgets all markings except the $0$-th markings on the inner and outer circles.

\begin{remark}
In the moduli space of treed annuli, we fixed the angle offset  between distinguished boundary markings as $\Phi:=\pi$. This choice is arbitrary. In fact, choosing any non-zero angle offset $\Phi \in (0,2\pi)$ produces a homeomorphic moduli space.  The angle offset zero $\Phi=0$ produces a different moduli space, which we will use in Section \ref{subsec:orth}.
\end{remark}

\begin{example} \label{innerandouter} 
We describe the moduli space of isomorphism classes of annuli with one inner boundary leaf and one outer boundary leaf with an angle offset of $\pi$. There is a homeomorphism 
\begin{equation}\label{eq:rhoann}
\rho: \ov{\mc M}{}^{\rm ann}_{1, 1, 0} \to [-\infty, +\infty]
\end{equation}
defined as follows (See Figure \ref{mannuli}). 
\begin{center}
\begin{figure}[h]
  \includegraphics[scale=0.65]{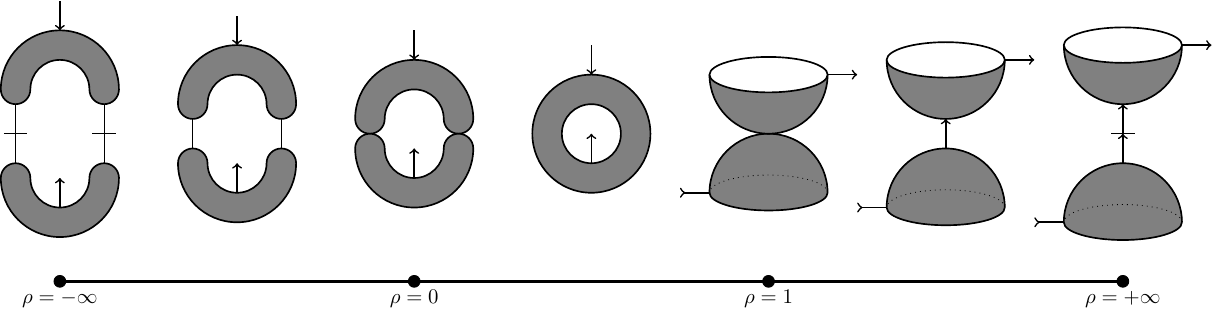}
\caption{Moduli of treed annuli with fixed non-zero angle offset} 
\label{mannuli}
\end{figure} 
\end{center}
For configurations containing an annulus component with inner radius $\rho_1$ and outer radius $\rho_2$ we define
\begin{equation} \label{rhoC}
\rho (C) = \frac{\rho_1 \rho_2^{-1}}{1 + \rho_1 \rho_2^{-1}}. 
\end{equation} 
In the case that the $0$-th markings on the inner and outer circles are contained in different disk components, suppose these two disks are connected by a path consisting of boundary edges $T_{e_1}, \ldots, T_{e_k}$ of lengths $\ell(e_1), \ldots, \ell(e_k)$; then we define
\[
\rho (C) = - \ell(e_1) - \cdots - \ell(e_k);
\]
the balanced condition implies that this value is independent of the choice of the path. On the other hand, if the disks containing two boundary circles are connected by a path consisting of interior edges $T_{e_1}, \ldots, T_{e_k}$ of lengths $\ell(e_1), \ldots, \ell(e_k)$, then define
\[
\rho(C) = 1 + \ell(e_1) + \cdots + \ell(e_k).
\]
\end{example} 

The description of the one inner-and-outer marking moduli space in Example \ref{innerandouter} leads to the following natural defined functions on moduli spaces with higher numbers of inner and outer markings: Composing the homeomorphism $\rho$ in \eqref{eq:rhoann} with the forgetful map \eqref{eq:forgetann}, we obtain a map 
\begin{equation}
  \label{eq:rho-d-ann}
  f : \ol{\M}_{d_-,d_+, d_\bullet}^\ann \to [-\infty,\infty].
\end{equation}
For any $\rho \in [-\infty,\infty]$ the fiber $f^{-1}(\rho)$ is the moduli space of annuli with a fixed ratio of inner and outer radii, and is denoted by
\begin{equation}
  \label{eq:fixedrho}
  \M_{d_-,d_+, d_\bullet}^{\ann, \rho} \subset \ol{\M}_{d_-,d_+, d_\bullet}^\ann.  
\end{equation}
Treed annuli can degenerate to broken configurations whose components can be treed disks (with no interior {gradient} leaves), open-closed domains, and closed-open domains. 

\begin{remark} \label{morient} 
The moduli space of treed annuli admits an orientation induced from choices of orientations on nodal annuli induced from the positions of the interior and boundary markings. We can identify each annuli of any width $\rho \in (0, 1)$ with a fixed annulus $A$; by recording the markings, one obtains a map
\[
{\mc M}_{d_-,d_+, d_\bullet}^{{\rm ann}, \rho} \hookrightarrow \left( {\rm Int}(A)^{d_\bullet} \times (\partial^+ A)^{d_+}\times (\partial^- A)^{d_-} \right) / S^1.
\]
The orientations on this stratum extends to a global orientation on the manifold with boundary $\ol{\M}_{d_-,d_+, d_\bullet}^\ann$. The boundary of $\ol{\M}_{d_-,d_+, d_\bullet}^\ann$ consists of configurations where the ratio $\rho$ is equal to $\infty$, configurations where the ratio $\rho$ is equal to $-\infty$ (in the sense that the lengths of the paths $\gamma_\pm$ above are infinite) and configurations where a collection of leaves $T_e$ have bubbled onto
  disks $S_v, v \in \Ver(\Gamma)$ attached to the outer boundary, and
  configurations where leaves $T_e$ have bubbled onto disks $S_v$ on
  the inner boundary.  The latter two types of boundary strata
  $\M_\Gamma, \Gamma = \Gamma_1 \# \Gamma_2$ have opposite
  orientations compared to the product orientation on
  $\M_{\Gamma_1} \times \M_{\Gamma_2}$.
\end{remark}

Regularizing families of holomorphic maps from treed annuli requires regularization of holomorphic disks, strips, and spheres as before.  We shall require that the perturbations for defining the equation for holomorphic treed annuli extend the existing ones, by induction on the type of the map. For each stable annuli type $\Gamma$, a map type $\bGamma$ consists of 
\begin{enumerate}
\item labelling of boundary markings by generators of the chain groups $CF^\bullet( \LB_i, \LB_{i+1})$, 
\item a labelling of interior markings by either components of the bulk deformation, the Donaldson hypersurface $D$, or their (transverse) intersections,
\item a labelling of interior {gradient} leaves by critical points of $f_X$, and 
\item a labelling of surface vertices by homology classes 
$\beta \in H_2(X, |{\mc L}|)$.  
\end{enumerate}
For each map type $\bGamma$ one has a moduli space of treed holomorphic annuli ${\mc M}_\bGamma^{\rm ann}(P_\Gamma)$ with respect to the perturbation $P_\Gamma$ (if $\Gamma$ is stable; otherwise we pull back $P_{\Gamma^{\rm st}}$).   A map type is called {\it essential} if, as in previous cases, that there is no finite edges $T_e$ with length
$ell(e) = 0$, no broken edges $T_e$, no sphere components $S_v, v \in \Ver_\black(\Gamma)$, that all interior markings $\ul{z}_v$ are labelled by the Donaldson hypersurface $D$ or components of the bulk deformation, and the width parameter $\rho$ is not equal to $-\infty$, $0$, $1$, or $+\infty$.  Given two sequences $\uds x^\pm = (x_0^\pm, \cdots, x_{d_\pm}^\pm)$ of generators of the Floer chain groups where $x_i^\pm \in {\mc I}(\LB_i^\pm, \LB_{i+1}^\pm)$ (we define $L_{d_\pm + 1}^\pm = L_0^\pm$) let
\[
{\mc M}^{\rm ann}(\uds x^-, \uds x^+)_i:= \bigsqcup_{\bGamma} {\mc M}_{\bGamma}(P_\Gamma)_i \quad i = 0, 1
\]
where the union is taken over all essential map types $\bGamma$ whose boundary labelling data are $\uds x_\pm$ and whose expected dimension is $i$.

\begin{lemma} \label{Abound} 
\begin{enumerate}

\item There exist a coherent system of strongly regular 
(Definition 
\ref{stronglyregular})
perturbations $P_\Gamma$ for all stable treed annuli that extend the existing perturbation data for treed disks (with no interior {gradient} leaves), open-closed domains, and closed-open domains. As a consequence, for each uncrowded map type $\bGamma$, the moduli space ${\mc M}_{\bGamma}(P_\Gamma)$ is regular of expected dimension.

\item For such a system of perturbations, for $d_\pm \geq 1$, the zero-dimensional moduli space ${\mc M}^{\rm ann} (\uds x^-, \uds x^+)_0$ is discrete and finite under each energy level. 
\item Moreover, the one-dimensional moduli space ${\mc M}^{\rm ann}(\uds x^-, \uds x^+)_1$ has true boundary corresponding to the following types of degenerations (all of which have no sphere bubbling).
\begin{enumerate}
    \item The width parameter $\rho$ goes to $+\infty$ and one interior edge breaks.
    \item The width parameter $\rho$ goes to $-\infty$ and two boundary edges breaks.
    \item $\rho$ is finite and different from $0$, $1$, while a boundary edge breaks.
    
\end{enumerate}
\end{enumerate}
\end{lemma}

\subsection{Commutativity}

We show that the Cardy diagram commutes at the level of cohomology. Using the moduli spaces of holomorphic treed annuli, we define a homotopy operator relating the composition of the maps around the two sides of the diagram in Theorem \eqref{cardy}.  This shows that the diagram in Theorem \ref{cardy} is commutative. For weakly unobstructed branes $\WB_0, \ldots, \WB_{d_+}$ with underlying Lagrangian submanifolds from ${\mc L}$, $w = w(\WB_i)=w({\bm K})$, and any nonnegative integer $d_-$, we define a linear map
\begin{equation}\label{eq:alephdef}
{\mc S}: CF^\bullet(\WB_0, \WB_{d_+}) \otimes \cdots \otimes CF^\bullet( \WB_0, \WB_1) \to {\rm Hom} \left( CF^\bullet( {\bm K}, {\bm K} )^{\otimes d_-}, CF^\bullet ( {\bm K}, {\bm K} )\right)
\end{equation}
by counting holomorphic treed annuli. More precisely, for $\uds a^+ = a_{d_+}^+\otimes \cdots \otimes a_0^+$ and $a_{d_-}^-, \ldots, a_1^-$, one has
\[
{\mc S}(\uds x^+)(a_{d_-}^- \otimes \cdots \otimes a_1^-) = \sum_{a_0^-} \sum_{u \in \ov{\mc M}{}^{\rm ann} (\uds x^-, \uds x^+)_0 }
 (-1)^{\heartsuit}\wt(u) .
\]
By summing over all possible ways of inserting weakly bounding cochains on $\WB_0, \ldots, \WB_{d_+}$ and the weakly bounding cochain on ${\bm K}$, we obtain a map (c.f. Abouzaid \cite[Equation 6.22]{Abouzaid_generation})
\[
{\mc S}: CC_\bullet( \Fuk_{{\mc L}}(X, \bb)_w ) \to {\rm Hom}(\WK, \WK).
\]

\begin{figure}[ht]
  \centering 
  \includegraphics[scale=0.9]{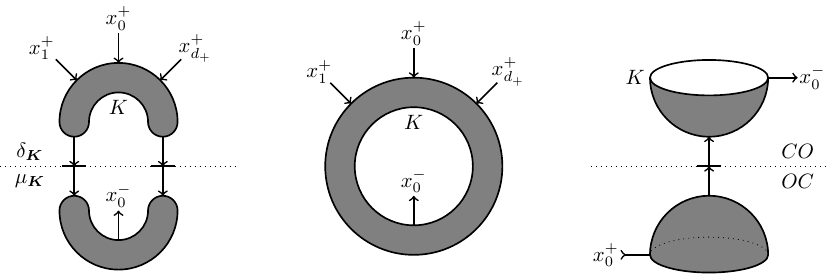}
  \caption{Cardy relation : End-points of a one-dimensional moduli space of holomorphic treed annuli. There could be insertions of weakly bounding cochains on both inner and outer circles.} 
  \label{fig:cardy}
\end{figure}

\begin{proof}[Proof of Theorem \ref{cardy}]
  It follows from the description of the boundary (see Figure
  \ref{fig:cardy}) in Lemma \ref{Abound} that the operator ${\mc S}$
  is a homotopy operator relating the two sides of the Cardy diagram:
  We have
  \begin{equation}\label{hop}
    m_{1, {\bm K}} \circ {\mc S} + {\mc S} \circ \delta_{CC} = (-1)^{\dim(X)(\dim(X) + 1)/2} CO(\bb) \circ OC(\bb) - \mu_\WK \circ \delta_\WK
  \end{equation}
  on $\Fuk_{\mc L}(X,\bb)_{w(\WK)}$.  The sign computation is carried
  out  in \cite{Abouzaid_generation}, and will not be repeated
  here. Therefore, on homology level, the Cardy diagram
  \eqref{cardy_diagram} commutes up to the expected sign. On
  $\Fuk_{\mc L}(X,\bb)_{w}$ with $w \neq w(\WK)$,
  $\mu_\WK \circ \delta_\WK$ is zero by definition, and
  $CO(\bb) \circ OC(\bb)$ vanishes as a consequence of the spectral
  decomposition results. In particular, by Theorem \ref{thm16}, 
  $OC(\bb)(\Fuk_{\mc L}(X,\bb)_{w})$ lies in the generalized
  eigenspace $QH^\bullet(X)_{D_q w} \subset QH^\bullet(X)$ of the
  quantum multiplication by $[\om]^\bb$, and the image
  $CO_K(QH^\bullet(X)_{D_q w})$ is zero by Theorem \ref{CO_spectral}.
\end{proof}

\subsection{Orthogonality for disjoint Lagrangians}\label{subsec:orth}

We prove another result about the orthogonality of images under the open-closed map, 
c.f. Corollary \ref{orthogonal1}.

\begin{theorem}\label{ortho} {\rm(Restatement of Theorem \ref{thm110})}
Suppose that $\cL_-, \cL_+ \subset \cL$ are disjoint collections of Lagrangian submanifolds in $X$, that is, $|{\mc L}_-|\cap |{\mc L}_+| = \emptyset$. Suppose 
\[
{\mf L}_\pm \subset MC({\mc L}_\pm).
\]
Then the images of elements
\begin{align*}
&\ [{\alpha}_-] \in HH_\bullet(\Fuk^\flat_{{\mf L}_-} (X,\bb)), \ &\ [{\alpha}_+] \in HH_\bullet (\Fuk^\flat_{{\mf L}_+} (X,\bb))
\end{align*}
under the open-closed map are orthogonal with respect to the intersection pairing.
\end{theorem} 
The proof of Theorem \ref{ortho} of the moduli space of treed holomorphic annuli as the width parameter
goes to zero.

\subsubsection{Intersection pairings and two chain-level open-closed maps}

We recall the chain-level definition of the intersection pairing. Recall that $(f_X, g_X)$ is a Morse-Smale pair on $X$ with the cochain complex $CM^\bullet(f_X)$. The pair $(-f_X, g_X)$ is also a Morse-Smale pair with complex $CM^\bullet(-f_X)$. Define 
\[
\langle \cdot \rangle: CM^\bullet(f_X) \otimes CM^\bullet(-f_X) \to \Lambda
\]
by 
\[
\left\langle \sum_i a_i x_i, \sum_j b_j x_j \right\rangle = \sum_i a_i \ov{b_i}.
\]
It is easy to see that the chain-level pairing descends to cohomology.

The two Morse-Smale pairs can define potentially different chain-level open-closed maps. By using perturbations on open-closed treed disks where we use either $(f_X, g_X)$ or $(-f_X, g_X)$, we can define two chain maps
\[
OC(\bb)_\pm: CC_\bullet (\Fuk^\flat_{\mf L}(X, \bb)) \to CM^\bullet( \pm f_X).
\]
On the homological level the two maps are identical. On the other hand, we can define a chain-level pairing
\[
\langle \cdot \rangle_\infty: CC_\bullet( \Fuk^\flat_{{\mf L}_-}(X, \bb)) \otimes CC_\bullet( \Fuk^\flat_{{\mf L}_+}(X, \bb)) \to \Lambda
\]
by 
\[
\langle \alpha_-, \alpha_+ \rangle_\infty:= \langle OC(\bb)_-(\alpha_-), OC(\bb)_+(\alpha_+) \rangle.
\]
We will prove that on the homology level the pairing is zero.

\subsubsection{Treed annuli}

Treed annuli used in the proof of the generation criterion are defined as follows. First, we modify the conditions on the marked annuli specified in Definition \ref{def:treedann}. We require that both inner $(\partial S)_-$ and outer circles $(\partial S)_+$ of an annulus $S$ contain boundary markings, and that the boundary markings on the outer circle $(\partial S)_+$ are counterclockwise ordered while those on the inner circle
$(\partial S)_-$ are clockwise ordered.  Given a marked annulus, 
we create a treed annulus by attaching to each boundary marking a semi-infinite edge and require that all these semi-infinite edges are incoming ones. We also require that, most importantly, the angle offset between $z_0^+$ and $z_0^-$ is $0$ instead of $\pi$. Hence, when we compactify the moduli space of treed annuli, when the width parameter $\rho$ approaches to zero, in the degenerate configurations $z_0^+$ and $z_0^-$ can be contained in the same surface component.
Figure \ref{fig:offset0} describes a compactified one-dimensional moduli space of such treed annuli. In general, there is a width parameter 
\[
\rho: \ov{\mc M}{}^{\rm ann}_{d_-, d_+, d_\bullet} \to [0, +\infty]
\]
on the moduli space of stable treed annuli with $d_-$ resp. $d_+$ boundary markings on the inner resp. outer circle and $d_\bullet$ interior markings.

Perturbations are defined on the universal curves of treed annuli. Notice that in the current situation, treed annuli can degenerate to broken configurations whose unbroken components can be either a treed disk with exactly one output (the $\rho=0$ slice of Figure \ref{fig:offset0}), or a treed disk of open-closed type but not closed-open type (the $\rho = +\infty$ slice of Figure \ref{fig:offset0}). This is also different from the case of the Cardy diagram. We require that when the annuli degenerate to two disks of open closed type, the perturbation on the component containing the outer resp. inner circle coincides with the perturbation chosen for the open-closed map for the Morse-Smale pair $(f_X, g_X)$ resp. $(-f_X, g_X)$.

Moduli spaces of treed annuli are defined as follows. We require that the labelling on the outer circle are from the branes in ${\mf L}_+$ and the labelling on the inner circle are from the branes in ${\mf L}_-$. A map type is called {\it essential} if it has no spheres, broken edges, or edges of length zero. Given boundary labelling data $\uds x_- = (x_{-, 0}, \ldots, x_{-, d_-})$, $\uds x_+ = (x_{+, 0}, \ldots, x_{+, d_+})$, homology classes labelling surface components, and interior labelling data, we can consider perturbed treed holomorphic annuli satisfying these constraints.  We require that, when the interior edge has positive length, the treed map satisfies the negative gradient flow equation for $(f_X, g_X)$ if we orient the edge from the component containing the outer circle to the component containing the inner circle. One can achieve transversality in the same way as before and we omit the details. Then let ${\mc M}^{\rm ann} (\uds x_-, \uds x_+)_i$ be the union of moduli spaces of essential map types $\bGamma$ of expected dimension $i$. When $i = 1$, by identifying fake boundary strata, we can describe the true boundaries of the closure $\ov{\mc M}^{\rm ann}(\uds x_-, \uds x_+)_1$. The true boundary strata include  configurations $u: C \to X$ where
\begin{enumerate}
    \item the width parameter $\rho$ is $+\infty$ and two disk components, no sphere components, and only one breaking on the interior edge.
    
    \item the width parameter $\rho$ is $0$ and there is one disk component and no sphere components; since the node on the infinitely thin annulus must have evaluation on $|{\mc L}_-| \cap |{\mc L}_+|$, which is empty. Hence this boundary stratum is empty;
    
    \item the width parameter $\rho$ is positive and finite and there is a breaking of a boundary edge $T_e$. 
\end{enumerate}

We define a chain-level map using treed annuli of varying width parameters. For $\uds x_\pm = x_{\pm, 0},\ldots, x_{\pm, d_\pm}$, consider map types $\bGamma$ with the outer resp. inner circles labelled by $\uds x_+$ resp. $\uds x_-$. First consider only the essential map types, so that the width parameter $\rho$ is positive and finite, having no boundary edges of length zero or boundary breakings. By counting rigid configurations, one obtains a map 
\begin{multline*}
{\mc T}: \Big( CF^\bullet( \LB_{-, d_-}, \LB_{-, 0}) \otimes \cdots \otimes CF^\bullet( \LB_{-, 0}, \LB_{-, 1}) \Big)\\ \otimes \Big( CF^\bullet( \LB_{+, d_+}, \LB_{+, 0}) \otimes \cdots \otimes CF^\bullet( \LB_{+, 0}, \LB_{+, 1}) \Big) \to \Lambda  
\end{multline*}
by 
\[
{\mc T}(\uds a_-, \uds a_+) = \sum_{u \in {\mc M}^{\rm ann} (\uds x_-, \uds x_+)_0} (-1)^{\heartsuit} {\rm wt}(u). 
\]
The map ${\mc T}$ induces a (not necessarily chain) map \label{lplm}
\[
{\mc T}: CC_\bullet( \Fuk^\flat_{{\mf L}_-}(X, \bb)) \otimes CC_\bullet( \Fuk^\flat_{{\mf L}_+}(X, \bb)) \to \Lambda.
\]

\begin{figure}[t]
  \centering 
  \includegraphics[scale=0.7]{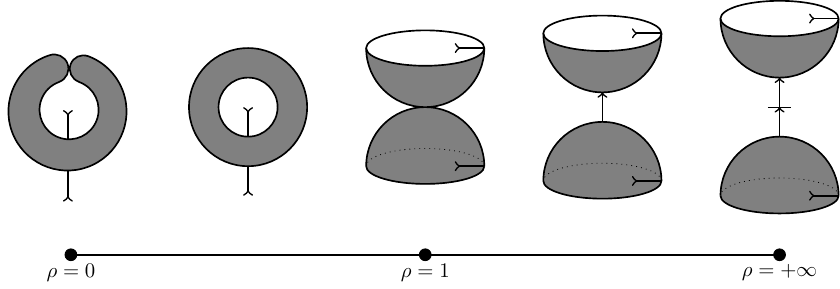}
  \caption{Moduli space of treed annuli with zero angle offset between distinguished leaves.} 
  \label{fig:offset0}
\end{figure}

\begin{proof}[Proof of Theorem \ref{ortho}]
Suppose $\alpha_\pm \in CC_\bullet( \Fuk^\flat_{{\mf L}_\pm}(X, \bb))$ are Hochschild cycles. Let $\uds a_\pm = a_{\pm, 0}\otimes \cdots \otimes a_{\pm, d_\pm}$ be a component of $\alpha_\pm$ with underlying 
critical points $\ul{x}_\pm$. Consider one-dimensional moduli spaces $\ov{\mc M}^{\rm ann}(\uds x_-, \uds x_+)_1$. The boundary strata consist of the following types:
\begin{enumerate}
    \item The  strata corresponding to $\rho = \infty$, denoted by ${\mc M}^{\rm ann}_\infty(\uds x_-, \uds x_+)_0$.  These strata contribute to the chain-level pairing  
    \[
    \langle OC(\bb)_- (\uds a_-), OC(\bb)_+ (\uds a_+) \rangle\in \Lambda.
    \]
    Indeed, on the broken edge, the treed map satisfies the negative gradient flow equation of $f_X$ (from the disk with the positive boundary to the disk with the negative boundary). We regard the map restricted to the semi-infinite edge attached to the disk with the negative boundary as the (perturbed) negative gradient flow equation of $-f_X$. Therefore, by the definition of the chain-level intersection pairing, the count of such configurations is exactly $\langle OC(\bb)_-(\uds a_-), OC(\bb)_+(\uds a_+) \rangle$.
    
    \item The union of strata corresponding to $\rho = 0$, denoted by ${\mc M}_0^{\rm ann}(\uds a_-, \uds a_+)_0$. Since $|{\mc L}_-|\cap |{\mc L}_+| = \emptyset$, this moduli space is always empty. 
    
    \item Configurations for $\rho \in (0, \infty)$ with one boundary breaking. These configurations contribute to 
    \[
    {\mc T}( \delta_{CC}(\alpha_-), \alpha_+) \pm {\mc T}( \alpha_-, \delta_{CC}(\alpha_+))
    \]
    which is zero.
\end{enumerate}
Therefore, it follows that on the chain level
\[
\langle OC(\bb)_-(\alpha_-), OC(\bb)_+(\alpha_+)\rangle = 0. \qedhere
\] 
\end{proof}

\section{Fukaya categories of blowups}\label{section5}

In this section,  we consider the special cases of previous constructions in the setting of the main theorem.  More precisely,
we study a perturbation scheme for which one has a correspondence
between treed disks in the original symplectic manifold
and its blowup.  

\subsection{The geometry of the blowup}

We fix an explicit construction of a family of blowups at the chosen point. 
From now on, $(X, \omega)$ denotes a rational symplectic manifold, ${\mc L}$ denotes a collection of rational Lagrangian submanifolds (see Definition \ref{srat}) satisfying Hypothesis \ref{lagrangian_assumption}, and $\bb$ is a bulk deformation. In addition, fix a point $p$ disjoint from the Lagrangians and the bulk deformation. We also fix a Donaldson hypersurface $D$ as before with the additional requirement that $p\notin D$, and a tamed almost complex structure $J_0$ satisfying (b) of Lemma \ref{countable}.   Let $U$ be a Darboux coordinate chart centered at $p$ that is disjoint from $|{\mc L}|$ and $\bb$ with Darboux coordinates $x_1, y_1, \ldots, x_n, y_n$. As $D\cap U = \emptyset$, we may assume that $J_0|_U$ is the standard complex structure with complex coordinates $z_i = x_i + \sqrt{-1} y_i$. 
The symplectic blowup $\tilde{X}$ of $X$ at $p$ is defined by removing
Darboux chart from $X$ and gluing in a neighborhood of
$\tilde{Z} = \C P^{n-1}$ in 
\[ \Bl_0(\C^n) := \Set{ (\ell,z) \in \C P^{n-1} \times \C^n | 
z \in \ell  }.\]
It admits 
an almost complex structure $\tilde J_0$ whose restriction to $\tilde U:= \pi^{-1}(U)$ is the integrable complex structure $J_{\tilde U}
: {\tilde U} \to {\tilde U}$ coming from the blowup.

We equip the blowup with a family of symplectic structures by symplectic cut. Following Lerman \cite{Lerman_1995}, for each $\epsilon>0$ sufficiently small, we may view $\tilde U$ as 
\[
\{ (z_1, \ldots, z_n)\in U \ |\ |z_1|^2 + \cdots + |z_n|^2 \geq \epsilon \}/\sim
\]
where $\sim$ is the relation collapsing the sphere 
\[ |z_1|^2 + \cdots + |z_n|^2 = \epsilon \] 
to $\mb{CP}^{n-1}$. In this way we obtain a family of symplectic forms $\tilde \omega_\epsilon \in \Omega^2(\tilde X)$ that agree with $\pi^* \omega$ outside $\tilde U$. Moreover, for all $\epsilon$, $\tilde J_0$ is $\tilde \omega_\epsilon$-tamed. In notation, we abbreviate $\tilde \omega_\epsilon$ by $\tilde \omega$. One can see that as cohomology classes, 
\begin{equation}\label{area_difference}
[\tilde \omega] = [\pi^* \omega] - \epsilon {\rm PD}([ \tilde Z]) \in H^2(\tilde X; {\mb R})
\end{equation}
where ${\rm PD}$ denotes the Poincar\'e dual. 

\subsubsection{The exceptional Lagrangians}\label{subsubsec:el}

In this section,  we introduce the additional Lagrangians needed to generate
the Fukaya category of the blowup.  First we realize blowup as a symplectic quotient. Consider a diagonal $S^1$-action on ${\mb C} \times {\mb C}^n$ with moment map 
\[
\Phi(z_0, z_1, \ldots, z_n) = - \frac{1}{2} \left( |z_0|^2 - |z_1|^2 - \cdots - |z_n|^2 \right).
\]
The symplectic quotient at the level $\Phi =\frac{\epsilon}{2}$ can be viewed as the $\epsilon$-blowup of $\mb{C}^n$ at the origin.  A neighborhood of the exceptional divisor $\mb{CP}^{n-1}$ can be identified with the neighborhood $\tilde U \subset \tilde X$. 
Consider
\[
\hat{L}_{\bm \epsilon} = \Set{ (z_0,\ldots, z_n) | \left|  z_i \right|^2 = \epsilon_i , \ i = 0,\ldots, n  } \subset {\mb C} \times {\mb C}^n
\]
for ${\bm \epsilon} = (\epsilon_0, \ldots, \epsilon_n) \in (\R_{> 0})^{n+1}$.  Suppose that 
\[
\epsilon_1 + \cdots + \epsilon_n - \epsilon_0 = \epsilon.
\]
In this case we have $\hat L_{\bm \epsilon} \subset \Phi^{-1}(\frac{\epsilon}{2})$ and and so the Lagrangian descends to a Lagrangian torus $L_{\bm \epsilon} \subset \tilde X$. \label{Ldec}

\begin{lemma}
When $\epsilon_0 = \epsilon_1 = \cdots = \epsilon_n = \frac{\epsilon}{n-1}$, $L_{\bm \epsilon}$ is a monotone Lagrangian in $\tilde U$.
\end{lemma}

\begin{proof}   Any disk bounding $L_{\bm \epsilon}$ lifts
to a disk in $\C^n \times \C$ bounding $\hat{L}$ with the 
same area and index.   Maps from disks to  $\C^n \times \C$ are products of disks in the factors.  The homology classes of such 
are generated by the disks of Maslov index two in each factor, all of which have the same area.   See \cite{Cho_Oh} for more details.  
\end{proof}

\subsubsection{Donaldson hypersurfaces in the blowup}

Since the pullback of the original Donaldson hypersurface is no longer a Donaldson hypersurface in the blowup, we need to choose a new Donaldson hypersurface to fit into the general framework. In order to use the explicit calculation in the previous section, we construct perturbations that have standard almost complex structures near the exceptional locus.

\begin{proposition}\label{new_divisor}
For each small rational $\epsilon$, there exist a Donaldson hypersurface $\tilde D \subset \tilde X$ and a tamed almost complex structure $\tilde J$ satisfying the following condition. 
\begin{enumerate}
    \item $\tilde D\subset \tilde X \setminus (|\tilde {\mc L}| \cup \tilde L_{{\bm \epsilon}})$ and the symplectic form $\tilde \omega$ is exact in the complement of $|\tilde {\mc L}| \cup \tilde L_{{\bm \epsilon}}$.

    \item $\tilde J$ coincides with $J_{\tilde U}$ inside $\tilde U$.
    
    \item $\tilde D$ is almost complex with respect to $\tilde J$ and is holomorphic inside $\tilde U$.
    
    \item $\tilde D$ intersects the exceptional locus $\tilde Z$ transversely.
    
    \item $\tilde D$ intersects the pullback hypersurface $\pi^{-1}(D)$ transversely.
    
    \item $\tilde D$ intersects the components of $\tilde {\mf b}_0$ transversely.
    
    \item $\tilde D$ intersects generic Maslov 2 disks in $\tilde U$ transversely.
\end{enumerate}
\end{proposition}

\begin{proof}[Sketch of proof]
The statement of the proposition is essentially a special case
 of  Auroux-Gayet-Mohsen 
\cite[Section 3.1]{Auroux_Gayet_Mohsen}, which describes how 
Donaldson's argument \cite{Donaldson_96} can be extended to a relative setting. More precisely, we identify $\tilde U$ with a neighborhood of the zero section of ${\mc O}(-1) \to \mb{CP}^{n-1}$. For small rational $\epsilon$, we can choose a generic holomorphic section $\tilde s_0$ of a sufficiently positive line bundle on ${\mc O}(-1)$ which intersect the zero locus and all the Maslov two disks transversely. Choose a cut-off function $\rho$ supported in $\tilde U$ which is identically 1 near $\tilde Z$. Then $\rho \tilde s_0$ is a smooth section of a positive line bundle over $\tilde X$ whose Chern form is a large multiple of $\tilde \omega$. Apply Donaldson's argument by using a collection of local sections of this line bundle (supported away from $\tilde Z$) and generic linear combination to achieve transversality to the given section $\rho \tilde s_0$.  
\end{proof}

\subsection{A perturbation system for the new branes}\label{subsection29}

In this section, we describe perturbation data on a blowup that is standard near the exceptional divisor.  We make explicit computations involving holomorphic disks whose boundary maps to exceptional branes.  To achieve symmetry properties of the composition maps, the perturbation data we consider are multivalued.  The symmetry property is used to show that a weak version of the divisor equation holds. \label{somesy}

We recall some geometric details about the neighborhood of the exceptional divisor needed for the construction of our perturbation data. Let $p \in X$ be the blowup point. Recall that the bulk deformation, the collection of Lagrangian branes, and the Donaldson hypersurface are all disjoint from $p$, hence disjoint from a Darboux chart $U \ni p$. Let $\tilde{U} \subset \tilde{X}$ be the preimage of $U$
under the projection $\tilde{X}\to X$. Fix the Darboux coordinate in $U$. Let $J_{\tilde{U}}$ be the integrable almost complex structure on $\tilde{U}$ that is the pullback from the standard complex structure with respect to the Darboux coordinates in $U$. The exceptional branes in ${\mc E}$ are all supported on an embedded Lagrangian
$L_{\bm\epsilon} \subset \tilde{U}$.

\subsubsection{Holomorphic disks bounding the exceptional Lagrangian}

We wish to classify the holomorphic disks of minimal area bounding the
exceptional Lagrangian.  Since the picture is locally toric, the
classification is a special case of the computations in Cho-Oh \cite{Cho_Oh}. In
particular, from the description in Section \ref{subsubsec:el}, a disk
$u: (\D, \partial \D) \to (\tilde X,L_{\bm \eps})$ whose image is
contained in a neighborhood of the exceptional divisor may be viewed
as a disk mapping to $(\C^{n+1},\hat L_{\bm \eps}) \qu S^1$, where
$S^1$ acts on $\C^{n+1}$ with weights $(-1,1,\dots,1)$ and
$\hat L_{\bm \eps}=\{|z_i|=\eps_i, i=0,\dots,n\} \subset \C^{n+1}$.
The disk $u$ lifts to a Blaschke product $\hat u$
whose definition we recall.

\begin{definition}
A {\em Blaschke product} of degree $(d_0,\ldots,d_n)$ with boundary in the Lagrangian $\hat L_{\bm\epsilon}$ is a map $\hat u : (\DD, \partial \D) \to (\mb{C}^{n+1}, \hat L_{\bm \eps})$ prescribed by coefficients
\[ | \zeta_i | = \epsilon_i , \quad a_{i,j} \in \C, \quad |a_{i,j}| < 1, \quad i \leq n+1, \quad j \leq d_i \]
and defined as 
\begin{equation} \label{blaschke} 
\hat u: (\D, \partial \D) \to (\C^{n+1}, \hat L_{\bm\epsilon}), \quad z
    \mapsto \left( \zeta_i \prod_{j=1}^{d_i} \frac{ z - a_{i,j}}{1 - z
        \ol{a_{i,j}}} \right)_{i = 0,\ldots,n} .\end{equation}
\end{definition}

We include the following proposition computing the areas and indices of Blaschke products from Cho-Oh \cite{Cho_Oh} for completeness:

\begin{lemma} \label{index} 
The descent $u : (\D, \partial D) \to (\tilde X, L_{\bm \eps})$ of the Blaschke product $\hat u: (\D, \partial \D) \to (\mb{C}^{n+1}, \hat L_{\bm \epsilon})$ given by \eqref{blaschke} has Maslov index
  \[ I(u) = \sum_{i=1}^{n+1} 2d_i \]
and area 
\[ A(u) = \pi \sum_{i=1}^{n+1} d_i \eps_i .\]
\end{lemma}

\begin{proof} 
  As in Cho-Oh \cite[Theorem 5.3]{Cho_Oh}, the products
  \eqref{blaschke} are a complete description of holomorphic disks
  with boundary in $\hat{L}_{\bm \eps}$.  Any Blaschke product
  $\hat{u} : (\D,\partial \D) \to (\C^{n+1}, \hat L_{\bm \eps})$ disjoint from the
  semistable locus descends to a disk $u: (\D,\partial \D) \to (\tilde X, L_{\bm \eps})$.
  We compute its Maslov index using the splitting (with notation
  $\partial u := u |_{\partial C}$)
  \[ (\hat{u}^* T\C^{n+1}, (\partial \hat{u})^* T\hat L_{\bm \eps}) \cong (u^* T\tilde X , (\partial u)^* TL_{\bm \eps})
  \oplus (\g_\C,\g) \]
where $\g_\C,\g$ denotes the trivial bundle and real boundary
condition with fiber $\g_\C \simeq \C^\times$ resp. $\g \simeq S^1$ the Lie algebras of the
complex resp. real torus acting on $\C^{n+1}$. We write 
\[ I(E,F) \in \Z \]  
for the Maslov index of a pair $(E,F)$ consisting of a complex vector
bundle $E$ on the disk $\DD$ and a totally real sub-bundle $F$ over the
boundary $\partial \DD$.  Since the Maslov index of bundle pairs is
additive, \label{bundlepairs}
\[ I( \hat{u}^* T \C^{n+1}, (\partial \hat{u})^* T\hat L_{\bm \eps} ) = I( u^* T \tilde{X}
, (\partial u)^* TL_{\bm \eps}) + I( u^* \g_\C, (\partial u)^* \g) .\]
The second factor has Maslov index
$I(u^* \g_\C, (\partial u)^* \g) = 0$, as a trivial bundle.  It follows that the
Maslov index of the disk $u$ is given by
\[ I(u) = I(\hat u^*T\C^{n+1}, (\partial \hat u)^* T \hat L_{\bm \eps}) = \sum_{i=1}^{n+1} 2 d_i = 2 \# u^{-1} \left(
  \sum_{i=1}^k [D_i] \right) ;\]
that is, $I(u)$ is twice the sum of the intersection number with the
anticanonical divisor.   That is,  
\[ [K^{-1}] = \sum_{i=1}^k [D_i] \in H^2(X,Z) \]
is the disjoint union of the prime invariant divisors
\[ D_i = [z_i = 0] \subset \C^{n+1} \qu \C, i = 1,\ldots, k . \] 
After an automorphism of the domain $\DD$, the disks of Maslov index two are
those maps $u_i: \DD \to X$ with lifts \label{withlifts} of the form
\[ \hat{u}_i: \D \to \hat{X}, \quad z \mapsto ( b_1,\ldots, b_{i-1},
b_i z, b_{i+1}, \ldots, b_{n+1}) .\]
%
We call these the {\em basic disks} and their homology classes {\em basic classes}.  The area of each such disk is
\[ A(u_i) = A(\hat{u}_i) = \eps_i \]
since 
\[ \int \hat{u}_i^* \hat{\omega} = \int_{r^2/2 = 0}^{r^2/2 =
  \eps_i/2\pi} r \d r \d \theta = \eps_i . \]
\label{primitive}
The homology class of higher index Maslov disks
$u: C \to X, I(u) > 2 $ is a weighted sum 
\[ [u ] = \sum d_i [u_i] \]
of homology classes of basic disks $u_i, i =1,\ldots, {n+1}$.  It follows that
the area $A(u) \in \R$ of such a disk $u$ is the weighted sum
\[ A(u) = \sum d_i A(u_i) \] 
of the areas $A(u_i)$ of disks $u_j$ of Maslov index $I(u_j) =
2$. The claim on the area follows.  \end{proof}

Next we describe the relation between the areas of disks in the blow-up and their projections.
Suppose that the
almost complex structures on $\tilde{X}, X$ are such that the projection 
\[ \pi: \tilde{X} \to X \] 
is almost complex, so that any holomorphic curve
$\tilde{u} : C \to \tilde{X}$ defines a holomorphic curve $u: C \to X$ by
projection.  Since the exceptional divisor $\tilde{Z}$ is almost complex,
the intersection number $\tilde{u}. \tilde{Z}$ is the sum of positive
intersection multiplicities at each of the intersection points
$\tilde{u}^{-1}(\tilde{Z})$, see for example \cite[Proposition
7.1]{Cieliebak_Mohnke}.

\begin{lemma} \label{arearel} The areas of $\tilde{u}$ and
  $u := \pi \circ \tilde{u}$ are related by
  $ A(\tilde{u}) = A(u) - \eps ( [\tilde{u}]. [\tilde{Z}] ) $.
\end{lemma} 

\begin{proof} By Mayer-Vietoris and the definition of the  
  symplectic form on the local model 
the 
  symplectic class $[\tilde{\omega}] \in H^2(\tilde{X})$ is equal to 
\[  [ \tilde{\omega}] = \pi^* [\omega] + \eps [\tilde{Z}]^ \dual \] 
  where $[\tilde{Z}]^\dual \in H^2(\tilde{X})$ is the dual class to the
  exceptional divisor $\tilde{Z}$.  Pairing with
  $[\tilde{u}] \in H_2(\tilde{X})$ proves the claim.
\end{proof}

\begin{proposition}\label{indcor}\label{prop236}
\begin{enumerate}
\item $(\tilde{U}, L_{\bm \epsilon}, \tilde \omega|_{\tilde{U}})$ is monotone with
  minimal Maslov index two.
    
\item The moduli space of $\tilde{J}_0$-holomorphic disks
  ${\cM}_{0,1}( \tilde{U}, L_{\bm \epsilon}, \tilde{J}_0)$ in $\tilde{U}$ with
  boundary in $L_{\bm \epsilon}$, with one boundary marking 
  no interior markings is regular and
  the evaluation map
  $\ev: {\cM}_{0,1}(\tilde{U}, L_{\bm \epsilon}, \tilde{J}_0) \to L_{\bm \epsilon}$ is a
  submersion.
    
    \item All nonconstant $\tilde{J}_0$-holomorphic spheres in $\tilde{U}$ have positive Chern numbers and are contained in the exceptional divisor $\tilde{Z}$. Moreover, the moduli space of these spheres with one marking is regular (as maps into $\tilde{Z}$) and the evaluation map at the marking is a submersion onto $\tilde{Z}$.
\end{enumerate}
\end{proposition}

\begin{proof} The first item follows from Lemma \ref{index}. 
For the second item, note that the torus action on $\tilde{U}$
induces an action on the moduli space of holomorphic disks bounding 
$L_{\bm \epsilon}$.  It follows that $D \ev$ is surjective at
any point.  The splitting in Oh \cite{oh:rh} implies that the boundary value problem defined by $u$ splits into one-dimensional summands with non-negative Maslov index.  In particular, the cokernel of $D_u$ vanishes, hence the regularity in the second item.   For
  the third item, note that any holomorphic sphere $u : \P^1 \to \tilde{U}$   defines a holomorphic sphere in $\tilde{Z}$ by projection, necessarily
  of degree $d$, together with a section of the pull-back of the
  normal bundle to $\tilde{Z}$, \label{tiY} necessarily a line bundle of
  degree $-d$.  Since such bundles have no sections, $u$ has image in
  the exceptional divisor $\ti{Z}$.  The claim follows from homogeneity of
  $\tilde{Z}$, and the fact that the Chern number of any degree $d$ map
  to $\tilde{Z}$ is $d(n-1)$.
\end{proof}

One needs the following simple result to help calculate the potential function for the exceptional torus. 

\begin{proposition} \label{alldisks} 
There exists $\eps_0$ such that for any $\eps \in (0, \eps_0] \cap {\mb Q}$, the following holds: For any smooth domain-dependent almost complex structure $J: {\mb D} \to {\mc J}_{\rm tame}(\tilde{X}, \tilde \omega)$ with $J|_{\tilde U} = J_{\tilde U}$, all $J$-holomorphic disks $u: {\mb D} \to \tilde{X}$ bounding $L_{{\bm \epsilon}}$ with energy at most $\eps$ are contained in $\tilde{U}$, and hence are the standard Blaschke products of Maslov index two.
\end{proposition}

\begin{proof}
The statement of the proposition is a consequence of the monotonicity property of pseudoholomorphic curves. Suppose the statement is not the case, so that for all $\eps$ there is a certain domain-dependent almost
  complex structure $J$ and a holomorphic map
  $u: {\mb D} \to \tilde{X}$ with area at most $\eps$ but not
  contained in the neighborhood $\tilde{U}$. Let
  $\tilde{U}''' \subset \tilde{U}'' \subset \tilde{U}' \subset \tilde{U}$ be a
  nested collection of open neighborhoods of the exceptional divisor $\ti{Z}$,
  so that in particular $u(\partial {\mb D}) \subset \tilde{U}'''$. Let
  $S \subset \tilde{U}$ be the closure of
  $u( {\mb D} ) \cap (\tilde{U}' \setminus \tilde{U}'')$, which is a compact
  minimal surface with boundary. The geometry between
  $\tilde{U}'''$ and $\tilde{U}$ is independent of $\eps$.  By the
  monotonicity property of minimal surfaces (see
  \cite[3.15]{Lawson_1974}, \cite[4.7.2]{Sikorav_1994} \cite[Lemma
3.4]{switching}) there is a constant $\delta_0 >0$ which is
  independent of $\eps$ such that for all non-constant compact minimal
  surface $\Sigma$ with nonempty boundary in the interior of
  $\tilde{U} \setminus \tilde{U}'''$ and $\delta < \delta_0$ we have
\[
x \in \Sigma,\ \partial \Sigma \cap B(x, \delta) = \emptyset \Longrightarrow {\rm Area} (\Sigma) \geq c \delta^2.
\]
Applying the monotonicity property to $S$ one sees that the
holomorphic map $u$ has an area lower bound that is independent of
$\eps$, a contradiction.
\end{proof}

\subsubsection{Multivalued perturbations}

\label{multival}
Next we introduce multivalued perturbations that are needed to establish a weak version of the divisor equation for the Fukaya
algebras of the exceptional tori.  

\begin{definition} Given a stable domain type $\Gamma$, a {\em multivalued perturbation} is a formal linear
combination of perturbations
\begin{equation} \label{mv} P_\Gamma = p_1 P_{\Gamma,1} + \ldots + p_k
  P_{\Gamma,k} \end{equation}
for real numbers $p_1,\ldots, p_k > 0$ summing to $1$.
\end{definition} 

Coherent
collections of multivalued perturbation data for all stable domain
types are defined as before.  Given a multivalued perturbation $P_\Gamma$ we write
\[
\ov{\mc M}_{\bGamma}(P_\Gamma):= \bigcup_{i=1}^k \ov{\mc M}_{\bGamma}(P_{\Gamma, i}).
\]
If each subset in the above union is regular, we consider it as weighted manifold with weights given by the coefficients $p_1,\ldots, p_k$.  We call each $\ov{\cM}_{\bGamma}(P_{\Gamma,i})$ a {\em branch} of $\ov{\cM}_{\bGamma}(P_{\Gamma})$. A multivalued perturbation $P_\Gamma = p_1 P_{\Gamma, 1} + \cdots + p_k P_{\Gamma, k}$ is (strongly) regular if each $P_{\Gamma, i}$ is (strongly) regular.  In fact, we only consider multivalued perturbations $P_\Gamma = (J_\Gamma, H_\Gamma, F_\Gamma, M_\Gamma)$ such that $J_\Gamma$, $H_\Gamma$, and $M_\Gamma$ are all single valued, but there
is no advantage in disallowing these components to be multi-valued also. Example \ref{divfail}
explains why multivalued perturbations are needed to prove the divisor equation.

\subsubsection{Perturbations needed for the divisor equation}

In this section, we identify the Floer cohomology rings of the tori near the exceptional locus with Clifford algebras. This requires a special version of the divisor equation (see Corollary \ref{diveqcor}).  
Recall that if the moduli spaces admit forgetful maps for 
omitting a marking and stabilizing if necessary, then the $A_\infty$ composition will satisfy the general divisor equation for any number of boundary insertions. Unfortunately, it is difficult to achieve existence of the forgetful maps using the perturbations
used in this paper.    Rather, we will achieve transversality while having the divisor equation for the $A_\infty$ algebras of the new branes in the blowup with only two insertions.   We first introduce a class of perturbations for which this restricted version of the divisor equation will hold.

\begin{notation}
$\Gamma^*$ is the stable domain type with only the root vertex $v_0$, two incoming unforgettable boundary leaves $e',e''$ and one outgoing boundary leaf $e_0$ (which must also be unforgettable), and $m(\epsilon)$ interior leaves where $m(\epsilon)$ is the expected number of intersections of the basic Maslov 2 disks with the Donaldson hypersurface $\tilde D$. Denote the segments corresponding to the two incoming edges by  $T_{e'}, T_{e''} \subset \ov{\mc U}_{\Gamma^*}$. Each multivalued perturbation $P_{\Gamma^*}$ restricts to two multivalued functions 
\[
F_{e'} = p_1 F_1' + \cdots + p_k F_k': T_{e'} \times L_{\bm \epsilon} \to {\mb R},\ F_{e''} = p_1 F_1'' + \cdots + p_k F_k'': T_{e''} \times L_{\bm \epsilon} \to {\mb R}.
\]
\end{notation}

\begin{definition}  
A perturbation $P_{\Gamma^*}$ is called {\it symmetric} if with respect to the obvious identification $T_{e'} \cong T_{e''}$, as multivalued functions (with weights) one has
$F_{e'} = F_{e''}.$
\end{definition} 

Now consider the situation of the divisor equation. As $L_{{\bm \epsilon}} \cong (S^1)^n$, there exists a perfect Morse function $F_{L_{{\bm \epsilon}}}$ that has exactly $2^n$ critical points. We call such a function a {\it minimal Morse function} on this torus. There are exactly $n$ critical points, denoted by $x_1, \ldots, x_n$ that have Morse index $n-1$. By choosing orientations on their unstable manifolds, $x_1, \ldots, x_n$ give a basis of $H^1(L_{{\bm \epsilon}})$. Also let $x_0$ be the unique critical point of index $n$, whose unstable manifold is oriented in the same way as $L_{{\bm \epsilon}}$. Let $\bGamma_{\tilde \beta, i, j}$ be the map type determined by a basic disk class $\tilde \beta \in H_2(\tilde X, L_{{\bm \epsilon}} )$, incoming critical points $x_i, x_j$ and outgoing critical point $x_0$. Let $\bGamma_{\tilde \beta}$ denote the map type without the incoming edges and only one output labelled by $x_0$.

\begin{lemma} There exists a symmetric multivalued perturbation $P_{\Gamma^*}$ such that the moduli spaces ${\mc M}_{\bGamma_{\tilde \beta, i, j}}(P_{\Gamma^*})$ are regular
for any basic disk  class $\tilde \beta$.
\end{lemma}

\begin{proof}  The proof is an averaging argument.  
Fix such a disk class $\tilde \beta$. Consider the moduli space of $J_{\tilde U}$-holomorphic disks $u: S \to X$ in this class with one boundary marked point $z_e \in S$. The Blaschke formula \eqref{blaschke} implies that this moduli space is a smooth manifold of dimension equal to ${\rm dim}L_{\bm \epsilon}$ and the evaluation map $u \mapsto u(z_e)$ at the boundary marking $z_e$ is a diffeomorphism onto $L_{\bm \epsilon}$. Therefore, the moduli space ${\mc M}_{\bGamma_{\tilde\beta}}$ contains only one configuration (up to permuting interior markings) whose boundary is an embedded circle $\partial \tilde\beta \subset L_{\bm \epsilon}$. Choose two perturbations to the negative gradient flow equation of $f_{L_{\bm \epsilon}}$, gives two perturbations of the unstable manifolds $W^u(x_i)$ for each $i$, denoted by $W^u_{e_1}(x_i)$, $W^u_{e_2}(x_i)$. We may require that $W^u_{e_1}(x_i)$, $W^u_{e_2}(x_j)$ always intersect transversely and intersect transversely with $\partial \tilde\beta$ so that $W^u_{e_1}(x_i) \cap W^u_{e_2}(x_j) \cap \partial \tilde \beta = \emptyset$. Switching the two perturbations does not alter this condition. Regarding the two perturbations as a perturbation of on the two incoming leaves of $\Gamma^*$ and the switching produces a 2-valued perturbation $P_{\Gamma^*}$. 
\end{proof}

\begin{lemma}\label{lemma511} 
The following divisor relation holds:
\begin{equation} \label{twice}
\# {\mc M}_{\bGamma_{\beta, i, j}} (P_\Gamma)  + \# {\mc M}_{\bGamma_{\beta,j,i}} (P_\Gamma) = \langle x_i, \partial \beta\rangle \langle x_j, \partial \beta \rangle  {\mc M}_{\bGamma_\beta} (P_\Gamma)
\end{equation}
Here $\langle x_i, \partial \beta\rangle$ is the intersection number between the unstable manifold $x_i$ and boundary class $\partial \beta \in H_1(L_{{\bm \epsilon}})$. 
\end{lemma}

\begin{proof}
Suppose that perturbations $F_{e_1}, F_{e_2}$ on the incoming edges have been chosen. For any perturbation datum $P_{\Gamma_0}$ we obtain a perturbation datum for $P_\Gamma$ by pull-back of $P_{\Gamma_0}$ everywhere except the edges
$e_1,e_2$ where we take the perturbation to equal $F_{e_1}, F_{e_2}$. 
Notice that the moduli space ${\mc M}_{\bGamma_\beta}(P_{\Gamma_0})$ is always transversely cut out. Any element of the moduli space ${\mc M}_{\bGamma_{\beta, i, j}^*}(P_{\Gamma^*})$ is determined by an element in ${\mc M}_{\bGamma_{\beta}^0}(P_{\Gamma^0})$ obtained by forgetting the edges together with attaching points of the edges $e_1,e_2$ which flow to $x_1,x_2$ under the perturbed gradient flow of $F_{e_1},F_{e_2}$. For any time-dependent perturbation $F_t$ of $F_{L_{\bm \epsilon}}$, the unstable manifold of $x_i$, which is the space of solutions to the equation
\[
\dot{x}(t) + \nabla F_t (x(t)) = 0,\ t \in (-\infty, 0],
\]
is still a cycle and represents the same class in $H_1(L_{\bm \epsilon})$ as the unperturbed unstable manifold. In the case when $i=j$,  it follows that the number of such configurations for any map of 
type $ \bGamma_0(\beta; x_0)$ is 
$ \frac{1}{2} \langle x_i, \partial \beta\rangle \langle x_j, \partial \beta \rangle  \# {\mc M}_{\bGamma_0(\beta; x_0)}(P_\Gamma)$, with the factor of $\frac{1}{2}$ appearing because the attaching points must appear in cyclic order, and the number of attaching points in either order are equal by the symmetric assumption. In the case when $i \neq j$, a map of 
type $ \bGamma_0(\beta; x_0)$ together with the data of attaching points contributes to exactly one of the two terms in the left hand side of \eqref{twice} depending on the cyclic ordering of $z_0$ 
 and the two attaching points. 
\end{proof}

\begin{example}\label{divfail}
  We give an example to show why the divisor equation \eqref{twice}
  does not hold if multivalued perturbations are not allowed. Consider
  $X=S^2$ with $L=S^1$ being the equatorial circle equipped with a
  minimal Morse function $F$. Let $x_0$ resp. $x_1 \in L$ be the
  minimum resp. maximum point of $F$. There are two basic classes
  $\beta^+$, $\beta^-$ in $H_2(X,L)$ corresponding to the upper and
  lower hemisphere of $S^2$, and there is one map each of class
  $\beta^+$, $\beta^-$ that has a single output mapping to $x_0$. We
  now count the elements in the moduli space
  ${\mc M}_{\bGamma_{\beta^\pm,1,1}} (P_\Gamma)$.  Let $F^{e_1}$ and
  $F^{e_2}$ be single-valued perturbations of $F$ defined on the treed
  segments $T_{e_1}$, $T_{e_2}$. The corresponding unstable manifolds
  $p_1:=W^u_{e_1}(x_1)$ and $p_2:=W^u_{e_2}(x_1)$ are distinct points in
  $L$. Depending on the cyclic ordering of $x_0$, $p_1$ and $p_2$,
  exactly one of the moduli spaces
  ${\mc M}_{\bGamma_{\beta^+,1,1}} (P_\Gamma)$,
  ${\mc M}_{\bGamma_{\beta^-,1,1}} (P_\Gamma)$ is non-empty. Thus the
  relation \eqref{twice} does not hold for the classes $\beta^+$,
  $\beta^-$.
\end{example}

Symmetric perturbations may not suffice to regularize other moduli spaces as one needs to break the symmetry in order to regularize. Nonetheless, a perturbation sufficiently close to a symmetric one will not change the divisor relation above. 

\begin{lemma} \label{semiinvlem}
For each sufficiently small $\epsilon$, there exist a coherent 
and strongly regular system of perturbations $\tilde P_\Gamma$ for treed holomorphic disks in $\tilde X$ so that for the basic disk classes, one has the relation \eqref{twice}.
\end{lemma}

\begin{proof} For each $\Gamma$
there are countably many map types
$\bGamma(\beta; x_i, x_j, x_0)$.
Since the countable intersection of comeager
sets is comeager, we may assume that 
$P_\Gamma$ has been chosen so that 
\eqref{twice} holds.  The rest of the construction of coherent perturbations now remains the same.  
\end{proof} 

We summarize all conditions that can be achieved in the following theorem.

\begin{theorem} \label{ereg}
There exists a coherent system of perturbation data $\uds{\tilde P} = (\tilde P_\Gamma)$ for treed disks in $\tilde X$ satisfying the following conditions.
\begin{enumerate}
    \item Each $\tilde P_\Gamma$ is strongly 
    regular (see Definition \ref{stronglyregular}).
    
    \item The system of perturbation data $\uds P = (P_\Gamma)$ obtained from $\uds{\tilde P}$ via the natural projection map is coherent and each $P_\Gamma$ is strongly regular.
    
    \item If $\Gamma$ is the domain type with a single vertex, two unforgettable incoming boundary edges and one outgoing boundary edge, no interior leaves, then \eqref{twice} holds.
    
    \item In particular, if $\tilde P_\Gamma = (\tilde J_\Gamma, \tilde H_\Gamma, \tilde F_\Gamma, \tilde M_\Gamma)$, then $\tilde J_\Gamma$ agrees with $J_{\tilde U}$ in $\tilde U$, $\tilde H_\Gamma$ vanishes identically in $\tilde U$, and $\tilde M_\Gamma$ is the identity in $\tilde U$.
\end{enumerate}
\end{theorem}

\subsection{Point constraint and restricted perturbations}

In this subsection,  we consider the Fukaya category, the quantum cohomology, the open-closed/closed-open maps with a special bulk deformation.   We first address some considerations for transversality of treed disks before we take the blowup. We consider bulk deformations of the form 
\[
{\mf b} = {\mf b}_0 + q^{-\epsilon} p
\]
where $\epsilon>0$ and ${\mf b}_0$ is a bulk deformation whose components are disjoint from $p$.  We first refine the combinatorial structures for domain types.  The datum for a domain type $\Gamma$
of treed disk includes a partition
\[
{\rm Leaf}_\bullet (\Gamma) = {\rm Leaf}_{\bullet, 0} (\Gamma) \sqcup {\rm Leaf}_{\bullet, {\rm ex}}(\Gamma)
\]
where ${\rm Leaf}_{\bullet, 0}(\Gamma)$ labels interior markings mapped into $D$ or components of ${\mf b}_0$, while the set of {\em exceptional leaves} ${\rm Leaf}_{\bullet, {\rm ex}}(\Gamma)$ labels interior markings (called exceptional markings) constrained by $p$. If $\Gamma$ is such a stable domain type, let $\Gamma'$ be the domain type obtained by forgetting ${\rm Leaf}_{\bullet, {\rm ex}}(\Gamma)$ and stabilizing. This operation induces a contraction map 
\[
\ov{\mc U}_\Gamma \to \ov{\mc U}_{\Gamma'}.
\]
We make the following restrictions on perturbations.

\begin{definition} \label{restrictedperturbations}
For each stable domain type $\Gamma$, a perturbation $P_\Gamma = (J_\Gamma, F_\Gamma, H_\Gamma, M_\Gamma)$ is called {\it restricted} if it satisfies the following conditions.
\begin{enumerate}
    \item $J_\Gamma|_U \equiv J_U$, $H_\Gamma|_U \equiv 0$, and $M_\Gamma|_U = {\rm Id}$.
    
    \item Let $\Gamma'$ be the domain type obtained by forgetting ${\rm Leaf}_{\bullet, {\rm ex}}(\Gamma)$. If $\Gamma'$ is not empty, then $P_\Gamma$ is the pullback of a function on $\ov{\mc U}_{\Gamma'}$ under the contraction $\ov{\mc U}_\Gamma \to \ov{\mc U}_{\Gamma'}$.
    
\end{enumerate}
\end{definition}

\begin{definition}
A map type $\bGamma$ is called $p$-uncrowded if on each ghost vertex $v\in {\rm Vert}(\Gamma)$ there is at most one exceptional leaf.
\end{definition}

We also modify the meanings of regular and strongly regular perturbations in Definition \ref{stronglyregular} by requiring the same conditions only for map types that are both uncrowded and $p$-uncrowded.

\begin{proposition}
There exists a coherent system of restricted and strongly regular perturbations $\uds{P} = (P_\Gamma)_\Gamma$.
\end{proposition}

\begin{proof}
The proof of the statement of the Proposition is similar to the proof of Theorem \ref{regular}.  Note that there are no nonconstant holomorphic spheres contained in the open subset $U$ where the almost complex structure
is unperturbed.
\end{proof}

\begin{proposition}\label{nonzero_derivative}
There exists a coherent system of restricted and strongly regular perturbations $\uds{P}$ such that, for each essential map type $\bGamma$ of expected dimension zero or one and for each element of ${\mc M}_{\bGamma}(P_\Gamma)$ represented by a treed disk $(C, u)$ for each exceptional marking $z \in C$, the derivative of $u$ at $z$ is nonzero.
\end{proposition}

\begin{proof}
The vanishing of derivatives at  markings is a phenomenon with  codimension two, hence generically cannot happen in a zero or one-dimensional moduli space.
\end{proof}

\subsection{Pullback perturbations and exceptional regularity}\label{subsection28}

In this subsection, we discuss the transversality issues related to embedding the Fukaya category to a blowup. In order to compare this category with  the Fukaya category before the blowup, we use the pullback perturbations which depends on markings mapped into the pullback of a Donaldson hypersurface $D \subset X$. However, $\pi^{-1}(D)$ is no longer a Donaldson hypersurface.  This requires a modification of the general construction.

\subsubsection{Exceptional regularity}


We first consider treed disks for defining the Fukaya category. Let $\Gamma$ be any stable domain type (without exceptional leaves). As restricted perturbations used downstairs  (see Definition \ref{restrictedperturbations})  are independent from exceptional leaves, they can be pulled back to a perturbation on $\ov{\mc U}_\Gamma$ for treed disks in $\tilde X$. Indeed, as $J_\Gamma$ is the standard almost complex structure $J_U$ in $U$, it lifts to the integrable almost complex structure $J_{\tilde U}$. The Hamiltonian perturbation $H_\Gamma$, the diffeomorphism $M_\Gamma$ both lifts as well. Therefore, each restricted perturbation $P_\Gamma$ corresponds to a pullback perturbation upstairs. We denote the pullback  by $\pi^* P_\Gamma$. 
A map type for treed disks in $\tilde X$ is denoted by $\tilde\bGamma$ and the corresponding moduli space is ${\mc M}_{\tilde\bGamma}(\pi^* P_\Gamma)$.


\label{excreg}

To regularize the moduli spaces of treed maps that have spherical components mapped to the exceptional divisor we require a
different notion of regularity, as the normal direction to the exceptional divisor may bring in obstructions in the usual sense.
We define a subgraph of the type of a map corresponding to components that map into the exceptional locus. 

\begin{definition}  {\rm (Exceptional subtype)}  Let $\Gamma$ be a domain type, $u: C \to \tilde{X}$ be a treed holomorphic disk of type $\Gamma$. 
 Let $\Gamma_{\rm ex}$ be the union of spherical subtrees  $\Gamma''$ of $\Gamma$ whose energy is positive and so that the images of the corresponding sub-curve $C''$ is contained in the exceptional divisor $\tilde{Z}$ (such a subtree may have ghost components).\footnote{A ghost spherical tree $\cup_{v \in V} S_v$ mapped into $\tilde{Z}$ with all neighboring components $S_{v'}$ not mapped into $\tilde{Z}$ is not contained in $\Gamma_{\rm ex}$.}  In general, a treed holomorphic disk $C$ of domain type $\Gamma$ is called a type $(\Gamma, \Gamma_{\rm ex})$-map if $\Gamma_{\rm ex}$ is union of all maximal spherical subtrees $\Gamma''$ with 
 \begin{align*}
 &\ \sum_{v \in {\rm Vert}(\Gamma'')} A(u_v) > 0\ \ \ \  {\rm and}\ \ \bigcup_{v\in {\rm Vert}(\Gamma'')} u_v(S_v) \subset \tilde Z.
 \end{align*}
\end{definition} 

The moduli space of maps to the blowup can be viewed as a fibre product in the following way. We assume for simplicity that the graph $\Gamma_{\rm ex}$ is connected and its complement, denoted by $\Gamma'$, is also connected. Let $\bGamma_{\rm ex}$ and $\bGamma'$ be the obviously induced map types. Let $u_{\rm ex}$ and $u'$ be the restriction of $u$ to these two parts. The perturbation datum $P_\Gamma$ induces a perturbation datum $P_{\Gamma'}$ on $\ov{\mc U}_{\Gamma'}$ and a perturbation datum $P_{\Gamma_{\rm ex}}$ on $\ov{\mc U}_{\Gamma_{\rm ex}}$. The map $u'$ represents an element of ${\mc M}_{\bGamma'}(P_{\Gamma'})$ and $u_{\rm ex}$ represents an element of ${\mc M}_{\bGamma_{\rm ex}}(P_{\Gamma_{\rm ex}})$. The moduli space of type $(\bGamma, \bGamma_{\rm ex})$ treed holomorphic disks (with respect to the perturbation $P_\Gamma$) can be identified with the fibre product
\[
{\mc M}_{\bGamma'}(P_{\Gamma'}) {}_{\rm ev} \times_{\rm ev} {\mc M}_{\bGamma_{\rm ex}}(P_{\Gamma_{\rm ex}})
\]
where the target set of the two evaluation maps is the exceptional divisor $\tilde Z$. 

\begin{definition}  The treed holomorphic disk $u$ is {\it regular as a type $(\bGamma, \bGamma_{\rm ex})$ map} if $u'$ and $u_{\rm ex}$ are both regular and the above fibre product is transverse at $([u'], [u_{\rm ex}])$.
\end{definition} 

In order to obtain corresponding regularity and compactness results, the notion of strong regularity of Definition \ref{stronglyregular} needs the following modification.

\begin{definition}\label{exceptionallyregular1}
Let $\Gamma$ be a stable domain type. A pullback perturbation $P_\Gamma$ (for treed disks in $\tilde X$) is called {\em exceptionally regular} if the following conditions are satisfied: For each subgraph $\Gamma_{\rm ex} \subset \Gamma$ whose vertices are all contained in ${\rm Vert}_\bullet(\Gamma)$, an uncrowded treed holomorphic disk $u: C \to \tilde{X}$ of domain type $(\Gamma, \Gamma_{\rm ex})$ is regular as a map of type $(\Gamma, \Gamma_{\rm ex})$.  
\end{definition}

\begin{proposition}\label{blowup_regular}
There exists a coherent system of perturbations $\uds P = (P_\Gamma)_\Gamma$ for treed disks in $X$ satisfying the following conditions.
\begin{enumerate}
    \item Each $P_\Gamma$ is strongly regular (Definition \ref{stronglyregular}.) 
        \item The lifted perturbation $\pi^* P_\Gamma$ is strongly regular for curves  in $\tilde X$ having no components mapped into $\tilde Z$.
        \item The lifted perturbation $\pi^* P_\Gamma$ is exceptionally regular.
\end{enumerate}
\end{proposition}

\begin{proof}  The proof is similar to that of Theorem \ref{ereg} and omitted.
\end{proof} 

\begin{remark}\label{exceptionallyregular2}
Exceptionally regularity implies regularity for the following maps obtained by the forgetful construction. Let $u: C \to \tilde{X}$ be a treed holomorphic disk of type $(\Gamma, \Gamma_{\rm ex})$. Let $C'$ be the (possibly disconnected) treed disk obtained by removing all spherical components $S_v$ labelled by vertices $v$ in $\Gamma_{\rm ex}$, and $u': C' \to \tilde{X}$ the induced map which has no nonconstant sphere components mapped into $\tilde{Z}$.  Equip $C'$ with new markings corresponding to nodes connecting $C'$ to its complement $C - C'$. 
Let $\Gamma'$ be the domain type (possibly disconnected) corresponding to $C'$.  (See Figure \ref{forget_exceptional}.) By the locality property of the perturbation data (see Definition \ref{locality}),
  $P_\Gamma$ induces a perturbation datum $P_{\Gamma'}$ so that $u'$ is $P_{\Gamma'}$-holomorphic. The moduli space $\M_{\bGamma,\bGamma_{\rm ex}}(\tilde{X})$ is then
  the fiber product of $\M_{\bGamma'}(X)$ with $\M_{\bGamma_{\rm ex}}(E)$,
  over some number $I$ of copies of $\tilde{Z}$ corresponding to edges
  connection $\Gamma'$ and $\Gamma_{\rm ex}$.  Since nonconstant spheres in
  $\ti{Z} \cong \mb{CP}^{n-1}\subset {\mc O}(-1)$ have obstructions to be
  deformed out of $\tilde{Z}$, the transversality at nodes connecting
  components in $\Gamma_{\rm ex}$ and not in $\Gamma_{\rm ex}$ implies the evaluation map at the
  new markings from the moduli space of $P_{\Gamma'}$-holomorphic
  treed disks is transversal to $(\ti{Z})^l$ at the point represented by
  $u'$.  In particular, the curves in $\M_{\Gamma'}(X)$ are regular.
 \end{remark}

\begin{figure}[ht]
    \centering
    \includegraphics[width=4in]{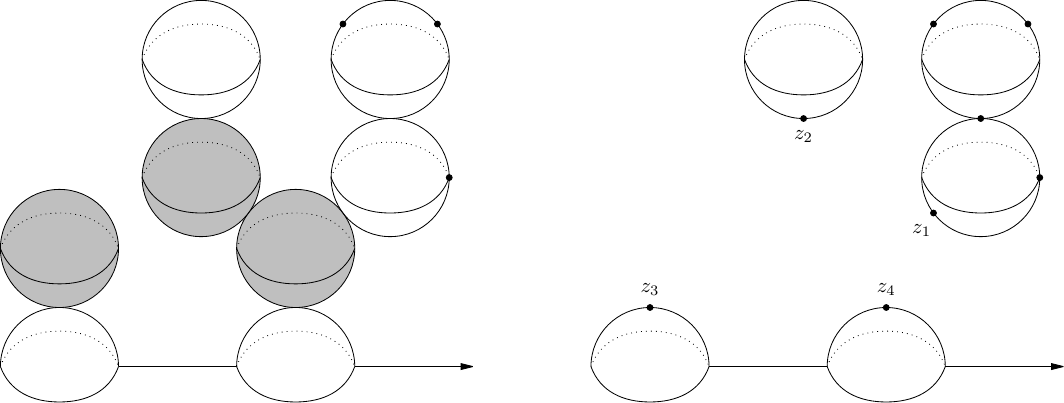}
    \caption{Forgetting sphere components mapped to the exceptional divisor. The gray spheres are (possibly constant) holomorphic spheres in the exceptional divisor. The markings supposed to be mapped to the Donaldson hypersurfaces are not drawn. The exceptional regularity requires regularity of the configuration on the right and the transversality to the exceptional divisor at the markings $z_1, z_2, z_3, z_4$.}
    \label{forget_exceptional}
\end{figure}

\subsubsection{Refined compactness for blowup}

Finally, we show that a version of Gromov compactness holds for the perturbations constructed as above, with complex structure standard in a neighborhood of the exceptional locus.  We first redefine the notion of {\it essential map types} in the blowup case (see Definition \ref{essentialtype}), since the pullback of the stabilizing divisor $\tilde D \subset X$ in $\ti{X}$ is no longer a Donaldson hypersurface, but rather represents the cohomology class $\pi^* [\omega] \in H^2(\ti{X})$.

\begin{definition}
A map type $\bGamma$ of treed holomorphic disks in $\tilde X$ is called {\it essential} if the edges $e \in \Edge({\bGamma})$ have no breakings, there are no edges $e$ of length $\ell(e)$ zero or infinity, no spherical vertices $v \in \Ver_\black({\bGamma})$,  all interior constraints on edges $e \in \Edge_\black({\bGamma})$ are either $(\tilde D, 1)$ or $\ti{\bb}$ and the following holds for each disk vertex $v\in {\rm Vert}({\bGamma}_\circ)$, if the homology class of the component $S_v$ is $\tilde \beta_v$, then the number of interior leaves meeting $v$ labelled by $(\tilde D, 1)$ is equal to $k \omega( \beta_v)$, where $\beta_v$ is the pushforward of $\tilde \beta_v$ to $X$.
\end{definition}

\begin{proposition}[Improved compactness]
\label{icompact}
For a coherent collection $ (\ti{P}_{\Gamma})$ of exceptionally regular perturbations, sequential compactness for moduli spaces
 $\tilde{\M}_{\bGamma}(\ul{L},\tilde{D}(\LL))$
of essential types of expected dimension at most one (exactly the same statement as Lemma \ref{refinedcompactness}) holds. In particular, the limit of a convergent sequence $u_\nu$ of elements $u$ 
in a moduli space 
$\tilde{\M}_{\bGamma}(\ul{L},\tilde{D}(\LL))$
 of essential map type of expected dimension at most one has no component mapped into the exceptional divisor $\ti{Z}$.
\end{proposition}

\begin{proof}
We extend the proof of Lemma \ref{refinedcompactness} to the case of sphere bubbling in the exceptional divisor, which is ruled out by an index argument. Consider an essential map type $\bGamma$  with index at most one and consider a sequence of treed holomorphic disks $u_i: C_i \to \tilde{X}$ representing a sequence of points in ${\cM}_{\bGamma}(\tilde P_\Gamma)$. 
By  general compactness results,  a subsequence converges to a limiting treed holomorphic disk $u: C \to \tilde{X}$ of some type $\bGamma'$. 
Indeed, Gromov compactness for Hamiltonian-perturbed pseudoholomorphic maps with Lagrangian boundary conditions shows that the maps on each surface component have stable limits, after passing to a subsequence, 
and convergence on the tree parts follows from compactness of 
the manifold.  To see that the limit has essential type, first one can as in the proof of Lemma \ref{refinedcompactness} (also the argument of Cieliebak--Mohnke \cite{Cieliebak_Mohnke}) remove crowded ghost components $u_v: S_v \to \tilde{X}$, and so assume that the map type $\bGamma'$ of $u$ is uncrowded. Second, if one can rule out the possibility of a non-constant sphere $u_v: S_v \to \ti{X}$ mapped into the exceptional divisor $\ti{Z}$, then the theorem follows from the same argument of the proof of Lemma \ref{refinedcompactness} as the exceptional regularity agrees with the regularity.

Suppose on the contrary that there are non-constant spherical components of $u$ mapped into the exceptional divisor. We derive a contradiction using the two types of regularity conditions of Definition \ref{exceptionallyregular1} and Remark  \ref{exceptionallyregular2}. Let $\Pi$ be the domain type of $u$. Consider maximal sphere bubble trees $\Pi_{\rm ex}$ in $\Pi$ whose energy is
positive and so that the corresponding maps $u': C' \to \ti{X}$ have images in $\tilde{Z}$. Let $\Pi'$ be the domain type obtained from $\Pi$ by removing $\Pi_{\rm ex}$. Suppose $\Pi_{\rm ex}$ has $m$ connected components $\Pi_{{\rm ex},1}, \ldots, \Pi_{{\rm ex},m}$ with positive degrees $d_1, \ldots, d_m$. Suppose $\Pi'$ have $k+1$ connected components $\Pi_0', \ldots, \Pi_k'$ where $\Pi_0'$ has boundary and $\Pi_1', \ldots, \Pi_k'$ are spherical trees. Suppose the homology class of $\Pi_i'$ is $\beta_i$ and $\Pi_i'$ has $l_i$ new markings. Suppose the component $\Pi_0'$ has map type ${\bPi}_0$ and $l_0$ new markings.  Since each removed node is replaced by 
two new markings, we have the equality
\[
l:= l_0 + l_1 + \cdots + l_k = k + m.
\]
To simplify the computation of the indices, without loss of generality, assume all the spherical trees $\Pi_{{\rm ex}, i}$ or $\Pi_j'$ ($j\geq 1$) have single vertices and the disk components have no bubbling of disks or breaking of edges; otherwise, the index
of the strata would be even lower.  The index of $u$ as a $(\Pi, \Pi_{\rm ex})$-type map (see Definition \ref{exceptionallyregular1}) is (here $2n$ is the dimension of $X$ and $d = d_1 + \cdots + d_m$ is the total degree of spheres in the exceptional divisor).
\begin{multline*}
\sum_{i=1}^k \underbrace{2n + m(\beta_i) + 2l_i - 6}_{{\rm index\ of\ } \Pi_i'} + \underbrace{{\rm ind}({\bPi}_0) + 2l_0}_{{\rm index\ of\ } \Pi_0'} + \underbrace{(2n-2) m + 2nd + 2l - 6m}_{{\rm index\ of\ } \Pi_{\rm ex}} - \underbrace{2nl}_{\rm matching\ condition}\\
  = (2n-6)k - (2n-4)l + (2n-8)m + {\rm ind}({\bGamma}) + 2d\\
  = 2d - 2k - 4m + {\rm ind}({\bGamma}) \geq 0 \end{multline*}
Hence
\begin{equation}\label{degreebound}
d\geq k + 2m.
\end{equation}
On the other hand, consider the induced object of type $\Gamma'$. The index is 
\begin{multline*}
  \sum_{i=1}^k \underbrace{2n + m(\beta_i) + 2l_i - 6}_{{\rm index\ of\ } \Pi_i'} + \underbrace{{\rm ind}({\bPi}_0) + 2l_0}_{{\rm index\ of\ } \Pi_0'} - \underbrace{2l}_{\rm constraints\ at\ new\ markings}\\
  =  (2n-6)k + {\rm ind}({\bGamma}) - 2(n-1)d \geq 0.
\end{multline*}
Hence 
\[
d \leq k - \frac{2k}{n-1}.
\]
This contradicts \eqref{degreebound}. Hence in the limit there cannot be any non-constant spherical components mapped into the exceptional divisor. On the other hand, using ordinary (but not the exceptional) regularity one can also see that there is no constant spherical component that is mapped into the exceptional divisor. 
\end{proof}

\subsubsection{Refined compactness for treed disks with point constraints}

The exceptional regularity achieved upstairs also implies the refined compactness for curves downstairs. Let $\bGamma$ be an essential map type of expected dimension $0$ or $1$ downstairs with $l$ exceptional markings. 

\begin{lemma} For  generic perturbations $P_\Gamma$, we can achieve an additional regularity condition: for each representative $(C, u)$ of points in ${\mc M}_{\bGamma}(P_\Gamma)$,
the fiber $u^{-1}(p)$ is  the set of the $l$ exceptional markings. 
\end{lemma} 

\begin{proof}
 Indeed, each additional point mapped to $p$ cuts down the dimension by $2n-2\geq 2$. Moreover, by Proposition \ref{nonzero_derivative}, the derivatives at the exceptional markings are nonzero. Hence $(C, u)$ lifts to a treed disk $(C, \tilde u)$ that intersects $\tilde Z$ at the positions of the exceptional markings. 
 \end{proof}
 
\begin{lemma}\label{equal_index}
Given a map ${\tilde u}: C \to {\tilde X}$, let $u = \pi \circ \tilde u$ denote the projection
with exceptional markings at $u^{-1}(p)$.  
The indices of $\bGamma$ of $u$ and $\tilde \bGamma$ of $\tilde u$ coincide. 
\end{lemma}

\begin{proof}
Let $n$ be the complex dimension of $X$. By the relation between canonical classes before and after the blowup, one has 
\[
c_1(\tilde X) = \pi^* c_1(X) - (n-1) {\rm PD}([\tilde Z]).
\]
where ${\rm PD}$ denotes the Poincar\'e dual. 
The lemma follows from Riemann-Roch.
\end{proof}

\begin{proposition}
Let $\bGamma$ be an essential map type of expected dimension $0$ resp. $1$. Then ${\mc M}_{\bGamma}(P_\Gamma)$ is compact resp. compact up to at most one of the degenerations listed in Lemma \ref{refinedcompactness}.
\end{proposition}

\begin{proof}
Suppose $(C_i, u_i)$ be a sequence of treed disks of map type $\bGamma$ and $(C_i, \tilde u_i)$ be the sequence of lifts. Suppose their map types are $\tilde\bGamma_i$. By the relation between symplectic forms (see \eqref{area_difference}), the topological energy of the types $\tilde\bGamma_i$ is uniformly bounded. Hence by Gromov compactness, we may assume that $\tilde\bGamma_i$ are all identical to a map type $\tilde\bGamma$. By Lemma \ref{equal_index}, the expected dimension of ${\mc M}_{\tilde\bGamma}(P_\Gamma)$ is either zero or one. The exceptional regularity of the perturbation implies that this moduli space is compact up to at most one codimension one degenerations listed in Lemma \ref{refinedcompactness}. In particular, the limiting configuration contains no spherical components $S_v, v \in \Ver_\black(\Gamma)$. 
\end{proof}

\subsection{The Fukaya category and open-closed/closed-open maps}

We describe how to construct the Fukaya category for blow-ups in 
the case of a point bulk deformation with small negative $q$-valuation.

\subsubsection{Insertions with negative $q$-valuations and convergence}

Consider a bulk deformation in $X$ of the form 
\[
{\mf b} + q^{-\epsilon} p
\]
We will define a curved $A_\infty$ category which is formally 
\[
\Fuk_{\mc L}^\sim (X, {\mf b} + q^{-\epsilon} p ).
\]
Notice that its definition does not automatically follow from the general case because of the negative exponent.

\begin{lemma}\label{lemma312}
There exists $\epsilon_0>0$ such that for any Riemann surface $\Sigma$ with boundary $\partial \Sigma$, any domain-dependent almost complex structure $J: \Sigma \to {\mc J}_{\rm tame}(X, \omega)$ whose restriction to $U$ is $J_U$, for any $J$-holomorphic curve $u: \Sigma \to X$ with $u(\partial \Sigma) \cap U = \emptyset$, we have the energy bound
\[
E(u)\geq 2\epsilon_0 \# u^{-1}(p).
\]
\end{lemma}

\begin{proof}
This follows from the generalization of Gromov's monotonicity result for $J$-holomorphic curves to the case with multiplicities (see \cite[Theorem 12]{Bao_2016}; the result can also be derived from \cite{Fish_2011}). 
\end{proof}

\begin{corollary}\label{corollary_fukaya_blowup}
There exists $\epsilon_0>0$ such that for all $\epsilon \in (0, \epsilon_0]$, the Fukaya category $\Fuk_{\mc L}^\sim (X, {\mf b}_0 + q^{-\epsilon} q)$, the quantum cohomology ring $QH^\bullet( X, \bb_0 + q^{-\epsilon} p)$, the open-closed map $[OC(\bb_0 + q^{-\epsilon} p)]$, and the closed-open map $[CO(\bb_0 + q^{-\epsilon} p)]$ are all well-defined.
\end{corollary}

\begin{proof}
We only prove for the case of $\Fuk_{\mc L}^\sim(X, \bb_0 + q^{-\epsilon} p)$; other cases are similar. Indeed, it suffices to show that the sum \eqref{composition} gives a well-defined element of $CF^\bullet (L_0, L_d)$. As the moduli space for each individual essential map type of expected dimension zero is compact, we need to show that when $\epsilon$ is small, for each $a>0$, there are only finitely many nonempty moduli spaces 
\[
{\mc M}_{\bGamma}(P_\Gamma) \subset {\mc M}(\uds x)_0
\]
that contribute to $m_d(a_1, \ldots, a_d)$ and that satisfy
\[
A( \bGamma) - \epsilon \# {\rm Leaf}_{\rm ex}(\Gamma) < a.
\]
Indeed, if $\epsilon$ is smaller than the $\epsilon_0$ of Lemma \ref{lemma312}, then the number of exceptional markings  is bounded in terms of the area.  By Gromov compactness, there can only be finitely many such map types. 
\end{proof}

\subsubsection{Categories of old branes in the blowup} 

We wish to identify the Fukaya category of the old branes with a subscategory of the Fukaya category of the blow-up.  We must
deal with the issue that the pullback hypersurface
$ \pi^{-1}(D) \subset \tilde X$ is not a Donaldson hypersurface in the blowup $\tilde X$, as its Poincar\'e dual is $k[\pi^* \omega]$. Therefore, below any given energy level, there might be infinitely many essential map types with unbounded numbers of interior markings contributing to the definition of the composition maps. 
To show that the structural maps (higher compositions, chain-level open/closed and closed/open maps) are defined, we need to show the
following:

\begin{lemma} Given any energy bound $E$ and constraints at semi-infinite edges/leaves, there are at most finitely many essential map types 
$\bGamma$ below the energy bound that have a nonempty moduli space.
\end{lemma} 

\begin{proof}  Given a non-empty moduli space ${\mc M}_{\tilde\bGamma}(P_\Gamma)$ and a point in it, after forgetting interior {gradient} leaves, a representative $(C, \tilde u)$ projects to a treed disk $(C, u)$ in $X$. The relation between symplectic classes (see \eqref{area_difference}) implies that 
\[
E(\tilde u) = E(u) - \epsilon \langle \tilde u, [\tilde Z]\rangle = E(u) - \epsilon \# u^{-1}(p).
\]
When $\epsilon$ is smaller than the $\epsilon_0$ of Lemma \ref{lemma312}, the energy bound upstairs implies a uniform bound on $\# u^{-1}(p)$. Therefore, $E(u)$, which is also the intersection number between $\tilde u$ and $\pi^{-1}(D)$, is uniformly bounded. It follows that there can be at most finitely many domain types supporting such map types with nonempty moduli spaces given an energy bound.
\end{proof} 

Gromov compactness then implies the finiteness of contributing moduli spaces and hence finite counts defining the coefficients of the structure maps.

\subsubsection{Homotopy invariance of the category of old branes}

In this section, we sketch the comparison between two constructions of the Fukaya category of old branes in the blowup, the general version provided in Section \ref{section2} and Section \ref{section3} and the special version using pullback perturbations. We use the strategy of Appendix \ref{sec:app}, although in the latter case the divisor used is not a Donaldson hypersurface. For simplicity, we assume that the bulk deformation $\tilde {\mf b}$ in the blowup is trivial. 

We first specify different types of domains. The domains used in the pullback construction are called {\it $\pi^{-1}(D)$-stabilized domains}. On the other hand, let $\tilde D \subset \tilde X$ be a Donaldson hypersurface with respect to the blowup symplectic structure $\tilde \omega$ and $\tilde J_0$ be an $\tilde \omega$-tamed almost complex structure such that $(\tilde D, \tilde J_0)$ satisfies conditions of Proposition \ref{new_divisor}. By following the general construction of Section \ref{section2} and Section \ref{section3}, we have a different version of Fukaya category for branes in ${\mc L}$. The domains used in this case are called {\it $\tilde D$-stabilized domains}. 

Now introduce treed disks with two types of interior markings to incorporate two stabilizing divisors. A $(\pi^{-1}(D), \tilde D)$-stabilized treed disk (or bi-stabilized treed disk) is a treed disk with a partition of the set of interior markings into two groups. By forgetting one group of markings and stabilizing one can obtain from a bi-stabilized treed disk either a $\pi^{-1}(D)$-stabilized treed disk or a $\tilde D$-stabilized treed disk. A perturbation used for $\pi^{-1}(D)$-stabilized treed disk resp. $\tilde D$-stabilized treed disk can be pulled back to a bi-stabilized treed disk. The pullback perturbation on longer satisfies the locality property in the sense of Definition \ref{locality}; but it is still {\it partly local} in the sense of Definition \ref{plocal}. Moreover, the composition maps defined using bi-stabilized treed disks with either one of the pullback perturbations agree with the composition maps obtained using just one type of markings.  Counts of quilted treed disks using generic homotopy between these two system of partly local pullback perturbations defines a homotopy equivalence of $A_\infty$ categories. Notice that we still need to use the argument of exceptional regularity (see Definition \ref{exceptionallyregular1}) to obtain the refined compactness result for zero or one-dimensional moduli spaces. We summarize the conclusion here.

\begin{theorem}
The Fukaya category $\Fuk_{\pi^{-1}({\mc L})}^\sim(\tilde X, \bb)$ defined using pullback restricted perturbations from $X$ is $A_\infty$ homotopy equivalent to the Fukaya category defined using a Donaldson hypersurface in $\tilde X$.
\end{theorem}

\section{Proof of the main theorem}\label{section6}

\subsection{Embedding of the downstairs Fukaya category}\label{embed}

In this section we prove the main Theorem \ref{gen} following the strategy sketched in the introduction.
We first prove Theorem \ref{oldembed}. Recall that $\tilde{X}$ is an $\eps$-blowup of $X$ at a point $p \in X$ and $\E \subset \tilde X$ is the exceptional divisor produced by the blowup. The basic ingredient of the proof is a correspondence between treed disks in $\tilde X$ and treed disks in $X$  defined as follows. Given a map $\tilde u : C \to \tilde X$, we obtain $u:C \to X$ 
by composing  $\tilde u$ with the projection map $\pi : \tilde X \to X$. 
For the map $u$, the points in $u^{-1}(p)$ are designated as
exceptional markings.

We introduce the following notation for moduli spaces with insertions at the exceptional locus or blowup point. Let $\bb$ be a bulk deformation in $X$ disjoint from $p$ and $\tilde{\bb} = \pi^{-1}(\bb)$ its preimage in $\tilde{X}$.  Let 
\begin{itemize} 
\item $\tilde{\bGamma}$ be an essential map type (see Definition \ref{essentialtype}) in $\tilde{X}$ with boundary conditions from the collection $\tilde {\mc L}$. For each vertex $v\in {\rm Vert}(\Gamma)$ (which must be a disk by the definition of essential map type), let $\tilde\beta_v$ be the labelling homology class. Then $\tilde\beta_v$ has a well-defined intersection number $d_v$ with the exceptional divisor $\tilde Z$;  
\item $\bGamma'$ be the map type in $X$ obtained from $\tilde{\bGamma}$ by replacing the decorations
  $\tilde \beta_v \in H_2( \tilde X, |\tilde {\mc L}|)$ with their projections $\beta_v\in H_2( X, |{\mc L}|)$  and adding to each vertex $v$ a set of $d_v$ exceptional leaves $\Leaf_{\rm ex}(\Gamma)$ (to be mapped to $p$). Notice that $\tilde\bGamma$ uniquely determines $\bGamma'$. 
\end{itemize}
Let $P_{\Gamma'}$ belong to the coherent system of perturbations (for treed disks in $X$) chosen in Section \ref{section2}. Remember that because $P_{\Gamma'}$ does not depend on the positions of the exceptional leaves, it lifts to a perturbation for treed disks in $X$ of domain type $\Gamma$, denoted by $P_\Gamma$.  The moduli spaces
are denoted ${\mc M}_{\tilde{\bGamma}}(P_\Gamma)$ upstairs and ${\mc M}_{\bGamma'}(P_{\Gamma'})$ downstairs.  
For each vertex $v$, let $d_{\rm ex}(\tilde\beta_v)$
denote the pairing of the homology class $\tilde \beta_v$
with the class of the exceptional divisor. 

\begin{theorem} \label{bijection} 
Suppose $\tilde \bGamma$ is an essential map type of expected dimension zero.  Composition with projection from $\ti{X}$ to $X$
induces a surjection
\begin{equation}\label{curve_lifting}
{\mc M}_{\bGamma'}(P_{\Gamma'}) \to \bigcup_{\tilde\bGamma \mapsto \bGamma'} {\mc M}_{\tilde\bGamma}(P_\Gamma).
\end{equation}
An element in $ {\mc M}_{\tilde\bGamma}(P_\Gamma)$ has 
\[
d_{\rm ex} (\tilde\bGamma):= \Big( \sum_{v \in {\rm Vert}(\Gamma)} d_{\rm ex}(\tilde\beta_v)\Big) !
\]
number of pre-images under \eqref{curve_lifting} that differ from each other in the ordering of the exceptional leaves.
\end{theorem} 

\begin{proof} 
By the definition of essential map types (see Definition \ref{essentialtype}), if $\tilde \bGamma$ is essential, so is $\bGamma'$. Moreover, by Lemma \ref{equal_index}, if $\tilde \bGamma$ has index zero, so does $\bGamma'$. Choose a point in ${\mc M}_{\bGamma'}(P_{\Gamma'})$ represented by a treed disks $u: C \to X$. Since $J_\Gamma$ coincides with $J_U$ inside $U$, $u$ lifts to a map
\[
\tilde u: C \setminus \pi^{-1}(p) \to \tilde X
\]
which projects down to $u$. Moreover, by the requirement on the perturbation, at each point of $u^{-1}(p)$, to derivative of $u$ is nonzero. Then $\tilde u$ extends continuously, hence smoothly, to a treed disk $\tilde u: C \to \tilde X$. If we remove the exceptional markings, then $\tilde u$ has an essential map type $\tilde\bGamma$ which descends to $\bGamma$. Moreover, as the perturbation $P_\Gamma$ does not depend on the positions of the exceptional markings, $\tilde u$ indeed represents an element of ${\mc M}_{\tilde\bGamma}(P_{\Gamma})$. Hence the map \eqref{curve_lifting} is defined. 

\label{bij}
Now we prove that the map \eqref{curve_lifting} is surjective and has the expected degree. Fix such a map type $\tilde\bGamma$ and let $\tilde u: C \to \tilde X$ represent an arbitrary point of ${\mc M}_{\tilde\bGamma}(P_{\Gamma})$. By the transversality condition in Definition 
\ref{exceptionallyregular1}, see also Remark \ref{exceptionallyregular2},  the curve $\tilde u$ intersects $\tilde Z$ transversely. Moreover, as $\tilde Z$ is almost complex, each intersection point contributes $1$ to the intersection number. Hence on each disk component $S_v \subset C$, $\tilde u$ intersects $\tilde Z$ at exactly $d_{\rm ex}(\tilde\beta_v)$ points. Therefore, the point represented by $\tilde u$ is in the image of \eqref{curve_lifting} where the intersection points with $\tilde Z$ are at the positions of the original exceptional markings. Furthermore, as there are $d_{\rm ex}(\tilde\bGamma)!$ many ways to label the exceptional markings, each point in ${\mc M}_{\tilde\bGamma}(P_{\Gamma})$ has exactly $d_{\rm ex}(\tilde \bGamma)!$ preimages. By comparing the orientations, this number is indeed the degree. 
\end{proof}

\begin{proof}[Proof of Theorem \ref{oldembed} from the Introduction] \label{oldembedproof}
The map \eqref{curve_lifting} constructed in Theorem \ref{bijection} preserves orientations $o(u)$, number of interior leaves $d_{\black}$, and (after the adjustment by $q^{-\eps}$ in the bulk insertion $p$) symplectic areas in the sense that
\[
A(\tilde{u}) = A(u) - \eps ( [\tilde{u}]. [\tilde{Z}] ).
\]
Indeed, any pseudoholomorphic curve in $\tilde{X}$ projects to a curve in $X$, with intersections $\tilde{u}^{-1}(\tilde{Z})$ with the exception locus $\tilde{Z}$ mapping to intersections $u^{-1}(p)$ with the blowup point $p$.

Regarding orientations, after capping off the the strip like ends as in \cite{Wehrheim_Woodward_orientation} we may assume that the boundary condition is given by a single totally real subbundle $(\partial u)^* TL$. Any deformation of the Lagrangian $(\partial u)^* TL$ to a trivial one for $\phi$ induces a similar isotopy for $\tilde{\phi}$. The pullback $\tilde{u}^* T\tilde{X}$ of the tangent bundle of $\tilde{X}$ around an intersection with the exceptional divisor $\tilde{Z}$ has a natural trivialization away from $\tilde{u}^{-1}(\tilde{Z})$. The projection $\pi$ naturally identifies sections of $\tilde{u}^{-1}(\tilde{Z})$ locally with sections of the $u^* TX$ vanishing at $0$. The orientations on moduli spaces of disks constructed in \cite{fooo} are defined by pinching off sphere bubbles on which the linearized operator has a complex kernel and cokernel, preserving the complex structure. It follows that the induced orientations on the determinant lines for $u$ and $\tilde{u}$ are equal. 
\end{proof}

\subsection{Open-closed maps from old branes}

Recall from Section \ref{quantumcohomology} that the quantum cohomology is defined, as a vector space, as the Morse homology of a Morse-Smale pair.  We choose the Morse-Smale pair $(f_X, h_X)$ on $X$ satisfying the following conditions (recall that $X$ is connected):
\begin{assumption}
\begin{enumerate}
\item $f_X$ has a unique critical point $x_{\max}$ of maximal Morse index and a unique critical point $x_{\min}$ of minimal Morse index.
\item For a critical point $x$ different from $x_{\max}$ resp. $x_{\min}$, $p$ is not contained in the unstable resp. stable manifold of $x$.
\end{enumerate}
In particular, $p$ is not a critical point.
\end{assumption}

On the other hand, the pullback $\pi^* f_X: \tilde{X} \to \R$ is a Morse-Bott function on the blowup $\tilde X$ that requires some perturbation. We choose a Morse-Smale pair $(f_{\tilde X}, h_{\tilde X})$ on $\tilde X$ satisfying the following conditions.

\begin{assumption}
\begin{enumerate}
\item $(f_{\tilde X}, g_{\tilde X})$ agrees with $(\pi^* f_X, \pi^* h_X)$ outside a small neighborhood of $\tilde Z$.
\item For each critical point $x \in {\rm crit} (f_X) \subset {\rm crit} (f_{\tilde X})$ that is not $x_{\max}$ resp. $x_{\min}$, the unstable resp. stable manifold of $x$ of the flow of $-\nabla f_{\tilde X}$ coincides with the unstable resp. stable manifold of $x$ of the flow of $- \nabla f_X$.
\end{enumerate}
\end{assumption}
The natural inclusion ${\rm crit}(f_X) \subset {\rm crit}(f_{\tilde X})$ extends to a linear map $CF^\bullet(X) \to CF^\bullet( \tilde X)$. The above conditions on the Morse-Smale pairs imply that it is a chain map and that the induced map on cohomology agrees with the (injective) pullback $QH^\bullet(X) \to QH^\bullet(\tilde X)$.

\begin{proposition} \label{converges}The following diagram is
  commutative: 
\begin{equation}\label{oc_commutative}
\vcenter{ \xymatrix{  HH_\bullet \big( \Fuk^\flat_{\cL}(X,\bb + q^{-\eps} p ) \big) \ar[rr] \ar[d]_{\pi^*} & & QH^\bullet (X,\bb + q^{-\eps} p )  \ar[d]^{\pi^*} \\
      HH_\bullet \big( \Fuk^\flat_{\tilde {\mc L}} (\tilde{X},\tilde{\bb}) \big) \ar[rr] & & QH^\bullet (\tilde{X},\tilde{\bb}) }}
      \end{equation}
where the horizontal arrows are $[OC (\bb + q^{-\eps} p)]$ resp.
$[OC (\tilde{\bb})]$.  In particular,
\[
{\rm dim}\big( [OC(\tilde \bb)]   ( HH_\bullet( \Fuk^\flat_{\tilde{\mc L}}(\tilde X, \tilde \bb))   ) \big) \geq {\rm dim} \big( {\rm Im}( [OC(\bb + q^{-\epsilon} p)] ) \big).
\]
\end{proposition}

\begin{proof} 
We check that the diagram \eqref{oc_commutative} commutes on the chain level by identifying the moduli spaces involved in the definition. The structure constants of the open-closed maps count treed disks with an interior constraint on an unstable manifold in $X$ resp. $\tilde{X}$. Suppose $x\in {\rm crit} (f_X) \setminus \{x_{\min}\}$.  By Theorem \ref{bijection}, treed disks in $X$ with the outgoing {gradient} leaf labelled by $x$ are in bijection (up to permuting constrained leaves labelled by $p$) with treed disks in $\tilde X$ with output the same constraint, as negative gradient trajectories starting from $x$ do not go near $p$. As in the proof of Theorem \ref{oldembed} on page \pageref{oldembedproof}, the bijection preserves the orientations $o(u)$ and the counting coefficients in defining the open-closed maps. Therefore, the diagram \eqref{oc_commutative} commutes up to multiples of the identities in the quantum cohomology. On the other hand, in the direction spanned by $x_{\min}$ (which is a Morse cocycle in both $X$ and $\tilde X$) the open-closed map always only has classical contributions (see Lemma \ref{oc_unit}). Hence \eqref{oc_commutative} commutes.
\end{proof}

\subsection{Floer cohomology of new branes}\label{subsec:floernew}

In this section, we discuss the Fukaya algebras of branes supported on the exceptional torus in the blowup.  The construction of perturbations relies on choosing a Donaldson hypersurface of the blowup, which is not the pullback of the Donaldson hypersurface $D \subset X$. Nevertheless, there exists a special Donaldson hypersurface $\tilde D \subset \tilde X$ which is holomorphic near the exceptional locus $\tilde Z$. We first recall the computation of the potential function and the Floer cohomology of these branes in \cite{Charest_Woodward_2017}.

\begin{theorem}\cite{flips} \label{thm63}
Let $L_{{\bm \epsilon}}$ be the exceptional Lagrangian, which is monotone in a neighborhood of $\tilde Z$. 
\begin{enumerate}
\item For each local system 
\[
y: H_1( L_{{\bm \epsilon}} ) \cong \mb{C}^n \to \Lambda^\times,\ y = (y_1, \ldots, y_n)
\]
the Fukaya algebra $CF^\bullet((L_{{\bm \epsilon}}, y), (L_{{\bm \epsilon}}, y))$ is weakly unobstructed. 

\item There exists a particular weakly bounding cochain $b_{\rm ex}(y) \in MC(L_{{\bm \epsilon}}, y)$ such that 
\begin{equation}\label{naive}
W(b_{\rm ex}(y)) = q^{\frac{\epsilon}{n-1}} \Big( y_1 +\cdots + y_n + y_1 \cdots y_n + {\rm h.o.t} \Big)
\end{equation}
where $\text{h.o.t.}$ denotes higher order terms measured by $q$-valuation. 

\item There are $n-1$ distinct local systems $y_{(k)}$, $k = 1, \ldots, n-1$, such that for $b_{(k)}:= b_{\rm ex}(y_{(k)})$, one has
\[
HF^\bullet( (L_{{\bm \epsilon}}, y_{(k)}, b_{(k)}), (L_{{\bm \epsilon}}, y_{(k)}, b_{(k)})) \cong H^*( L_{{\bm \epsilon}}, \Lambda).
\]
\end{enumerate}
\end{theorem}

\begin{proof}
The computation of the potential function in \cite{flips} was carried out in the following way. First, by a neck-stretching argument along the hypersurface $\partial \tilde U \cong S^{2n-1}$, the Fukaya algebra (possibly with a bulk deformation supported away from the exceptional divisor $\tilde Z$) of $L_{{\bm \epsilon}}$ with any local system is $A_\infty$ homotopy equivalent to a ``broken Fukaya algebra'' defined by counting holomorphic buildings. The holomorphic buildings contains levels in $X$ and in certain toric pieces. Second, by turning on gradient flows of a Morse function $H$ on $\partial \tilde U/ S^1\cong \mb{P}^{n-1}$, the broken Fukaya algebra is $A_\infty$ homotopy equivalent to another $A_\infty$ algebra defined by counting holomorphic buildings whose levels are separated by Morse gradient lines of $H$ of any fixed length $\tau$. Third, while the $A_\infty$ homotopy type of the Fukaya algebra does not depend on $\tau$, when $\tau$ goes to $\infty$, by dimension counting, the holomorphic buildings must be Maslov index two disks in the level containing $L_{{\bm \epsilon}}$. Denote by $m_k^{\tau = \infty}$  the composition maps of the last Fukaya algebra.  The potential in the neck-stretched limit is  
\[
m_0^{\tau  =\infty}(1) = q^{\frac{\epsilon}{n-1}} \left( y_1 + \cdots y_n + y_1 \cdots y_n + {\rm h.o.t} \right) 1_{L_{{\bm \epsilon}}}^\blackt=: W_{\rm ex}(y_1, \cdots, y_n) 1_{L_{{\bm \epsilon}}}^\blackt.
\]
Using the positivity of the toric piece, one can see that $b^{\tau = \infty}:= W_{\rm ex}(y) 1_{L_{{\bm \epsilon}}}^\greyt$ is a weakly bounding cochain. As an $A_\infty$ homotopy equivalence identifies Maurer-Cartan solution spaces and preserves the potential function, the original Fukaya algebra of $(L_{{\bm \epsilon}}, y)$ is weakly unobstructed, with the weakly bounding cochain $b^{\tau = \infty}$ identified with a weakly bounding cochain $b_{\rm ex}(y) \in MC(L_{{\bm \epsilon}}, y)$, at which the potential function has the value $W_{\rm ex}(y)$.

To identify nontrivial Floer cohomologies, consider the leading order term 
\[
W_0 =q^{\frac{\epsilon}{n-1}} ( y_1 + \cdots y_n + y_1 \cdots y_n)
\]
and its critical points. Indeed,
\[
dW_0 = 0 \Longrightarrow  y_1\cdots \widehat{y_i} \cdots y_n = -1
\]
which has $n-1$ solutions $y_{0, (k)}$, $k = 1, \ldots, n-1$ where
\begin{equation}\label{eqn64}
y_{0, (k)} = \left( \exp \left( \frac{ (2k-1) \pi \sqrt{-1}}{n-1}\right), \ldots, \exp \left( \frac{ (2k-1) \pi \sqrt{-1}}{n-1}\right) \right).
\end{equation}
Computing the second-order derivatives shows that the Hessian is non-degenerate at those critical points. The higher order terms in $W_{\rm ex}$ will not change the number of critical points and the non-degeneracy of the Hessian. Let $y_{(1)}, \ldots, y_{(n-1)}$ be the corresponding critical points.  Standard arguments as
in \cite[Theorem 4.10]{FOOO_toric_1} that for these local systems, the Floer cohomology (for the $\tau = \infty$ Fukaya algebra with the weakly bounding cochain) is isomorphic to the ordinary cohomology of $L_{{\bm \epsilon}}$. As $A_\infty$ homotopy equivalence preserves Floer cohomology, the last assertion is proved.
\end{proof}

\begin{definition}\label{exceptional_brane}
The exceptional collection of branes in the blowup $\tilde X$ is  
\[
{\mf E}:= \{ \WB_{(k)} = (L_{\bm \epsilon}, y_{(k)}, b_{(k)})\ |\ k = 1, \ldots, n-1 \}
\]
where $y_{(1)}, \ldots, y_{(n-1)}$ are the critical points of $W_{\rm ex}$ and $b_{(k)}$ are weakly bounding cochains provided above. Notice that the collection also depends on the bulk deformation $\tilde {\mf b}$ in $\tilde X$.
\end{definition}

To compute the ring structure on the Floer cohomologies, we need a version of the {\em divisor equation} as in \cite[Proposition 6.3]{Cho}.

\begin{proposition}\label{diveqcor} 
If the perturbation data for treed disks in $\tilde X$ are chosen such that \eqref{twice} holds, then the following (restricted) {\em divisor equation} holds. For any two Morse cocycle $x_1, x_2$ on $L_{{\bm \epsilon}}$ of degree $1$ (i.e. linear combinations of critical points of Morse indices $1$) and any basic disk class $\beta$
\begin{equation}\label{diveq}
m_{2,\beta}(x_1,x_2) + m_{2,\beta}(x_2,x_1) = \lan [x_1], \partial \beta \ran \lan [x_2] , \partial \beta \ran m_{0,\beta}(1).
\end{equation}
\end{proposition}

\begin{proof}
The statement of the Proposition is a direct consequence of Lemma \ref{lemma511}.
\end{proof}

\begin{proposition} \label{cliffprop}
The branes $\WB_{(k)} = (L_{{\bm \epsilon}}, y_{(k)}, b_{(k)}) \in {\mf E}$, $k = 1,\ldots, n-1$ have distinct values of the potential function and so generate orthogonal summands of the Fukaya category $\Fuk^\flat_{\mf E} (\tilde X, \pi^{-1}(\bb))$. Moreover, each Floer cohomology ring $HF^\bullet (\WB_{(k)}, \WB_{(k)})$ is isomorphic to a Clifford algebra corresponding to a non-degenerate quadratic form whose leading order is the Hessian of $W_0$ at $y_{0, (k)}$ (see \eqref{eqn64}).
\end{proposition} 

\begin{proof} 
Direct calculation shows that the critical values of $W_0$ are all distinct. As the bulk deformation has positive $q$ valuations, the actual potential function $W_{\rm ex}$ is a higher order deformation of $W_0$. So the critical values remain distinct. By definition of the spectral decomposition, $L_{(k)}$ span orthogonal summands in $\Fuk^\flat_{\mf E}(\tilde X, \pi^{-1}({\mf b}))$.

Now we prove the second claim. For each $E>0$ and $x \in CF^\bullet (\LB_{(k)}, \LB_{(k)})$, let $x^{\leq E}$ be the truncation of $x$ at the energy level $E$. Then \eqref{diveq} implies that 
%
%
%
%
 %
%
%
for generators $a_1,a_2$ living over Morse cocycles $x_1, x_2$ of degree $1$, one has
\begin{multline*}
m_2(a_1, a_2)^{\leq \frac{\epsilon}{n-1}} + m_2(a_2, a_1)^{\leq \frac{\epsilon}{n-1}} = \sum_{\beta} \langle x_1, \partial \beta \rangle \langle x_2, \partial \beta \rangle m_0(1)^{\leq \frac{\epsilon}{n-1}} \\
= \partial_{x_1} \partial_{x_2} W_{\rm ex}(y_{(k)})^{\leq \frac{\epsilon}{n-1}}
\end{multline*}
where the summation runs over all the basic disk classes $\beta$.  In this computation, $m_0(1)$ is viewed as a function of the representation $y$ defined by the local system $y$ and taking the second derivative with respect to $y$. By direct calculation, the right hand side is a non-degenerate quadratic form. It follows that $HF^\bullet ( \WB_{(k)}, \WB_{(k)} )$ is a deformation of the Clifford algebra of a non-degenerate quadratic form (i.e. the Hessian of $W_0$ at the $k$-th critical point). Such Clifford algebras are rigid by Lemma \ref{cliff} below so $HF^\bullet( \WB_{(k)}, \WB_{(k)} )$ is itself a Clifford algebra.
\end{proof}

\begin{lemma}\label{cliff} 
Let $A$ be the Clifford algebra of a non-degenerate quadratic form on a vector space $V$ of dimension $n$.  The $\Z_2$-graded $\Ext$ groups\footnote{In Sheridan \cite[Section 6]{Sheridan_2016} the Hochschild cohomology groups are  written in terms of the $\Ext$ groups by recombining the bi-gradings.   However, in the current situation the algebra $A$ is only $\Z_2$-graded 
and we wish to avoid combining the $\Z$-grading and $\Z_2$-grading.}  $\Ext^s(A,A)$ vanish for $s > 0$.  The Hochschild homology $HH_\bullet (A,A)$ is one-dimensional and generated by the class in $HH_0(A,A)$ of $v_1 \ldots v_n \in A$, where $v_1 \ldots v_n$ is an orthogonal basis for $V$. 
\end{lemma} 

\begin{proof} 
(See also \cite[Lemma 3.8.5]{FOOO_asterisque}) the graded Hochschild homology of Clifford algebras is computed in Kassel \cite[Section 6,Proof of Proposition 1]{Kassel_1986}. For the result on 
the $\Ext$ groups see Sheridan \cite[(6.1.6)]{Sheridan_2016}.
\end{proof}

\begin{remark}
By definition, a formal deformation of an algebra $A$ over a field $\Lambda$ of characteristic zero is a $\Lambda[[\hbar]]$-algebra structure over $A[[\hbar]]$ (where $\hbar$ is a formal variable) whose zero order term is the algebra $A$. The Floer cohomology $HF^\bullet(\WB_{(k)}, \WB_{(k)})$ provides a ``first-order'' deformation of the Clifford algebra associated to the Hessian of $W_0$ at its $k$-th critical point, which can be extended to a formal deformation. 
\end{remark}

\begin{corollary} \label{oned} 
For $k = 1, \ldots, n-1$, the Hochschild homology of the $A_\infty$ algebra $\Hom^\bullet (\WB_{(k)}, \WB_{(k)})$ is   one-dimensional.
\end{corollary} 

\begin{proof}
The Clifford algebra is intrinsically formal by Sheridan \cite[Corollary 6.4]{Sheridan_2016}, meaning that the $A_\infty$ algebra $\Hom^\bullet( \WB_{(k)}, \WB_{(k)})$ is quasi-isomorphic to its cohomology algebra,
which in this case is a Clifford algebra. As quasiisomorphisms of $A_\infty$ algebras admits homotopy inverses (see \cite[Corollary 1.14]{Seidelbook}), the $A_\infty$ algebra $\Hom^\bullet( \WB_{(k)}, \WB_{(k)})$ is $A_\infty$ homotopy equivalent to the cohomology algebra, hence has isomorphic Hochschild homology.
\end{proof}

\subsection{Open-closed map from the new branes}\label{subsec:ocnew}

In this section, we examine the open-closed maps on the collections of branes in the blowup constructed above. For this, we need to specify a Morse function that facilitates the calculation. Let $(z_1, \ldots, z_n)$ be the Darboux coordinates in the neighborhood $U$ of $p$ used for constructing the blowup. Inside the exceptional divisor $\tilde Z \subset \tilde X$ we specify the following cycles
\[\label{zkconds}
\tilde Z_k = \big\{ [z_1, \ldots, z_{k+1}, 0, \ldots, 0]\in \tilde Z \cong \mb{CP}^{n-1} \big\} \cong \mb{CP}^k.
\]
Let $[\tilde Z_k]\in H_{2k}(\tilde X)$ be the homology class.   Then 
\[
\tilde H_\bullet( \tilde Z) = {\rm span} \big\{ [\tilde Z_1], \ldots, [\tilde Z_{n-}]\big\}
\]
where $\tilde H_\bullet(\tilde Z)$ is the reduced homology.

\begin{lemma}
For any $\delta > 0$, there exists a Morse-Smale pair $(f_{\tilde X}, h_{\tilde X})$ on $\tilde X$ satisfying the following conditions.
\begin{enumerate}

\item There is a Morse-Smale pair $(f_X, h_X)$ on $X$ with a unique local maximum at $p$ such that outside a neighborhood $\tilde V$ of $\tilde Z$, $(f_{\tilde X}, h_{\tilde X}) = (\pi^* f_X, \pi^* h_X)$ and $|f_{\tilde X} - \pi^* f_X| < \delta$.

\item ${\rm crit}(f_{\tilde X}) \cap \tilde Z$ is the $n$ toric fixed points of $\mb{CP}^{n-1}$ with Morse indices $2, 4, \ldots, 2n-2, 2n$. We call them the exceptional critical points.

\item Each exceptional critical point is $\delta_{\rm Morse}$-closed but not exact.

\item The gradient vector field of $f_{\tilde X}$ is tangent to $\tilde Z$ and the Hessian of $f_{\tilde X}$ is negative definite on the normal direction of $\tilde Z$. The stable manifolds of the exceptional critical points are $\tilde Z_{n-1} \setminus \tilde Z_{n-2}$, $\ldots$, $\tilde Z_1 \setminus \tilde Z_0$, and $\tilde Z_0$ respectively.

\end{enumerate}
\end{lemma}

\begin{proof}
Choose a Morse-Smale pair $(f_X, h_X)$ on $X$ such that $p$ is the unique local maximum. We may assume that in the Darboux neighborhood $U$ of $p$, 
\[ f_X (z_1, \ldots, z_n) =  - |z_1|^2 - \cdots - |z_n|^2 . \] 
Consider the function 
\[ (z_1, \ldots, z_n) \mapsto \sum_{i=1}^n a_i |z_i|^2 .\] 
For generic real numbers $a_i$ this induces a Morse function $f_{\tilde Z}$ on $\tilde Z = \mb{CP}^{n-1}$ with critical points equal to  the toric fixed points. Regard a neighborhood of $\tilde Z$ as a neighborhood of the zero section in ${\mc O}(-1)$ and denote a normal vector by $\xi$.  Define 
\[
f_{\tilde X} = \delta \rho(|\xi|) f_{\tilde Z} - |\xi|^2.
\]
where $\rho: {\mb R} \to {\mb R}$ is a smooth cut-off function supported near $0$. 
For any $\delta>0$, $f_{\tilde X}$ coincides with $\pi^* f_X$ outside the support of $\rho$, and so defines a function on $\tilde X$. Moreover, when $\delta$ is sufficiently small, the only critical points near $\tilde Z$ are the critical points of $f_{\tilde Z}$.  We take $h_{\tilde X}$ to be $\pi^* h_X$ outside the support of $\rho$, the Fubini-Study metric on ${\mc O}(-1)$ near $\tilde Z$, and a generic interpolation in between. This makes $(f_{\tilde X}, h_{\tilde X})$ a Morse-Smale pair. Moreover, the stable manifolds of these exceptional critical points are contained in $\tilde Z$ as the Hessian in the normal direction is negative definite. 
\end{proof}

We fix the following notations. 
Choose a Morse-Smale pair $(f_{\tilde X}, h_{\tilde X})$ as above to define the open-closed map. 
Let $x^* \in L_{{\bm \epsilon}}$ be the (only) critical point of $F_{L_{{\bm \epsilon}}}$ of Morse index $0$. We consider the unit disk $\DD \subset \C$ equipped with the distinguished points $0 \in \on{int}(\D)$ and $1 \in \partial \D$. 

\begin{lemma} \label{uniquemap} For each $k = 1,\ldots, n-1$ there is a unique map $u_k: \D \to \tilde{X}$ of Maslov index $2(n-k)$ bounding $L_{{\bm \epsilon}}$ and satisfying 
\begin{equation}\label{eqn66}
u_k(0) \in \tilde{Z}_k, \quad u_k(1) = x^*, \quad A(u_k) \leq \frac{(n-k)\eps}{n-1}.
\end{equation}
The map $u_k$ is regular as a map with these constraints (that is, the linearized operator restricted to sections lying in   $T\tilde{Z}_k$ \label{tzk} at $0$ and $1$ is surjective) and there are no other stable disks with these properties.
\end{lemma} 

\begin{proof} 
Let $u$ be a holomorphic disk satisfying \eqref{eqn66}. The requirement that $A(u) = \frac{(n-k)\epsilon}{n-1}$ prevents the map from leaving the toric neighborhood $\tilde U$, by Proposition \ref{alldisks}.  Therefore, we may write 
\[
u = [u_0, \ldots, u_n]
\]
using the homogeneous coordinates viewing the local model as a toric quotient. By the Blaschke classification of holomorphic disks (see \eqref{blaschke}), the condition $u(0) \in \tilde Z_k$ requires that $u_0$ and $u_{k+2}, \ldots, u_n$ have degree at least one and have a common zero at an interior point. Hence the Maslov index of $u$ is at least $2(n-k)$. As $L_{{\bm \epsilon}}$ is monotone in $\tilde U$, the energy of $u$ is at least $\frac{(n-k)\epsilon}{n-1}$. Hence the degrees of $u_1, \ldots, u_{k+1}$ are all zero. The Blaschke classification then implies that there is exactly one such disk satisfying in addition $u(1) = x^*$, up to $PSL(2;{\mb R})$ symmetry. Denote such a map by $u_k$. Its regularity follows from the regularity of Blashcke disks in the toric case (see \cite{Cho_Oh}). The uniqueness of $u_k$ as a stable disk with one output is also obvious.
\end{proof} 

We compute the open-closed map on the exceptional Lagrangian branes. Similar computations will also appear in the example of the Clifford torus in the projective space considered below in Subsection \ref{proj_example}. Recall that the Vandermonde matrix 
\begin{equation}
  \label{eq:vdm}
  T(a_0, \ldots, a_m):=  \left[
    \begin{array}{cccc} 1 & 1 & \cdots & 1 \\
      a_0 & a_1 & \cdots & a_m \\
      \vdots & \vdots & \ddots & \vdots \\
      a_0^m & a_1^m & \cdots & a_m^m
    \end{array}
  \right]    
\end{equation}
has determinant 
\[
\det T(a_1,\ldots, a_m) = \prod_{i< j}(a_j - a_i).
\]
This determinant is non-singular when $a_i \neq a_j$. Denote 
\begin{equation}\label{varsigma}
\varsigma_k:= \exp \left( \frac{(2k-1) \pi \sqrt{-1} }{n-1} \right),\ k = 1, \ldots, n-1
\end{equation}
which are the components of the critical points \eqref{eqn64}. Define an $(n-1)\times (n-1)$ matrix $\on{FFT}_q$ whose $(i, j)$-entry is $(q^\epsilon\varsigma_j)^i$. Namely
\begin{equation}
  \label{fftq}
{\rm FFT}_q = \left[ \begin{array}{cccc} q^\epsilon \varsigma_1 & q^\epsilon \varsigma_2 & \cdots & q^\epsilon \varsigma_{n-1} \\
                           (q^\epsilon \varsigma_1)^2 & (q^\epsilon \varsigma_2)^2 & \cdots & (q^\epsilon \varsigma_{n-1} )^2\\
                            \vdots & \vdots & \ddots & \vdots \\
                           (q^\epsilon \varsigma_1)^{n-1} & (q^\epsilon \varsigma_2)^{n-1} & \cdots & (q^\epsilon \varsigma_{n-1} )^{n-1}  \end{array} \right].
\end{equation}
Its determinant is 
\begin{eqnarray*} 
\det (\on{FFT}_q) &=& q^{(n-1)\epsilon} \varsigma_1 \cdots \varsigma_{n-1} \det T( q^\epsilon \varsigma_1, \ldots, q^\epsilon \varsigma_{n-1}) \\
&=& q^{(n-1)\epsilon} \varsigma_1 \cdots \varsigma_{n-1} \prod_{i<j} ( q^\epsilon \varsigma_j - q^\epsilon \varsigma_i) \neq 0.
\end{eqnarray*}
Hence $\on{FFT}_q$ defines an invertible linear map.  

We will show that the leading order term of the open-closed map is given by such a finite Fourier
transform.   Write $QH^\bullet(\tilde{X},\tilde{\bb})$ as the direct sum (as vector spaces) of the image of $QH^{\bullet}(X,\bb)$ under pull-back and a collection of cycle classes $[\tilde{Z}_1],\ldots, [\tilde{Z}_{n-1}]$, supported on the exceptional divisor $\tilde{Z} \cong \P^{n-1}$ with each $\tilde{Z}_k$ diffeomorphic to a complex projective space $\P^k$. Note that the point class $[\tilde{Z}_0] = [ \pt]$ is not an additional generator. Thus we have a splitting of vector spaces
\begin{equation} \label{vsplitting} QH^\bullet(\tilde{X},\tilde{\bb}) \cong
  QH^{\bullet}(X,\bb) \oplus QH^\bullet(\P^{n-1})/\Lambda [\pt] \cong
  QH^{\bullet}(X,\bb) \oplus \Lambda^{n-1} . \end{equation}
Recall the definition of the exceptional collection ${\mf E}$ from Definition \ref{exceptional_brane}.

\begin{lemma} \label{fft2} There exists $\delta > 0 $ so that for any $\eps > 0 $ sufficiently small, the leading order term in the restriction of the open-closed map $OC( \tilde{\bb}) | HH_\bullet (\Fuk^\flat_{\mf E} (\tilde{X},\tilde{\bb}))$ composed with projection
\[
QH^\bullet(\tilde{X},\tilde{\bb}) \to QH^\bullet(\tilde{X},\tilde{\bb})/\pi^* QH^{\bullet}(X,\bb + q^{-\eps} p) \cong \on{span}([\tilde{Z}_1],\ldots, [\tilde{Z}_{n-1}])
\]
is of the form
\[
OC(\tilde{\bb}) | HH_\bullet(\Fuk^\flat_{\mf E}(\tilde{X},\tilde{\bb}))  ) \ \on{mod} \ QH^{\bullet}(X,\bb + q^{-\eps} p ) = \eps \on{FFT}_q   \ \mod \ q^\delta 
\]
with respect to the bases $\{ L_{(i)} \}, \{ \tilde{Z}_j \}$ where $\on{FFT}_q$ is the matrix \eqref{fftq}.  \label{explicit} As a result, for $\eps$ sufficiently small $OC(\tilde{\bb}) | HH_\bullet(\Fuk^\flat_{\mf E}(\tilde{X},\tilde{\bb})) )$ surjects onto
$QH^\bullet(\tilde{X},\tilde{\bb})/ \pi^* QH^{\bullet}(X,\bb + q^{-\eps} p )$.
\end{lemma}

\begin{proof}
The proof is similar to the proof of surjectivity for the Clifford torus in Theorem \ref{ctorus} below. Via the Blaschke classification \eqref{blaschke}, there is a unique disk of Maslov index $2k$ with an interior point mapping to $\tilde{Z}_k$ and boundary on $L_\epsilon$.  Let
\[ \gamma_1 ,\ldots, \gamma_n \in \pi_1((S^1)^n)  \] 
be the standard set of generators for $\pi_1((S^1)^n)$ and the representation defined by the local system $y\in  H^1(L_{{\bm \epsilon}}, \Lambda)$ is written in coordinates as 
\[
(y_1, \ldots, y_n) = (y(\gamma_1), \ldots, y(\gamma_n)).
\]
By Proposition \ref{alldisks}, there exists a constant $\delta > 0$ independent of $\eps$ so that any disk that leaves the fixed exceptional region $\tilde U$ and bounds $L_{{\bm \epsilon}}$ must have energy greater than $\delta$. Hence, the leading order contributions in the open-closed map $OC(\tilde{\bb})$ come from configurations with no interior insertions labelled by the bulk deformation $\tilde \bb$.
These are holomorphic disks $u: \D \to \tilde X$ with a single point constraint $u(z) = x$ on the boundary $z \in \partial \D$.  It follows that for each brane ${\bm L}_{(k)} = (L_{{\bm \epsilon}}, y_{(k)}, b_{(k)})$, the open-closed map $OC(\tilde{\bb})$ sends the point class $[\pt]_{(k)} \in HF^\bullet ( {\bm L}_{(k)}, {\bm L}_{(k)})$ to
\[
(OC(\tilde{\bb}))([\pt]_{(k)})
\on{mod} \ QH^{\bullet}(X,\bb + q^{-\eps} p )
= (y(\gamma_1), y(\gamma_1 \gamma_2), \ldots, y(\gamma_1 \ldots \gamma_n) ) + \text{h.o.t.} 
\]
similar to the terms in \eqref{oftheform}. Recall that the representation defined by the local system $y_{(k)}$ in the brane ${\bm L}_{(k)}$ is a higher order perturbation of the representation
\[
\varsigma_{(k)} = (\varsigma_k, \ldots, \varsigma_k)
\]
where $\varsigma_k$ is in \eqref{eqn64}. Then we see 
\[
OC(\tilde\bb)([\pt]_{(k)}) = ( q^\epsilon \varsigma_k, \ldots, q^{(n-1)\epsilon} \varsigma_k^{n-1} ) + {\rm h.o.t}.
\]
Therefore, in our preferred basis, $OC(\tilde \bb)$ is the matrix ${\rm FFT}_q$ plus a higher order perturbation, hence is invertible. \label{missingtilde}
\end{proof}

For the exceptional collection one has the following result.

\begin{lemma} \label{disjoint} 
The intersection pairing on the image of $HH_\bullet(\Fuk^\flat_{\mf E}(\tilde{X},\tilde{\bb}))$ is non-degenerate. 
\end{lemma}

\begin{proof} 
The image of $HH_\bullet(\Fuk^\flat_{\mf E}(\tilde{X},\tilde{\bb}))$ is the span of the exceptional cycles, up to higher order corrections. Since $\ti{Z}^n = \lan \ti{Z}, c_1(\mO(-1))^n \ran = (-1)^{n-1}$ in $H(\ti{X})$ is non-zero, the powers of $\ti{Z}$ give a basis for the span of exceptional classes on which the pairing is non-degenerate.
\end{proof}

\subsection{Split-generation for the blowup}\label{subsec:splitgen}

We conclude by proving the main theorem, which now follows from a dimension count.

\begin{proof}[Proof of Theorem \ref{gen}] \label{genproof} 
For sufficiently small $\eps$, $\pi^{-1}({\mc L})$ and $L_{{\bm \epsilon}}$ are disjoint. Hence by Theorem \ref{ortho} the images of
\begin{align*}
&\ HH_\bullet(\Fuk^\flat_{\pi^{-1}(\cL)}(\tilde{X}, \pi^{-1}(\bb))),\ &\ HH_\bullet (\Fuk^\flat_{\mf E} (\tilde{X},\pi^{-1}(\bb)))
\end{align*}
under the open-closed map $OC(\pi^{-1}(\bb))$ are orthogonal with respect to the intersection pairing. 
By Lemma \ref{disjoint}, these two images have trivial intersection. Therefore it suffices to show that their images have complementary dimensions. Indeed, by Corollary \ref{oned},
\[
\dim HH_\bullet (\Fuk^\flat_{\mf E} (\tilde{X},\tilde{\bb}))) = n-1.
\]
By Lemma \ref{fft2}, 
\[
{\rm dim} \left( OC(\pi^{-1}(\bb)) \big( HH_\bullet( \Fuk^\flat_{\mf E}(\tilde X, \pi^{-1}(\bb))) \big) \right) = n-1.
\]
On the other hand, by Theorem \ref{oldembed}
\[
\dim \left( OC( \pi^{-1}(\bb)) \big( HH_\bullet (\Fuk^\flat_{\pi^{-1}(\LL)} (\tilde{X}, \pi^{-1}(\bb)))\big) \right) = 
 \dim QH^\bullet( X, \bb + q^{-\epsilon} p).
\]
\label{missingtilde2}
The claim now follows.
\end{proof}

\begin{proof}[Proof of Corollary \ref{under}] \label{underproof}
Equation \eqref{dpi} is an immediate consequence of the split generation statement and the fact that the exceptional and   unexceptional Lagrangians are disjoint.  For the splitting of quantum cohomology, consider the decomposition of the quantum cohomology $QH^\bullet( \tilde{X}, \pi^{-1}(\bb))$ according to subspaces generated the collections $MC(\pi^{-1}(\LL))$ and ${\mf E}$. By proof of Theorem \ref{gen} above, these have orthogonal images with trivial intersection. Hence we have
\[
QH^\bullet( \tilde X, \pi^{-1}(\bb)) \cong QH^\bullet_{\pi^{-1}({\mc L})}(\tilde X, \pi^{-1}(\bb)) \oplus QH^\bullet_{\mf E} (\tilde X, \pi^{-1}(\bb))
\]
where the dimension of the second summand is $n-1$. By the calculation of the potential function, the bulk-deformed quantum cohomology $QH^\bullet_{\mf E}(\tilde X, \pi^{-1}(\bb))$ is the direct sum of the $n-1$ generalized eigenspaces of the quantum multiplication by $[\omega]$ corresponding to the eigenvalues equal to the $n-1$ critical values of the potential function $W_{\rm ex}$. Hence 
\[
QH^\bullet_{\mf G}(\tilde X, \pi^{-1}(\bb)) \cong QH^\bullet (\pt)^{\oplus n-1} .
\]
and the second claim of Corollary \ref{under} follows. 
\end{proof}

\subsection{The example of projective spaces}\label{proj_example}

Lastly, we show that there is a nonempty set of examples for which our theorem applies. The argument is an explicit computation for a projective space and the Clifford torus.    For $X = \mb{CP}^n$, we normalize the toric invariant symplectic form $\omega$ such that its integral $\int_{\P^1} v^* \omega$ over the standard generator $v: \P^1 \to \P^n$ of $H_2$ is $1$. Let $L \subset \mb{P}^n$ be the Clifford torus
\[
L \cong (S^1)^n = \left\{[z_0, \ldots, z_n]\ |\ |z_0| = \cdots = |z_n| \right\},
\]
which is the only member of the collection $\LL$. 
The potential function of the Clifford torus is 
computed in  \cite{Cho} and \cite{Cho_Oh}.  For the standard complex structure $J_{\mb{P}^n}$ on $\mb{P}^n$, all holomorphic disks are regular. Hence, there exists a Donaldson hypersurface $D\subset \mb{P}^n$ which intersects
the Maslov index two holomorphic disks transversely. We can then require that the domain-dependent almost complex structures on domains with minimal number of interior markings (which is the case for Maslov index two disks that have minimal areas) actually coincide with $J_{\mb{P}^n}$.  The inductive construction of the coherent system of perturbation data extends to this case. Therefore, under our general framework, the count of Maslov index two disks coincides with the count of the standard Maslov index two disks. For any local system with representation $y$ with corresponding brane ${\bm L}$, one has 
\[
m_0(1) = q^{\frac{1}{n+1}} \left( y_1 + \cdots + y_n + \frac{1}{y_1 \cdots y_n}\right) 1_{\bm L}^\blackt=: W(y) 1_{\bm L}^\blackt.
\]
One can also verify that 
\[
m_k \big( 1_{\bm L}^\greyt, \ldots, 1_{\bm L}^\greyt \big) = \left\{ \begin{array}{cc} 1_{\bm L}^\circt - 1_{\bm L}^\blackt,\ &\ k = 1 \\                                                                    0   ,\ &\ k \geq 2. \end{array} \right.
\]
Hence
\[
\sum_{k\geq 0} m_k\big( W(y) 1_{\bm L}^\greyt, \ldots, W(y) 1_{\bm L}^\greyt \big) = m_0(1) + W(y) m_1( 1_{\bm L}^\greyt) = W(y) 1_{\bm L}^\circt.
\]
Therefore, we obtain a distinguished weakly bounding cochain $b_{y} = W(y) 1_{\bm L}^\greyt$ for ${\bm L}$. 

To compute the Floer cohomology, note that  the critical points of the potential function are
\begin{equation}\label{Clifford_torus_brane}
y_{(k)} = (\varsigma_{(k)}, \ldots, \varsigma_{(k)} ) = \left( \exp \left( \frac{2k \pi \sqrt{-1}}{n+1}\right), \ldots, \exp \left( \frac{2k\pi \sqrt{-1}}{n+1} \right) \right),\ 0 \leq k \leq n.
\end{equation}
We choose the following set of weakly unobstructed branes:
\[
{\mf L}:= \big\{ \WB_{(k)} = (L, y_{(k)}, b_{y_{(k)}})\ |\ k = 0, \ldots, n \big\}.
\]

For each bulk deformation $\bb$, consider the flat $A_\infty$ category $\Fuk^\flat_{\mf L}(X, \bb)$ with the above $n+1$ objects. Recall that by the definition of Hochschild homology there is a linear map 
\[
\bigoplus_{k=0}^n HF^\bullet( \WB_{(k)}, \WB_{(k)}) \to HH_\bullet( \Fuk^\flat_{\mf L}(X, \bb)).
\]

\begin{lemma}
$HH_\bullet( \Fuk^\flat_{\mf L}(\mb{CP}^n))$ is $(n+1)$-dimensional with a basis given by the images of $[\pt_{(k)}] \in HF^\bullet (\WB_{(k)}, \WB_{(k)})$.
\end{lemma}

\begin{proof}
The argument is similar to that for Theorem \ref{thm63}. Indeed, one can identify the Floer cohomology $HF^\bullet(\WB_{(k)}, \WB_{(k)})$ as the Clifford algebra associated to the (non-degenerate) Hessian of $W$ at the $k$-th critical point $y_{(k)}$. 
\end{proof}

\begin{theorem} \label{ctorus} 
Let $\bb = 0$ be the trivial bulk deformation. Let $H\in H^2(\mb{CP}^n, {\mb Z})$ be the hyperplane class. We have the following: 
\begin{enumerate}

\item \label{itemb} The matrix of the open-closed map 
\[
[OC(0)]: HH_\bullet( \Fuk^\flat_{\mf L}(\mb{CP}^n)) \to QH^\bullet( \mb{CP}^n)
\]
equals
$T(q^{\frac{1}{n+1}} \varsigma_{(0)}, \ldots, q^{\frac{1}{n+1}}
\varsigma_{(n)} )$ plus a higher order term with respect to the basis
$[\pt_{(0)}], \ldots, [\pt_{(n)}]$ of the Hochschild homology and the
basis $1, H, \ldots, H^n$ of the quantum cohomology. Here $T$ is the
Vandermonde matrix from \eqref{eq:vdm} and $\varsigma_{(i)}$ are the
 $(n+1)$-th roots of unity from \eqref{Clifford_torus_brane}. 
In particular, the open-closed map is an isomorphism.

\item \label{itemc} The closed-open maps $CO_{0, \WB_{(k)}}: QH^{\bullet} (\mb{CP}^n ) \to HF^\bullet ( \WB_{(k)}, \WB_{(k)})$ are given
  by
\[
H^l \mapsto q^{\frac{n-l}{n+1}} \varsigma_{(k)}^{n-l} [1_{\WB_{(k)}}], \quad k, l =   0,\ldots, n
\]
where $[1_{\WB_{(k)}}]\in HF^\bullet( \WB_{(k)}, \WB_{(k)})$ is the identity element.
\end{enumerate} 
\end{theorem}

\begin{proof} 
We choose a particular Morse-Smale pair on $\mb{CP}^n$ to simplify the computation. The function 
\[
f(z_0, \ldots, z_n) = \sum_{i=0}^n a_i |z_i|^2
\]
for $a_0 > a_1 > \cdots > a_n$ descends to a Morse function on $\mb{CP}^n$ whose critical points are the toric fixed points. The closure of the unstable manifolds are the cycles
\[
Z_k = \{ [z_0, \ldots, z_k, 0, \ldots, 0] \in \mb{CP}^n \},\ k = 0, \ldots, n
\]
which represents the classes $1, H, H^2, \ldots, H^n$. 

We give an explicit computation of the open-closed map using the Blaschke classification.  Recall by Proposition \ref{cliffprop} that the branes $\WB_{(l)}$ have different values of the potential for distinct $l$, and so $\Hom (\WB_{(l)}, \WB_{(m)}) = 0$ by definition for $l,m$ distinct.
To prove \eqref{itemb}, recall from Proposition \ref{cliffprop} that the Floer cohomology $HF^\bullet ( \WB_{(m)}, \WB_{(m)})$ a non-degenerate Clifford algebra corresponding to the Hessian $\partial^a \partial^b W(y)$ of the potential $W(y)$.  By Corollary \ref{oned} the Hochschild homology $HH_\bullet( HF^\bullet (\WB_{(m)}, \WB_{(m)}))$ 
has a single  generator, which must be the the point class $[\pt] \in HF^n (\WB_{(m)}, \WB_{(m)})$ since its image under the  open-closed map is non-trivial.  Via the Blaschke classification \eqref{blaschke} there is a unique disk $u: \D \to X$ of Maslov index $I(u) = 2k$ with an interior point $z \in \D$ mapping to $Z_k$ and boundary on $L$.  Identify $L \cong (S^1)^n$ via the local model and let
\[
\gamma_1 ,\ldots, \gamma_n \in H_1((S^1)^n)
\]
be the standard set of generators for $H_1((S^1)^n)$.  By \eqref{blaschke} again, \label{exactly} the contributions in the open-closed map $OC(0)$ arise from disks 
    \[ u: C \to X, \quad u(z_e) \subset Z_k \]
    with a single point constraint $z_e$ in the interior of $C$.  It
    follows that the open-closed map $OC(0)$ is given as a function of
    the representation defined by the local system $y$ on the point class $[\pt] \in HF^\bullet ( \WB_{(k)}, \WB_{(k)})$
    by
    \begin{equation} \label{oftheform} [OC(0)]([\pt]) =
      (1,y(\gamma_1), y(\gamma_1 \gamma_2), \ldots , y(\gamma_1 \ldots
      \gamma_n) ) \end{equation}
  As a result, the point class in the brane with representation defined by local system $y_{(k)}$ is mapped under
  the open-closed map $OC(0)$ to
\[
[Z_0] + q^{1/(n+1)} \varsigma [Z_1] + q^{2/(n+1)} \varsigma^2 [Z_2] + \ldots + q^{n/(n+1)} \varsigma^n[Z_n].
\]
In the basis given by $[Z_0],\ldots, [Z_{n}]$ the open-closed map has the matrix as claimed. 

For \eqref{itemc} note that for each cycle $\P^\ell$, the Blaschke products mapping $0$ to $\P^\ell$ with index $2(n-\ell)$ are those
with the first $n-\ell$ components
\[
(u_1,\ldots, u_{n-\ell})(z) = u: \D \to \C^{n+1}, \quad z \mapsto \left( \zeta_i  \frac{ z - a}{1 - z \ol{a}} \right)_{i = 1,\ldots,n-\ell} .
\]
are non-vanishing with a common root at some $a \in \D$.  Hence the moduli space of holomorphic disks bounding $L_\epsilon$ with one interior input labelled by $Z_l$ and one boundary output is non-empty only if the output is a point constraint. In the case of a point constraint there is a single disk with an interior point mapping to $Z_k$ and the contribution is $y_{(k),1} \ldots y_{(k),n-\ell} 1_{\WB_{(k)}} \in HF( {\bm L}_{(k)}, {\bm L}_{(k)})$.
\end{proof}

\begin{remark}
The closed-open map is a ring homomorphism as predicted by Theorem \ref{homo}. For example, in quantum cohomology we have $[\P^{n-1}]^{n+1} = q$ while in Hochschild cohomology 
\[
( CO_{0,\WB_{(k)}}( [\P^{n-1}] )^{n+1} =     (q^{1/(n+1)} y_{(k),1})^{n+1} = q
\]
for any of the branes ${\bm L}_{(k)}$ in question.  
\end{remark}

\subsubsection{Point bulk deformations}

We extend the above calculation to the bulk deformed case considered in this paper. First we know from Corollary \ref{corollary_fukaya_blowup} that the bulk-deformed curved Fukaya category $\Fuk^\sim_{\mc L}(\mb{CP}^n, q^{-\epsilon} p)$ is well-defined when $\epsilon$ is sufficiently small. Its homotopy equivalence class is also independent of the choice of the point $p$. Now we recalculate the potential function for the choice $p = [0,\ldots, 0, 1]$. Suppose we have a rigid holomorphic disk $u: \mb{D} \to \mb{CP}^n$ passing through $p$ at interior markings $l$ times. By the Blaschke classification of holomorphic disks, we know that the Maslov index $\mu(u)$ of $u$ is at least $2ln$.  The dimension of the moduli space of such marked disks with $1$ boundary marking is 
\[
n + \mu(u) - 2 + 2l - 2ln \geq n + 2l -2.
\]
In order to contribute to $m_0(1)$, the dimension is at most $n$. Hence $l = 0, 1$. While the $l = 0$ case corresponds to the original calculation ....., when $l = 1$, there is exactly one Maslov $2n$ disk passing through $p$ in the interior and passing through a fixed point on $L$. Hence in this case  
\[
m_0(1) = \left( q^{\frac{1}{n+1}} \left( y_1 + \cdots + y_n + \frac{1}{y_1 \cdots y_n} \right) + q^{\frac{n}{n+1} - \epsilon} y_1 \cdots y_n \right) 1_{\bm L}^\blackt=: W_\epsilon(y) 1_{\bm L}^\blackt. 
\]
(Note that $n \geq 2$.) When $\epsilon$ is small, $m_0(1)$ is a higher order perturbation of $W(y)$. By the non-degeneracy of the Hessian of $W$ at critical points, there are exactly $n+1$ critical points 
\[
y_{\epsilon, (k)} = (\varsigma_{\epsilon, (k)}, \ldots, \varsigma_{\epsilon, (k)})
\]
where $\varsigma_{\epsilon, (k)} \in \Lambda$ is a solution to 
\[
\frac{1}{x^n} - q^{\frac{n-1}{n+1} - \epsilon} x^n = x
\]
as a higher order perturbation of $\varsigma_{(k)}$. For the branes $\LB_{\epsilon, (k)}$ corresponding to the local systems $y_{\epsilon, (k)}$, one has a canonical weakly bounding cochain
\[
b_{\epsilon, (k)}:= W_\epsilon( y_{\epsilon, (k)}) 1_{{\bm L}}^\greyt.
\]
Take the corresponding weakly unobstructed branes, one obtains the flat $A_\infty$ category
\[
\Fuk^\flat_{{\mf L}_\epsilon}(\mb{CP}^n, q^{-\epsilon}p).
\]

\begin{corollary}\label{piso} 
For any sufficiently small $\epsilon$, the bulk-deformed open-closed map 
\[
[OC(q^{-\eps} p)]: HH_\bullet (\Fuk^\flat_{\mf L}(\mb{CP}^n, q^{-\epsilon} p)) \to H^\bullet( \mb{CP}^n) 
\]
is a linear isomorphism.
\end{corollary}

\begin{proof} 
We compute the disks with bulk insertions at a point. We still take $p = [ 0,  0,\ldots,0, 1]$. As $[OC(q^{-\epsilon} p)]$ in the direction of $1 = {\rm PD}([Z_n])$) is always only the classical contribution (see Lemma \ref{oc_unit}), we only need to compute in the directions of $H, H^2, \ldots, H^n$, whose representatives $Z_{n-1}$, $\ldots$, $Z_0$ are all disjoint from $p$. The requirement that the disk passes through $p$ forces $n$ additional roots in the Blaschke product \eqref{blaschke}.  So the total number of roots in any Blaschke disk with at least one point constraint at $p$ that also contribute to the open-closed map in the directions of $Z_{n-1}, \ldots, Z_0$, is at least $n$. It follows the matrix of $[OC(q^{-\epsilon} q)]$ is that of $[OC(0)]$ plus terms with $q$-valuation at least $\frac{n}{n+1} - \eps$ or greater. Hence $[OC(q^{-\eps} p)]$ is still an isomorphism. \end{proof}

\appendix

\section{Partly-local domain-dependent almost complex structures} \label{sec:app}

In this appendix, we fill a gap pointed out by Nick Sheridan in the proof of independence of genus zero Gromov-Witten invariants from the choice of divisor in the Cieliebak-Mohnke perturbation scheme
 \cite{Cieliebak_Mohnke}.  We then use the same argument to show that the Fukaya category defined using stabilizing divisors is independent of the choice of stabilizing divisor.

\subsection{Independence of Gromov-Witten invariants}

The proof of independence of genus zero Gromov-Witten invariants of a rational symplectic manifold $X$ from the choice of
Donaldson hypersurfaces in \cite[8.18]{Cieliebak_Mohnke} depends on the
construction of a parametrized moduli space for the following
situation: Given a type $\Gamma$ of stable marked curve let
$\ol{\cU}_\Gamma \to \ol{\cM}_\Gamma$ denote the universal curve over
the compactified moduli space $\ol{\cM}_\Gamma$ of curves of type
$\Gamma$.  Let $\J_\tau(X,\omega)$ denote the space of $\omega$-tamed
almost complex structures on the given symplectic manifold
$(X,\omega)$ with rational symplectic class
$[\omega] \in H^2(X,\omega)$.  A {\it domain-dependent almost complex
  structure} is a map
\[ J_\Gamma: \ol{\cU}_\Gamma \to J_\tau(X,\omega) .\]
Associated to a coherent collection of sufficiently generic choices
$\ul{J} = (J_\Gamma)$ is a Gromov-Witten pseudocycle
$\ol{\M}_{0,n}(X,\beta) \subset X^n$ for each number of markings $n$
and each class $\beta \in H_2(X)$.

Naturally, one wishes to show that the resulting pseudocycle is independent, up to cobordism between pseudocycles, from the choice of Donaldson hypersurface.  Suppose that $D', D'' \subset X$ are two
Donaldson hypersurfaces and
$J' = (J'_{\Gamma'}), J'' = (J''_{\Gamma''})$ are two collections of
domain dependent almost complex structures depending on the
intersection points with $D'$ resp. $D''$, depending on some
combinatorial type $\Gamma'$ resp. $\Gamma''$.  Consider the pullback
\[ (f'')^* J'_{\Gamma'}, \ (f')^* J''_{\Gamma''}: 
\ol{\U}_\Gamma \to \J(X,D',D'') \] 
to a common universal curve $\ol{\U}_\Gamma$ for some type $\Gamma$
recording both sets of markings (so that if $\Gamma'$ resp. $\Gamma''$
has $n'$ resp. $n''$ leaves then $\Gamma$ has $n'+n''$ leaves).  One
wishes to construct a homotopy between
$(f'')^* J'_{\Gamma'}, \ (f')^* J''_{\Gamma''}$ to construct a
cobordism between the corresponding pseudocycles $\ol{\M}'_n(X,\beta)$ and
$\ol{\M}_n''(X,\beta)$ .  Unfortunately, as pointed out by Nick
Sheridan, the pullbacks
$(f'')^* J'_{\Gamma'}, \ (f')^* J''_{\Gamma''}$ do not satisfy the
locality condition used to show compactness. 
That is, the restriction
of the almost complex structures $(f')^* J_{\Gamma''}''$ (or
$(f'')^* J_{\Gamma'}'$) to some irreducible component $S_v$ of the
domain curve $C$ are not independent of markings on other components
$S_{v'} \neq S_v$, because collapsed components $S_v$ may map to
non-special points $f'(S_v) = \{ w \} \in f'(C)$ under the forgetful
map $f'$.

In this appendix we modify the definition of the locality on the collapsed
components so that one may homotope between the two domain-dependent
almost complex structures without losing compactness. 
Instead of directly homotoping between the given pull-backs, one first
homotopes each pullback to an almost complex structure that is equal
to a base almost complex structure near any special point.

\subsection{Partly local perturbations}

We introduce the following notation for stable maps with two types of
markings.  Let $\Gamma$ be a combinatorial type of genus zero stable
curve with $n = n'+n''$ markings.  Let $D',D''$ be Donaldson
hypersurfaces in the symplectic manifold $(X,\omega)$, that is,
symplectic hypersurfaces representing large multiples $k' [\omega]$
resp.  $k'' [\omega]$ of the symplectic class
$[\omega] \in H^2(X,\Q)$.  Suppose $D'$ and $D''$ intersect
transversely.  Let $\mathcal{J} (X, D', D'')$ be the space of
$\omega$-tamed almost complex structures on $X$ that make $D'$ and
$D''$ almost complex. Let
$\J^E(X,D',D'') \subset \mathcal{J} (X, D', D'')$ be some contractible
subset of almost complex structures $J: TX \to TX$ preserving $TD'$
and $TD''$ taming the symplectic form $\omega$ and so that any
non-constant pseudoholomorphic $J$-holomorphic map $u: C \to X$ with
some given energy bound $E(u) < E$ to $X$ meets $D',D''$ each in at
least three but finitely many distinct points
$u^{-1}(D'), u^{-1}(D'')$ in the domain $C$ as in
\cite[8.18]{Cieliebak_Mohnke}.  Let
\[ J_{D',D''} \in \bigcap_E  \J^E(X,D',D'') \] 
be a base almost complex structure that satisfies these conditions
without restriction on the energy of the map $u: C \to X$.  

The universal curve breaks into irreducible components corresponding
to the vertices of the combinatorial type.  Let
$\ol{\cU}_\Gamma \to \ol{\cM}_\Gamma$ be the closure of the universal
curve of type $\Gamma$.  For each vertex $v \in \Ver(\Gamma)$ let
$\Gamma(v)$ denote the tree with the single vertex $v$ and edges those
of $\Gamma$ meeting $v$.  Let
$\ol{\cU}_{\Gamma,v} \subset \ol{\cU}_\Gamma$ be the component
corresponding to $v$, obtained by pulling back $\ol{\cU}_{\Gamma(v)}$
so that $\cU_\Gamma$ is obtained from the disjoint union of the curves
$\cU_{\Gamma,v} \to \cM_\Gamma$ by identifying at nodes.

Cieliebak-Mohnke \cite{Cieliebak_Mohnke} requires that the almost complex
structure is equal to the base almost complex structure near the
nodes.  This condition is not true for domain-dependent almost complex
structures pulled back under forgetful maps, and so must be relaxed as
follows.  Recall that Knudsen's (genus zero) universal curve
$\ol{\cU}_\Gamma$ \cite{Knudsen_1983} is a smooth projective variety, and in
particular a complex manifold.  A {\it domain-dependent almost complex
  structure} for type $\Gamma$ of stable genus zero curve is an almost
complex structure
\[ J_\Gamma:  T(\ol{\cU}_\Gamma \times X ) \to 
T(\ol{\cU}_\Gamma \times X)  \]  
that preserves the splitting of the tangent bundle
$T(\ol{\cU}_\Gamma \times X)$ into factors
$T\ol{\cU}_\Gamma \times TX$ and that is equal to the standard complex
structure on the tangent space to the projective variety
$\ol{\cU}_{\Gamma}$, and gives rise to a map from $\ol{\cU}_\Gamma$ to
$\J(X,D',D'')$ with the same notation $J_\Gamma$.  Let
\[ \J_\Gamma^E(X,D',D'') \subset \Map(\ol{\cU}_\Gamma,
\J^E(X,D',D'')) \]
denote the space of such maps taking values in $\J^E(X,D',D'')$.  With
this definition, the standard proof of Gromov convergence applies: Any
sequence $u_\nu: C_\nu \to X$ of $J_\Gamma$-holomorphic maps with
energy $E(u) <E $ may be viewed as a finite energy sequence of maps to
$\ol{\cU}_\Gamma \times X$.  Therefore it has a subsequence with a
Gromov limit $u: C \to X$ where the stabilization $C^s$ of $C$ is a
fiber of $\ol{\cU}_\Gamma$ and $u$ is pseudoholomorphic for the
pull-back of the restriction of $J_\Gamma$ to $C^s$.  If we restrict
to sequences of maps $u_\nu: C_\nu \to X$ sending the markings to $D'$
or $D''$ then in fact $C^s$ is equal to $C$, since non-constant
components of $u$ with fewer than three markings are impossible.

We distinguish components of the curve that are collapsed under
forgetting the first or second group of markings.  Let
\[ f': \ol{\cU}_{\Gamma} \to \ol{\cU}_{\Gamma''}, \quad 
f'':\ol{\cU}_{\Gamma} \to \ol{\cU}_{\Gamma'} \]  
denote the forgetful maps forgetting the first $n'$ resp. last $n''$
markings and stabilizing.  Call a component of $C$ {\it $f'$-unstable}
if it is collapsed by $f'$, and {\it $f'$-stable} otherwise, in which
case it corresponds to a component of $f'(C)$.  $f''$-unstable
components are defined similarly.

\begin{definition}\label{plocal} 
{\rm (Local and partly local almost complex
    structures) } 
\begin{enumerate} 
 \item  A domain-dependent almost complex structure
\[ J_\Gamma: \ol{\cU}_\Gamma \to \J(X,D',D'') \]
is {\it local} if and only if for each $v \in \Ver(\Gamma)$ the restriction
$J_\Gamma | \ol{\cU}_{\Gamma,v}$ is local in the sense that
$J_\Gamma | \ol{\cU}_{\Gamma,v}$ is pulled back from some map
$J_{\Gamma,v}$ defined on the universal curve $\ol{\cU}_{\Gamma(v)}$
and equal to $J_{ D', D''}$ near any special point of
$\ol{\cU}_{\Gamma,v}$.
\item A domain-dependent almost complex structure
 \[ J_\Gamma: \ol{\cU}_\Gamma \to \J(X,D',D'') \]
is {(\it$f'$-local} if and only if 
\begin{enumerate} 
\item for each $v \in \Ver(\Gamma)$ such that $\cU_{\Gamma,v}$ is
  $f'$-stable (that is, has sufficiently many $D''$ markings) then
  $J_\Gamma | \ol{\cU}_{\Gamma,v}$ is local in the sense that
  $J_\Gamma | \ol{\cU}_{\Gamma,v}$ is pulled back from some map
  $J_{\Gamma,v}$ defined on the universal curve $\ol{\cU}_{\Gamma(v)}$
  and equal to $J_{D',D''}$ near any point $z \in C$ mapping to a
  special point $f'(z)$ of $f'(C)$, and
\item for each $v \in \Ver(\Gamma)$ such that $\cU_{\Gamma,v}$ is
  $f'$-unstable (that is, does not have sufficiently many $D''$
  markings) then $J_\Gamma | \ol{\cU}_{\Gamma,v}$ is constant on each fiber of $\ol{\cU}_{\Gamma,v}\to {\mc M}_\Gamma$.
\end{enumerate} 
The definition of $f''$-local is similar.  
In either case, 
we say that $J_\Gamma$ is {\it partly-local.}
\end{enumerate} 
\end{definition}

\begin{remark} \label{weaker} Note that $f'$-pullbacks
  $ (f')^* J_{\Gamma''}$ are $f'$-local, and local almost complex
  structures are $f'$-local.  The condition that an almost complex
  structure be $f'$-local is weaker than the condition that it be
  pulled back under $f'$, because the restriction
  $J_\Gamma | \ol{\cU}_{\Gamma,v}$ is allowed to depend on special
  points $z \in \ol{\cU}_{\Gamma,v}$ that are forgotten under
  $f'$. \end{remark}

\begin{remark} \label{pullback} One can reformulate the $f'$-local
  condition as a pullback condition for a forgetful map that forgets
  almost the same markings as those forgotten by $f'$.  Let $C$ be a
  curve of type $\Gamma$.  Let $C^{\us} \subset C$ be the locus
  collapsed by $f'$.  For each connected component
  $C_i, i = 1,\ldots, k$ of $C^{\us}$ mapping to a marking of $f'(C)$
  choose $j(i)$ so that $z_{j(i)} \in C_i$.  Let
  $I^{\us} \subset \{ 1, \ldots, n \}$ denote the set of indices $j$
  of markings $z_j \in C^{\us}$ with $z_j \neq z_{j(i)}, \forall i$.
  Forgetting the markings with indices in $I^{\us}$ and collapsing
  defines a map $f: C \to f(C)$ such that any collapsed component of
  $C$ maps to a special point of $f(C)$.  Let $\Gamma^f$ denote the
  combinatorial type of $f(C)$.  Then $J_\Gamma$ is $f'$-local if and
  only if $J_\Gamma = f^* J_{\Gamma^f}$ is pulled back from a local
  domain-dependent almost complex structure
  $J_{\Gamma^f}: \ol{\cU}_{\Gamma^f} \to \J(X,D',D'')$.  Indeed, the
  collapsed components under $C \to f(C)$ are the same as those of
  $f': C \to f'(C)$ since adding a single marking $z_i$ on the
  components that collapse $S_v$ to markings $f(S_v) \subset f(C)$
  does not stabilize $S_v$.  So the pull-back condition
  $J_\Gamma = f^* J_{\Gamma^f}$ requires $J_\Gamma$ to be constant on
  the components $S_v$ such that $\dim(f(S_v)) = 0$.  On the other
  hand, any irreducible component of $f(C)$ is isomorphic, as a stable
  marked curve, to an irreducible component of $C$ not collapsed under
  $f'$.
\end{remark} 

\begin{remark} \label{both} There also exist domain-dependent almost
  complex structures that are both $f'$ and $f''$-local.  Indeed,
  suppose that $C$ is a curve of type $\Gamma$, and
  $K \subset \{1,\ldots, n' + n'' \} $ is the set of markings on
  components collapsed by $f'$ or $f''$.  Forgetting the markings
  $z_k, k \in K$ defines a forgetful map
  $f^{{\rm s}{\rm s}}: C \to f^{{\rm s}{\rm s}}(C)$, where $f^{{\rm s}{\rm s}} (C)$ is of some
  (possibly empty) type $\Gamma^{\rm ss}$.  Let
  $J_{\Gamma^{\rm ss}}: \ol{\cU}_{\Gamma^{\rm ss}} \to \J(X,D',D'')$ be a
  domain-dependent almost complex structure for type $\Gamma^{\rm ss}$.
  Then $(f^{\rm ss})^* J_{\Gamma^{\rm ss}}$ is both $f'$ and $f''$-local
  (taking the constant structure $J_{D',D''}$ if $\Gamma^{\rm ss}$ is
  empty.)
\end{remark}

\begin{lemma} \label{extends} The space of $f'$-local
  resp. $f''$-local resp. $f'$ and $f''$-local almost complex
  structures tamed by or compatible with the symplectic form $\omega$
  is contractible.  Any $f'$-local resp. $f''$-local resp. $f'$ and
  $f''$-local $J_\Gamma | \partial \ol{\U}_\Gamma$ defined on the
  boundary
  $\partial \ol{\U}_\Gamma := \ol{\U}_\Gamma | \partial
  \ol{\M}_\Gamma$
  extends to a $f'$-local resp. $f''$-local resp. $f'$ and $f''$-local
  structure $J_\Gamma$ over an open neighborhood of the boundary
  $\partial \ol{\U}_\Gamma$ in $\ol{\U}_\Gamma$.
\end{lemma} 

\begin{proof} Contractibility follows from the contractibility of
  tamed or compatible almost complex structures.  Since the space of
  $f'$-local tamed almost complex structures is contractible, it
  suffices to show the existence of an extension of $J_\Gamma$ near
  any stratum $\ol{\U}_{\Gamma_1} \subset \ol{\U}_\Gamma$ and then
  patch together the extensions.  {Local domain-dependent almost
  complex structures $J_\Gamma$ extend by a gluing construction in
  which open balls $U_+, U_-$ around a node are replaced by a
  punctured ball $V \cong U_+^\times \cong U_-^\times$ on which the
  almost complex structure is equal to the base almost complex
  structure $J_{D', D''}$.}

  In the partly-local case, recall from Remark \ref{pullback} that
  $J_\Gamma$ is the pull-back of a local almost complex structure
  $J_{\Gamma^f}$ near any particular fiber of the universal curve.
  Define an extension of $J_\Gamma$ near curves of type $\Gamma_1$ by
  first extending $J_{\Gamma^f}$ and then pulling back.  In more
    detail, let $C$ be such a curve and let $C_1,\ldots, C_k$ denote
    the connected components of $C$ collapsed by $f'$ to a non-special
    point of $f'(C)$.  Choose a marking $z_i \in C_i$ and let
    $\Gamma^{\s}$ resp. $\Gamma_1^{\s}$ denote the type obtained from
    $\Gamma$ resp. $\Gamma_1$ by forgetting all markings on $C_i$
    except $z_i$, for each $i = 1,\ldots, k$.  Consider the forgetful
    map
  \[ f: \ol{\cU}_{\Gamma} \to \ol{\cU}_{\Gamma^f} \]
  that forgets all but the marking $z_i$ on $C_i$.  As discussed
in Remark \ref{pullback}
  $J_{\Gamma_1}$ is  the pullback of a complex structure
  \[ J_{\Gamma_1^f}: \ol{\cU}_{\Gamma_1^f} \to \J(X,D',D'') . \]
  Since the complex structure $J_{\Gamma_1^f}$ is constant equal to
  the base almost complex structure $J_{D',D''}$ near the nodes (which
  must join non-collapsed components) $J_{\Gamma_1^f}$ naturally
  extends to a domain-dependent almost complex structure
  $J_{\Gamma^f}$ on a neighborhood $\cN_{\Gamma_1^f}$ of
  $\ol{\cU}_{\Gamma_1^f}$ in $\ol{\cU}_{\Gamma^f}$ by taking
  $J_{\Gamma^f}$ to equal $J_{D',D''}$ near the nodes.  Now take
  $J_\Gamma = f^* J_{\Gamma^f}$ to obtain an extension of $J_\Gamma$
  from $\cU_{\Gamma_1}$ to a neighborhood $f^{-1}(\cN_{\Gamma_1^f})$.
  The proof for $f'$ local or $f'$ and $f''$-local structures is
  similar.
\end{proof}

\subsection{Transversality}
 
We wish to inductively construct partly-local almost complex
structures so that the moduli spaces of stable maps define
pseudocycles. Recall that the combinatorial type of a stable map is
obtained from the type of stable curve by decorating the vertices with
homology classes; we also wish to record the intersection
multiplicities with the Donaldson hypersurfaces.  More precisely, a
{\it type} of stable map $u$ from $C$ to $(X,D',D'')$ consists of a
type $\Gamma$ the stable curve $C$ (the graph with vertices
corresponding to components and edges corresponding to markings and
nodes) with the labelling of vertices $v \in \Ver(\Gamma)$ by homology
class $d(v) = [u |S_v] \in H_2(X)$, labelling of the semi-infinite
edges $e$ by either $D'$ or $D''$,\footnote{To obtain evaluation maps
  one should allow additional edges, but here we ignore evaluation
  maps.} and by the intersection multiplicities $m'(e), m''(e)$ with
$D'$ and $D''$ (possibly zero if the corresponding marking does not
map to $D'$ or $D''$.  A stable map is {\it adapted} of type $\Gamma$
if each connected component of $u^{-1}(D')$ resp.  $u^{-1}(D'')$
contains at least one marking $z_e$ corresponding to an edge $e$
labelled $D'$ resp. $D''$, and each marking $z_e$ maps to $D'$ or
$D''$ depending on its label.  A stable map is {\it adapted} of type
$\Gamma$ if
\begin{enumerate}
\item each connected component of $u^{-1}(D')$ resp. $u^{-1}(D'')$ contains at least one marking $z_e$ corresponding to an edge $e$ with labelling $m'(e) \geq 1$ resp. $m''(e) \geq 1$, and

\item if $m'(e) \geq 1$ resp. $m''(e) \geq 1$, then the marking $z_e$ is mapped to $D'$ resp. $D''$.
\end{enumerate}

By forgetting the extra data and stabilization one can associate to
each type of stable maps to a type of stable curves. In notation we do
not distinguish the two notions of types.  Given a type of stable map
$\Gamma$ choose a domain-dependent almost complex structure
$J_\Gamma$.  Denote by $\M_\Gamma(X,J_\Gamma)$ the moduli space of
adapted $J_\Gamma$-holomorphic stable maps $u: C \to X$ of type
$\Gamma$, such that for each $v \in {\rm Vert}(\Gamma)$ with
$d(v) \neq 0$, the image of $u_v$ is not contained in $D' \cup D''$,
and for each semi-infinite edge $e$ attached to $v$, the local
intersection number of $u_v$ with $D'$ resp. $D''$ at $z_e$ is equal
to $m'(e)$ resp. $m''(e)$.   {The moduli space $\M_\Gamma(X,J_\Gamma)$
is locally cut out by a smooth map of Banach manifolds: Given a local
trivialization of the universal curve given by an subset
$\cM^i_\Gamma \subset \cM_\Gamma$ and a trivialization
$C \times \cM^i_\Gamma \to \cU_\Gamma^i = \cU_\Gamma
|_{\cM^i_\Gamma}$,
we consider the space of maps $\Map(C,X)_{k,p}$ of Sobolev class $k,p$
for $p \ge 2$ satisfying the above constraints and $k$ sufficiently
large to the space of $0,1$-forms with values in $TX$ given by the
Cauchy-Riemann operator $\olp_{J_\Gamma}$ associated to $J_\Gamma$.
The linearization of this operator is denoted $D_{u}$ (or
$D_{u,J_\Gamma}$ to emphasize dependence on $J_\Gamma$) and the map
$u$ is called {\it regular} if $D_{u}$ is surjective.}  We call a type
$\bGamma$ of stable map $u: C \to X$ {\it crowded} if there is a
maximal ghost subtree of the domain $C_1 \subset C $ with more than
one marking $z_e \in C_1$ and {\it uncrowded} otherwise.  It is not in
general possible to achieve transversality for crowded types using the
Cieliebak-Mohnke perturbation scheme.

\begin{definition} We say a domain-dependent almost complex structure
  $J_\Gamma$ is {\it regular} for a type of map type $\bGamma$ with 
  underlying domain type $\Gamma$ if
  \begin{enumerate}
  \item if $\bGamma$ is uncrowded then every element of the moduli
    space $\M_\Gamma(X,J_\Gamma)$ of adapted $J_\Gamma$-holomorphic
    maps is regular; and
  \item If $\bGamma$ is crowded then there exists a regular
    $J_{\Gamma^{\s}}$ for some uncrowded type $\bGamma^{\s}$ obtained
    by forgetting all but one marking $z_e$ on each maximal ghost
    component for curves of type $\bGamma$ such that $J_{\Gamma^{\s}}$
    is equal to $J_\Gamma$ on all non-constant components, that is,
    all components of $\ol{\cU}_\Gamma$ on which the maps
    $u: C \to X$ in $\M_\bGamma(X,J_\Gamma)$ are non-constant.
\end{enumerate} 
\end{definition} 

Recall the construction by Floer { \cite[Lemma 5.1]{Floer_unregularized}} of a
subspace of smooth functions with a {separable} Banach space
structure.  Let $\ul{\eps} = ( \eps_\ell, \ell \in \Z_{\ge 0})$ be a
sequence of constants converging to zero.  Let
$\J_\Gamma(X)_{\ul{\eps}}$ denote the space of domain-dependent almost
complex structures of finite Floer norm as in \cite[Section 5]{Floer_unregularized}.
In particular, $\J_\Gamma(X)_{\ul{\eps}}$ allows variations with
arbitrarily small support near any point.

\begin{proposition} 
\begin{enumerate} 
\item \label{c} For a regular domain-dependent almost complex
  structure $J_{\Gamma''}$ the pull-back $ (f')^* J_{\Gamma''}$ is
  regular, and similarly for the pull-back $(f'')^* J_{\Gamma'}$ for
  regular $J_{\Gamma'}$.
\item \label{d} Suppose that $J_\Gamma | \partial \ol{\cU}_\Gamma$ is
  $f'$-local and is a regular domain-dependent almost complex
  structure defined on the boundary
  $\partial \ol{\cU}_\Gamma \to \partial \ol{\M}_\Gamma$.  The set of
  regular $f'$-local extensions is comeager, that is, is the
  intersection of countably many sets with dense interiors.
\item \label{e} Any parametrized-regular homotopy
  $J_{\Gamma,t} | \partial \ol{\cU}_\Gamma$ between two regular
  $f'$-local domain-dependent almost complex structures
  $J_{\Gamma,0}, J_{\Gamma,1}$ on the boundary
  $\partial \ol{\cU}_\Gamma$ may be extended to a parametrized-regular
  one-parameter family of $f'$-local structures $J_{\Gamma,t}$ equal
  to $J_{\Gamma,t}$ over $\ol{\cU}_\Gamma$.
\end{enumerate} 
\end{proposition} 

\begin{proof}
  Item \eqref{c} is immediate from the definition, since any variation
  of $J_{\Gamma''}$ induces a variation of $(f')^* J_{\Gamma''}$.
  \eqref{d} is an application of Sard-Smale applied to a universal
  moduli space.  We sketch the proof, which is analogous to that in
  Cieliebak-Mohnke \cite[Chapter 5]{Cieliebak_Mohnke}.  By Lemma
  \ref{extends}, $J_\Gamma | \partial \ol{\U}_\Gamma$ has an extension
  over the interior.  For transversality, first consider the case of
  an uncrowded type $\bGamma$ of stable map with domain type $\Gamma$.  Choose open subsets
    $L_\Gamma,N_\Gamma \subset \ol{\cU}_\Gamma$ of the boundary resp.
    markings and nodes, such that $L_\Gamma$ is union of fibers of
    $\ol{\cU}_\Gamma$ containing the restriction
    $\ol{\cU}_\Gamma | \partial \M_\Gamma$ and $N_\Gamma$ is
    sufficiently small so that the intersection of the complement of
    $N_\Gamma$ with each component of each fiber of $\cU_\Gamma$ not
    meeting $L_\Gamma$ is non-empty.  Let $\M^{\univ}_\Gamma(X)$
    denote the universal moduli space consisting of pairs
    $(u,J_\Gamma)$, where $u: C \to X$ is a $J_\Gamma$-holomorphic map
    of some Sobolev class $W^{k,p}, kp \ge 3, p \ge 2$ on each
    component (with $k$ sufficiently large so that the given vanishing
    order at the Donaldson hypersurfaces $D',D''$ is well-defined).
    Let $\J^E_\Gamma(X,N_\Gamma,S_\Gamma) \subset \J^E_\Gamma(X)$
    denote the space of $J_\Gamma \in \J_\Gamma^E(X)_{\ul{\eps}}$ that
    are $f'$-local domain-dependent almost complex structures that
    agree with $J_{D',D''}$ on the neighborhood $N_\Gamma$ of the
    nodes and markings $z \in \ol{\cU}_\Gamma$ that map to special
    points $f'(z) \in \ol{\cU}_{f'(\Gamma)}$ as in Definition
    \ref{plocal}, and equal to the given extension in the neighborhood
    $L_\Gamma$ of the boundary, and constant on the components
    required by $f'$-locality in Definition \ref{plocal}.  By elliptic
    regularity, $\M^{\univ}_\Gamma(X)$ is independent of the choice of
    Sobolev constants used in its construction.

    The universal moduli space is a smooth Banach manifold by an
    application of the implicit function theorem for Banach manifolds.
    Let $\cU_\Gamma^i \to \cM_\Gamma^i , i = 1,\ldots, m$ be a
    collection of open subsets of the universal curve
    $\cU_\Gamma \to \cM_\Gamma$ on which the universal curve is
    trivialized via diffeomorphisms
    $\cU_\Gamma^i \to \cM_\Gamma^i \times C$.  The space of pairs
    $(u : C \to X, J_\Gamma)$ with $[C] \in \cM_\Gamma^i$, $u$ of type
    $\Gamma$ of class $W^{k,p}$ on each component, and
    $J_\Gamma \in \J^E_\Gamma(X,N_\Gamma,S_\Gamma)$ is a smooth
    separable Banach manifold.  Since we assume that $J_\Gamma$ is
    regular on the boundary $\partial \cU_\Gamma$, an argument using
    Gromov compactness shows that by choosing $L_\Gamma$ sufficiently
    small we may assume that $\ti{D}_{u,J_\Gamma}$ is surjective for
    $[C] \in L_\Gamma$, since regularity is an open condition in the
    Gromov topology \cite[Section 10.7]{McDuff_Salamon_2004}.  Let
    $\ti{D}_{u,J_\Gamma}$ the linearization of
    $(u,J_\Gamma) \mapsto \olp_{J_\Gamma} u$, and suppose that $\eta$
    lies in the cokernel of $\ti{D}_{u,J_\Gamma}$.  We have
    $D_u^* \eta^s = 0$ where $D_u$ is the usual linearized
    Cauchy-Riemann operator \cite[p. 258]{McDuff_Salamon_2004} for the map; in the
    case of vanishing constraints at the Donaldson hypersurfaces see
    Cieliebak-Mohnke \cite[Lemma 6.6]{Cieliebak_Mohnke}.  By variation of the
    almost complex structure $J_\Gamma$ and unique continuation,
    $\eta$ vanishes on any component on which $u$ is non-constant.  On
    the other hand, for any constant component $u_v: S_v \to X$, the
    linearized Cauchy-Riemann operator $D_{u_v}$ on a trivial bundle
    $u_v^* TX$ is regular with kernel $\ker(D_{u_v})$ the space of
    constant maps $\xi: C_u \to (u_v)^* TX$.  It follows by a standard
    inductive argument that the same holds true for a tree
    $C' = \cup_{v \in V} S_v , \d u |_{C'} = 0 $ of constant
    pseudoholomorphic spheres so the element $\eta$ vanishes on any
    component $S_v \subset C$ on which $u$ is constant.  It follows
    that $\cM^{\univ,i}_\Gamma(X)$ is a smooth Banach manifold.  For a
    comeager subset $\J_\Gamma^{\reg}(X) \subset \J_\Gamma(X)$ of
    partly almost complex structures in the space above, the moduli
    spaces $\cM^i_\Gamma(X) = \cM_\Gamma(X) |_{\cM^i_\Gamma}$ are
    transversely cut out for each $i = 1,\ldots, m$.  The transition
    maps between the local trivializations
    $\cM_\Gamma^i \cap \cM_\Gamma^j \to \Aut(C)$ induce smooth maps
    $\cM_\Gamma^{i}(X) |_{\cM_{\Gamma}^i \cap \cM_\Gamma^j} \to
    \cM_\Gamma^{j}(X) |_{\cM_{\Gamma}^i \cap \cM_\Gamma^j}$
    making $\cM_\Gamma(X)$ into a smooth manifold.

  Next, consider a crowded type $\bGamma$ with domain type $\Gamma$.  Let 
  $f: \Gamma \to f(\Gamma)$ be a map forgetting all but one marking on
  each maximal ghost component $C' \subset C$ and stabilizing; the
  multiplicities $m'(e),m''(e)$ at any marking $z_e$ is the sum of the
  multiplicities of markings in its preimage $f^{-1}(z_e)$.  Define
  $J_{\Gamma^f}$ as follows.
\begin{enumerate} 
\item If $\cU_{\Gamma^f,v} \cong \cU_{\Gamma,v}$ let
  $J_{\Gamma^f} | \cU_{\Gamma^f,v}$ be equal to
  $J_{\Gamma} | \cU_{\Gamma,v}$.
\item Otherwise let
  $J_{\Gamma^f} : \cU_{\Gamma^f,v}\to \J^E(X,D',D'')$ be constant
  equal to $J_{D',D''}$.
\end{enumerate}
\noindent The map $J_{\Gamma^f}$ is continuous because any non-collapsed ghost
component $S_v \subset C$ must connect at least two non-ghost
components $C_{v_1} , C_{v_2} \subset C$ and the connecting points of
the non-ghost components $f'(C_{v_1}), f'(C_{v_2})$ is a node of the
curve $f'(f(C))$ of type $f'(\Gamma^f)$.  For a comeager subset of
$J_\Gamma$ described above, the complex structures $J_{\Gamma^f}$ are
also regular by the argument for uncrowded types.  Item \eqref{e} is a
parametrized version of \eqref{d}.
\end{proof} 

\begin{corollary} There exists a regular homotopy
  $J_{\Gamma,t}, t \in [-1,1]$ between $ (f'')^* J'_{\Gamma'}$ and
  $ \ (f')^* J''_{\Gamma''}$ in the space of maps
  $\ol{\U}_\Gamma \to \J(X,D',D'')$ that are $f'$-local for
  $t \in [-1,0]$ and $f''$-local for $t \in [0,1]$
\end{corollary}

\begin{proof}  
  Let $\ul{J} = (J_\Gamma)$ be a collection of regular
  domain-dependent almost complex structures that are both $f'$ and
  $f''$-local, 
as in Remark \ref{both}.   By part \eqref{d} above, for each type $\Gamma$ there
  exists a regular homotopy from $J_\Gamma$ to $(f')^* J_{\Gamma''}$
  resp. $(f'')^* J_{\Gamma'}$ extending given homotopies on the
  boundary.  The existence of a regular homotopy now follows by
  induction.
\end{proof}

\subsection{Homotopy invariance of Fukaya categories}

We wish to show, as claimed in Remark \ref{indep}, that 
the \ainfty homotopy type of $\Fuk_{\cL} (X,\bb)$ (as a curved \ainfty algebra with curvature with positive $q$-valuation over the Novikov ring $\Lambda_{\ge 0}$) is independent of the choice of almost complex structures, perturbations,\footnote{We do not use Hamiltonian perturbations of Lagrangians in this paper, so the Fukaya category we use here is defined over $\Lambda_{ \ge   0}$.} stabilizing divisors, and depend only on the isotopy class of bulk deformation. The argument uses a moduli space of {\em quilted disks} with seams labelled by the diagonal, as in \cite[Section 5.5]{flips}.   The hardest part is showing independence
  of the choice of Donaldson hypersurface.  Let $D',D''$ be two Donaldson hypersurfaces that intersect transversely.   Let $C = S \cup T$ be a quilted treed disk of type $\Gamma$.  Each 
  component $S_v, v \in \Ver(\Gamma)$ has some distance
\begin{equation} 
\label{dv} d(v) = \sum_{e \in \Edge(\Gamma)_{v}^{v'} } \ell(e) \in \R 
\end{equation} 
  measuring the sum of the lengths of edges $e \in \Edge(\Gamma)_{v}^{v'}$ 
  to a vertex $v'$ corresponding to a quilted component. 
 Thus $d(v)$ is negative if $v$ comes after the quilted components in order of components starting with the incoming edges,  positive distance if it comes before, and zero distance if $S_v$ is itself quilted.    We now consider
  perturbations $P_\Gamma$ for $D' \cup D''$-adapted maps, with each marking
  labelled by the divisor to which it maps.    The perturbations 
  $P_\Gamma$ are required to satisfy the following properties:
  \begin{enumerate} 
  \item On the components $S_v$ with $d(v) = + \infty$, the perturbation $P_\Gamma$
  is required to be a partly-local perturbation obtained by pull-back of maps forgetting the $D''$-markings of some perturbation scheme $P_{f''(\Gamma})$ for 
  markings mapping to $D'$.
  \item On the components $S_v$ with $d(v) = - \infty$, the perturbation $P_\Gamma$
  is required to be partly-local perturbations obtained by pull-back of maps forgetting the $D'$-markings of some perturbation scheme $P_{f'(\Gamma})$ for 
  markings mapping to $D''$.
  \end{enumerate} 
  One obtains from such a scheme an \ainfty morphism 
  \[ \phi_d: CF(L_{d-1},L_d;D') 
  \otimes \ldots \otimes CF(L_{0},L_1;D')
  \to CF(L_0,L_d;D'')[1-d] \] 
  where the inclusion of $D'$ or $D''$ in the notation indicates which perturbation scheme is being used.   To justify the existence of such a perturbation scheme,
  note that as in the proof of Lemma \ref{extends}, any partly local 
  perturbation scheme may be homotoped to a local one by homotoping 
  to the base almost complex structure on certain components, and the space of local perturbations is contractible.  Similarly, reversing the roles of $D'',D'$
  one obtains an \ainfty morphism 
  \[ \psi_d: CF(L_{d-1},L_d;D'')
  \otimes \ldots CF(L_0,L_1;D'') 
  \to CF(L_0,L_d;D')[1-d] .\] 
  Then an argument using twice-quilted treed disks produces \ainfty homotopies
  from $\psi \circ \phi, \phi \circ \psi$ to the relevant identities.  
  
There are some minor differences between the case of genus zero Gromov--Witten invariants and the case of Fukaya category. For example, the universal curve in the Gromov--Witten case is itself a manifold, a fact which can be used to simplify the description of the space of perturbations. 

\begin{remark} \label{finalhequiv}
We also require a version of homotopy invariance which allows one of the divisors to be stabilizing only for disks in a certain subset, as in the case of the inverse image of a Donaldson hypersurface in the blow-down discussed in the main body of the paper.  We consider the following situation:
Let $U \subset X$ be an open subset disjoint from $L$, $J$ an almost 
complex structure on $X$ 
and $\ti{Z} \subset U$ a $J$-almost complex submanifold  
with the property that
any non-constant holomorphic sphere in $U$ is contained in $\ti{Z}$ and has positive Chern number.  Let 
$D',D''$
be codimension two $J$-almost complex submanifolds with the property that any holomorphic sphere not contained in $\ti{Z}$ meets $D', D''$ in finitely many but at least three points, and any holomorphic disk bounding $L$ meets $D', D''$ at least once.    By perturbing the almost complex structure using domain-dependent perturbations away from $U$ one finds that the moduli spaces of holomorphic disks of expected dimension at most one are regular and define 
Fukaya categories $\Fuk_{\cL}^\sim (X,\bb;D')$ 
and $\Fuk_{\cL}^\sim (X,\bb;D'')$, as in Section 
\ref{excreg}, with compactness as in Lemma \ref{icompact}. The argument above now gives the desired homotopy equivalence. \end{remark}

\bibliographystyle{amsalpha} 
\bibliography{Reference}

\end{document}